\documentclass[11pt]{amsart}
\usepackage[leqno]{amsmath}
\usepackage{amssymb}
\usepackage{amsfonts}
\usepackage{enumerate}
\usepackage{color}
\usepackage{subfigure}
\usepackage{graphicx}
\usepackage{tabularx}
\usepackage[table,dvipsnames]{xcolor}
\usepackage[protrusion=true,expansion=true]{microtype}
\usepackage{hyperref}
\usepackage[all]{xy}
\numberwithin{equation}{section}
\numberwithin{figure}{section}

\newtheorem{theorem}{Theorem}[section]
\newtheorem{remark}[theorem]{Remark}
\newtheorem{lemma}[theorem]{Lemma}
\newtheorem{proposition}[theorem]{Proposition}
\newtheorem{corollary}[theorem]{Corollary}

\newcommand{\E}{\mathbf{E}}

\newcommand{\N}{\mathbf{N}}
\newcommand{\Z}{\mathbf{Z}}
\newcommand{\p}{\mathbf{P}}
\newcommand{\Q}{\mathbf{Q}}
\newcommand{\R}{\mathbf{R}}
\newcommand{\s}{\mathbf{S}}

\newcommand{\CB}{\mathcal {B}}
\newcommand{\CC}{\mathcal {C}}
\newcommand{\CD}{\mathcal {D}}

\newcommand{\CF}{\mathcal {F}}
\newcommand{\CI}{\mathcal {I}}

\newcommand{\CL}{\mathcal {L}}

\newcommand{\CR}{\mathcal {R}}
\newcommand{\CS}{\mathcal {S}}
\newcommand{\CT}{\mathcal {T}}

\newcommand{\CZ}{\mathcal {Z}}

\newcommand{\eps}{\epsilon}

\newcommand{\CG}{\mathcal {G}}
\newcommand{\CH}{\mathcal {H}}

\newcommand{\SLE}{{\rm SLE}}

\newcommand{\diam}{\text{diam}}

\newcommand{\cov}{\text{cov}}

\newcommand{\one}{{\bf 1}}
\newcommand{\wt}{\widetilde}
\newcommand{\wh}{\widehat}
\newcommand{\ol}{\overline}
\newcommand{\ul}{\underline}

\newcommand{\giv}{\,|\,}

\definecolor{slightblue}{rgb}{.8, .8, 1}
\definecolor{hair}{RGB}{100,225,190}
\definecolor{ruby}{RGB}{220,50,120}
\definecolor{grass}{RGB}{150,220,110}
\definecolor{ceruleanblue}{rgb}{0.16, 0.32, 0.75}
\definecolor{deepcarmine}{rgb}{0.66, 0.13, 0.24} 
\definecolor{otterbrown}{rgb}{0.4, 0.26, 0.13}
\definecolor{sapphire}{rgb}{0.03, 0.15, 0.4}
\definecolor{blue2}{rgb}{0.2, 0.2, 0.6}

\renewcommand{\sp}[1]{\mathfrak{#1}}

\newcommand{\gf}{{\mathrm {GF}}}

\newcommand{\X}{{\protect\scalebox{1.5}{$x$}}}

\newcommand{\fb}[3]{B_{#1}^\bullet(#2,#3)}
\newcommand{\dimH}{\dim_{{\mathrm H}}}
\newcommand{\distH}{d_{{\mathrm H}}}
\newcommand{\distHos}{d_{\mathrm H}^1}

\newcommand{\bandlaw}[2]{\mathbf{P}_{{\mathrm{Band}}}^{L=#1,W=#2}}
\newcommand{\slicelaw}[2]{\mathbf{P}_{\mathrm{Slice}}^{L=#1,W=#2}}
\newcommand{\bmlaw}[1]{\mu_{\mathrm{BM}}^{A=#1}}
\newcommand{\bminflaw}{\mu_{\mathrm{BM}}}
\newcommand{\innerboundary}{\partial_{\mathrm {In}}}
\newcommand{\outerboundary}{\partial_{\mathrm {Out}}}
\newcommand{\leb}{\mu_{\mathrm{Leb}}}

\begin{document}

\title[Geodesics in the Brownian map]{Geodesics in the Brownian map: \\ Strong confluence and geometric structure}

\author{Jason Miller}
\address{Statistical Laboratory, Center for Mathematical Sciences, University of Cambridge.}
\email {jpmiller@statslab.cam.ac.uk}

\author{Wei Qian}
\address{City University of Hong Kong. On leave from CNRS and Laboratoire de Math\'ematiques d'Orsay, Universit\'e Paris-Saclay.}
\email {wei.qian@universite-paris-saclay.fr}

\begin{abstract}
We study geodesics in the Brownian map $(\CS,d,\nu)$, the random metric measure space which arises as the Gromov-Hausdorff scaling limit of uniformly random planar maps.  Our results apply to \emph{all} geodesics including those between exceptional points.  \medskip

\noindent First, we prove a strong and quantitative form of the \emph{confluence of geodesics} phenomenon which states that any pair of geodesics which are sufficiently close in the Hausdorff distance must coincide with each other except near their endpoints.  \medskip

\noindent  Then, we show that the intersection of any two geodesics minus their endpoints is connected, the number of geodesics which emanate from a single point and are disjoint except at their starting point is at most $5$, and the maximal number of geodesics which connect any pair of points is $9$. For each $1\le k \le 9$, we obtain the Hausdorff dimension of the pairs of points connected by exactly $k$ geodesics. For $k=7,8,9$, such pairs have dimension zero and are countably infinite. 
Further, we classify the (finite number of) possible configurations of geodesics between any pair of points in $\CS$, up to homeomorphism, and give a dimension upper bound for the set of endpoints in each case.  \medskip

\noindent Finally, we show that every geodesic can be approximated arbitrarily well and in a strong sense by a geodesic connecting $\nu$-typical points.  In particular, this gives an affirmative answer to a conjecture of Angel, Kolesnik, and Miermont that the \emph{geodesic frame} of $\CS$, the union of all of the geodesics in $\CS$ minus their endpoints, has dimension one, the dimension of a single geodesic.  
\end{abstract}

\maketitle

\setcounter{tocdepth}{1}

\let\oldtocsection=\tocsection
 
\let\oldtocsubsection=\tocsubsection
 
\let\oldtocsubsubsection=\tocsubsubsection
 
\renewcommand{\tocsection}[2]{\hspace{0em}\oldtocsection{#1}{#2}}
\renewcommand{\tocsubsection}[2]{\hspace{2em}\oldtocsubsection{#1}{#2}}

\tableofcontents

\parindent 0 pt
\setlength{\parskip}{0.20cm plus1mm minus1mm}

\section{Introduction}
\label{sec:intro}

\subsection{Background and overview}
\label{subsec:background_intro}

The Brownian map is in a certain sense the canonical model for a metric space chosen ``uniformly at random'' among metric spaces which have the topology of the two-dimensional sphere~$\s^2$, and has been a subject of extensive study in recent years.  More specifically, the Brownian map $(\CS, d, \nu)$ is a random geodesic metric space equipped with a measure~$\nu$, which arises as the Gromov-Hausdorff scaling limit of several natural classes of planar maps chosen uniformly at random.  Recall that a planar map is a graph together with an embedding into $\s^2$ defined up to orientation preserving homeomorphism.  If one makes the restriction that each face of the map has $p$ adjacent edges (a $p$-angulation), and fixes the total number of faces, then there are only a finite number of possibilities, so one can pick one uniformly at random.  This is the simplest example of a \emph{random planar map}.  A planar map can be viewed as a metric space by equipping it with its graph metric.  The Brownian map was proved independently by Le Gall \cite{lg2013bm} and Miermont \cite{m2013bm} to be the Gromov-Hausdorff scaling limit of uniformly random quadrangulations with $n$ faces as $n \to \infty$ (as well as triangulations and $2p$-angulations for all integers $p\ge 2$ in \cite{lg2013bm}).  It was subsequently shown to be the limit of a number of other classes of discrete random maps (e.g., \cite{MR3256874,MR3498001,MR3706731,ada2019oddangulations}).  It was proved by Le Gall and Paulin \cite{lgp2008sphere} to be homeomorphic to $\s^2$  (also see a later proof by Miermont \cite{m2008sphere}), and by Le Gall \cite{MR2336042}  to have Hausdorff dimension  $4$  (even though its topological dimension is $2$). 
The Brownian map is also equivalent as a metric measure space to  $\sqrt{8/3}$-Liouville quantum gravity (LQG) \cite{ms2015lqgtbm1,ms2016lqgtbm2,ms2016lqgtbm3}, which serves to equip it with a canonical embedding into $\s^2$.

The present work is devoted to the study of geodesics in the Brownian map, with the aim of providing a global description of the behavior of \emph{all} geodesics at the same time. Geodesics in the Brownian map which emanate from $\nu$-typical points (i.e., $\nu$ a.e.\ point) are now well understood \cite{lg2010geodesics}, thanks to the Brownian snake encoding of the rooted Brownian map \cite{MR2031225, MR2294979, MR2336042}. In contrast, much less is known about geodesics between exceptional points (i.e., points that are not typical, and belong to a set with zero $\nu$ measure). For example, it was not known (before the present work) whether there exist in the Brownian map any two points which are connected by \emph{infinitely} many geodesics; nor was it known whether there exists any point from which \emph{infinitely} many disjoint (except at the starting point) geodesics emanate.  In this work, we aim to fill this gap by proving precise results about the geometric structure of all geodesics together in the Brownian map.

\begin{figure}[h!]
\centering
\includegraphics[width=.7\textwidth]{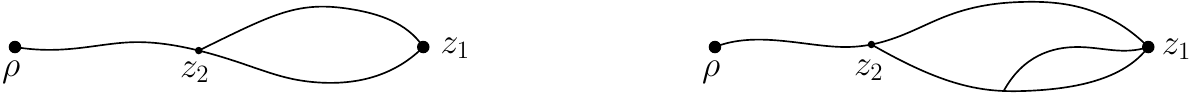}
\caption{Topology of the geodesics from a point which has $2$ (left) or $3$ (right) distinct geodesics to the root $\rho$.}
\label{fig:root_geodesics}
\end{figure}

In the pioneering work \cite{lg2010geodesics},  Le Gall completely classified all geodesics in the Brownian map starting from a distinguished point $\rho\in\CS$ called the root. 
Among other things, he showed that it is a.s.\ the case that for $\nu$ a.e.\ point in $\CS$ there is a unique geodesic connecting this point to the root. 
He also characterized the set of points which are connected to the root by more than one geodesic. These points constitute a dense set with zero $\nu$-measure. In particular, he showed that  it is a.s.\ the case that  every point in $\CS$ is connected to the root by at most~$3$ distinct geodesics. His results also imply that if a point is connected to the root by $2$ or $3$ geodesics, then these geodesics must have the topology of Figure~\ref{fig:root_geodesics}, namely they start being disjoint (except at the starting point) before merging into the same geodesic ending at the root.

Furthermore, Le Gall identified an important feature of the Brownian map, called the \emph{confluence of geodesics} phenomenon.  It states that it is a.s.\ the case that for all $z_1,z_2 \in \CS \setminus \{\rho\}$ any pair of geodesics from $z_1,z_2$ to the root $\rho$ intersect and coalesce before reaching $\rho$ (see Figure~\ref{fig:confluence}).  This property plays a major role in the works \cite{lg2013bm,m2013bm} that identify the Brownian map as the scaling limit of uniform random maps, as well as in the proof of the equivalence of $\sqrt{8/3}$-LQG with the Brownian map \cite{ms2015lqgtbm1,ms2016lqgtbm2,ms2016lqgtbm3}.  
Since the Brownian map is invariant in law under re-rooting \cite{lg2010geodesics}, the root of the map is just a point sampled independently according to $\nu$.
Consequently, the above results hold for geodesics starting from $\nu$ a.e.\ point in $\CS$ (we call them \emph{typical} points). 

\begin{figure}[h]
\centering
\begin{minipage}[b]{.57\textwidth}
  \centering
  \includegraphics[width=.42\linewidth]{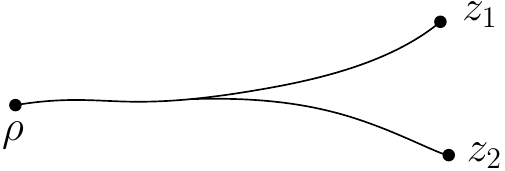}
  \caption{Confluence of geodesics at the root.}
  \label{fig:confluence}
\end{minipage}\hfill
\begin{minipage}[b]{.43\textwidth}
  \centering
  \includegraphics[width=.7\linewidth]{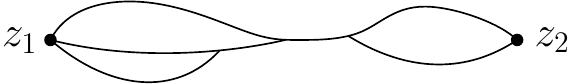}
  \caption{A normal $(3,2)$-network.}
  \label{fig:network}
\end{minipage}
\end{figure}

However, these results do not describe the behavior of geodesics whose endpoints are both not typical (such points constitute a set with zero $\nu$-measure and we call them \emph{exceptional} points).
In \cite{akm2017stability}, Angel, Kolesnik, and Miermont studied the set of pairs of points which are connected by a collection of geodesics with a specified topology which they call a \emph{normal network}.  See Figure~\ref{fig:network}. Two points $z_1, z_2\in\CS$ are said to be connected by a normal $(j,k)$-network if there are $j$ disjoint geodesics (disjoint except at the starting point) which emanate from $z_1$ and then all coalesce into the same geodesic, before branching into $k$ disjoint geodesics (disjoint except at the ending point) ending at $z_2$. They showed that the set of pairs of points that are connected by a normal $(j,k)$-network is non-empty if and only if $j, k \in\{1,2,3\}$, and computed their Hausdorff dimensions. 
Note that if $j,k\ge 2$, then both endpoints of a normal $(j,k)$-network are exceptional points, since the confluence of geodesics phenomenon does not occur at these points. 

Nevertheless, the above results do not rule out the existence of other exceptional points between which the collection of geodesics has a topology which is not that of a normal network.
In fact, it is easy to see that there do exist such other points. For example,  in both configurations in Figure~\ref{fig:root_geodesics}, the geodesics between $z_1$ and $z_2$ do not form a normal network.

In the present work, we aim to describe and classify \emph{all} geodesics in the Brownian map. 
First of all (see Section~\ref{sec:intro1}), we will prove the \emph{strong confluence of geodesics} phenomenon which holds simultaneously for all geodesics in the Brownian map. It states that any two geodesics that are close in the Hausdorff distance must coalesce very rapidly away from their endpoints (see Figure~\ref{fig:strong_confluence}). 

\begin{figure}[h!]
\centering
\includegraphics[width=.4\textwidth]{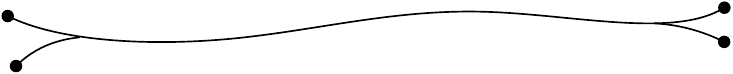}
\caption{Strong confluence of geodesics. Theorems~\ref{thm:strong_confluence2}, \ref{thm:strong_confluence} imply that any two geodesics that are close in the Hausdorff distance must coincide with each other except in a small neighborhood of their endpoints.}
\label{fig:strong_confluence}
\end{figure}

Secondly (see Section~\ref{sec:intro2}), we will show that if we consider the intersection of any two geodesics minus their endpoints, it must be a connected set, which rules out the configurations in Figure~\ref{fig:bump12}.
We will also obtain dimension upper bounds on the \emph{geodesic stars},  which were studied by Miermont in ~\cite{m2013bm} and  played an essential role in the proof of the convergence of quadrangulations towards the Brownian map. More precisely, a point $z\in\CS$ is called a \emph{$k$-star point}, if there are $k$ geodesics which emanate from $z$ and are disjoint except at their starting point.
Our results will in particular imply that there a.s.\ do not exist $k$-star points for $k\ge 6$, which confirms the prediction in \cite{m2013bm}.

We will further classify all possible configurations of geodesics between any pair of points into a finite number of cases, up to homeomorphism, and obtain  upper bounds on the Hausdorff dimension of such pairs of points for each case. Our results will in particular imply that it is a.s.\ the case that the maximal number of geodesics between any pair of points is $9$.
In addition, for each $1\le k \le 9$, we will obtain the a.s.\ Hausdorff dimension of the pairs of points connected by exactly $k$ geodesics. This dimension is positive for $1\le k\le 6$, and equal to zero for $k=7,8,9$. For each $k\in\{7,8,9\}$, we will show that the union of all endpoints connected by exactly $k$ geodesics is a.s.\ dense in $\CS$ and countably infinite.

\begin{figure}[h!]
\centering
\includegraphics[width=.28\textwidth]{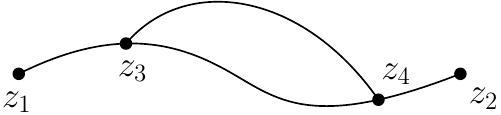}\qquad \qquad \qquad
\includegraphics[width=.28\textwidth]{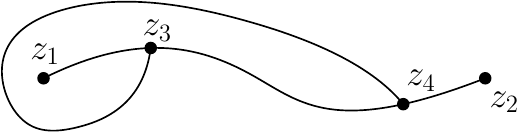}
\caption{\label{fig:bump12}  On both sides, we depict two geodesics between $z_1$ and $z_2$ which coincide for a positive length between $z_1$ and $z_3$, and then between $z_4$ and $z_2$, but take disjoint routes between $z_3$ and $z_4$.  Theorem~\ref{thm:intersection_of_geodesics} rules out the existence of such configurations.}
\end{figure}

Finally  (see Section~\ref{sec:intro3}), we will show that every geodesic $\eta$ in the Brownian map can be approximated arbitrarily well by a geodesic connecting $\nu$-typical points, in the sense that the latter geodesic agrees with $\eta$ except possibly in a small neighborhood of its endpoints. 
As a consequence, we will confirm a conjecture by Angel, Kolesnik, and Miermont \cite{akm2017stability} that the \emph{geodesic frame} of $\CS$, which is the union of all geodesics in $\CS$ minus their endpoints, has dimension one, the dimension of a single geodesic.
In other words, this indicates that all geodesics in the Brownian map go through a few common ``highways'', and that most points in the Brownian map are not traversed by any geodesic (they are just endpoints of geodesics). See Figure~\ref{fig:geodesics} for a numerical simulation.
This exhibits a striking difference between the metric of the Brownian map and that of the Euclidean plane.

We remark that the proofs in the previous works \cite{lg2010geodesics,akm2017stability} on geodesics in the Brownian map  primarily make use of the Brownian snake encoding of the Brownian map \cite{MR2031225, MR2294979, MR2336042}, which is the continuous analog of the Cori-Vauquelin-Schaeffer bijection for quadrangulations \cite{MR638363, schaeffer98}. 
In this encoding, the Brownian map is built from a labeled continuous random tree (CRT) \cite{ald1991crt, MR1207226}. This approach corresponds to the \emph{depth-first} construction of the Brownian map. 
The present work will differ in that we will primarily make use of the \emph{breadth-first} exploration of the Brownian map,  which is the continuum analog of the \emph{peeling by layers} algorithm \cite{ambjorn1997quantum,WATABIKI1995119, a2003peeling} for random planar maps.  There are a number of works which have developed this perspective, including in \cite{clg2016plane, MR3606744, bck2018growth,bbck2018martingales, ms2015axiomatic}.  As we will see later in this work, the breadth-first exploration is particularly amenable for establishing independence properties along geodesics which will lead to our main results.

The term Brownian map is often used to refer to the standard unit area Brownian map in which case $\nu(\CS) = 1$.  In this work we will primarily make use of the infinite measure $\bminflaw$ on (doubly-marked) Brownian maps (see Section~\ref{sec:preliminaries} for the definition) under which $\nu(\CS) \in (0,\infty)$ is not fixed.  Conditioning $\bminflaw$ so that $\nu(\CS) = 1$ yields a probability measure which coincides with the unit area Brownian map.  We will state all of our main results in terms of $\bminflaw$, although by scaling they also all apply to the standard unit area Brownian map. 

One can also consider similar questions in the setting of Brownian surfaces with other topologies, such as the Brownian disk \cite{bm2017disk}, plane \cite{cl2014plane}, or half-plane \cite{gm2017uihpq} (also see \cite{lg2020spine} for their relations).  It is shown in \cite{bet2016geodesics} that for a general class of Brownian surfaces, geodesics to a uniformly chosen random point exhibit similar behavior as for the Brownian map.  We expect that our results can also be transferred to other Brownian surfaces, but for the sake of brevity we will not develop this further here.

Let us finally mention a few works that study geodesics in discrete maps such as the uniform infinite planar triangulations and quadrangulations (e.g.\ \cite{krikun, bg2008geodesics, MR3783207, MR3580091, MR3963287}). The present work focuses exclusively on the continuous object, but can also possibly shed light on the discrete setting.

\begin{figure}[h!]
\begin{center}
\includegraphics[width=0.65\textwidth]{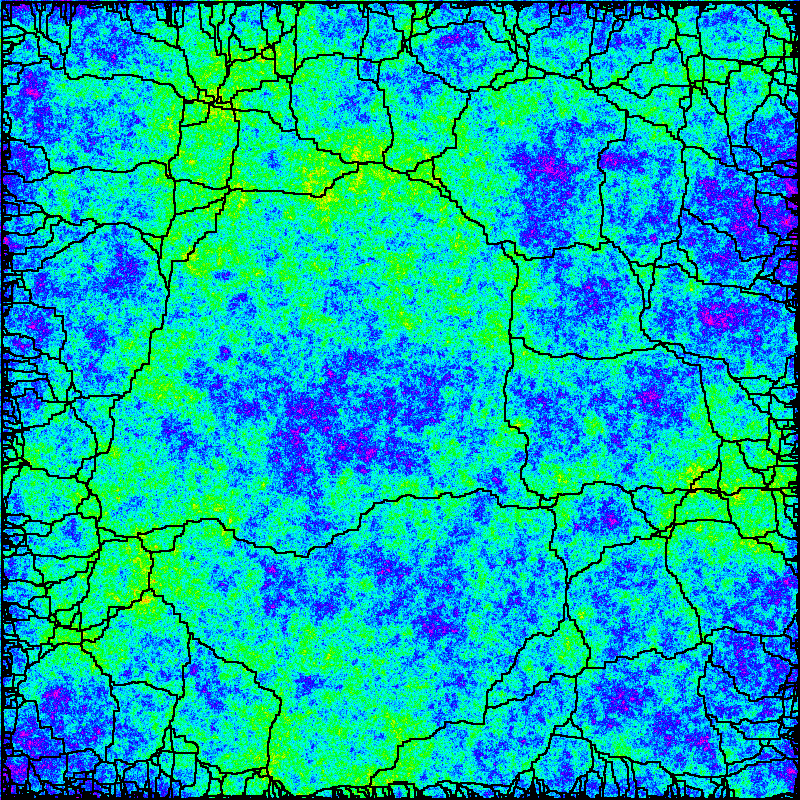}
\end{center}
\caption{\label{fig:geodesics} Shown is an instance of the discrete Gaussian free field $h$ on a $1000 \times 1000$ box $B$ in $\Z^2$.  Darker (resp.\ lighter) colors indicate higher (resp.\ lower) values of $h$.  Shown also are the geodesics which connect every pair of boundary points associated with the random metric on $B$ where the length of each path $P$ is given by $\sum_{x \in P} e^{h(x)/\sqrt{6}}$ (note $\sqrt{8/3}/4=1/\sqrt{6}$).  It is believed that this metric converges to the $\sqrt{8/3}$-LQG metric and therefore the simulation represents an embedding of a Brownian surface (a related approximation scheme was shown to converge in \cite{dddf2019tightness,gm2019metric}).}
\end{figure}

\subsection{Strong confluence of geodesics}\label{sec:intro1}
We have earlier explained the confluence of geodesics phenomenon at the root of the Brownian map discovered in \cite{lg2010geodesics}. Let us also mention that a certain strengthened form of this phenomenon was proved in~\cite{akm2017stability}, but its statement is again associated with typical points.
We also point out that this type of phenomenon does not hold for all points in the Brownian map simultaneously. Indeed, one counterexample is given by the endpoints of a normal network (see Figure~\ref{fig:network}).

In the present work, we show that a different form of  the confluence of geodesics phenomenon holds for all geodesics in the Brownian map.
Before stating our results, let us fix some notation.  For a metric space $X$, $x \in X$, and $\epsilon > 0$ we let $B(x,\epsilon)$ denote the open metric ball centered at~$x$ of radius~$\epsilon$.  For a set $A \subseteq X$ and $\epsilon > 0$ we let $A(\epsilon) = \cup_{x \in A} B(x,\epsilon)$ be the $\epsilon$-neighborhood of $A$.  We also recall that the \emph{Hausdorff distance} between closed sets $A,B \subseteq X$ is defined as
\[ \distH(A,B) = \inf\{ \epsilon > 0 : A \subseteq B(\epsilon),\ B \subseteq A(\epsilon)\}.\]
\begin{theorem}[Strong confluence of geodesics]
\label{thm:strong_confluence2}
The following holds for $\bminflaw$ a.e.\ instance of the Brownian map $(\CS,d,\nu)$.  For each $\sp{u}>0$, there exists $\epsilon_0 > 0$ so that for all $\epsilon \in (0,\epsilon_0)$ the following is true. Let $\delta =  \eps^{1-\sp{u}}$.  Suppose that $\eta_i \colon [0,T_i] \to \CS$ for $i=1,2$ are geodesics with $T_i = d(\eta_i(0), \eta_i(T_i)) \geq 2\delta$  and
$ \distH(\eta_1([0, T_1]),\eta_2([0,T_2])) \leq \epsilon.$  
Then $\eta_i([\delta, T_i-\delta]) \subseteq \eta_{3-i}$ for $i=1,2$.
\end{theorem}

The strong confluence of geodesics, as well as the intermediate results in its proof, will allow us to deduce a number of other consequences about the behavior of geodesics (see Sections~\ref{sec:intro2} and~\ref{sec:intro3}).

In fact, we have an even more precise version of the strong confluence of geodesics, if we further assume that $\eta_2$ is consistently close to either the left or right side of $\eta_1$.
Let us fix more notation. Given a metric space $(X, d)$ and $S\subseteq X$, let $d_S$ be the \emph{interior-internal metric} on $S$, whereby the $d_S$ distance between any two points $u,v \in S$ is given by the infimum of the $d$-length of paths which are contained in the interior of $S$, except possibly at its endpoints.
Suppose $\eta_1$ and $\eta_2$ are two geodesics in a Brownian map $(\CS, d, \nu)$. Then $\CS\setminus \eta_1$ is a simply connected set whose boundary is the union of two parts $\eta_1^{\mathrm L}$ and  $\eta_1^{\mathrm R}$ which respectively correspond to the left and right sides of $\eta_1$. 
Let $\ell_{\mathrm L}$ (resp.\ $\ell_{\mathrm R}$) be the Hausdorff distance between $\eta_1^{\mathrm L}$ (resp.\ $\eta_1^{\mathrm R}$) and $\eta_2 \setminus \eta_1$ with respect to the interior-internal metric $d_{\CS \setminus \eta_1}$.
We define the \emph{one-sided Hausdorff distance} from $\eta_1$ to $\eta_2$ to be
\begin{align}\label{eq:one_sided_hausdorff}
\distHos(\eta_1, \eta_2)= \min (\ell_{\mathrm L}, \ell_{\mathrm R}).
\end{align}
Note that we always have $\distH(\eta_1, \eta_2)\le \distHos(\eta_1, \eta_2)$ where $\distH$ is the Hausdorff distance with respect to $d$.

\begin{theorem}[Strong confluence of geodesics for the one-sided Hausdorff distance]
\label{thm:strong_confluence}
There exists $c > 0$ such that the following holds for $\bminflaw$ a.e.\ instance of the Brownian map $(\CS,d,\nu)$.   There exists $\epsilon_0 > 0$ so that for all $\epsilon \in (0,\epsilon_0)$ the following is true. Let $\delta =  c \eps \log \eps^{-1}$.  Suppose that $\eta_i \colon [0,T_i] \to \CS$ for $i=1,2$ are geodesics with $T_i = d(\eta_i(0), \eta_i(T_i)) \geq 2\delta$  and
$ \distHos(\eta_1([0, T_1]),\eta_2([0,T_2])) \leq \epsilon.$  
Then $\eta_i([\delta, T_i-\delta]) \subseteq \eta_{3-i}$ for $i=1,2$.
\end{theorem}

\begin{figure}[h!]
\centering
\includegraphics[width= .9\textwidth]{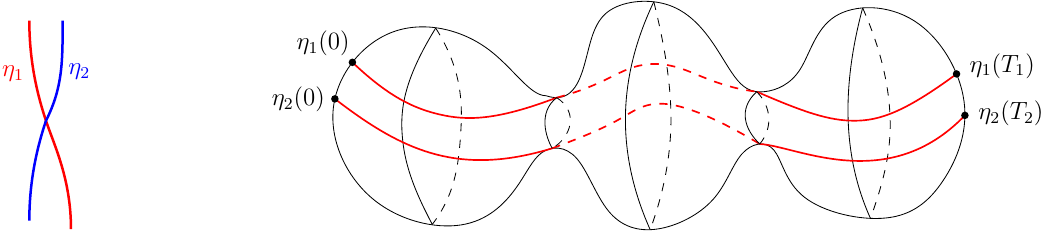}
\caption{{\bf Left:} $\eta_1$ and $\eta_2$ cross each other. {\bf Right:} $\eta_1$ starts being close to the left side of $\eta_2$, but bottlenecks of the Brownian map allow them to ``switch sides'' (so that $\eta_1$ can become close to the right side of $\eta_2$ instead).}
\label{fig:bottleneck}
\end{figure}

We believe that the order of magnitude $\eps \log \eps^{-1}$ is optimal in the statement of Theorem~\ref{thm:strong_confluence}. 
We will first prove Theorem~\ref{thm:strong_confluence}, and then use it to deduce Theorem~\ref{thm:strong_confluence2}.
We depict in Figure~\ref{fig:bottleneck} the two situations where $\eta_1$ and $\eta_2$ are close in the Hausdorff distance, but not in the one-sided Hausdorff distance. If $\eta_1$ and $\eta_2$ cross each other (see the left side of Figure~\ref{fig:bottleneck}), then we can apply  Theorem~\ref{thm:strong_confluence} separately to the portions of $\eta_1$ and $\eta_2$ which do not cross each other. In the right side of Figure~\ref{fig:bottleneck}, we present a more complicated situation where $\eta_1$ and $\eta_2$ do not cross each other, but switch sides at the bottlenecks of the Brownian map.
We will show that for each $\sp{u}>0$ and all sufficiently small $\epsilon > 0$, there are at most $\eps^{-\sp{u}}$ such bottlenecks along any geodesic. Since Theorem~\ref{thm:strong_confluence} ensures that the lengths of $\eta_1$ and $\eta_2$ between every two bottlenecks are at most $c \eps \log \eps^{-1}$, Theorem~\ref{thm:strong_confluence2} holds for this situation.

\subsection{Geometric structure of geodesics}\label{sec:intro2}
We will prove a number of results on the geometric structure of geodesics in the Brownian map.
Among other things, we will show  that the number of geodesics from any point which are otherwise disjoint is at most $5$, and the maximal number of geodesics between any pair of points is $9$.
These results in particular rule out the possibility of infinitely many geodesics between two points, and infinitely many geodesics from a point which are otherwise disjoint.

First of all, we restrict the intersection behavior of geodesics, which rules out the configurations of geodesics in Figure~\ref{fig:bump12}. 
\begin{theorem}[Intersection behavior of geodesics]
\label{thm:intersection_of_geodesics}
The following holds for $\bminflaw$ a.e.\ instance of the Brownian map $(\CS,d,\nu)$.  Suppose that $\eta_i \colon [0,T_i] \to \CS$ are geodesics for $i=1,2$.  Then $\{t \in (0,T_i) : \eta_i(t) \in \eta_{3-i}\}$ is connected for $i=1,2$.
\end{theorem}

Then, we obtain the following dimension upper bound on the set $\Psi_k$ of $k$-star points, namely points $z \in \CS$ from which there emanate $k$ geodesics which are disjoint (except at $z$).

\begin{theorem}[Geodesic stars]
\label{thm:maximum_geodesics}
The following holds for $\bminflaw$ a.e.\ instance of the Brownian map $(\CS,d,\nu)$.  The set $\Psi_{k}$ is empty if $k \geq 6$, and satisfies $\dimH(\Psi_{k}) \leq 5-k$ if $k \leq 5$.
\end{theorem} 
In \cite{m2013bm}, Miermont conjectured that there exist $k$-star points for $1\le k\le 4$ and that there do not exist $k$-star points for $k\ge 6$. Our result confirms the second part of this conjecture.
While we were revising this article, Le Gall established the matching lower bound for $1\le k\le 4$ in the recent work \cite{lg2021star}, so one actually has $\dimH(\Psi_{k}) = 5-k$ for $k\le 5$. We believe that, the techniques and ideas of this article can also allow us to obtain the relevant second moment estimate, leading to a matching lower bound for Theorem~\ref{thm:maximum_geodesics} (see Remark~\ref{remark}). 
It remains an interesting open question to determine whether there exist $5$-star points.

Note that Theorems~\ref{thm:intersection_of_geodesics} and~\ref{thm:maximum_geodesics} together already rule out the possibility of infinitely many geodesics between any pair of points, and reduce the possible configurations of geodesics between any pair of points to a finite number of cases up to homeomorphism.
In the next result, for each of the finite number of configurations of geodesics between pairs of points (up to homeomorphism),  we will  provide an upper bound on the Hausdorff dimension of the endpoints of these geodesics.
In order to give the statement, we first need to introduce the notion of a splitting point. Suppose that $(\CS,d,\nu)$ is an instance of the Brownian map and $u,v \in \CS$ are distinct.  
We say that $z$ is a \emph{splitting point} from $v$ towards $u$ of multiplicity at least $k \in \N$ if there exists $0<r<t<d(u,v)$ and geodesics $\eta_1, \ldots,\eta_{k+1}$ from $v$ to $u$ so that $\eta_i(t)=z$ for each $1\le i \le k+1$, and
$\eta_i|_{[t-r,t]} = \eta_j|_{[t-r,t]}$, $\eta_i((t,t+r]) \cap \eta_j((t,t+r]) = \emptyset$ for all $1\le i < j \le k+1$. 
The \emph{multiplicity} of $z$ is equal to the largest integer $k$ so that the above holds.

\begin{figure}[h!]
\begin{center}
\includegraphics[scale=0.85]{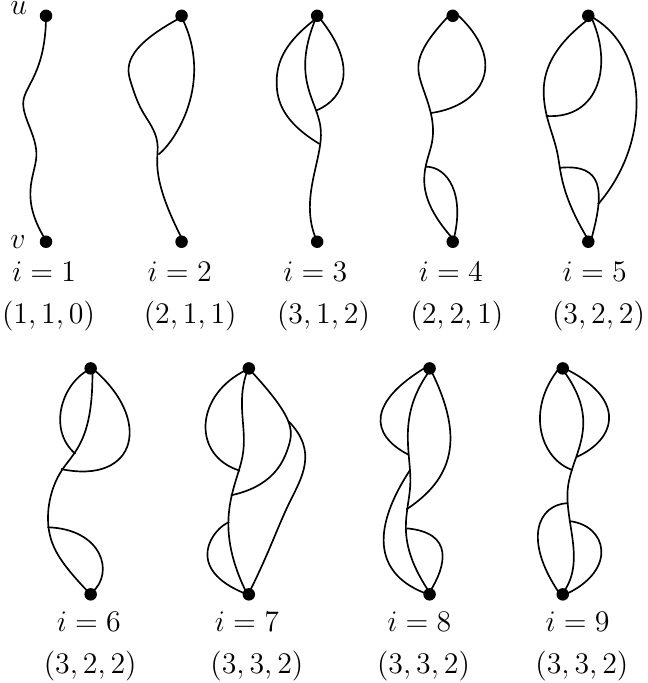}	
\end{center}
\caption{\label{fig:optimal_configurations} Shown are configurations of geodesics which minimize $I+2J+K$ from Theorem~\ref{thm:finite_number_of_geodesics} between points connected by exactly $i$ geodesics with $1 \leq i \leq 9$.
We also indicate the triplet $(I, J, K)$ associated to each configuration.
There are no pairs of distinct points which are connected by $10$ or more geodesics.  Altogether, this leads to the upper bounds in Theorem~\ref{thm:number_of_geodesics}.}
\end{figure}

\begin{theorem}
\label{thm:finite_number_of_geodesics}
The following holds for $\bminflaw$ a.e.\ instance of $(\CS,d,\nu)$.  
For any $u,v\in\CS$ distinct, any geodesic from $v$ to $u$ contains at most $2$ splitting points from $v$ towards $u$, and the multiplicity of any such splitting point is $1$.
Let $\Phi_{I,J,K}$ be the set of $(u,v)$ such that  $u,v \in \CS$ are distinct and that there exists $r > 0$ so that the following is true.
\begin{enumerate}[(i)]
\item There are geodesics $\eta_1,\ldots,\eta_I$ from $u$ to $v$ so that the sets $\eta_i((0,r))$ for $1 \leq i \leq I$ are pairwise disjoint.
\item There are geodesics $\eta_1,\ldots,\eta_J$ from $v$ to $u$ so that the sets $\eta_i((0,r))$ for $1 \leq i \leq J$ are pairwise disjoint.
\item There are $K$ splitting points from $v$ towards $u$.
\end{enumerate}
If $11-(I+2J+K) \geq 0$, we have that
\begin{equation}
\label{eqn:geo_haus_formula}
\dimH(\Phi_{I,J,K}) \leq 11 - (I+2J+K).
\end{equation}
Otherwise, we have that $\Phi_{I,J,K} = \emptyset$.  
\end{theorem}

See Figure~\ref{fig:optimal_configurations} for an illustration of the definition of $I, J, K$ in several cases. 
The reason for the asymmetry in $I$ and $J$ in Theorem~\ref{thm:finite_number_of_geodesics} is that there is an asymmetry in the definition of a splitting point.  
In the language of \cite{akm2017stability}, if $u$ and $v$ are connected by a normal $(j,k)$-network, we have that $I = j$, $J=k$, and $K=j-1$.  Therefore Theorem~\ref{thm:finite_number_of_geodesics} implies that the dimension of the set of such pairs of points is at most $11-(j+2k+(j-1)) = 12-2(j+k)$.  This matches the dimension computed in \cite{akm2017stability}.

Theorem~\ref{thm:finite_number_of_geodesics} is one main ingredient leading to the following result on the number of geodesics between a pair of points.
\begin{theorem}[Number of geodesics between a pair of points]
\label{thm:number_of_geodesics}
The following holds for $\bminflaw$ a.e.\ instance of the Brownian map $(\CS,d,\nu)$. The set $\Phi_{i}$ of pairs of distinct points in $\CS$ which are connected by exactly $i$ geodesics is empty if $i \geq 10$. For $1 \leq i \leq 9$ we have that 
\begin{equation*}
\begin{gathered}
\dimH(\Phi_1) = 8,\quad \dimH(\Phi_2) = 6,\quad \dimH(\Phi_3)= 4,\quad \dimH(\Phi_4)= 4\\
\dimH(\Phi_5) = 2,\quad \dimH(\Phi_6)=2,\quad \dimH(\Phi_7) = 0,\quad \dimH(\Phi_8) = 0,\quad \dimH(\Phi_9) = 0.
\end{gathered}
\end{equation*}
The sets $\Phi_7, \Phi_8, \Phi_9$ are all countably infinite. For each $1\le i \le 9$, the set of endpoints $u\in\CS$ such that there exists $v\in\CS$ with $(u,v) \in \Phi_i$ is dense in $\CS$.
\end{theorem}
In Figure~\ref{fig:optimal_configurations}, we illustrate configurations of geodesics which minimize $I+2J+K$ between points connected by exactly $i$ geodesics with $1 \leq i \leq 9$. Together with Theorem~\ref{thm:finite_number_of_geodesics}, this gives the upper bound of $\dimH(\Phi_{i})$ for $1\le i \le 9$.
The matching lower bounds of $\dimH(\Phi_{i})$ for $2\le i \le 6$, as well as the description of the sets $\Phi_7, \Phi_8, \Phi_9$, are obtained via different arguments (the case $i=1$ is trivial, since $\Phi_1$ has full measure in $\CS \times \CS$).
\begin{itemize}
\item For $i \in \{2,3,4,6, 9\}$, the optimal configurations in Figure~\ref{fig:optimal_configurations} are normal networks.  It was shown in \cite{akm2017stability} that the dimension of the set of pairs of points connected by a normal $(j,k)$-network is $12- 2(j+k)$. Since the endpoints of normal $(j,k)$-networks form a subset of $\Phi_{jk}$, this gives the matching lower bounds of $\dimH(\Phi_{i})$ for $i \in \{2,3,4,6\}$.  
It was also shown in \cite{akm2017stability} that there is a.s.\ a dense and countably infinite set of pairs of points which are connected by a normal $(3,3)$-network. Theorem~\ref{thm:finite_number_of_geodesics} also implies that there exists no other configuration (which is not a normal $(3,3)$-network) giving rise to $9$ geodesics.
\smallskip

\item For $i \in\{5, 7, 8\}$,  the optimal configurations in Figure~\ref{fig:optimal_configurations} are \emph{not} normal networks. We will go through separate procedures to obtain the matching lower bound for $\dimH(\Phi_5)$, and to show that $\Phi_7, \Phi_8$ are countably infinite. This part of the proof uses techniques which are different from the rest of the article. This will be the subject of Section~\ref{sec:578}.
\end{itemize}

\subsection{Approximation by geodesics between typical points}
\label{sec:intro3}
As a consequence of Theorems~\ref{thm:strong_confluence2} and~\ref{thm:intersection_of_geodesics}, we can deduce the following result.

\begin{theorem}
\label{thm:ghost}
For $\bminflaw$ a.e.\ instance of the Brownian map $(\CS,d,\nu)$, the following holds.  For every geodesic $\eta \colon [0,T] \to \CS$,  every $0 < s < t < T$ and $\eps>0$, there exist $\delta>0$ such that every geodesic $\xi \colon [0, S] \to \CS$ with $\xi(0) \in B(\eta(s), \delta)$ and $\xi(S)\in B(\eta(t), \delta)$ satisfies
\[ \xi([\eps, S-\eps]) \subseteq \eta \quad\text{and}\quad \eta([s+\eps,t-\eps]) \subseteq \xi.\]
\end{theorem}
In the statement above, we can choose the endpoints of $\xi$ to be $\nu$-typical points. This implies that every geodesic in the Brownian map can be arbitrarily well approximated in a strong sense by a geodesic connecting typical points. In particular, the behavior of any geodesic away from their endpoints is the same as that of a geodesic between typical points. This result is used in the proof of Theorem~\ref{thm:finite_number_of_geodesics}.

The \emph{geodesic frame} $\gf(\CS)$ is the union of all of the geodesics in $\CS$ minus their endpoints.  Since the Hausdorff dimension of a single geodesic is $1$, it immediately follows that $\dimH \gf(\CS) \geq 1$.  In \cite{akm2017stability}, Angel, Kolesnik, and Miermont proved that the geodesic frame of the Brownian map is of first Baire category, and further conjectured that $\dimH \gf(\CS) = 1$.  We confirm this conjecture as a consequence of Theorem~\ref{thm:ghost}.

\begin{corollary}
\label{cor:geodesic_frame}
For $\bminflaw$ a.e.\ instance of the Brownian map $(\CS,d,\nu)$, we  have that $\dimH \gf(\CS) = 1$.
\end{corollary}

Corollary~\ref{cor:geodesic_frame} is similar in spirit to \cite{m2018fan}, where it is proved that the dimension of all the flow lines of the Gaussian free field with different angles is equal to the dimension of a single $\SLE$ path.  This type of result seems to be a special feature of random, fractal spaces.

\subsection{Outline}\label{sec:intro4}

We now give a detailed outline of the remainder of this article as well as the general strategy to prove the main theorems. We emphasize that our proofs are guided by clear geometric intuitions, which can be explained in a relatively simple language, even though the actual proofs involve a lot of technicalities. The most essential tool in our proofs is the breadth-first exploration of the Brownian map, which roughly speaking decomposes the Brownian map into concentric annuli that are independent metric spaces (conditionally on their boundary lengths).

We denote by $(\CS,d,\nu,x,y)$ an instance of the doubly-marked Brownian map sampled from $\bminflaw$.  As we will explain in more detail, this means that the conditional law of $x,y$ given $(\CS,d,\nu)$ is that of independent samples from $\nu$ (in particular $x, y$ are $\nu$-typical points) and for each $a > 0$ the conditional law of $(\CS,d,\nu)$ given $\nu(\CS) = a$ is that of the standard Brownian map with total area $a$.  We will review the construction of $\bminflaw$ in Section~\ref{sec:preliminaries} as well as describe the breadth-first construction of the Brownian map as developed in \cite{ms2015axiomatic}.

\begin{figure}[h!]
\centering
\includegraphics[width=\textwidth]{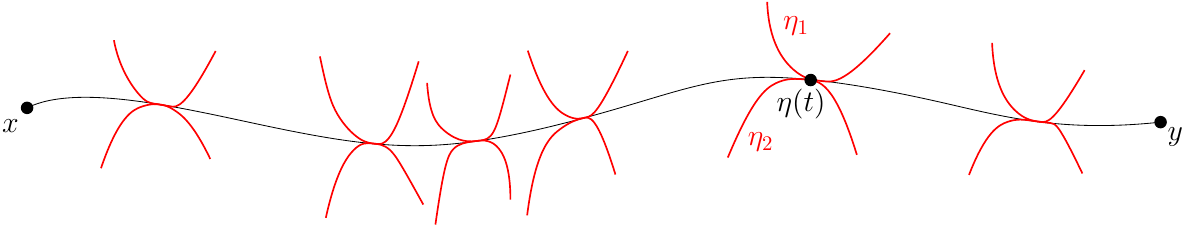}
\caption{Along the geodesic $\eta$ between two typical points $x,y$, with overwhelming probability, there is a very dense set of times at which there is an \X\ along $\eta$.}
\label{fig:outline_X}
\end{figure}

The purpose of Section~\ref{sec:strong_confluence} is to prove a weaker version of Theorem~\ref{thm:strong_confluence}.  Namely, we will prove that two geodesics which are sufficiently close in the one-sided Hausdorff distance intersect each other near their endpoints.  The general idea to prove this result is to show that with overwhelming probability for a geodesic $\eta$ between two $\nu$-typical points a certain event occurs at a very dense set of times along $\eta$.  The event is designed so that if another geodesic passes near $\eta$ at a place where the event occurs then it is forced to intersect $\eta$.  Roughly, the event occurs for $\eta$ at time $t$ if there are two auxiliary geodesics $\eta_1,\eta_2$ of $\CS$ which are respectively to the left and right of $\eta$, both contain $\eta(t)$, and $\eta_i$ for $i=1,2$ is the unique geodesic connecting its endpoints (see Figure~\ref{fig:outline_X}).  
This means that if another geodesic $\wt{\eta}$ intersects the parts of $\eta_i$ (for $i=1$ or $2$)  before and after $\eta_i$ hits $\eta(t)$ then $\wt{\eta}$ must also hit $\eta(t)$ (since $\wt{\eta}$ has to agree with $\eta_i$ between its first and last intersection times by uniqueness of $\eta_i$).  
We refer to these configurations of geodesics as $\X$'s (as the union of $\eta_1$ and $\eta_2$ has the topology of the letter $\X$).  In order to prove this result, we will consider the probability that an \X\ occurs in the successive concentric annuli (which consist of what we call \emph{metric bands}) arising from the breadth-first decomposition centered at $\eta(0)$, and use the independence property across these annuli.

The purpose of Section~\ref{sec:finite_number_of_geodesics} is to rule out the existence of infinitely many geodesics  between any pair of points. 
We will first show a weaker version of Theorem~\ref{thm:maximum_geodesics} which states that there is a deterministic constant $C$ so that the number of geodesics which emanate from any point in the Brownian map and are otherwise disjoint is at most $C$.  
The proof uses the result from Section~\ref{sec:strong_confluence} and a compactness argument. This method is soft and will not allow us to deduce anything about the value of $C$.  
Then, we will compute the probability of having points $u,v$ respectively within distance $\eps$ of two $\nu$-typical points $x,y$ such that the sum of the multiplicities of the splitting points of the geodesics from $u$ to $v$ is exactly $K$. (We will later show in Theorem~\ref{thm:finite_number_of_geodesics} that the multiplicity of each splitting point is $1$, but do not know this at this stage.) 
We will again use the independence property across the successive concentric annuli centered at $x$ and show that the cost of having a splitting point in each annulus is $\eps$ to the power of its multiplicity.
Therefore, the probability of the preceding event is $O(\eps^{K+o(1)})$ and this result will be an important input in the proof of Theorem~\ref{thm:finite_number_of_geodesics} later in Section~\ref{sec:dimension}.
This will allow us to show that the collection of geodesics which connect any pair of points in the Brownian map has at most $9$ splitting points.  Combining these properties, we will deduce that every pair of points is connected by at most a constant number of geodesics (but at this point we do not have any control over this constant).

In Section~\ref{sec:geodesic_structure}, we will complete the proofs of Theorems~\ref{thm:strong_confluence2}, \ref{thm:strong_confluence}, \ref{thm:intersection_of_geodesics}, \ref{thm:ghost}, and Corollary~\ref{cor:geodesic_frame}.  The strategy is first to show that Theorem~\ref{thm:intersection_of_geodesics} holds (i.e., that the intersection set of two geodesics minus their endpoints is connected).  The idea is to show that if it is not the case then there must exist a pair of points in the Brownian map which are connected by infinitely many geodesics.  In other words, we will obtain a contradiction to the results obtained in Section~\ref{sec:finite_number_of_geodesics}.  Upon establishing Theorem~\ref{thm:intersection_of_geodesics}, Theorem~\ref{thm:strong_confluence} (i.e., strong confluence of geodesics for the one-sided Hausdorff distance) will immediately follow from the results established in Section~\ref{sec:strong_confluence}.  Theorem~\ref{thm:ghost} and Corollary~\ref{cor:geodesic_frame} also follow quickly.  Indeed, for the former if we have a geodesic $\eta \colon [0,T]\to \CS$ in the Brownian map and $0 < s < t < T$ fixed and a sequence of geodesics $(\eta_n)$ connecting $\nu$-typical points which converge to $\eta(s)$ and $\eta(t)$ then $\eta_n$ must converge to $\eta|_{[s,t]}$ for otherwise we would obtain a pair of geodesics whose intersection (minus the endpoints) set is not connected.  Corollary~\ref{cor:geodesic_frame} then immediately follows from Theorem~\ref{thm:ghost}.  As explained earlier, we will deduce Theorem~\ref{thm:strong_confluence2} from Theorem~\ref{thm:strong_confluence} by controlling the number of bottlenecks which can occur in the Brownian map.

\begin{figure}[h!]
\centering
\includegraphics[width=.6\textwidth]{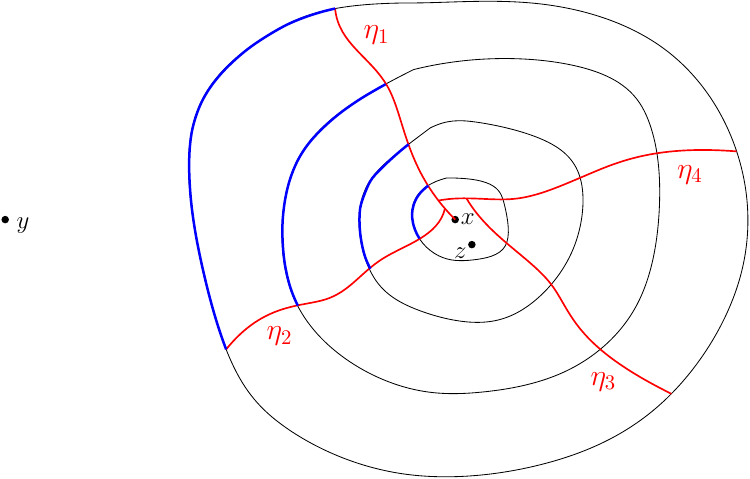}
\caption{We depict the event that there are $k=4$ geodesics going to a typical point $x$ which stay disjoint before reaching distance $\eps$ of $x$. We also depict the boundaries of the (filled) metric balls centered at $x$ at various radii, and color the part of the metric ball boundary between $\eta_1$ and $\eta_2$ in blue, for successive radii. The boundary length of the blue part evolves as a $3/2$-stable continuous state branching process (CSBP) as the radius decreases. The evolution of the boundary lengths between $\eta_i$ and $\eta_{i+1}$ (let $\eta_{k+1}=\eta_1$) for different $i$ are independent.}
\label{fig:outline_disjoint_geodesics}
\end{figure}

In Section~\ref{sec:exponent_disjoint_geodesics}, we will obtain the exponent for there being a point $z$ within distance $\epsilon$ of the marked point $x$ so that there are $k$ geodesics that emanate from $z$ and are otherwise disjoint.  The value of the exponent turns out to be $k-1$.  The proof of this result will involve many technicalities, but heuristically it is possible to arrive at this exponent relatively quickly. Let us consider a slightly different event, which states that there are $k$ geodesics going to $x$ which stay disjoint before reaching distance $\eps$ of $x$ (see Figure~\ref{fig:outline_disjoint_geodesics}).
In the breadth-first construction of the Brownian map, one can describe the evolution of the boundary lengths along the parts of the boundary of the filled metric ball centered at $x$ (i.e., we fill in all the holes of the metric ball except the one containing $y$) between these $k$ geodesics as the radius of the metric ball is reduced: they evolve as $k$ independent continuous state branching processes (CSBPs, we will review this in Section~\ref{sec:preliminaries}).  The time at which one of these processes first hits $0$ corresponds to when the associated pair of geodesics first merges.  For $k$ geodesics going to $x$ not to merge before getting within distance $\epsilon$ of $x$, it must be that these $k$ processes first reach $0$ within time $\epsilon$ of each other.  If we condition on when one of the processes first hits $0$, the conditional probability of the other $k-1$ processes hitting $0$ within $\epsilon$ of this time will be of order $\epsilon^{k-1}$ (as the hitting time of $0$ has a smooth density).  The technical difficulties arise because we want to make a statement about geodesics to $z$ (not $x$) and when one performs the above exploration it is not possible to condition on which geodesics eventually reach $z$ without destroying the Markovian property.  The strategy will be to use the strong confluence results to get that at all but finite number of scales, the geodesics towards $z$ must agree with geodesics towards $x$.  Once we have obtained this exponent, Theorem~\ref{thm:maximum_geodesics} quickly follows.

In Section~\ref{sec:dimension} we will complete the proofs of Theorems~\ref{thm:maximum_geodesics} and~\ref{thm:finite_number_of_geodesics}, as well as deduce the dimension upper bounds in Theorem~\ref{thm:number_of_geodesics} from Theorem~\ref{thm:finite_number_of_geodesics}.
To obtain the dimension upper bound for the endpoints of geodesics in each configuration, we will use as a main ingredient the exponent for the number of splitting points computed in Section~\ref{sec:finite_number_of_geodesics} and the exponent for the number of disjoint geodesics computed in Section~\ref{sec:exponent_disjoint_geodesics}. To deduce that every geodesic contains at most $2$ splitting points (from one endpoint to the other) and that each splitting point has multiplicity $1$, we will use Theorem~\ref{thm:ghost} (which states every geodesic can be approximated by geodesics between typical points) and results \cite{lg2010geodesics} on geodesics to the root.

Finally, in Section~\ref{sec:578}, we will complete the proof of Theorem~\ref{thm:number_of_geodesics}. More precisely, we will obtain the dimension lower bound for $\dimH(\Phi_5)$,  show that $\Phi_7, \Phi_8$ are countably infinite, and that the endpoints of $\Phi_5, \Phi_7, \Phi_8$ are dense in $\CS$. 
It is enough to study the optimal configurations for $\Phi_5, \Phi_7, \Phi_8$ as shown in Figure~\ref{fig:optimal_configurations}. The idea is that one can construct a configuration of geodesics between $u, v\in\CS$, by merging and concatenating a set of geodesics from $u$ to a typical point and another set of geodesics from $v$ to some other typical point. For example, the points $u$ and $v$ with $(u,v) \in \Phi_5$ would respectively belong to the set of points which have $2$ and $3$ distinct geodesics to a typical point (as in Figure~\ref{fig:root_geodesics}). It was shown in \cite{lg2010geodesics} that such sets respectively have dimension $2$ and $0$. We will construct the sets of geodesics from $u$ and $v$ in a relatively independent way, so that the dimension of $\Phi_5$ should be just $2+0=2$. The idea here is similar to how the dimension of the endpoints of normal networks were computed in \cite{akm2017stability}, but our setting is more complicated, because  the geodesics in our configurations do not all pass through any common point (except their common endpoints).

\subsection*{Acknowledgements}  JM was supported by ERC Starting Grant 804166 (SPRS).
WQ was in University of Cambridge, and partially supported by EPSRC grant EP/L018896/1 and a JRF of Churchill college when the project was initiated. WQ was then in CNRS and Universit\'e Paris-Saclay where the project was continued.
The revision of this paper was partially carried out while WQ participated in a program hosted by the Mathematical Sciences Research Institute in Berkeley, California, during the Spring 2022 semester, supported by the National Science Foundation under Grant No. DMS-1928930. 
WQ was also partially supported by CityU Start-up Grant 7200745 while completing the revision.
We are indebted to the referee for a very careful reading of this paper and for many helpful comments.
We thank Ewain Gwynne, Jean-Fran\c cois Le Gall and Pierre Nolin for useful comments on an earlier version of this article.

\section{Preliminaries}
\label{sec:preliminaries}

\subsection{Brownian map review}
\label{subsec:brownian_map}

\newcommand{\CRT}{\mathrm{CRT}}
\newcommand{\BM}{\mathrm{BM}}

We will now give a brief review of the definition and basic properties of the Brownian map.  We direct the reader to \cite{lg2014sphere,m2014stflour} for a more complete review.  The starting point for the construction of the standard unit area Brownian map is the \emph{Brownian snake} \cite{MR1127710}, which is a random process $(X,Y)$ from $[0,1]$ to $\R_+ \times \R$, defined as follows.  Let $X$ be a Brownian excursion on $[0,1]$ (see \cite{ry1999mg}).  Let $\CT$ be the continuum random tree (CRT) \cite{ald1991crt, MR1207226} encoded by $X$.  That is, for $s,t \in [0,1]$ with $s \leq t$ we let
\[ d_X(s,t) = X_s + X_t - 2\inf_{r \in [s,t]} X_r.\]
We say that $s \sim t$ if and only if $d_X(s,t) = 0$.  Then $\CT$ is given by the metric quotient $[0,1] / \sim$.  Let $\rho_{\CRT} \colon [0,1] \to \CT$ be the associated projection map.  Then $\CT$ is equipped with a measure which is given by the pushforward of Lebesgue measure on $[0,1]$ to $\CT$ using $\rho_{\CRT}$.  Given $X$, let $Y$ be the mean-zero Gaussian process on $[0,1]$ with covariance function
\[ \cov(Y_s,Y_t) = \inf_{r \in [s,t]} X_r  \quad\text{for}\quad 0 \leq s \leq t \leq 1.\]
It follows from the Kolmogorov continuity criterion that $Y$ has an a.s.\ $(1/4-\sp{a})$-H\"older continuous modification for any $\sp{a} \in (0,1/4)$.  Note that $s \sim t$ implies that $Y_s = Y_t$ so that $Y$ can be viewed as a Gaussian process indexed by $\CT$.

One can now define the Brownian map as a random metric measure space encoded by the Brownian snake $(X,Y)$, see \cite{MR2336042}.
For $s,t \in [0,1]$ with $s < t$ and $[t,s] = [0,1] \setminus (s,t)$, let
\begin{equation}
\label{eqn:d_circ_def}
d^\circ(s,t) = Y_s + Y_t - 2\max\left(\inf_{r \in [s,t]} Y_r, \inf_{r \in [t,s]} Y_r \right).
\end{equation}
For $a, b \in \CT$, we set
\[ d_\CT^\circ(a,b) = \min\{ d^\circ(s,t) : \rho_{\CRT}(s) = a,\ \rho_{\CRT}(t) = b\}.\]
Finally, for $a,b \in \CT$ we set
\begin{equation}
\label{eqn:d_def}
 d(a,b) = \inf\left\{ \sum_{j=1}^k d_\CT^\circ(a_{j-1},a_j) \right\}
\end{equation}
where the infimum is over all $k \in \N$ and $a_0=a,a_1,\ldots,a_k=b$ in $\CT$.  We say that $a \cong b$ if and only if $d(a,b) = 0$.  Let $p \colon \CT \to \CT / \cong$ be the associated projection map and let $\rho_{\BM} = p \circ \rho_{\CRT}$.  Then $(\CS,d,\nu) = \CT/\cong$ is the Brownian map instance encoded by the Brownian snake $(X,Y)$.  It is a geodesic metric measure space where $\nu$ is given by the pushforward under $p$ of the natural measure on $\CT$.  Equivalently, $\nu$ is the pushforward of Lebesgue measure on $[0,1]$ under the projection $\rho_\BM$.  It was proved by Le Gall and Paulin \cite{lgp2008sphere} and independently by Miermont \cite{m2008sphere} that $(\CS,d,\nu)$ is a.s.\ homeomorphic to $\s^2$.

The Brownian map instance $(\CS,d,\nu)$ is marked by two special points.  The first is the \emph{root} and is given by $x = \rho_{\BM}(s^*)$ where $s^*$ is the a.s.\ unique value of $s \in [0,1]$ where $Y$ attains its infimum (see \cite[Section~2.5]{lgw2006conditioned} for a proof of the uniqueness of $s^*$). The second is the \emph{dual root} and is given by $y = \rho_{\BM}(0) = \rho_{\BM}(1)$.  The reason for the terminology is that $x$ is the root of the tree of geodesics from $x$ to every point in $\CS$ and $y$ is the root of the dual tree, the projection to $\CS$ of $\CT$ under $p$.  It turns out that $x,y$ are independently distributed according to $\nu$.  That is, the law of $(\CS,d,\nu,x,y)$ is invariant under the operation of resampling $x,y$ independently using $\nu$ \cite[Theorem~8.1]{lg2010geodesics}.

We remark that the root and dual root of the Brownian map are often denoted by $\rho$ and $\rho^*$ in other works. In the present work, we denote them by $x$ and $y$, to emphasize that they are just points independently sampled according to $\nu$.  That is, the conditional law  of $\rho,\rho^*$ given the metric measure space $(\CS,d,\nu)$ is that of independent samples from $\nu$.  In particular, throughout the present article, we will often resample the root and the dual root according to $\nu$.

We let $\bmlaw{1}$ denote the law of $(\CS,d,\nu,x,y)$; the superscript $A=1$ in the notation is to emphasize that $\nu(\CS) = 1$.  By replacing the unit length Brownian excursion $X$ with a Brownian excursion of length $a > 0$, we can similarly define $\bmlaw{a}$.  This is the law on Brownian map instances with total area $a$.  By applying the Brownian scaling property twice (i.e., for $X$ and then for $Y$), we note that if $(\CS,d,\nu,x,y)$ has law $\bmlaw{1}$ then the metric measure space obtained by scaling distances by the factor $a^{1/4}$ and areas by the factor $a$ has law $\bmlaw{a}$.

In many situations, it is useful to consider the Brownian map with \emph{random} area rather than fixed area.  This is defined by replacing the Brownian excursion $X$ with a ``sample'' from the (infinite) Brownian excursion measure.  We recall that the Brownian motion excursion measure can be ``sampled'' from as follows:
\begin{itemize}
\item Pick a lifetime from the infinite measure $c t^{-3/2} dt$ where $dt$ denotes Lebesgue measure on $\R_+$ and $c > 0$ is a constant.
\item Given $t$, sample a Brownian excursion $(X_s)_{0\le s\le t}$ of length $t$. We recall that for all $t>0$, $(X_s)_{0\le s\le t}$ is equal in law to $( t^{1/2} \wt{X}_{s/t})_{0\le s\le t}$ where $\wt{X}$ is a Brownian excursion of length $1$. 
\end{itemize}
Note that the total amount of area of the corresponding Brownian map instance is given by the length $t$ of the Brownian excursion $X$.

We let $\bminflaw$ be the distribution of $(\CS,d,\nu,x,y)$ when $X$ is ``sampled'' from the infinite Brownian excursion measure as defined above.  
For each $a>0$, the conditional law of $(\CS,d,\nu,x,y)$ given $\nu(\CS) = a$ is exactly the probability measure $\bmlaw{a}$. By an abuse of notation, we also say an instance $(\CS, d, \nu)$ is distributed according to $\bminflaw$ (or $\bmlaw{a}$) meaning that it is obtained by sampling $(\CS,d,\nu,x,y)$ according to $\bminflaw$ (or $\bmlaw{a}$) and then forgetting about the marked points $x,y$.

We finish this subsection by collecting a few results on the upper and lower bounds for the volume of balls in the Brownian map and a result about covering the Brownian map by a union of balls centered at $\nu$-typical points.
We first record in the following lemma a result from \cite[Corollary~6.2]{lg2010geodesics}.
\begin{lemma}[\cite{lg2010geodesics} Corollary 6.2]
\label{lem:le_gall_volume}
Let $(\CS, d, \nu)$ be sampled from $\bmlaw{1}$. Fix $\sp{a} \in (0,1]$ and let $S= \sup_{\eps>0} (\sup_{z\in\CS} \nu(B(z, \eps)) \eps^{-4+\sp{a}})$. Then $\E[S^k] <\infty$ for every $k\in\N$.
\end{lemma}

\begin{lemma}
\label{lem:bm_volume_estimates}
For $\bminflaw$ a.e.\ instance $(\CS,d,\nu)$ and each $\sp{a} > 0$ there exists $\epsilon_0 > 0$ so that
\[ \epsilon^{4+\sp{a}} \leq \nu(B(z,\epsilon)) \leq \epsilon^{4-\sp{a}} \quad\text{for all}\quad \epsilon \in (0,\epsilon_0) \quad\text{and}\quad z \in \CS.\]
\end{lemma}
\begin{proof}
By scaling, it suffices to prove the result for $(\CS,d,\nu)$ sampled from $\bmlaw{1}$.   Lemma~\ref{lem:le_gall_volume} implies that off an event with probability decaying faster than any power of $\eps$, we have $\nu(B(z,\eps)) \le \eps^{4-\sp{a}}$ for every $z\in\CS$. Then by the Borel-Cantelli lemma, the upper bound follows. Let us now deduce the lower bound.  Let $(X,Y)$ be the Brownian snake instance which encodes $(\CS,d,\nu)$ and recall that $Y$ is a.s.\ $(1/4-\sp{b})$-H\"older continuous for each $\sp{b} \in (0,1/4)$.  Fix $s,t \in [0,1]$ with $s \leq t$.  Recalling~\eqref{eqn:d_circ_def} and~\eqref{eqn:d_def}, we thus have for a constant $c > 0$ that
\[ d(\rho_\BM(s), \rho_\BM(t)) \leq Y_s + Y_t - 2 \inf_{r \in [s,t]} Y_r \leq c |s-t|^{1/4-\sp{b}}.\]
Fix $\sp{a} > 0$.  Then by taking $\sp{b} \in (0,1/4)$ sufficiently small, we see that there exists $\epsilon_0 > 0$ so that $\rho_\BM([s,s+\epsilon^{4+\sp{a}}]) \subseteq B(\rho_\BM(s),\epsilon)$ for all $\epsilon \in (0,\epsilon_0)$.  By the definition of $\nu$, this implies that $\nu(B(\rho_\BM(s),\epsilon)) \geq \epsilon^{4+\sp{a}}$ for all $\epsilon \in (0,\epsilon_0)$.  This completes the proof of the lower bound since $s \in [0,1]$ was arbitrary.
\end{proof}

\begin{lemma}
\label{lem:typical_points_dense}
The following holds for $\bminflaw$ a.e.\ instance of the Brownian map $(\CS,d,\nu)$.  Let $(z_j)$ be a sequence of i.i.d.\ points chosen from $\nu$.  For each $\sp{a} > 0$ there exists $\epsilon_0 > 0$ so that for all $\epsilon \in (0,\epsilon_0)$ we have that $\CS \subseteq \cup_{j=1}^{N_\epsilon} B(z_j,\epsilon)$ where $N_\epsilon = \epsilon^{-4-\sp{a}}$.	
\end{lemma}
\begin{proof}
By scaling, it suffices to prove the result for $(\CS,d,\nu)$ sampled from $\bmlaw{1}$.  Fix $\sp{a} > 0$.  Lemma~\ref{lem:bm_volume_estimates} implies that there a.s.\ exists $r_0 > 0$ so that $\nu(B(z,r)) \geq r^{4+\sp{a}}$ for all $r \in (0,r_0)$ and $z \in \CS$.   For each $k \in \N$, we let $N_k = 2^{(4+2\sp{a})k}$.  There exists $k_0\in\N$ such that for all $k\ge k_0$, given $(\CS,d,\nu)$, conditionally on $\{r_0\ge 2^{-k}\}$, the probability that the following event $A_k$ does not hold is at most $\exp(-2^{k\sp{a}/2})$.
\begin{enumerate}
\item[$(A_k)$] For all $z\in\CS$, there exists $1 \leq i \leq N_k$ such that $z_i \in B(z, 2^{-k})$.
\end{enumerate}
Indeed, to see this we note that on $A_k^c$ there exists $z \in \CS$ so that $d(z,z_i) \geq 2^{-k}$ for all $1 \leq i \leq N_k$.  Since $\nu(B(z,2^{-k})) \geq 2^{-k(4+\sp{a})}$ and $\nu(\CS) = 1$, this would imply that the probability conditionally on $A_k^c$ that $d(z_{N_k+1},z_i) \geq 2^{-k}$ for all $1 \leq i \leq N_k$ is at least $2^{-(4+\sp{a})k}$. Using the fact that for events $A,B$ if $\p[ A \giv B] \geq p$ then $\p[B] \leq p^{-1}\p[A]$, this implies that the probability of $A_k^c$ is at most $2^{(4+\sp{a})k}$ times the probability that $d(z_{N_k+1},z_i) \geq 2^{-k}$ for all $1 \leq i \leq N_k$.  The latter probability is at most $(1-2^{-(4+\sp{a})k})^{N_k} \leq \exp(-2^{k\sp{a}})$.  This proves the claim.  Therefore the Borel-Cantelli lemma implies that there $\bminflaw$-a.e.\ exists $k_0 \in \N$ with $2^{-k_0} \le r_0$ so that for all $k \geq k_0$ the union of $B(z_j, 2^{-k})$ for $1 \leq j \leq N_k$ covers $\CS$.
\end{proof}

\subsection{Breadth-first exploration of the Brownian map}
\label{subsec:breadth_first}

The Brownian snake construction of the Brownian map $(\CS,d,\nu,x,y)$ given in Section~\ref{subsec:brownian_map} corresponds to a \emph{depth-first} exploration of the Brownian map because the curve $t \mapsto \rho_\BM(t)$ is the Peano curve between the tree of geodesics from the root $x$ and the dual tree.  In this work, the \emph{breadth-first} exploration which has been studied in a number of works, including \cite{clg2016plane, MR3606744, bck2018growth,bbck2018martingales, ms2015axiomatic}, will play an important role.
Here, we will only focus on the continuous aspect, but mention that this point of view also naturally arises when one considers the peeling process \cite{a2003peeling} of the uniform infinite planar triangulation \cite{MR2013797}.

\subsubsection{Continuous state branching processes}
\label{subsubsec:csbp}

We begin by recalling basic properties of the continuous state branching processes (CSBPs) \cite{MR0101554,MR208685,MR0217893}. A CSBP with branching mechanism $\psi$ is the Markov process $Y$ on $\R_+$ which is defined through its Laplace transforms
\begin{equation}
\label{eqn:csbp_laplace_transform}
\E[ \exp(-\lambda Y_t) \giv Y_s] = \exp(-Y_s u_{t-s}(\lambda)) \quad\text{for}\quad \lambda \geq 0, \quad t\ge s\ge 0
\end{equation}
where
\[ \frac{\partial u_t}{\partial t}(\lambda) = -\psi(u_t(\lambda)) \quad\text{for}\quad u_0(\lambda) = \lambda.\]
CSBPs are related to L\'evy processes through the \emph{Lamperti transform} (see \cite{kyp2006levy}).  Namely, if $X$ is a L\'evy process with Laplace exponent $\psi$ and
\begin{equation}
\label{eqn:lamperti_levy_to_csbp}
s(t) = \inf\{r > 0 : \int_0^r \frac{1}{X_u} du \geq t\}
\end{equation}
then the time-changed process $X_{s(t)}$ is a CSBP with branching mechanism $\psi$.  Conversely, if $Y$ is a CSBP with branching mechanism $\psi$ and
\begin{equation}
\label{eqn:lamperti_csbp_to_levy}
t(s) = \inf\{r > 0 : \int_0^r Y_u du \geq s\}
\end{equation}
then the time-changed process $Y_{t(s)}$ is a L\'evy process with Laplace exponent $\psi$.

For $\alpha\in (1,2)$, the Laplace exponent of an $\alpha$-stable L\'evy process with only upward jumps is given by $\psi(\lambda) = c \lambda^\alpha$ for a constant $c > 0$. We call the corresponding CSBP an $\alpha$-stable CSBP.  Note that in this case we have that $u_t(\lambda) = (\lambda^{1-\alpha} + ct)^{1/(1-\alpha)}$.  This combined with~\eqref{eqn:csbp_laplace_transform} implies that $\alpha$-stable CSBPs satisfy the following scaling property.  If $Y$ is an $\alpha$-stable CSBP and $C > 0$, then $Y_{C^{\alpha-1} t}$ is equal in distribution to $C Y_t$, up to a change of starting point.
Suppose that $Y$ is an $\alpha$-stable CSBP and $\zeta = \inf\{t \geq 0 : Y_t = 0\}$.  Then we have that
\begin{equation}
\label{eqn:csbp_extinction_time}
\p[ \zeta > t ]  = \p[ Y_t > 0] = 1 - \lim_{\lambda \to \infty} \E[ e^{-\lambda Y_t}] = 1- \exp(-(c t)^{1/(1-\alpha)} Y_0).\end{equation}

For each $\alpha$-stable L\'evy process $X$ with only upward jumps, there is a naturally associated infinite measure $\mu_\alpha$ on $\alpha$-stable L\'evy excursions from $0$ (see \cite{MR1465814}). Let $I_t = \inf_{0 \leq s \leq t} X_s$ be the running infimum of $X$.   Then the succession of excursions of $X-I$  from $0$ is distributed as a Poisson point process with intensity measure given by the product of Lebesgue measure and $\mu_\alpha$.
The measure $\mu_\alpha$ can be ``sampled'' from using the following steps 

\begin{itemize}
\item Pick a lifetime $t$  from the infinite measure $c_\alpha t^{-1-1/\alpha} dt$ where $dt$ denotes Lebesgue measure on $\R_+$ and $c_\alpha > 0$ is a constant.
\item Given $t$, sample (according to a probability measure) an $\alpha$-stable L\'evy excursion $(Z_s)_{0\le s\le t}$ of length $t$. We refer the reader to \cite[Chapter VIII]{b1996levy} for more details on this excursion measure. Here, we just mention that for all $t>0$, $(Z_s)_{0\le s\le t}$ is equal in law to $( t^{1/\alpha} \wt{Z}_{s/t})_{0\le s\le t}$ where $\wt{Z}$ is an $\alpha$-stable L\'evy excursion of length $1$. 
\end{itemize}
By performing the time-change for $\alpha$-stable L\'evy excursions as in the definition of the Lamperti transform~\eqref{eqn:lamperti_levy_to_csbp}, we also get an infinite measure $M_\alpha$ on $\alpha$-stable CSBP excursions.  As in the case of the Brownian excursion measure, both the $\alpha$-stable L\'evy and CSBP excursion measures have the following property.  If for each $\epsilon > 0$ we let $\tau_\epsilon$ be the first time $t$ that the L\'evy (resp.\ CSBP) excursion hits $\epsilon$, the conditional law of the remainder of the process given $\{\tau_\epsilon < \infty\}$ is that of an $\alpha$-stable L\'evy process (resp.\ CSBP) starting from $\epsilon$ and stopped at the first time that it hits $0$.  Let us now derive the law on the lifetime $t$ under $M_\alpha$.  For each $S,T > \delta > 0$ we have that
\begin{align*}
	\frac{M_\alpha[ t \geq T ]}{M_\alpha[ t \geq S ]}
&= \lim_{\epsilon \to 0} \frac{M_\alpha[ t \geq T, \tau_\epsilon \leq \delta]}{M_\alpha[ t \geq S, \tau_\epsilon \leq \delta]}
 \leq \left( \frac{T-\delta}{S+\delta} \right)^{1/(1-\alpha)},
\end{align*}
where we have applied~\eqref{eqn:csbp_extinction_time} at the stopping time $\tau_\epsilon$ in the inequality.  Since $\delta > 0$ was arbitrary we have that the left hand side above is at most $(T/S)^{1/(1-\alpha)}$.  By an analogous argument, we have that it is also at least $(T/S)^{1/(1-\alpha)}$.  Therefore it is also equal to $(T/S)^{1/(1-\alpha)}$.  This implies that there exists a constant $c > 0$ so that the density for the lifetime under $M_\alpha$ is given by $c t^{1/(1-\alpha)-1}$.

\subsubsection{Boundary length and the conditional independence of the inside and outside of filled metric balls}
\label{sec:exploration}

Suppose that $(\CS,d,\nu,x,y)$ is distributed according to $\bminflaw$.  For each $r \geq 0$, we let $\fb{y}{x}{r}$ be the \emph{filled metric ball} centered at $x$ of radius $r$ with respect to $y$.  That is, $\fb{y}{x}{r}$ is the complement of the $y$-containing component of $\CS \setminus B(x,r)$ where $B(x,r)$ is the metric ball of radius~$r$ centered at~$x$.

It is shown in \cite[Section~4]{ms2015axiomatic} that it is possible to associate with $\partial \fb{y}{x}{r}$ a \emph{boundary length}~$L_r$ in a manner which is measurable with respect to $(\CS,d,\nu,x,y)$.
Moreover, the process $(L_r)$ indexed by $r\in [0,  d(x,y)]$ has the same (infinite) distribution as the time-reversal of a $3/2$-stable CSBP excursion.
Note that the infinite mass of $\bminflaw$ is carried by the infinite distribution of the time length $d(x,y)$ of the process $(L_r)$. For any $r>0$, when we restrict to the event $d(x,y) >r$, $\bminflaw$ has finite total mass.

For each $r > 0$, we can view $\fb{y}{x}{r}$ as a metric measure space which is marked by $x$ and equipped with the measure which is given by restricting $\nu$ to $\fb{y}{x}{r}$.  
We can similarly view $\CS \setminus \fb{y}{x}{r}$ as a metric measure space which is marked by $y$.
We equip both spaces with the interior-internal metric as defined before Theorem~\ref{thm:strong_confluence}.

It is shown in \cite{ms2015axiomatic} that the boundary length $L_r$ is a.s.\ determined  by $\fb{y}{x}{r}$.  It is also shown in \cite{ms2015axiomatic} that $L_r$ is a.s.\ determined by $\CS \setminus \fb{y}{x}{r}$.  Moreover, on the event $d(x,y)>r$,  both $\fb{y}{x}{r}$ and $\CS \setminus \fb{y}{x}{r}$ as marked metric measure spaces are conditionally independent given $L_r$. 
The same also holds if we replace $r$ with $s = d(x,y)-r$. This allows us to perform a \emph{reverse metric exploration}, in which case we observe $\CS \setminus \fb{y}{x}{s}$ as $r$ increases from $0$ to $d(x,y)$ so that $s$ decreases from $d(x,y)$ to $0$.  In this case, the unexplored region is the filled metric ball $\fb{y}{x}{s}$ and its law only depends on $L_s = L_{d(x,y)-r}$.  Then the boundary length process $L_{d(x,y)-r}$ evolves as a $3/2$-stable CSBP.  We will often use the notation $Y_r$ for $L_{d(x,y)-r}$. As we will see, in many cases it is actually more convenient to perform a reverse rather than a forward metric exploration.

\subsubsection{Metric bands}
\label{subsec:metric_bands}

Iterating the reverse metric exploration allows us to decompose an instance of the Brownian map into conditionally independent \emph{metric bands}. See Figure~\ref{fig:metric_band}.  More precisely, suppose that we have fixed $0 < r_1 < r_2 < \cdots < r_k$ and we let $s_j = d(x,y) - r_j$ for each $1 \leq j \leq k$.  Then we can view each $\CB^j = \fb{y}{x}{s_{j}} \setminus \fb{y}{x}{s_{j+1}}$, $1 \leq j \leq k-1$, as a metric measure space with its interior-internal metric $d_{\CB^j}$ and measure $\nu_{\CB^j} := \nu|_{\CB^j}$.  On the event $d(x,y)>r_j$, $\CB^j$ is non-empty, and it is either a topological annulus if $d(x,y)>r_{j+1}$, or disk if $r_{j+1}\ge d(x,y)$.  Its inner (resp.\ outer) boundary is the component of $\partial \CB^j$ whose distance to $x$ is $s_{j}$ (resp.\ $s_{j+1}$).  If $\CB^j$ is a topological disk, then it corresponds to a filled metric ball and in this case we will define the outer boundary to be the center point of the ball. We denote the inner (resp.\ outer) boundary of $\partial \CB^j$ by $\innerboundary \CB^j$ (resp.\ $\outerboundary \CB^j$).  We note that $\innerboundary \CB^j$ is naturally marked by the point visited by the a.s.\ unique geodesic connecting $x$ and $y$.  The \emph{width} of $\CB^j$ is $s_{j} - s_{j+1} = r_{j+1} - r_{j}$.  The independence property for the reverse metric exploration implies that $\CB^j$ is conditionally independent of $\CB^1,\ldots,\CB^{j-1}$ given the boundary length $Y_{r_j}$ of $\innerboundary \CB^j$.

\begin{figure}[h!]
\centering
\includegraphics[width=.56\textwidth]{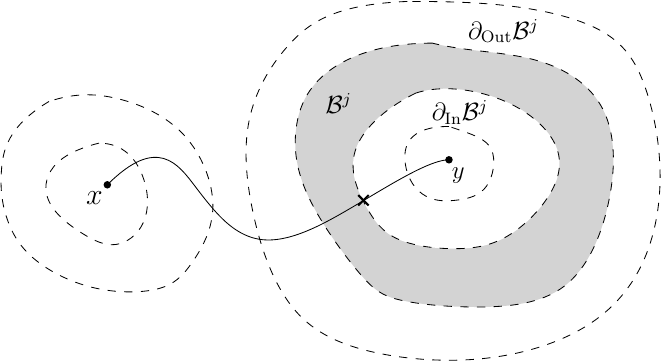}
\caption{\label{fig:metric_band} Illustration of the metric bands (drawn in the plane, but should be viewed as in the Riemann sphere). We depict one metric band $\CB^j$ in grey.}
\end{figure}

For each $\ell,w > 0$, let $\bandlaw{\ell}{w}$ be the probability law on metric bands $(\CB,d_\CB,\nu_\CB,z)$ with inner boundary length $\ell$, width $w$, and marked by a point $z$ on the inner boundary of $\CB$. 
By Section~\ref{sec:exploration} and the definition of the metric bands, one can again perform a reverse metric exploration inside~$\CB$.  Let $d_0$ be the distance from $\innerboundary \CB$ to $\outerboundary \CB$ (this distance is equal to $w$ if $\CB$ is a topological annulus, and is at most $w$ otherwise).
For each $t \in (0, w)$ we let $\CB_t$ be the set of points in $\CB$ disconnected from $\innerboundary \CB$ by the $(d_0-t)$-neighborhood of $\outerboundary \CB$ (let $\CB_t$ be empty if $t>d_0$).  Then $\CB_t$ is a metric band of width $w-t$.  If $Y_t$ denotes the boundary length of $\innerboundary \CB_t$, then $(Y_t)_{t \in [0,w]}$ evolves as a $3/2$-stable CSBP starting from $\ell$ and which stays at $0$ once it hits $0$.

Finally, let us remark that the metric bands satisfy the same scaling property as the Brownian map. For all $a>0$, if $(\CB,d_\CB,\nu_\CB,z)$ has law $\bandlaw{\ell}{w}$, and we rescale distances by $a$, boundary lengths by $a^2$, and areas by $a^4$, then we obtain a sample from $\bandlaw{a^2 \ell}{a w}$.

\subsubsection{Geodesic slices}
Suppose that $(\CB,d_\CB,\nu_\CB,z)$ is distributed according to $\bandlaw{\ell}{w}$.  Then we can decompose $\CB$ further into \emph{slices} as follows, see Figure~\ref{fig:slices}.  Suppose that $z_1,\ldots,z_k$ are points on $\innerboundary \CB$ given in counterclockwise order  and chosen in a way which is independent of $\CB$. For each $j$ there is a.s.\  a unique geodesic $\eta_j$ from $z_j$ to $\outerboundary \CB$.  For each $j$, let $\CG_j$ be the component of $\CB \setminus (\eta_j \cup \eta_{j+1})$ with $\eta_j$ (resp.\ $\eta_{j+1}$) on its left (resp.\ right) side.  Then $\CG_j$ is a \emph{geodesic slice}.  Moreover, the $\CG_j$'s (viewed as metric measure spaces with the interior-internal metric) are independent of each other.  We define the inner boundary $\innerboundary \CG_j$ of $\CG_j$ to be the intersection of $\partial \CG_j$ with $\innerboundary \CB$.  If the geodesics which make up the left and right boundaries of $\CG_j$ have not merged upon getting to distance $w$ from $\innerboundary \CB$, then we define the outer boundary $\outerboundary \CG_j$ to be the intersection of $\partial \CG_j$ with $\outerboundary \CB$.  Otherwise, we define $\outerboundary \CG_j$ to be the merging point of the two geodesics.  The width of each slice $\CG_j$ is equal to $w$ (even though the geodesics which bound its left and right sides may merge before getting to distance $w$ from $\innerboundary \CB$) and the inner boundary length of $\CG_j$ is equal to the boundary length $\innerboundary \CG_j$ (equivalently, of the counterclockwise segment of $\innerboundary \CB$ from $z_j$ to $z_{j+1}$).  We note that if $k=1$ so that we cut $\CB$ with a single geodesic then the resulting slice has width $w$ and inner boundary length $\ell$.  In particular, if we take such a slice and glue its left and right boundaries together then we obtain a band with inner boundary length $\ell$ and width $w$.

\begin{figure}[h!]
\centering
\includegraphics[width=.45\textwidth]{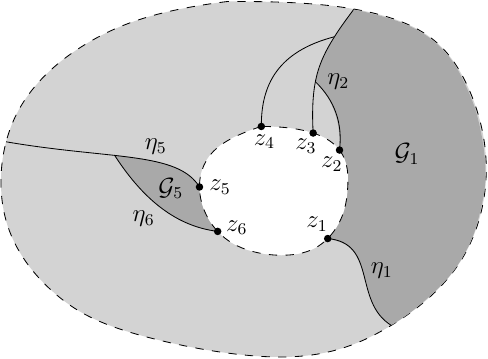}
\caption{\label{fig:slices} We decompose a metric band $\CB$ of width $w$ into several geodesic slices. We depict in dark grey two geodesic slices. The left and right boundaries of $\CG_1$ have not merged, while those of $\CG_5$ have merged before reaching distance $w$ from $\innerboundary \CB$.}
\end{figure}

For $\ell, w>0$, let $\slicelaw{\ell}{w}$ be the  probability law on slices $\CG$ with inner boundary length $\ell$ and width $w$.  As in the case of metric bands, we can also perform a reverse metric exploration from~$\innerboundary\CG$.  To explain this point in more detail, let $d_0$ be the distance from $\innerboundary \CG$ to $\outerboundary \CG$.  For each $t \in (0, w)$ we let~$\CG_t$ be the set of points in~$\CG$ disconnected from $\innerboundary \CG$ by the $(d_0-t)$-neighborhood of $\outerboundary \CG$.  Then $\CG_t$ is a slice of width $w-t$.  If $Y_t$ denotes the boundary length of $\innerboundary \CG_t$, then $(Y_t)_{t \in [0,w]}$ evolves as a $3/2$-stable CSBP starting from $\ell$ and which stays at $0$ once it hits $0$.  Also, for each $a > 0$, if $(\CG,d_\CG,\nu_\CG)$ has law $\slicelaw{\ell}{w}$, and we rescale distances by $a$, boundary lengths by $a^2$, and areas by $a^4$, then we obtain a sample from $\slicelaw{a^2 \ell}{a w}$.

\subsection{Brownian disk review}

We will now record some basic facts about the Brownian disk.  There are two different ways to define it. One is using a Brownian snake based construction, which is developed in \cite{bm2017disk}.  The other is based on realizing it as a complementary component when performing a metric exploration of the Brownian map.  This approach is developed in \cite{ms2015axiomatic} and the two definitions were proved to be equivalent in \cite{lg2019disksnake}.  The two definitions of the Brownian disk each include a notion of boundary length and these notions were proved to be equivalent in \cite{lg2020equivalent}.

\newcommand{\bdisklaw}[1]{\p_{\mathrm{BD}}^{L=#1}}
\newcommand{\bdisklawweighted}[1]{\p_{\mathrm{BD,W}}^{L=#1}}
\newcommand{\mnet}[3]{{\mathrm{MetNet}}(#1;#2,#3)}

Let us first recall the definition given in \cite{ms2015axiomatic}. Suppose that $r > 0$.  Consider the metric measure space which is given by equipping $\CS \setminus \fb{y}{x}{r}$ with its interior-internal metric, the restriction of $\nu$, and the boundary length measure associated with $\partial \fb{y}{x}{r}$.  We note that this metric measure space is marked by the point $y$.  This is the \emph{marked Brownian disk} and we will denote its law by $\bdisklawweighted{\ell}$ when its boundary length is equal to $\ell$. There is also a law on unmarked Brownian disks so that one can obtain the law on marked Brownian disks from the former by weighting the area of the unmarked disk and then adding a marked point which is sampled from the area measure. However, we will only consider marked Brownian disks in the present article.

\newcommand{\BD}{{\mathrm{BD}}}

We now recall the Brownian snake definition of $\bdisklawweighted{\ell}$ as developed in \cite{bm2017disk}.  Let $X$ be a standard Brownian motion and let $\tau = \inf\{t \geq 0 : X_t = -\ell\}$.  For each $0 \leq s \leq t \leq \tau$, we let $\ul{X}_{s,t} = \inf_{r \in [s,t]} X_r$.  We then define for $0 \leq s \leq t \leq \tau$,
\[ d_X(s,t) = X_s + X_t - 2 \ul{X}_{s,t}.\]
Then $d_X$ defines a pseudometric on $[0,\tau]$.  By setting $s \sim t$ if and only if $d_X(s,t) = 0$, one can define a metric space by considering $[0,\tau] / \sim$.  Given $X$, we let $Y^0$ be the Gaussian process with covariance function
\[ \cov( Y_s^0, Y_{s'}^0) = \inf_{u \in [s \wedge s', s \vee s']} (X_u - \ul{X}_u)\]
where $\ul{X}_u = \inf_{0 \leq v \leq u} X_v$.  We also let $B$ be a Brownian bridge of duration $\ell$ so that
\begin{equation}
\label{eqn:bb_cov}
\cov(B_s, B_t) = \frac{s(\ell-t)}{\ell} \quad\text{for}\quad 0 \leq s \leq t \leq \ell.
\end{equation}
Finally, we let
\[ Y_s = Y_s^0 + \sqrt{3} B_{T^{-1}(s)}\]
where $T^{-1}$ denotes the right-continuous inverse of the local time of $X$ at its running infimum.  The Brownian disk is then defined in terms of a metric quotient from $(X,Y)$ in the same way that the Brownian map is defined as a metric quotient from the Brownian snake in that case (recall Section~\ref{subsec:brownian_map}).  Let $\rho_{\BD} \colon [0,\tau] \to \CD$ be the natural projection map.  The marked point $y$ of the Brownian disk is given by $\rho_\BD(s^*)$ where $s^*$ is the a.s.\ unique point where $Y$ attains its overall infimum.  The image of the set of times at which $X$ hits a record minimum under $\rho_\BD$ gives the boundary $\partial \CD$ of $\CD$.  Let $\rho_{\BD,\partial} \colon [0,\ell] \to \partial \CD$ be the natural projection map to $\partial \CD$.  Then the boundary length measure $\nu_\partial$ on $\partial \CD$ is given by the projection of Lebesgue measure on $[0,\ell]$ to $\partial \CD$ under $\rho_{\BD,\partial}$.  If $x \in \partial \CD$ is given by $\rho_{\BD,\partial}(s)$ for some $s \in [0,\ell]$ then $d(x,y) = \sqrt{3} B_s - Y_{s^*}$.  That is, $B$ serves to encode the distance of points on $\partial \CD$ to $y$.  In particular, the unique place where $B$ attains its overall infimum corresponds to the unique boundary point which is closest to $y$.

It follows from the Brownian snake construction that the law of the area of a sample from $\bdisklawweighted{\ell}$ is equal to the law of the amount of time it takes a standard Brownian motion on $\R$ starting from $0$ to hit $-\ell$.  Recall that the density for this law with respect to Lebesgue measure on $\R_+$ at $a$ is given by
\begin{equation}
\label{eqn:bdiskweighted_area_density}
\frac{\ell}{\sqrt{2 \pi a^3}} \exp\left(-\frac{\ell^2}{2a} \right).	
\end{equation}

The following is a restatement of \cite[Lemma~3.2]{gm2019gluing}, which we will use several times later in this article.  In the lemma statement, for $u,v \in \partial \CD$ we let $[u,v]_{\partial}$ denote the counterclockwise arc in $\partial \CD$ from $u$ to $v$.

\begin{lemma}[\cite{gm2019gluing} Lemma 3.2]
\label{lem:boundary_length_distance}
Fix $\ell >0$ and suppose that $(\CD,d,\nu,\nu_\partial,y)$ has law $\bdisklawweighted{\ell}$.  For every $\zeta > 0$ there a.s.\ exists $C > 0$ so that for all $u,v \in \partial \CD$ we have that
\[ d(u,v) \leq C \nu_{\partial}([u,v]_{\partial})^{1/2} \left(	|\log \nu_{\partial}([u,v]_{\partial}) | + 1\right)^{7/4+\zeta}.\]
Moreover, if $C$ is the smallest constant so that the above holds for all $u,v \in \partial \CD$ then we have that $\bdisklawweighted{\ell}[ C > A]$ decays to $0$ as $A \to \infty$ faster than any negative power of $A$.
\end{lemma}

\section{Close geodesics must intersect near their endpoints}
\label{sec:strong_confluence}

The purpose of this section is to prove the following proposition, which is a key ingredient in proving Theorem~\ref{thm:strong_confluence}.

\begin{proposition}
\label{prop:strong_confluence}
There exists $c>0$ such that the following holds for $\bminflaw$ a.e.\ instance of the Brownian map $(\CS,d,\nu,x, y)$.  There exists $\epsilon_0 > 0$ so that for all $\epsilon \in (0,\epsilon_0)$ the following is true. Let $\delta =  c \eps \log \eps^{-1}$. Suppose that $\eta_i \colon [0,T_i] \to \CS$ for $i=1,2$ are geodesics with $T_i = d(\eta_i(0), \eta_i(T_i)) \geq 2\delta$  and
$ \distHos(\eta_1([0, T_1]),\eta_2([0,T_2])) \leq \epsilon.$  
Then 
\begin{align}\label{eq:prop_str_conf}
 \eta_i( (0, \delta]) \cap \eta_{3-i} \neq \emptyset \quad \text{and} \quad  \eta_i([T_i-\delta, T_i)) \cap \eta_{3-i} \neq \emptyset \quad \text{for} \quad i=1,2.
\end{align}
\end{proposition}

We begin by showing in Section~\ref{subsec:x_with_pos_prob} that for a geodesic between typical points, a certain event which forces nearby geodesics to merge occurs within a metric band of unit width and inner boundary length bounded below with uniformly positive probability.  Then in Section~\ref{subsec:x_concentration}, we will combine this with a concentration argument and the independence property across metric bands when performing a reverse metric exploration  to show that with overwhelming probability this event happens on a very dense set of points along a geodesic starting from a typical point.
Finally in Section~\ref{sec:proof_prop_3.1}, we will complete the proof of Proposition~\ref{prop:strong_confluence} by putting in between $\eta_1$, $\eta_2$ a geodesic starting from a typical point so that it forces $\eta_1$, $\eta_2$ both to merge with it hence intersect each other.

Throughout the section, when there is a unique geodesic from $a$ to $b$ in the Brownian map or a metric band, we denote it by $[a,b]$. If we have not ruled out the possibility of multiple geodesics between $a$ and $b$, we still sometimes use $[a,b]$ to denote one of the geodesics from $a$ to $b$ which we will make clear from the context.  Whenever we have defined a geodesic $[a,b]$, then for any points $c, d \in[a,b]$, we use $[c,d]$ to denote the  geodesic from $c$ to $d$ which is a subset of $[a,b]$.

\subsection{\X ~occurs with positive probability within a band}
\label{subsec:x_with_pos_prob}

Let $(\CS, d, \nu, x, y)$ be an instance of the Brownian map and let $\eta$ be a geodesic in $\CS$. We say that there exists an \X\ along $\eta$ if there exist $a,b,c,d \in \CS$ so that the following hold. See Figure~\ref{fig:band_x} (Left).
\begin{enumerate}[(i)]
\item\label{itm:X1} There is a unique geodesic from $a$ to $c$ which intersects $\eta$ on its left side. Moreover, there exist $e, g\in\eta$ so that $[a,c]\cap \eta=[e,g]$ and $[a,e]\cap [c,g]=\emptyset$ so that $e$ (resp.\ $g$) is the first place on $\eta$ hit by $[a,c]$ (resp.\ $[c,a]$).
\item\label{itm:X2} There is a unique geodesic from $b$ to $d$ which intersects $\eta$ on its right side. Moreover, there exist $f, h\in \eta$ so that $[b,d]\cap \eta=[f,h]$ and $[b,f]\cap [h,d]=\emptyset$ so that $f$ (resp.\ $h$) is the first place on $\eta$ hit by $[b,d]$ (resp.\ $[d,b]$).
\item\label{itm:X3} The intersection between $[e,g ]$ and $[f,h]$ is non-empty.
\end{enumerate}
We call $[e,g]\cup [f,h]$ the \emph{center} of the \X\ and call $[a,e]$, $[b,f]$, $[g,c]$, $[h,d]$ the four \emph{branches} of the \X.
We say that the \emph{size} of the {\X} is at least $s > 0$ if the balls of radius $s$ centered at $a, b, c, d$ are all disjoint from $\eta$. 
\begin{figure}[h!]
\centering
\includegraphics[scale=0.8]{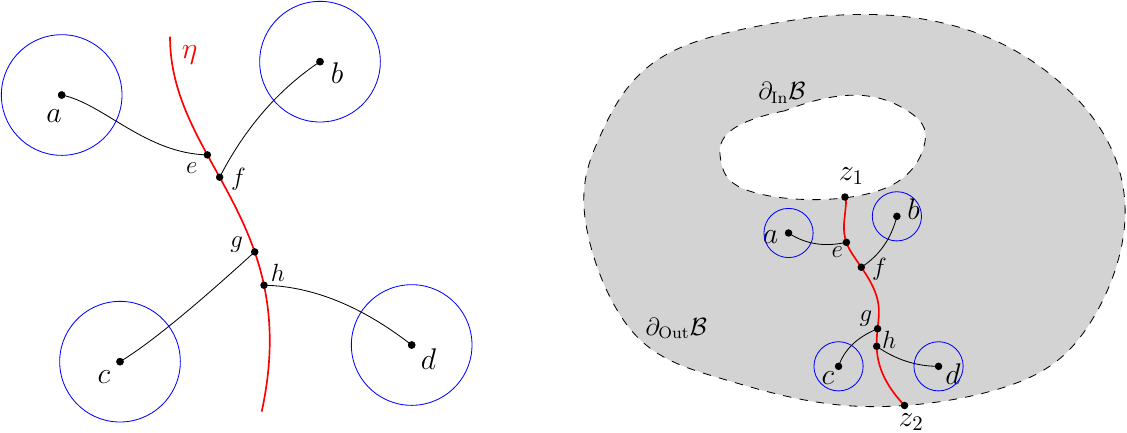}
\caption{\label{fig:band_x} \textbf{Left:} An \X\ along $\eta$ in a Brownian map of size at least $s$. \textbf{Right:} An {\X} in a metric band along the geodesic $[z_1,z_2]$.}
\end{figure}

In the remainder of this subsection, we will mostly be interested in \X's in a metric band.
Fix $\ell, w > 0$  and let $(\CB,d_\CB,\nu_\CB,z_1)$ be sampled according to $\bandlaw{\ell}{w}$.  There is a.s.\ a unique geodesic in $\CB$ from~$z_1$ to $\outerboundary \CB$.  Let $z_2$ be the terminal point of this geodesic in $\outerboundary \CB$.  We say that there exists an {\X} along $[z_1,z_2]$ in $\CB$ if there exist $a, b, c,d \in \CB$ so that the conditions \eqref{itm:X1}--\eqref{itm:X3} hold for $\eta=[z_1, z_2]$. See Figure~\ref{fig:band_x} (Right).
We say that the \emph{size} of the {\X} is at least $s > 0$ if the balls of radius $s$ centered at $a, b, c, d$ are all disjoint from $\eta$, and also disjoint from $\partial \CB$. 
Note that if $\CB$ is a band in a Brownian map instance, then the balls $B(a,s)$, $B(b,s)$, $B(c, s)$, $B(d, s)$ with respect to $d_\CB$ are also balls with respect to the Brownian map metric.  In particular, if there is an \X\ of size at least $s$ along a geodesic $[z_1, z_2]$ in a  metric band which is embedded in a Brownian map and $a,b,c,d$ are as above, then $[a,c]$ and $[b,d]$ form an \X\ of size at least $s$ inside the Brownian map along any geodesic which contains $[z_1, z_2]$.

Let $E(\CB,z_1,s)$ be the event that there exists an $\X$ along $[z_1,z_2]$ in $\CB$ of size at least~$s$.  Note that this event is measurable with respect to $(\CB,d_\CB,\nu_\CB,z_1)$.

Let us first prove the following lemma.

\begin{lemma}
\label{lem:X_prob}
There exists $\ell_0,s_0,p_0 > 0$ so that for all $\ell \geq \ell_0$, we have 
\begin{align*}
\bandlaw{\ell}{1}[E(\CB, z_1,s_0)] \geq p_0.
\end{align*}
\end{lemma}
Our strategy is to first construct an instance of the Brownian map $(\CS,d,\nu,x,y)$ which exhibits a certain behavior with positive ``probability'' under $\bminflaw$. Since $\bminflaw$ is an infinite measure, what we actually construct is $(\CS,d,\nu,x,y)$ conditioned on a certain event with $\bminflaw$ measure in $(0,\infty)$. 
Next,  we explore the conditioned $(\CS,d,\nu,x,y)$ in a breadth-first way. We show that with positive (conditional) probability, we can cut out a metric band $\CB$ which contains an \X. Finally, we complete the proof by adjusting the width and length of the metric band.

\begin{proof}

Let $(\CS,d,\nu,x,y)$ have the law of $\bminflaw$ conditioned on the event $\{d(x,y) > 1\}$.  Note that this corresponds to conditioning the  $3/2$-stable CSBP excursion associated with the reverse metric exploration from $y$ to $x$ (and also from $x$ to $y$) on having time length at least $1$. Since  $\bminflaw(d(x,y) > 1)\in (0,\infty)$, we obtain a probability measure after performing the conditioning.  We fix $r > 0$ small and then perform a reverse metric exploration from $y$ to $x$ (resp.\ from $x$ to $y$) and consider $\fb{y}{x}{t_0}$ (resp.\ $\fb{x}{y}{t_0}$) where $t_0 = d(x,y) - r$.  We can choose $r>0$ sufficiently small so that it is a positive probability event that $\partial \fb{y}{x}{t_0}$ and $\partial\fb{x}{y}{t_0}$ are disjoint (note that the probability of this event tends to $1$ as $r\to 0$, since we have conditioned on $d(x,y)>1$). From now on, we further condition on this event (see Figure~\ref{fig:construction}, Left).

Let $x_1$ (resp.\ $y_1$) be the a.s.\ unique intersection point of $[x,y]$ with $\partial \fb{x}{y}{t_0}$ (resp.\ $\partial \fb{y}{x}{t_0}$).  
If $(a_n)$ is a sequence of points on $\partial \fb{y}{x}{t_0}$ so that the boundary length of the counterclockwise arc from $y_1$ to $a_n$ is equal to $1/n$ and $u_n$ is the point where the unique geodesic from $a_n$ to $x$ merges with $[y_1,x]$, then $u_n \to y_1$ as $n \to \infty$.  Similarly, if $(b_n)$ is a sequence of points on $\partial \fb{x}{y}{t_0}$ so that the boundary length along the clockwise arc of $\partial \fb{x}{y}{t_0}$ from $b_n$ to $x_1$ is equal to $1/n$ and $v_n$ is where the unique geodesic from $b_n$ to $y$ merges with $[x_1,y]$, then we have that $v_n \to x_1$ as $n \to \infty$.  Therefore there exists $N \in \N$ so that if we let $a = a_N$, $b = b_N$, $u=u_N$, and $v=v_N$, then the geodesics $[a,x]$ and $[b,y]$ overlap on a non-trivial interval $[u,v]$.
Note that $[x,y], [a,x], [b,y]$ are each the unique geodesic connecting their endpoints, but it is \emph{a priori} not true that the concatenation of $[a,v]$ and $[v,b]$ is a geodesic from $a$ to $b$.

\begin{figure}[h!]
\centering
\includegraphics[scale=0.85]{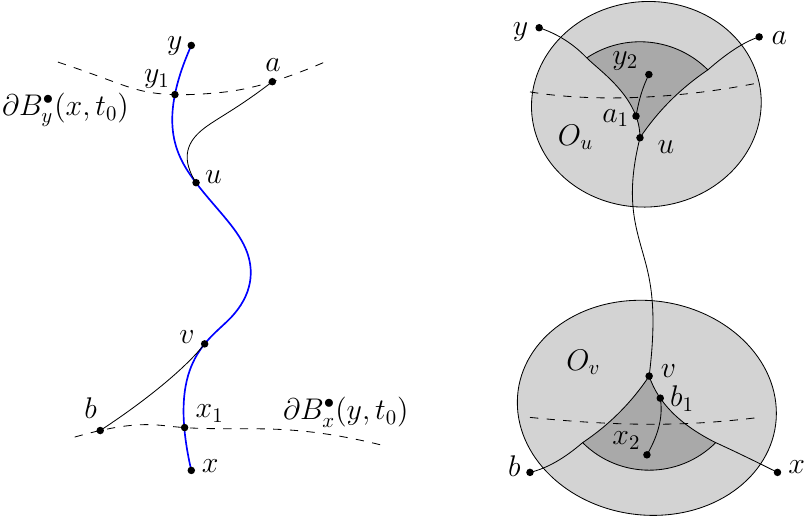}
\caption{\label{fig:construction} Construction of an \X.}
\end{figure}

In the following, we will use the root invariance of the Brownian map to resample the root and dual root of $(\CS,d,\nu,x,y)$, so that with positive probability, we get a new instance $(\CS,d,\nu,x_2,y_2)$ that exhibits some desired properties (see Figure~\ref{fig:construction}, Right).

Choose a neighborhood $O_u$ of $u$ and a neighborhood $O_v$ of $v$, so that (i) $y, a\not\in O_u$, $b,x\not\in O_v$, (ii) $O_u, O_v$ are disjoint, and (iii)  $O_u \setminus ([y,u] \cup [a,u] \cup [u,v])$ and $O_v \setminus ([x,v] \cup [b,v] \cup [u,v])$ each have three connected components (this is possible since the Brownian map is homeomorphic to the sphere $\s^2$).  Let $\wt O_u$ (resp.\ $\wt O_v$) denote the connected component of  $O_u \setminus ([y,u] \cup [a,u] \cup [u,v])$ (resp.\ $O_v \setminus ([x,v] \cup [b,v] \cup [u,v])$)  which is bounded between $[y,u]$ and $[a,u]$ (resp.\ $[b,v]$ and $[x,v]$). Let $l_2=\partial \wt O_u \cap [y,u]$, $l_3=\partial \wt O_u \cap[u,a]$, $l_4=\partial \wt O_v \cap[v,b]$, $l_5=\partial \wt O_v \cap[v,x]$. 
Let $l_1=\partial \wt O_u \setminus (l_2 \cup l_3) $ and  $l_6=\partial \wt O_v \setminus (l_4 \cup l_5) $. 
See Figure~\ref{fig:X}.
Let $(O_u^n)$ (resp.\ $(O_v^n)$) be a sequence of such neighborhoods of $u$ (resp.\ $v$) contained in  $O_u$ (resp.\ $O_v$) shrinking to the point $u$ (resp.\ $v$) as $n\to\infty$.
For each $n \in \N$, among the three connected components of $O_u^n \setminus ([y,u] \cup [a,u] \cup [u,v])$ (resp.\ $O_v^n  \setminus ([x,v] \cup [b,v] \cup [u,v])$), let $\wt O_u^n$ (resp.\ $\wt O_v^n$) denote the one which is bounded between $[y,u]$ and $[a,u]$ (resp.\ $[b,v]$ and $[x,v]$).
The rest of this paragraph is dedicated to the proof of the following claim.
\begin{enumerate}
\item[$(*)$] It is a.s.\ the case that there exists $N \in\N$ such that for any $y' \in O_u^N$ and $x'\in O_v^N$, any geodesic from $y'$ to $x'$ is contained in $O_u \cup O_v \cup [u,v]$. 
\end{enumerate}
Let us prove $(*)$ by contradiction, and suppose that it is not the case. Then with positive probability, for all $n\in\N$, there exist $y_n \in O_u^n$ and $x_n \in O_v^n$ and a geodesic $[x_n, y_n]$ from $x_n$ to $y_n$ which is not contained in $O_u \cup O_v \cup [u,v]$.
Since $\CS$ is compact, by the Arzel\`a-Ascoli theorem, there is a subsequence $(n_k)$ for which $[x_{n_k}, y_{n_k}]$ converges to a limiting geodesic $\eta$. Since $(x_n)$ and $(y_n)$ respectively converge to $v$ and $u$, the endpoints of $\eta$ are  $u$ and $v$. Since there is a unique geodesic between $u$ and $v$ (because there is a unique geodesic between $x$ and $y$),
$\eta$ must be equal to $[u,v]$. 
We must be in (at least) one of the following situations, depending on where $[y_{n_k}, x_{n_k}]$ (resp.\ $[x_{n_k}, y_{n_k}]$) first hits $\partial \wt O_u$ (resp.\ $\partial \wt O_v$). 
\begin{enumerate}[(i)]
\item See Figure~\ref{fig:X} (b). There exists $k\in\N$ such that $[y_{n_k}, x_{n_k}]$ first hits $\partial \wt O_u$ at some point $u_{n_k}\in l_2$, and $[x_{n_k}, y_{n_k}]$ first hits $\partial \wt O_v$ at some point $v_{n_k}\in l_4$.  Then, due to the uniqueness of the geodesic $[y,b]$, the part of the geodesic $[y_{n_k}, x_{n_k}]$ between $u_{n_k}$ and $v_{n_k}$ must coincide with $[y,b]$. In this case $[x_{n_k}, y_{n_k}] \subset O_u \cup O_v \cup [u,v]$. 

\item There exists $k\in\N$ such that $[y_{n_k}, x_{n_k}]$ first hits $\partial\wt O_u$ in $l_3$ and $[x_{n_k}, y_{n_k}]$ first hits $\partial\wt O_v$ in $l_5$. This case is similar to (i), due to the uniqueness of the geodesic $[a,x]$.  In this case $[x_{n_k}, y_{n_k}] \subset O_u \cup O_v \cup [u,v]$. 

\item There exists $k\in\N$ such that $[y_{n_k}, x_{n_k}]$ first hits $\partial\wt O_u$ in $l_2$ and $[x_{n_k}, y_{n_k}]$ first hits $\partial\wt O_v$ in $l_5$. This case is similar to (i) due to the uniqueness of the geodesic $[y,x]$. In this case $[x_{n_k}, y_{n_k}] \subset O_u \cup O_v \cup [u,v]$. 

\item See Figure~\ref{fig:X} (a). For all $k_0\in\N$ there exists $k\ge k_0$ such that either $[y_{n_k}, x_{n_k}]$ first hits $\partial\wt O_u$ in $l_1$ or $[x_{n_k}, y_{n_k}]$ first hits $\partial\wt O_v$ in $l_6$. In this case, the Hausdorff distance between $[y_{n_k}, x_{n_k}]$ and $[u,v]$ is bounded away from $0$, hence it is impossible.

\item For $k\in\N$ large enough,  $[y_{n_k}, x_{n_k}]$ first hits $\partial\wt O_u$ at some point $u_{n_k}\in l_3$, and $[x_{n_k}, y_{n_k}]$ first hits $\partial\wt O_v$ at some point $v_{n_k}\in l_4$. See Figure~\ref{fig:X} (c). The difference with (i)-(iii) is that the concatenation of $[a,v]$ and $[v,b]$ is \emph{a priori} not necessarily a geodesic from $a$ to $b$.  However, we will show that, there exists $k\in\N$ such that the part between $u_{n_k}$ and $v_{n_k}$ of the geodesic $[y_{n_k}, x_{n_k}]$ in fact coincides with the concatenation of $[u_{n_k},v]$ and $[v, v_{n_k}]$.
Suppose that it is not true. For each $k\in \N$, let $u_{n_k}'$ (resp.\ $v_{n_k}'$) be the point where $[u_{n_k}, v_{n_k}]$ (resp.\ $[v_{n_k}, u_{n_k}]$) first leaves $[u_{n_k},v] \cup [v, v_{n_k}]$.
 Due to the uniqueness of $[a,x]$ and $[b,y]$, the part $[u_{n_k}', v_{n_k}']$
 has to go around $[b,y] \cup [a,x]$, see Figure~\ref{fig:X} (d). However, in this case, the Hausdorff distance between $[y_{n_k}, x_{n_k}]$ and $[u,v]$ is bounded away from $0$, hence this is impossible. This implies that there exists $k\in\N$ such that $[y_{n_k}, x_{n_k}]$ follows $[u,v]$, hence  $[x_{n_k}, y_{n_k}] \subset O_u \cup O_v \cup [u,v]$. 
\end{enumerate}
\begin{figure}[h!]
\centering
\includegraphics[width=\textwidth]{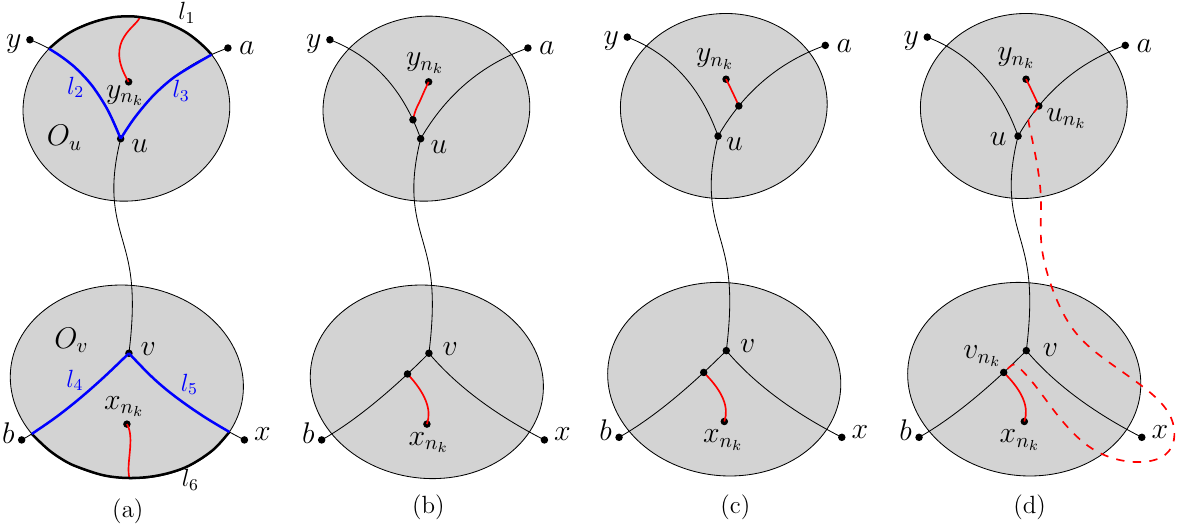}
\caption{\label{fig:X} The different cases in the proof of $(*)$.}
\end{figure}
In each of the above cases (when it is possible), we have $[x_{n_k}, y_{n_k}] \subset O_u \cup O_v \cup [u,v]$,  which contradicts the definition of $[y_{n_k}, x_{n_k}]$. We have therefore proved $(*)$.

Now, let us resample the root and dual root of $\CS$ independently according to $\nu$. Since $\wt O^N_u$ and $\wt O^N_v$ have non-zero area with respect to $\nu$, there is a positive probability that the new root $x_2$ is in $\wt O^N_v$ and the new dual root $y_2$ is in $\wt O^N_u$. From now on, we condition on this event. Let $a_1$ be the point at which the geodesic $[y_2, x_2]$ first hits $[y,u] \cup [u,a]$, and let $b_1$ be the point at which the geodesic $[x_2, y_2]$ first hits $[x,v] \cup [b,v]$ (these points exist, by the definition of $\wt O^N_u$ and $\wt O^N_v$). Moreover, the geodesic $[a_1, b_1]$ is also a subset of $[y,b]\cup [a,x]$, see Figure~\ref{fig:construction}.

The lengths of the three parts $[y_2, a_1], [a_1, b_1], [b_1, x_2]$ form a triple of random variables $(L_1, L_2, L_3)$, taking values in $(\R_{>0})^3$. There exists $\eps>0$ sufficiently small and $\ell_1, \ell_2, \ell_3 > 5\eps$ such that there is a positive probability that $L_i\in (\ell_i-\eps, \ell_i+\eps)$ for $i=1,2,3$. We further condition on this event.  Let $r=\ell_1 -\eps$ and $w=\ell_2 + 3 \eps$.  We perform a reverse metric exploration from $y_2$ to $x_2$ and consider the metric band $\CB=\fb{y_2}{x_2}{s} \setminus \fb{y_2}{x_2}{s-w}$, where $s=d(x_2, y_2) -r$. Due to our conditioning, we have
\begin{align*}
L_1>r, \quad L_1+L_2<r+w,  \quad d(x_2,y_2)=L_1+L_2+L_3>r+w.
\end{align*}
This implies that there is an \X~in $\CB$ along the geodesic $[z_1, z_2]$, where $z_1$ (resp.\ $z_2$) is the unique intersection point of $[x_2,y_2]$ with $\innerboundary \CB$ (resp.\ $\outerboundary \CB$). Note that $\CB$ has width $w$ and a random length $L_0>0$. It follows that there exist $\wt \ell_0, \wt s_0, p_0>0$ such that 
\[ \bandlaw{\wt \ell_0}{w} \!\left[ E(\CB, z_1, \wt s_0)\right] >  p_0.\]
By applying the scaling property for metric bands and letting $\ell_0=\wt \ell_0/w^2$, $s_0=\wt s_0/w$, we have
\begin{align}
\label{eq:eb1}
\bandlaw{\ell_0}{1} \!\left[ E(\CB, z_1, s_0) \right] >p_0.
\end{align}
\begin{figure}[t]
\centering
\includegraphics[width=\textwidth]{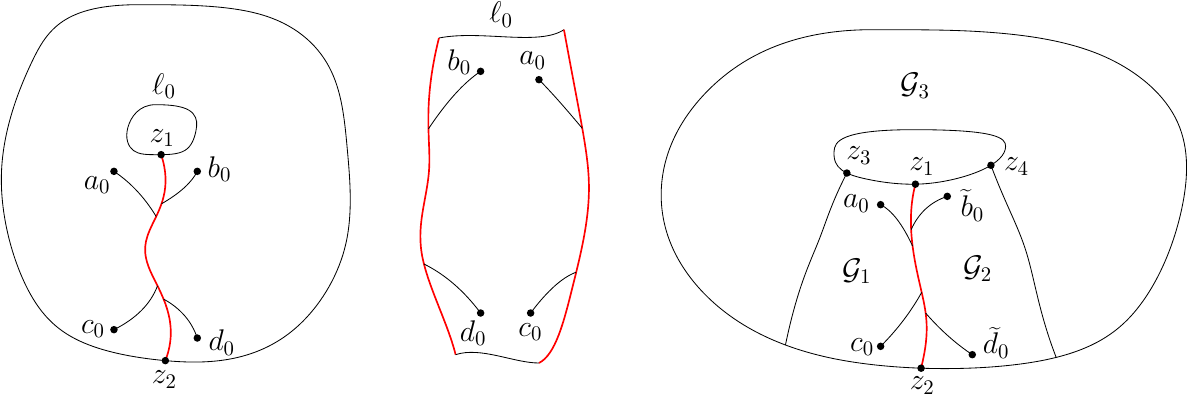}
\caption{\label{fig:X_distance} Gluing three independent slices to construct a band with an arbitrarily big inner boundary length which contains an \X.}
\end{figure}
To complete the proof, let us show that for every $\ell > 2\ell_0$, we also have
\begin{align}\label{eq:eb2}
\bandlaw{\ell}{1} \!\left[ E(\CB, z_1, s_0)\right] > p_0^2.
\end{align}
Fix $\ell>2\ell_0$. To show \eqref{eq:eb2}, we will construct a band $\CB$ with law $\bandlaw{\ell}{1}$ by gluing together three independent slices, as in Figure~\ref{fig:X_distance}.
Suppose that $\CB_1$ is a band with law $\bandlaw{\ell_0}{1}$ for which $E(\CB_1, z_1, s_0)$ holds. Let $[z_1, z_2]$ be the unique geodesic in $\CB_1$ from $z_1$ to $\outerboundary\CB_1$.
Then there exist a geodesic $[a_0, c_0]$ on the left of $[z_1, z_2]$ and a geodesic $[b_0, d_0]$ on the right of $[z_1, z_2]$ which form an \X~with size at least $s_0$. In particular, the balls of radius $s_0$ around $a_0, b_0, c_0, d_0$ are disjoint from $[z_1, z_2], \innerboundary\CB_1, \outerboundary\CB_1$.
If we cut open $\CB_1$ along the geodesic $[z_1, z_2]$, then we get a slice $\CG_1$ with inner boundary length $\ell_0$ and width $1$. The geodesics $[a_0, c_0]$ and $[b_0, d_0]$ in $\CB_1$ remain geodesics in $\CG_1$. Now, we also sample an independent copy $\CG_2$ with the same law as $\CG_1$, equipped with two geodesics $[\wt a_0, \wt c_0]$ and $[\wt b_0, \wt d_0]$. Let $\CG_3$ be a third independent slice with law $\slicelaw{\ell-2\ell_0}{1}$. Gluing together the slices $\CG_1, \CG_2$ and $\CG_3$ in a counterclockwise manner as in Figure~\ref{fig:X_distance}, we get a metric band $\CB$ with inner boundary length $\ell$ and width $1$. Let $z_1, z_3, z_4$ be the three points on $\innerboundary\CB$ which mark the separation between the slices,  as in Figure~\ref{fig:X_distance}.
Note that the geodesic $[a_0, c_0]$ in $\CG_1$ remains a geodesic in $\CB$, because the geodesic between $a_0$ and $c_0$ in $\CB$ cannot cross any of the two geodesics which bound $\CG_1$, i.e., the geodesics from $z_1$ and $z_3$ to $\outerboundary\CB$, hence must stay in $\CG_1$. Indeed, if the geodesic from $a_0$ to $c_0$ in $\CB$ crosses the geodesic starting from $z_1$ (resp.\ $z_3$), then it will create more than one geodesic from  $z_1$ (resp.\ $z_3$) to $\outerboundary\CB$, which is impossible. Similarly, the geodesic $[\wt b_0, \wt d_0]$ in $\CG_2$ is also a geodesic in $\CB$. Therefore, the geodesics $[a_0, c_0]$ and $[\wt b_0, \wt d_0]$ form an \X~in $\CB$ along $[z_1,z_2]$.
It is easy to see that this \X~also has size at least $s_0$. To summarise, we have constructed $\CB$ for which $E(\CB, z_1, s_0)$ occurs, using two independent bands $\CB_1$ and $\CB_2$ for which $E(\CB_1, z_1, s_0)$ and $E(\CB_1, z_1, s_0)$ respectively occur, and a third slice $\CG_3$ independent of everything else. Due to~\eqref{eq:eb1}, we can deduce~\eqref{eq:eb2}.
This completes the proof of the lemma, with $2\ell_0$ and $p_0^2$ in place of $\ell_0$ and $p_0$.
\end{proof}

Now, let us consider an event where we impose some further conditions on the \X.
Fix $\ell > 0$ and let $(\CB,d_\CB,\nu_\CB,z_0)$ be sampled according to $\bandlaw{\ell}{7}$ and let $d_0$ be the distance between $\innerboundary \CB$ and $\outerboundary \CB$.  Note that $\CB$ is the union of the three bands $\CB_1, \CB_2, \CB_3$ where 
\begin{itemize}
\item $\CB_3$ is  the set of points in $\CB$ disconnected from $\innerboundary\CB$ by the $(d_0-4)$-neighborhood of $\outerboundary\CB$. If $d_0\le 4$, then let $\CB_3$ be empty and let $\innerboundary\CB_3$, $\outerboundary\CB_3$ be equal to the point $\outerboundary \CB$. Let $d_3$ be the distance between $\innerboundary \CB_3$ and $\innerboundary \CB$.
\item $\CB_2$ is the set of points in $\CB \setminus \CB_3$ disconnected from $\innerboundary\CB$ by the $(d_3-3)$-neighborhood of $\innerboundary\CB_3$. If $d_3\le 3$, then let $\CB_2$ be empty and let $\innerboundary\CB_2$, $\outerboundary\CB_2$ be equal to the point $\outerboundary \CB$.
\item $\CB_1$ is the set of points in $\CB \setminus (\CB_2 \cup \CB_3)$.
\end{itemize}
There is a.s.\ a unique geodesic $\eta$ in $\CB$ from $z_0$ to $\outerboundary \CB$.  
For $s_0>0$ and $s_1\in (0, s_0/2)$, let $F(\CB, s_0, s_1)$ be the following event. See Figure~\ref{fig:goodX}.
\begin{enumerate}[(i)]
\item We have $d_0=7$.
\item There is an \X\ of size at least $s_0$ along $\eta([3, 4])$ in the band $\CB_2$.
\item\label{itm: U_s_1} Let $U$ denote the $s_1$-neighborhood of $\eta([1, 6])$ with respect to  $d_\CB$. Then $\ol U$ is disjoint from $\innerboundary\CB\cup\outerboundary\CB$ and does not disconnect $\innerboundary\CB$ from $\outerboundary\CB$. Moreover, every connected component of $\CB\setminus U$ whose closure is disjoint from $\innerboundary\CB \cup \outerboundary\CB$ has diameter at most $s_0/2$.
\end{enumerate}

\begin{figure}[h!]
\centering
\includegraphics[width=.7\textwidth]{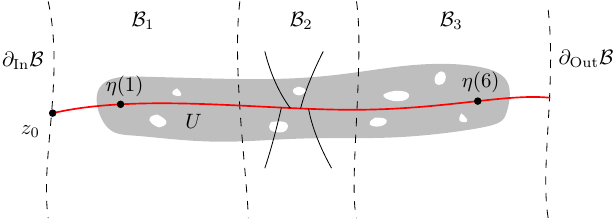}
\caption{Illustration of the event $F(\CB, s_0, s_1)$.}
\label{fig:goodX}
\end{figure}

\begin{lemma}\label{lem:X_prob2}
There exists $\ell_1,s_0, p_1 > 0$ and $s_1\in(0,s_0/2)$ so that for all $\ell \geq \ell_1$, we have 
\begin{align*}
\bandlaw{\ell}{7}[F(\CB, s_0, s_1)] \geq p_1.
\end{align*}
\end{lemma}
\begin{proof}
Let $\ell_0, s_0, p_0$ be as in Lemma~\ref{lem:X_prob}.
Fix $\ell_1=2 \ell_0$ so that with positive probability, the boundary length of $\innerboundary \CB_2$ is at least $\ell_0$. By  Lemma~\ref{lem:X_prob}, the event $E(\CB_2, \eta(3), s_0)$ happens with probability at least $p_0$. Conditionally on $E(\CB_2, \eta(3), s_0)$, the law of $\CB_3$ depends on $\CB_1\cup \CB_2$ only through $\outerboundary \CB_2$. There is again a positive conditional probability that the distance between $\innerboundary \CB_3$ and $\outerboundary \CB_3$ is at least $3$, so that $d_0=7$. We further condition on this event.
For $s>0$, let $U_s$ denote the $s$-neighborhood of $\eta([1, 6])$ with respect to  $d_\CB$. Almost surely, as $s\to 0$, $U_s$ converges in Hausdorff distance to $\eta([1, 6])$ which is disjoint from $\innerboundary\CB \cup \outerboundary\CB$ and does not contain any holes. Therefore, there exists $s_1>0$ such that the event~\eqref{itm: U_s_1} occurs with positive conditional probability. This completes the proof.
\end{proof}

Finally, let us turn back to the notion of an \X\ in a Brownian map instance $(\CS, d, \nu, x,y)$.
For each $s_0, s_1, \delta>0$, for any geodesic $\eta$ in the Brownian map with length $T>0$ and any \X\ along $\eta$ in $\CS$, we say that this \X\ is \emph{$(s_0, s_1,  \delta)$-good} for $\eta$ if the following holds. 
\begin{enumerate}[(i)]
\item\label{itm:goodX} The size of this \X\ is at least $s_0$.
\item\label{itm:goodX_center} The center of this \X\ is $\eta([u_1, u_2])$ where $T-\delta >u_2 > u_1 >\delta$. 
\item\label{itm:goodX_hole} Let $U_0$ be the $s_1$-neighborhood of $\eta([u_1-\delta, u_2+\delta])$. Let $\CF(U_0)$ be the complement in $\CS$ of the $\eta(0)$-containing connected component of $\CS \setminus U_0$. None of the four branches of the \X\ is contained in $\CF(U_0)$.
\item\label{itm:goodX_dist} For any $t\le u_1-\delta/2$ (resp.\ $t\ge u_2+\delta/2$) such that $t\in[0,T]$, the distance from $\eta(t)$ to any of the four branches of the \X\ is at least $u_1-\delta/2 - t+s_0$ (resp.\ $t-u_2-\delta/2 + s_0$). 
\end{enumerate}

\begin{lemma}
\label{lem:good_band_map}
Suppose that $(\CB,d_\CB,\nu_\CB,z_0)$ is distributed according to $\bandlaw{\ell}{7}$ and is embedded in a Brownian map instance $\CS$. Suppose that $\eta$ is a geodesic in $\CS$ which contains the geodesic in $\CB$ from $z_0$ to $\outerboundary\CB$. Then on the event $F(\CB, s_0, s_1)$, the \X\ in $\CB$ is an $(s_0, s_1, 2)$-good \X\  along $\eta$ inside $\CS$.
\end{lemma}
\begin{proof}
Suppose that $z_0=\eta(t_0)$ for some $t_0\ge 0$. Let $T$ be the time length of $\eta$ so that $T \ge t_0 + d$ where $d$ is the distance between $\innerboundary \CB$ and $\outerboundary \CB$.
On the event $F(\CB, s_0, s_1)$, there is an \X\ along $\eta$ contained in the middle part $\CB_2$ of $\CB$, hence
the center of the \X\ is $\eta([u_1, u_2])$ where $t_0+3 \le u_1< u_2 \le t_0 +4$. 
It is clear that~\eqref{itm:goodX} and~\eqref{itm:goodX_center} hold. 

For any point $z\in \CF(U_0)$, since $U_0$ does not contain any hole with diameter at least $s_0/2$, we have
\[ d(z, \eta([t_0+1,t_0+6])) \le d(z, \partial U_0) + d(\partial U_0,  \eta([t_0+1,t_0+6])) \le s_0/2 + s_1 \le s_0.\]
Since the size of the \X\ is at least $s_0$, each branch of the \X\ must exit $\CF(U_0)$. Thus~\eqref{itm:goodX_hole}  holds.

For any $t\le t_0 +3$, the distance from $\eta(t)$ to any branch of the \X\ is at least the distance from  $\eta(t)$ to $\innerboundary \CB_2$ plus the distance from $\innerboundary \CB_2$ to any branch of the \X. This is at least $t_0+3-t+ s_0 \ge u_1 -1-t +s_0$. Since $u_1-1\le t_0+3$, the same is true for any $t\le u_1-1$. Similarly, we can deduce that for any $t \ge u_2 +1$, the distance $\eta(t)$ to any branch of the \X\ is at least $t-u_2-1+s_0$. Thus~\eqref{itm:goodX_dist}  holds.
\end{proof}

\subsection{There are \X's everywhere along a geodesic between typical points}
\label{subsec:x_concentration}

The goal of this subsection is to prove Lemma~\ref{lem:x_concentration}, as stated below.

\begin{lemma}
\label{lem:x_concentration}
Let  $s_0, s_1$ be as in Lemma~\ref{lem:X_prob2}. Fix $r_0>0$. There exist $\sp{b}_0, c_0>0$ and $\eps_0>0$ so that for all $r\in(0,r_0)$, $\eps\in(0,\eps_0)$ and $c>c_0$ the following is true.  
Suppose that $(\CS,d,\nu,x,y)$ is sampled from $\bminflaw$ conditioned on $\{d(x,y) > r\}$. We denote this conditional probability measure by $\mathbf{P}_{r}$.
The $\p_{r}$ probability of the following event is at most $\eps^{\sp{b}_0 c}$.  There exist $t\in (0, d(x,y)-r)$, a geodesic $\eta$ from $\partial \fb{y}{x}{t}$ to $x$ and times $2\eps \leq t_1 < t_1  +c \epsilon \log \eps^{-1} \leq t_2 \leq t-2\eps$, so that
$\eta|_{[t_1,t_2]}$ does not pass through any \X\ which is $(s_0\eps, s_1 \eps, 2 \eps)$-good for $\eta$.
\end{lemma}

Before proving Lemma~\ref{lem:x_concentration}, let us prove the following lemma which will be used later.

\begin{lemma}
\label{lem:num_of_geodesics}
There exists a constant $c > 0$ so that the following is true.  Suppose that $(\CB,d_\CB,\nu_\CB,z)$ is sampled from $\bandlaw{\ell}{w}$.  Let $N$ be the number of points on $\outerboundary \CB$ which are visited by a geodesic from $\innerboundary \CB$ to $\outerboundary \CB$.  Then $N$ is stochastically dominated by the law of $2Z$ where $Z$ is a Poisson random variable with mean $c \ell/w^2$.
\end{lemma}
\begin{proof}
By scaling, we only need to  give the proof in the case that $\ell=1$.  For each $k \in \N$, let $z_1^k,\ldots,z_{2^k}^k$ be points on $\innerboundary \CB$ with $z_1^k = z$ and $z_1^k,\ldots,z_{2^k}^k$ ordered counterclockwise on $\innerboundary \CB$ so that the boundary length of the counterclockwise segment of $\innerboundary \CB$ from $z_j^k$ to $z_{j-1}^k$ is equal to $2^{-k}$.  Let $\eta_j^k$ be the a.s.\ unique geodesic from $z_j^k$ to $\outerboundary \CB$.  Let $N_k$ be the number of points in $\outerboundary \CB$ visited by the $\eta_j^k$ for $j=1,\ldots, 2^k$.  Note that $N_k$ is equal to the number of indices $j$ so that $\eta_j^k$ does not merge with $\eta_{j-1}^k$ before hitting $\outerboundary \CB$.  The probability that $\eta_j^k$ does not merge with $\eta_{j-1}^k$ before hitting $\outerboundary \CB$ is equal to the probability that a $3/2$-stable CSBP starting from $2^{-k}$ hits $0$ after time $w$.  By~\eqref{eqn:csbp_extinction_time}, there exists a constant $c > 0$ so that this probability is explicitly given by
\[ p_{k,w} = 1-\exp(c w^{-2} 2^{-k}) = c w^{-2} 2^{-k} + o(2^{-k}).\]
By the independence of geodesic slices, we thus have that $N_k$ is a binomial random variable with parameters $p=p_{k,w}$ and $n=2^k$.  It therefore follows that $N_k$ converges in distribution as $k \to \infty$ to a Poisson random variable with mean $c/w^2$.  Let $P$ be the set of points on $\outerboundary \CB$ which are visited by a geodesic starting from a point $\wt{z}$ on $\innerboundary \CB$ where $\wt{z}$ is such that the boundary length along $\innerboundary \CB$ from $z$ to $\wt{z}$ in the counterclockwise direction is a dyadic rational.  Then the above implies that $|P|$ is distributed as a Poisson random variable with mean $c/w^2$.  Let $\wt{P}$ be the set of points on $\outerboundary \CB$ which are visited by a geodesic starting from a point on $\innerboundary \CB$ to $\outerboundary \CB$.  Suppose that $a \in \wt{P} \setminus P$.  Then $a$ is visited by a geodesic from a point $b$ on $\innerboundary \CB$ which is not the leftmost or the rightmost geodesic from $b$ to $\outerboundary \CB$ (because the leftmost and rightmost geodesics can be written as limits of the $\eta_j^k$).  Since there can be at most $3$ geodesics from any point $b$ in $\innerboundary \CB$ to $\outerboundary \CB$ \cite[Theorem~1.4]{lg2010geodesics} (since if we realized $\CB$ as a metric band inside of an ambient Brownian map instance associated with filled metric balls centered at the root, then these geodesics are each part of a geodesic from $b$ to the root), it follows that there can be at most one point in~$\wt{P}$ between any pair of points in~$P$.  That is, $|\wt{P}| \leq 2|P|$.
\end{proof}

We are now ready to prove Lemma~\ref{lem:x_concentration}.

\begin{proof}[Proof of Lemma~\ref{lem:x_concentration}]
Fix $r_0>0$ and $r\in(0,r_0)$. Let $(\CS,d,\nu,x,y)$ be sampled from $\bminflaw$ conditioned on $\{d(x,y) > r\}$.  
Fix $\ell \ge r$ and let $s=d(x,y)-\ell$. Fix $c>0.$ Let $E$ be the following event.
\begin{enumerate}[($E$)]
\item We have $d(x,y)>\ell$, i.e., $s>0$, and there exist a geodesic $\eta$ from $\partial \fb{y}{x}{s}$ to $x$ and times $2\eps \leq t_1 < t_1 + c \eps \log \eps^{-1} \leq t_2 \leq s-2\eps$, so that 
$\eta|_{[t_1,t_2]}$ does not pass through any \X\ which is $(s_0\eps, s_1 \eps, 2 \eps)$-good for $\eta$.
\end{enumerate}
Let us focus on showing that there exist  $\sp{b}_1,c_1 >0$ and $\eps_0>0$ such that $\mathbf{P}_{r}[E] \le  \eps^{\sp{b}_1 c}$ for all $c>c_1$ and $\eps\in(0,\eps_0)$, where $\sp{b}_1, c_1, \eps_0$ do not depend on $\ell$ or $c$. 

This will enable us to complete the proof of the lemma as follows.
Note that for any $\sp{a}>0$, we have (recall the discussion at the end of Section~\ref{subsubsec:csbp} on the density of the lifetime of an $\alpha$-stable CSBP excursion)
\begin{align}\label{eq:dxy}
\mathbf{P}_r [ d(x,y) \ge \eps^{-\sp{a}} ] = O( \eps^{2\sp{a}}),
\end{align}
where the implicit constant is uniform over $r\in(0,r_0)$.
Applying the union bound to $\ell=r+k\eps$ for $k\in [0,  \eps^{-1-\sp{a}}] \cap \N$ yields that the event in the lemma holds with probability at most $  \eps^{\sp{b}_1 c-1-\sp{a}} + O(\eps^{2\sp{a}})$ for all $c>c_1$ and $\eps > 0$ sufficiently small. Take $c_0=\max(2/\sp{b}_1, c_1)$ and $\sp{b}_0=\sp{b}_1/4$. For all $c>c_0$, take $\sp{a} \in(\sp{b}_1 c/8, 3\sp{b}_1 c/4-1)$. Then for all $c>c_0$, we have $\sp{b}_1c-1-\sp{a}>\sp{b}_0c$ and $2\sp{a}>\sp{b}_0c$. This implies that the event in the lemma holds with probability at most $\eps^{\sp{b}_0c}$ for $\eps$ sufficiently small, which completes the proof.

Let us now fix $\ell \in [r, r+\eps^{-\sp{a}}]$.
Note that $\p_r[E] \le \p_{\ell}[E]$, hence we only need to prove that  there exist $\sp{b}_1, c_1>0$ and $\eps_0>0$ such that $\mathbf{P}_{\ell}[E] \le \eps^{\sp{b}_1 c}$ for all $c>c_1$ and $\eps\in(0,\eps_0)$, where $\sp{b}_1, c_1, \eps_0$ do not depend on $\ell$ or $c$.

As a first step, we fix $\eps > 0$ and perform $\eps$ unit of reverse metric exploration on $\fb{y}{x}{s}$. Let  $N$ be the number of points on $\partial \fb{y}{x}{s-\eps}$ visited by all of the geodesics from $\partial \fb{y}{x}{s}$ to $x$. Let us prove that there exists $c_3>0$ such that for all $\sp{b}>0$, we have
\begin{align}
\label{eq:reverse_eps}
\p_{\ell}[ N \ge c_3 \eps^{-2-\sp{b}} ] =O( \eps^\sp{b}).
\end{align}

Let $Y_{\ell}$ denote the boundary length of $\partial \fb{y}{x}{s}$.  Then $Y_\ell$ is the value at time $\ell$ of a $3/2$-stable CSBP excursion conditioned to have length at least $\ell$.  The maximum of a $3/2$-stable CSBP excursion has the same distribution as the maximum of the corresponding $3/2$-stable L\'evy excursion by the Lamperti transform~\eqref{eqn:lamperti_csbp_to_levy}.  Since the lifetime of a $3/2$-stable L\'evy excursion has distribution given by a constant times $t^{-5/3} dt$ where $dt$ is Lebesgue measure on $\R_+$, the scaling property of a $3/2$-stable L\'evy excursion implies that the distribution of the maximum is given by a constant times $t^{-2} dt$.  This implies in particular that for some constant $c_2 > 0$, we have
\begin{align}\label{eq:Yl_upper_bound}
\p_{\ell}[ Y_\ell \ge \eps^{-\sp{b}} ] \le c_2 \eps^{\sp{b}}.
\end{align}
We now further condition on the event  $\{Y_\ell < \eps^{-\sp{b}}\}$.  Lemma~\ref{lem:num_of_geodesics} then implies that on this event the number $N$ of points on $\partial \fb{y}{x}{s-\epsilon}$ which are visited by a geodesic from $\partial \fb{y}{x}{s}$ to $x$ is stochastically dominated by $2Z$ where $Z$ is a Poisson random variable with mean $\lambda = c_3 \epsilon^{-2-\sp{b}}$ for some absolute constant $c_3>0$.
Recall the elementary tail bound for Poisson random variables:
\[ \p[ Z \geq \alpha \lambda] \leq \exp(\lambda(\alpha-\alpha \log \alpha -1)) \quad\text{for all}\quad \alpha > 1.\]
Applying this with $\alpha = 3$ so that $\alpha -\alpha \log \alpha - 1 < -1$ implies
\begin{align*}
\p_{\ell} [ N \ge 6 c_3 \eps^{-2-\sp{b}} \giv Y_\ell < \eps^{-\sp{b}} ] \le \exp (- c_3 \eps^{-2-\sp{b}}).
\end{align*}
Combining this with~\eqref{eq:Yl_upper_bound} and adjusting the value of $c_3$ implies~\eqref{eq:reverse_eps}. 
\smallskip

Now it only remains to prove that the following is true.  Fix $z \in \partial \fb{y}{x}{s-\epsilon}$ chosen in a way which is measurable with respect to $\CS \setminus \fb{y}{x}{s-\epsilon}$ and let $\eta$ be the a.s.\ unique geodesic from $z$ to $x$ in $\fb{y}{x}{s-\epsilon}$.  Then there exist $\sp{b}_2, c_4>0$ and $\eps_0>0$ such that for all $c>c_4$ and $\eps\in(0,\eps_0)$, the following event happens with probability at most  $\eps^{\sp{b}_2 c}$.
There exist times $2\eps \leq t_1 < t_1 + c \eps \log \eps^{-1} \leq t_2 \leq s-2\eps$ so that $\eta|_{[t_1,t_2]}$ does not pass through an {\X} which is $(s_0 \eps, s_1 \eps, 2\eps)$-good for $\eta$.  Indeed,  upon showing this, since the metric band $\fb{y}{x}{s} \setminus \fb{y}{x}{s-\epsilon}$ is independent of the filled metric ball $\fb{y}{x}{s-\epsilon}$ given its boundary length $Y_{\ell+\epsilon}$, 
by applying a union bound to the at most $c_3 \eps^{-2-\sp{b}}$ points $z$ in $\partial \fb{y}{x}{s-\epsilon}$ which are visited by a geodesic from $\partial \fb{y}{x}{s}$ to $x$, we get that $\p_\ell[E] \le c_3 \eps^{\sp{b}_2 c -2-\sp{b}} + O(\eps^\sp{b})$. Let $\sp{b}_1=\sp{b}_2/4$ and $c_1=\max(c_4, 5/\sp{b}_2)$. For all $c>c_1$, choose $\sp{b} \in(\sp{b}_2c/4, 3\sp{b}_2c/4-2)$. This implies that there exists $\eps_0>0$ such that  $\p_\ell[E] \le \eps^{\sp{b}_1c}$ for all $c> c_1$ and $\eps\in(0,\eps_0)$ as desired.

Assume that we have fixed $z$ and $\eta$ as just above.  For each $k \in \N$, we let $\CF_k$ be the $\sigma$-algebra generated by the metric measure space $\CS \setminus \fb{y}{x}{s-(7k+1)\epsilon}$ equipped with the interior-internal metric and the restriction of $\nu$ to $\CS \setminus \fb{y}{x}{s-(7k+1)\epsilon}$.  The boundary length $Y_{\ell+(7k+1)\epsilon}$ of $\partial \fb{y}{x}{s-(7k+1)\epsilon}$  is $\CF_k$-measurable.  Moreover, the conditional law of $\fb{y}{x}{s-(7k+1)\epsilon} \setminus \fb{y}{x}{s-(7k+8)\epsilon}$ given $\CF_k$ is $\bandlaw{Y_{\ell+(7k+1)\epsilon}}{7\epsilon}$.  

Let $\ell_0$ be as the $\ell_1$ in Lemma~\ref{lem:X_prob2}.  Let $\tau_{1} = \min\{k \geq 0 : Y_{\ell+(7k+1) \epsilon} \geq \ell_0 \epsilon^2\}$ and for each $j \geq 1$ we inductively let $\tau_{j+1} = \min\{k \geq \tau_j+1 : Y_{\ell+(7k+1)\epsilon} \geq \ell_0 \epsilon^2\}$.  Lemmas~\ref{lem:X_prob2} and~\ref{lem:good_band_map} imply that there exist constants $s_0, s_1, p_1 > 0$ so that on the event $\{\tau_j < \infty\}$, the conditional probability given $\CF_{\tau_j}$ such that $\eta$ passes through an {\X} which is $(s_0 \eps, s_1 \eps, 2\eps)$-good for $\eta$ in the time-interval $[s-(7\tau_j+8)\epsilon,s- (7 \tau_j+1) \epsilon]$ is at least $p_1$.  Let $N_1$ be the first $j \geq 0$ that either $\eta$ passes through an {\X} which is $(s_0 \eps, s_1 \eps, 2\eps)$-good for $\eta$ in the time-interval  $[s-(7\tau_j+8)\epsilon,s- (7 \tau_j+1) \epsilon]$ or $Y_{\ell+(7\tau_j+1) \eps}=0$, i.e., $\tau_j=\infty$.  Inductively let $N_{i+1}$ be the first $j \geq N_i+1$ so that either $\eta$ passes through an {\X} which is $(s_0 \eps, s_1 \eps, 2\eps)$-good for $\eta$ in the time-interval  $[s-(7\tau_j+8)\epsilon,s- (7 \tau_j+1) \epsilon]$ or $Y_{\ell+(7\tau_j+1)\eps}=0$.  Then it follows that on the event $Y_{\ell+(7\tau_{N_i}+1) \epsilon} > 0$,  conditionally on $\CF_{\tau_{N_i}}$, the probability of $ N_{i+1}-N_i \geq n $ is at most $(1-p_1)^n$.  This in particular implies that there exist constants $\sp{b}_3 > 0$ and $\eps_0>0$ so that for all $c>0$ and $\eps\in(0,\eps_0)$, we have
\begin{equation}
\label{eqn:nj_tail}
\mathbf{P}_{\ell} \!\left[ N_{i+1}-N_i \geq (c/40) \log \eps^{-1} \giv Y_{\tau_{N_i}} > 0,\, \CF_{\tau_{N_i}} \right] \leq \eps^{\sp{b}_3c}.
\end{equation}
On the other hand, the number of $k\in\N$ such that $Y_{\ell+(7k+1)\epsilon} \in (0, \ell_0 \epsilon^2)$ is stochastically dominated by a geometric random variable uniformly in $\epsilon > 0$.  Indeed, by the scaling property of $3/2$-stable CSBPs, on $\{Y_{\ell+(7k+1)\epsilon} \in (0, \ell_0 \epsilon^2)\}$ we have that the conditional probability given $\CF_k$ of the event that $Y$ hits $0$ in $[\ell+(7k+1)\epsilon, \ell+(7k+8)\epsilon]$ is positive uniformly in $\epsilon > 0$. It follows that there exists $\sp{b}_4>0$ and $\eps_0>0$ such that the probability that there are at least $(c/40) \log \eps^{-1}$ such values of $k$ is at most $\eps^{\sp{b}_4 c}$ for all $c>0$ and $\eps\in(0,\eps_0)$.

To complete the proof,  it suffices to show that there exist $\sp{b}_2>0$ and $\eps_0>0$ such that for all $c>0$ and $\eps\in(0,\eps_0)$, with probability at most $\eps^{\sp{b}_2 c}$, there exists $n \in \N_0$ such that $(n+1) (c/2) \eps \log\eps^{-1} \le s$ and  $\eta$ does not pass through an {\X} which is $(s_0 \eps, s_1 \eps, 2\eps)$-good for $\eta$ in $I_n:=[s- (n+1) (c/2) \eps \log\epsilon^{-1}, s - n (c/2)\eps \log\epsilon^{-1}]$. Indeed, each interval $[t_1, t_2]$ where  $2\eps \leq t_1 < t_1 + c \eps \log \eps^{-1} \leq t_2 \leq s-2\eps$ contains an interval  $I_n$ for some $n \in \N_0$ with  $(n+1) (c/2) \eps\log\eps^{-1} \le s$.

Fix $n \in \N_0$. On the event $\{(n+1) (c/2) \eps \log\eps^{-1} \le s\}$,  there are at least $ (c/20) \log \eps^{-1}$ values of $k$ for which $\ell+(7k+1)\epsilon \in I_n$. We have previously shown that the probability that there are at least $(c/40) \log \eps^{-1}$ values of $k$ with $Y_{\ell+(7k+1)\epsilon} \in (0, \ell_0 \epsilon^2)$ is at most $\eps^{\sp{b}_4 c}$.  Combining this with~\eqref{eqn:nj_tail} implies that the probability that $\eta$ does not pass through an {\X} in $I_n$ is at most $\eps^{\sp{b}_5 c}$ with $\sp{b}_5 =\max(\sp{b}_3, \sp{b}_4)$.

On the other hand, \eqref{eqn:csbp_extinction_time} implies that for all $\sp{b}>0$,
\begin{align*}
\mathbf{P}_{\ell} [s \geq \eps^{-\sp{b}} \giv Y_\ell < \eps^{-\sp{b}} ] =O( \eps^{\sp{b}}).
\end{align*}
Recall that $\p_\ell[Y_\ell \ge \eps^{-\sp{b}}] \le c_2 \eps^\sp{b}$ by \eqref{eq:Yl_upper_bound}. Now apply the union bound to the probability in the previous paragraph by summing over $n\in[0, \eps^{-\sp{b}-2}] \cap\N_0$. We get that the  probability that there exist times  $2\eps \leq t_1 < t_1 + c \eps \log \eps^{-1} \leq t_2 \leq s-2\eps$ so that $\eta|_{[t_1,t_2]}$ does not pass through any {\X} which is $(s_0 \eps, s_1 \eps, 2\eps)$-good for $\eta$ is at most $\eps^{\sp{b}_5c -\sp{b}-2} + c_2 \eps^\sp{b}+O(\eps^{\sp{b}})$. Take $\sp{b}_2=\sp{b}_5/4$ and $c_4=5/\sp{b}_5$. For all $c>c_4$, choose $\sp{b} \in (\sp{b}_5c/4, 3\sp{b}_5c/4-2)$. Then we have $\sp{b}_5 c-\sp{b}-2 > \sp{b}_2c$ and $b> \sp{b}_2c$, hence the previously mentioned probability is at most $\eps^{\sp{b}_2c}$ for all $c>c_4$ and $\eps$ sufficiently small. This completes the proof.
\end{proof}

\subsection{Proof of Proposition~\ref{prop:strong_confluence}}
\label{sec:proof_prop_3.1}
In this subsection, we will complete the proof of Proposition~\ref{prop:strong_confluence}.  We will make use of the following strategy.  First, for all sufficiently large $k \in \N$ we place an $(\eps_{k}=2^{-k})$-net of typical points in the Brownian map and prove using Lemma~\ref{lem:x_concentration} that there are \X's everywhere along every geodesic emanating from points in each $2^{-k}$-net. Then we will show that for any pair of geodesics $\eta_1, \eta_2$ which are close in the one-sided Hausdorff distance, we can always find $k \in \N$ and a point in the $2^{-k}$-net and a geodesic $\sigma$ emanating from this point which stays between $\eta_1$ and $\eta_2$. This will force $\eta_1$ and $\eta_2$ to both intersect the same \X\ along $\sigma$ so that they will also intersect each other.

We will focus on proving the assertion of Proposition~\ref{prop:strong_confluence} for a.e.\ $(\CS,d,\nu,x,y)$ which is sampled from $\bminflaw$ conditioned on $\{1 \leq \nu(\CS) \leq 2\}$.  It will then follow that for Lebesgue a.e.\ value of $a \in [1,2]$, the same result holds a.s.\ for $(\CS,d,\nu,x,y)$ with law $\bmlaw{a}$. The result in the case of a sample from $\bmlaw{a}$ for every value of $a > 0$ thus holds by the scaling property of the Brownian map.  This will complete the proof of Proposition~\ref{prop:strong_confluence}.

Throughout this subsection, we suppose that $(\CS,d,\nu,x,y)$ is sampled from $\bminflaw$ conditioned on $\{1 \leq \nu(\CS) \leq 2\}$. Let $(z_i)$ be a sequence of points chosen i.i.d.\ with respect to the measure $\nu$.  Let $s_0, s_1, r_0, \sp{b}_0, c_0$ be given by Lemma~\ref{lem:x_concentration}. Fix $M>0$ that we will adjust later.  For each $k\ge 1$, let $\eps_k=2^{-k}$ and $N_k = \eps_k^{-4-\sp{a}}$. Let $L_k= c_1 (M\eps_k) \log (M\eps_k)^{-1}$.

\begin{lemma}\label{lem:eps_net_points} 
Fix $c_1> 2 c_0$ so that $\sp{b}_0 c_1/2 \ge 9$. Fix $\sp{a} \in(0, 1/2)$. There a.s.\ exists $k_0 \in \N$ so that for all $k \geq k_0$, the following events occur
\begin{enumerate}
\item[$(A_k)$] For all $z\in\CS$, there exists $1 \leq i \leq N_k$ such that $z_i \in B(z, \eps_k)$.
\item[$(B_k)$] For all $1 \leq i \leq N_k$, for every geodesic $\omega: [0, L] \to \CS$ with length $L\ge L_k$ starting from $z_i$ and $2 M\eps_k \leq t_1 \leq t_1 + L_k \leq t_2 \leq L -2 M\eps_k$, $\omega|_{[t_1,t_2]}$ passes through an {\X} which is $(s_0 M\eps_k, s_1 M\eps_k, 2M\eps_k)$-good.
\end{enumerate}

\end{lemma}
\begin{proof}
By Lemma~\ref{lem:typical_points_dense}, there exists $k_0\in\N$ such that for all $k\ge k_0$, the event $A_k$ holds.
Fix $k \in \N$ and $1 \leq i, j \leq N_k$ distinct. Fix $r\in(0,r_0)$ where $r_0$ is as in Lemma~\ref{lem:x_concentration}.
Applying Lemma~\ref{lem:x_concentration} to $(\CS,d,\nu,z_i,z_j)$ in place of $(\CS,d,\nu,x,y)$, we deduce that by possibly increasing $k_0$, the following event has probability at most a constant times $\eps_k^{9}$. There exists $t\in (0, d(z_i, z_j)-r)$, a geodesic $\eta$ from $\partial \fb{z_j}{z_i}{t}$ to $z_i$ and times $2M\eps_k \le t_1< t_1 + L_k/2 \le t_2\le t -2 M\eps_k$ so that $\eta|_{[t_1, t_2]}$ does not pass through any \X\  which is $(s_0 M\eps_k, s_1 M\eps_k, 2M\eps_k)$-good. 
Since this is true for all $r\in(0,r_0)$, the preceding event with $t\in (0, d(z_i, z_j))$ in place of $t\in (0, d(z_i, z_j)-r)$ also occurs with probability at most a constant times $\eps_k^{9}$.
We can then apply a union bound for the $N_k(N_k-1)$ pairs of distinct $(i,j)$. This implies that the following event $E_k$ holds with probability at least $1- O( \eps_k^{1-2\sp{a}})$. 
\begin{enumerate}
\item[$(E_k)$]  For all $1 \leq i,j \leq N_k$ such that $d(z_i, z_j)> L_k/2$, for all geodesics $\omega \colon [0,L] \to \CS$ with length $L \ge L_k/2$ from $z_i$ to $\partial \fb{z_j}{z_i}{L}$ and $2M\eps_k \leq t_1 \leq t_1 + L_k/2 \leq t_2 \leq L-2M\eps_k$, $\omega|_{[t_1, t_2]}$ passes through an {\X} which is $(s_0 M\eps_k, s_1 M\eps_k, 2M\eps_k)$-good.
\end{enumerate}
By the Borel-Cantelli lemma, we get that, by possibly increasing the value of $k_0 \in \N$, $A_k \cap E_k$ holds for every $k\ge k_0$.

Now, let us show that the event $A_k \cap E_k$ implies the event $A_k \cap B_k$, which will complete the proof of the lemma.
Fix $L \ge L_k$. On $A_k$, any geodesic $\omega: [0, L] \to \CS$  starting from $z_i$ ends at $\omega(L) \in B(z_j, \eps_k)$ for some $1 \leq j \leq N_k$.  Note that $\fb{z_j}{z_i}{d(z_i, z_j) - \eps_k}$ is always disjoint from $B(z_j, \eps_k)$.
This implies that $\omega$ intersects $\partial \fb{z_j}{z_i}{d(z_i, z_j) - \eps_k}$ at a unique point $\omega(\wt s)$ for some $\wt s>0$. Moreover, $L- \wt s \in (0, 2\eps_k)$ by the triangle inequality. 
On $A_k\cap E_k$, for any $2M\eps_k \leq t_1 \leq t_1 + L_k \leq t_2 \leq L-2M\eps_k$, $\omega|_{[t_1, t_1 + L_k/2]}$ must pass through an \X\ which is $(s_0 M\eps_k, s_1 M\eps_k, 2M\eps_k)$-good, hence so do $\omega|_{[t_1, t_2]}$.
Since this is true for all $1 \leq i \leq N_k$ and every geodesic $\omega$ starting from $z_i$, the event $A_k\cap B_k$ also holds.  
\end{proof}

Let us collect the following general fact about a pair of geodesics in a metric space which roughly speaking says that if they are close in the Hausdorff sense then their lengths are close and their endpoints are also close.

\begin{lemma}\label{lem:hausdorff_close}
Suppose that $(X,d)$ is a geodesic metric space and $\eta_i \colon [0,T_i] \to X$ for $i=1,2$ are geodesics such that $\distH(\eta_1([0, T_1]),\eta_2([0,T_2])) \leq \epsilon$.  Then, by possibly reversing the time of $\eta_2$, we have
\begin{align}\label{eq:triangular}
|T_1- T_2| \le 2\eps,\quad d(\eta_1(0), \eta_2(0))\le 5\eps, \quad d(\eta_1(T_1), \eta_2(T_2))\le 5\eps.
\end{align}
\end{lemma}
\begin{proof}
Fix $0 \leq s_1 < t_1 \leq T_1$. There exist $0 \leq s_2,t_2 \leq T_2$ such that 
\begin{align}\label{eq:eta_distance}
d(\eta_1(s_1),\eta_2(s_2))\le \eps, \quad d(\eta_1(t_1),\eta_2(t_2))\le \eps.
\end{align}
By the triangle inequality, we have that
\begin{align*}
t_1 - s_1 = d(\eta_1(s_1),\eta_1(t_1)) \leq 2 \epsilon + d(\eta_2(s_2),\eta_2(t_2)) = 2\epsilon + |t_2 - s_2|.
\end{align*}
Taking $s_1 = 0$ and $t_1 = T_1$ we see that 
\begin{align}
\label{eq:st1}
T_1 \leq |t_2-s_2| + 2\epsilon \le T_2 +2\eps.
\end{align}
By swapping the roles of $\eta_1$ and $\eta_2$, we also have 
\begin{align}\label{eq:st2}
T_2 \leq T_1 + 2 \epsilon.  
\end{align}
Combining, we see that
\begin{align*}
    T_2 - |t_2-s_2|
 &\leq T_2 - (T_1 - 2\epsilon) \quad\text{(by~\eqref{eq:st1})}\\
 &\leq 4 \epsilon \quad\text{(by~\eqref{eq:st2})}.
\end{align*}
By possibly reversing the time of $\eta_2$, we can assume that $s_2 \le t_2$, in which case $s_2 + (T_2 - t_2) \le 4\eps$.  Since $s_2 \geq 0$ and $T_2 - t_2 \geq 0$, this inequality implies both $s_2 \leq 4 \epsilon$ and $T_2 - t_2 \leq 4 \epsilon$.  Together with~\eqref{eq:eta_distance} applied to $s_1=0$ and $t_1=T_1$, we deduce both
\begin{align*}
&d(\eta_1(0),\eta_2(0)) \leq d(\eta_1(0), \eta_2(s_2)) + d(\eta_2(s_2), \eta_2(0))\le \eps + s_2 \le 5\eps \quad\text{and}\\
&d(\eta_1(T_2),\eta_2(T_2)) \leq d(\eta_1(T_1), \eta_2(t_2)) + d(\eta_2(t_2), \eta_2(T_2))\le \eps + (T_2 -t_2) \le 5\eps.
\end{align*}
This completes the proof.
\end{proof}

From now on, fix $k_0$ and $c_1$ as given by Lemma~\ref{lem:eps_net_points}. Fix $\eps_0 \le 2^{-k_0}$ and $\eps\in(0,\eps_0)$. Fix $\delta =  8 c_1 M \eps \log \eps^{-1}$.
Suppose that $\eta_i \colon [0,T_i] \to \CS$ for $i=1,2$ are geodesics with $T_i = d(\eta_i(0), \eta_i(T_i)) \geq 2\delta$ and $\distHos(\eta_1([0, T_1]),\eta_2([0,T_2])) \leq \epsilon.$  
Without loss of generality, suppose that $\eta_2$ is close to the right side of $\eta_1$, namely $\distHos(\eta_1, \eta_2)= \ell_\mathrm{R}$ in~\eqref{eq:one_sided_hausdorff}. 

By possibly reversing the time of $\eta_2$, suppose that~\eqref{eq:triangular} holds.
Let $\gamma_0$ be a geodesic from $\eta_1(0)$ to $\eta_2(0)$ and $\gamma_1$ be a geodesic from $\eta_1(T_1)$ to $\eta_2(T_2)$. See Figure~\ref{fig:sigma}.  Let $U$ be the open set of points surrounded clockwise by the concatenation of $\eta_1$, $\gamma_1$, the time-reversal of $\eta_2$, and the time-reversal of $\gamma_0$. 
There exist $0\le \alpha_1 < \alpha_2 \le T_1$, paths $\gamma_2$ from $\eta_1(\alpha_1)$ to $\eta_2(10 \eps)$ and $\gamma_3$ from $\eta_1(\alpha_2)$ to $\eta_2(T_2- 10 \eps)$ such that both $\gamma_2, \gamma_3$ are contained in $\ol U$ and have length at most $\eps$.
Let $V$ be the open set of points surrounded by the concatenation of $\eta_1|_{[\alpha_1, \alpha_2]}$, $\gamma_3$, the time-reversal of $\eta_2|_{[10\epsilon,T_2-10\epsilon]}$, and the time-reversal of $\gamma_2$. 
Define 
\[R_V = \sup\{ r > 0 : \exists w \in V, B(w,r) \subseteq V\} \quad \text{and} \quad R= \min(R_V, \eps). \]
Choose $k \in \N$ so that $R \in (\eps_k, 2\eps_k]$. 
Pick $w \in V$ so that $B(w,R) \subseteq V$. 
As $R<\eps$ so $k \geq k_0$, we know by Lemma~\ref{lem:eps_net_points} that there exists $1 \leq j \leq N_k$ so that $z_j \in B(w,\eps_k) \subseteq V$.

\begin{figure}[h!]
\begin{center}
\includegraphics[width=.92\textwidth]{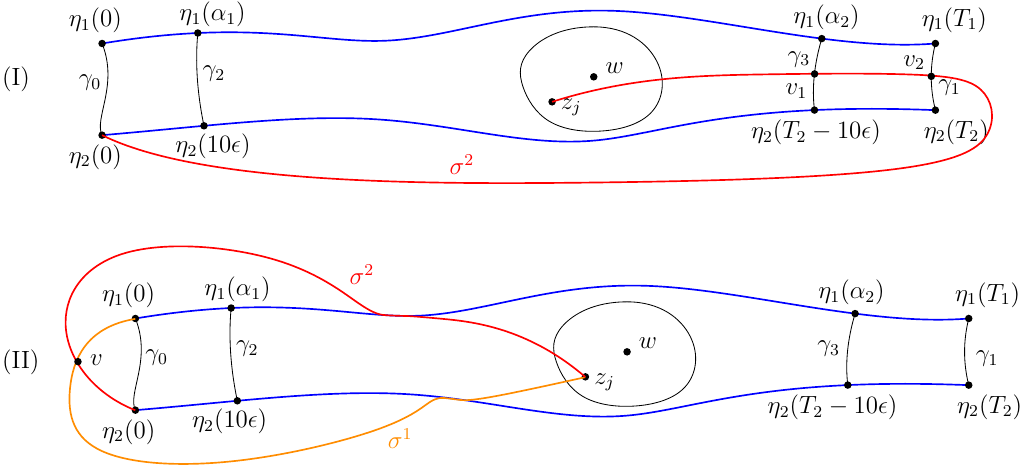}	
\end{center}
\caption{\label{fig:sigma} Illustration of the setup and the proof of Lemma~\ref{lem:sigma_contained_U}.}
\end{figure}

\begin{lemma}\label{lem:sigma_contained_U}
There exists $i\in\{1,2\}$ and a geodesic $\sigma^i$ from $z_j$ to $\eta_i(0)$ which is contained in $\ol U$.
\end{lemma}

\begin{proof}
Suppose that it is not the case.  Then for $i=1,2$, every geodesic from $z_j$ to $\eta_i(0)$ exits~$\ol U$. 
Let $\sigma^2$ be a geodesic from $z_j$ to $\eta_2(0)$. Then $\sigma^2$ must exit $\ol U$ through either $\eta_1, \eta_2$ or $\gamma_0, \gamma_1$.

First of all, $\sigma^2$ cannot first exit $\ol U$ through $\eta_2$. Otherwise, let $v_0$ be the point where $\sigma^2$ first hits $\eta_2$, then the concatenation of the part of $\sigma^2$ from $z_j$ to $v_0$ and the part of $\eta_2$ from $v_0$ to $\eta_2(0)$ is also a geodesic from $z_j$ to $\eta_2(0)$ and it is contained in $\ol U$, which contradicts our assumption.

Similarly, $\sigma^2$ cannot first exit $\ol U$ through $\gamma_0$. Otherwise, let $v_0$ be the point where $\sigma^2$ first hits $\gamma_0$, then the concatenation of the part of $\sigma^2$ from $z_j$ to $v_0$ and the part of $\gamma_0$ from $v_0$ to $\eta_2(0)$ is also a geodesic from $z_j$ to $\eta_2(0)$ and it is contained in $\ol U$, which contradicts our assumption.

Now, let us show that $\sigma^2$ cannot first exit $\ol U$ through $\gamma_1$. More generally, we will show that
\begin{align}\label{eq:sigma_gamma_empty}
\sigma^2 \cap \gamma_1 =\emptyset.
\end{align}
Suppose in the contrary that $\sigma^2$ intersects $\gamma_1$ at some point $v_2$. Before intersecting $\gamma_1$, $\sigma^2$ must first exit $V$. Suppose that there exists a point $v_1 \in \sigma^2\cap \partial V$ (we illustrate in Figure~\ref{fig:sigma} (I) the case $v_1 \in \gamma_3$). Then we must have
\begin{align}\label{eq:d_v1}
d(v_1, \eta_2(0)) \le T_2 -10\eps + \eps.
\end{align}
This clearly holds if $v_1 \in \eta_2([10\eps, T_2 -10\eps])$. If $v_1 \in \gamma_2$, then we have $d(v_1, \eta_2(0)) \le d(v_1, \eta_2(10\eps)) + d(\eta_2(10\eps), \eta_2(0))\le 11\eps$, hence~\eqref{eq:d_v1} also holds.
If $v_1 \in\gamma_3$, then $d(v_1, \eta_2(0)) \le d(v_1, \eta_2(T_2 - 10\eps)) + d(\eta_2(T_2 - 10\eps), \eta_2(0))\le \eps+ T_2 -10\eps$, hence~\eqref{eq:d_v1} holds. If $v_1 \in \eta_1([\alpha_1, \alpha_2])$, then there exists $t_2 \in [10\eps, T_2 -10\eps]$ such that $d(v_1, \eta_2(t_2)) \le \eps$. Therefore, $d(v_1, \eta_2(0)) \le d(v_1, \eta_2(t_2)) + d(\eta_2(t_2), \eta_2(0)) \le \eps + T_2 -10 \eps$. This completes the proof of~\eqref{eq:d_v1} for all cases. 
Recall that we have assumed that $\sigma^2$ first hits $v_1$ before hitting $v_2$, so we also have
\begin{align*}
d(v_1, \eta_2(0)) =  d(v_1, v_2) + d(v_2, \eta_2(0)) \ge d(v_2, \eta_2(0)) \ge d (\eta_2(T_2), \eta_2(0)) - d (\eta_2(T_2), v_2) \ge T_2 -\eps.
\end{align*}
This contradicts~\eqref{eq:d_v1}, hence proves~\eqref{eq:sigma_gamma_empty}.

The only remaining possibility is that $\sigma^2$ first exits $\ol U$ through $\eta_1$. See Figure~\ref{fig:sigma} (II).
Moreover, by possibly modifying $\sigma^2$, we can assume that after that  $\sigma^2$ exits $\ol U$ through $\eta_1$, it does not reenter $U$. By~\eqref{eq:sigma_gamma_empty}, we know that $\sigma^2$ cannot intersect $\gamma_1$.  
If $\sigma^2$ intersects $\eta_1$ again after leaving it, then we can just replace the part of $\sigma^2$ between the first and last time that it intersects $\eta_1$ by the part of $\eta_1$ between these two points. 
If  $\sigma^2$ intersects $\eta_2$ or $\gamma_0$, then we can modify $\sigma_2$ so that if follows $\eta_2$ or $\gamma_0$ since the first time that it intersects $\eta_2$ or $\gamma_0$ until it reaches $\eta_2(0)$.

By symmetry, there also exists a geodesic $\sigma^1$ from $z_j$ to $\eta_1(0)$ which first exits $\ol U$ through $\eta_2$ and remains outside of $U$ afterwards. Since the Brownian map has the topology of a sphere, these two geodesics $\sigma^1$ and $\sigma^2$ must intersect each other. Let $v$ be such an intersection point.  (See Figure~\ref{fig:sigma} (II) for an illustration.) It follows that the concatenation of the part of $\sigma^1$ from $z_j$ to $v$ and the part of $\sigma^2$ from $v$ to $\eta_2(0)$ is another geodesic from $z_j$ to $\eta_2(0)$ and this geodesic first exits $\ol U$ through $\eta_2$. We have shown earlier that this is impossible. This leads to a contradiction and we have completed the proof.
\end{proof}

We are finally ready to prove Proposition~\ref{prop:strong_confluence}.

\begin{proof}[Proof of Proposition~\ref{prop:strong_confluence}]
Recall that $k$ is chosen so that $R \in (\eps_k, 2\eps_k]$ and $\eps_k=2^{-k}$. 
We will complete the proof in 3 steps.
In Step 1, we will construct a ladder of paths between $\eta_1$ and $\eta_2$ with a spacing of $280 \eps_k$, and show that the longest path in this ladder has length $L\le 128 \eps_k$, which is less than the spacing. 
This will be used in Step 2 to show that when there is a good \X\ of size proportional to $\eps_k$ along a geodesic starting from $z_j$ which stays in $\ol U$, this \X\ must exit $U$ by intersecting both $\eta_1$ and $\eta_2$ and force them to merge. This will imply that $\eta_1$ and $\eta_2$ intersect near  one endpoint, and we will complete the proof in Step 3 by showing that $\eta_1, \eta_2$ also intersect near the other endpoint.

\begin{figure}
\centering
\includegraphics[width=\textwidth]{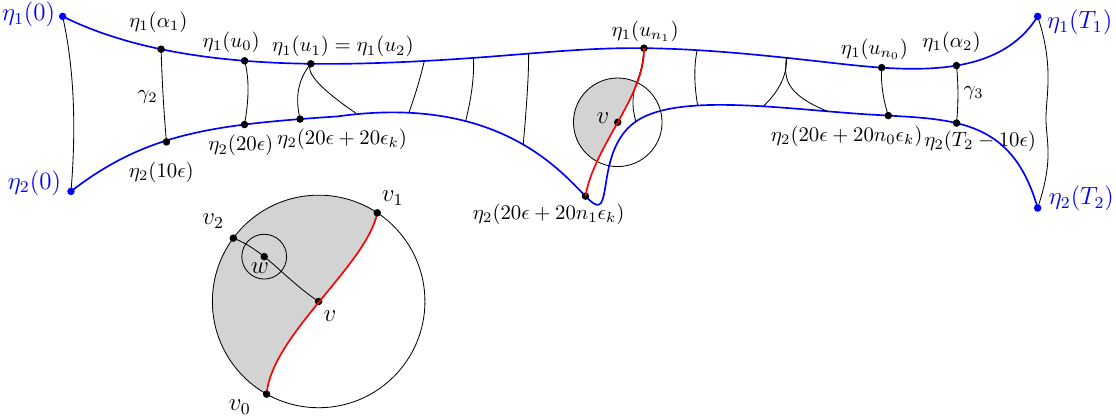}
\caption{Construction of the ladder. The paths $\wh\xi_m$ can possibly overlap, but do not cross each other.
The region $W_0$ is colored in grey. We also show a zoomed version of $B(v, L/8)$ in the case that it intersects $\eta_2$. Both $B(v, L/8)$ and $W_0$ can have holes, but we do not draw them for simplicity.}
\label{fig:proof_X2}
\end{figure}

\emph{Step 1. Construction of a ladder of paths between $\eta_1$ and $\eta_2$.} See Figure~\ref{fig:proof_X1} for the final ladder and see Figure~\ref{fig:proof_X2} for the intermediate steps. 
Let $n_0=  \lfloor (T_2-40 \eps) /(20\eps_k) \rfloor -1$.
For each $0\le n \le n_0$, let $u_n$ be the number  in $[\alpha_1, \alpha_2]$ which minimizes the distance from $\eta_2(20\eps+ 20 n  \eps_k)$ to $\eta_1(u_n)$  with respect to the interior-internal metric $d_V$ (if there are several such numbers, let $u_n$ be the largest one).
Let $\wh\xi_n$ be a shortest path from  $\eta_2(20\eps+ 20 n  \eps_k)$ to $\eta_1(u_n)$ contained in $\ol V$ (if there are more than one such paths, choose the rightmost one).  Note that the length of $\wh{\xi}_n$ is at most $\epsilon$.  Moreover, by~\eqref{eq:triangular} and the triangle inequality, we have
\begin{align}\label{eq:u0un0}
14 \eps\le u_0  \le 26 \eps, \quad 14 \eps \le T_2 - u_{n_0} \le 26 \eps.
\end{align}
For $0\le n\le n_0-1$, we have $u_{n+1} \ge u_n$. Let $L$ denote the maximum of the lengths of the $\wh\xi_n$ for $0\le n \le n_0$. Note that we clearly have $L\le \distHos(\eta_1, \eta_2) \le \eps$. However it is possible that $R_V$ (hence $\eps_k$) is much smaller than $\eps$, and it is not immediately clear that $R_V$ cannot be much smaller than $L$, but we will show that it is not the case. More precisely, we will show that 
\begin{align}\label{eq:rv_prop_L}
R_V \ge L/64.
\end{align}
Suppose that $\wh\xi_{n_1}$ has length $L$ where $0 \le n_1 \le n_0$, and let $v=\wh \xi_{n_1}(L/2)$. Then $B(v, L/8)$ must be disjoint from $\eta_1$, because otherwise the distance from $\eta_2(20\eps + 20 n_1 \eps_k)$ to $\eta_1$ would be at most $5L/8$, which contradicts the definition of $u_{n_1}$. 
It is also clear that $B(v, L/8)$ is disjoint from $\gamma_2$ or $\gamma_3$, because otherwise  the distance from $\eta_2(20\eps + 20 n_1 \eps_k)$ to $\eta_2(10 \eps)$ or $\eta_2(T_2 -10\eps)$ would be at most $L \le \eps$, which is impossible. 
If $B(v, L/8)$ does not intersect $\eta_2$, then $R_V\ge L/8$ hence~\eqref{eq:rv_prop_L} holds.
Otherwise, if $B(v, L/8)$ intersects $\eta_2$, it cannot intersect both $\eta_2([0, 20\eps+ 20 n_1 \eps_k])$ and $\eta_2([20\eps+ 20 n_1 \eps_k, T_2])$. Indeed, if we assume the contrary, then there exist $t_1\in [0,20\eps+ 20 n_1 \eps_k]$ and $t_2 \in[20\eps+ 20 n_1 \eps_k, T_2]$ such that both $\eta_2(t_1)$ and $\eta_2(t_2)$ are in $B(v, L/8)$. Then $t_2 - t_1 = d (\eta_2(t_2), \eta_2(t_1)) \le L/4$. On the other hand, we have
\begin{align*}
t_2-t_1\ge t_2 - (20\eps+ 20 n_1 \eps_k) \ge d(v, \eta_2(20\eps+ 20 n_1 \eps_k)) - d(v, \eta_2(t_2)) \ge 3L/8.
\end{align*}
This is a contradiction. Without loss of generality, we can suppose that $B(v, L/8)$ intersects  $\eta_2([20\eps+ 20 n_1 \eps_k, T_2])$.  
Let $W_0$ be the connected component of $B(v, L/8) \setminus \wh \xi_{n_1}$ which lies to the left of $ \wh \xi_{n_1}$. Note that $W_0$ is disjoint from $\eta_1, \eta_2$ and $\gamma_2, \gamma_3$, hence $W_0 \subset V$. Let $v_0 =\wh \xi_{n_1}(3L/8)$ and $v_1=\wh \xi_{n_1}(5L/8)$ so that they lie on the boundary of $B(v, L/8)$. We have $d(v_0, v_1)= L/4$. By the continuity of the space, there exists a point $v_2 \in V$ on the clockwise part of the boundary of $B(v, L/8)$ from $v_0$ to $v_1$  which satisfies $d(v_2, v_0)\ge L/8$ and $d(v_2, v_1) \ge L/8$.  We draw a geodesic from $v_0$ to $v$ so that it is contained in $\ol W_0$, and then let $w$ be the point on this geodesic which has distance $L/32$ to $v_2$. We claim that $B(w, L/64) \subset W_0$. It is enough to show that $B(w, L/ 64) \cap \wh \xi_{n_1} =\emptyset$. Suppose the contrary so that there exists a point $u$ on the part of $\wh \xi_{n_1}$ which is in $B(w, L/64)$. Then $d(u, v_2) \le d (u, w) + d(w, v_2) \le L/64 + L/32=3L/64.$ On the other hand, the distance from $u$ to one of $v_0, v, v_1$ must be at most $L/16$ and the distance from each of  $v_0, v, v_1$ to $v_2$ is at least $L/8$. By the triangle inequality, we have $d(u, v_2) \ge L/8 - L/16 =L/16$, leading to a contradiction. We have thus proved $B(w, L/64) \subset W_0$, hence~\eqref{eq:rv_prop_L} holds.

Recall that $R=\min(R_V, \eps)$. Noting that $L \le \eps$, \eqref{eq:rv_prop_L} implies that $R \ge L /64$ and consequently 
\begin{align}\label{eq:eps_k_L}
128 \eps_k \ge L.
\end{align}
Let $m_0=\lfloor n_0/14 \rfloor$. For $0\le m \le m_0$, let $v_m=20 \eps+280 m \eps_k$ and $w_m=u_{14m}$. Let $\xi_m=\wh\xi_{14m}$ be the path from $\eta_2(v_m)$ to $\eta_1(w_m)$. Then for each $0\le m \le m_0-1$, we have $\xi_{m+1}\cap \xi_{m}=\emptyset$, because otherwise, we would have $280\eps_k=v_{m+1} - v_m\le 2L$ which contradicts~\eqref{eq:eps_k_L}.
Moreover, 
\begin{align*}
|(w_{m+1} - w_m) - 280 \eps_k |=&|(w_{m+1} - w_m) - (v_{m+1} - v_m)|\\
=&| d(\eta_1(w_{m+1}), \eta_1(w_m)) - d (\eta_2(v_{m+1}), \eta_2(v_m))| \\
\le &d(\eta_1(w_{m+1}), \eta_2(v_{m+1})) + d(\eta_1(w_m), \eta_2(v_m)) \le 2L  \le 256 \eps_k.
\end{align*}
This implies that
\begin{align}\label{eq:w_m_spacing}
 24 \eps_k \le w_{m+1} - w_m \le 536 \eps_k.
\end{align}
The collection of paths $\xi_1, \ldots, \xi_{m_0}$ form the ladder that we want to construct.

\begin{figure}
\centering
\includegraphics[width=\textwidth]{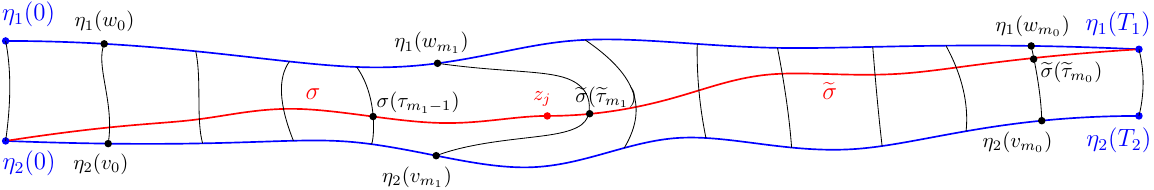}
\caption{Illustration of the ladder and the setup of Step 2 in the proof of Proposition~\ref{prop:strong_confluence}.}
\label{fig:proof_X1}
\end{figure}

\emph{Step 2. Intersection of $\eta_1$ and $\eta_2$ near one end.} More precisely, we aim to prove 
\begin{align}\label{eq:prop1}
 \eta_i( (0, \delta]) \cap \eta_{3-i} \neq \emptyset  \quad \text{for} \quad i=1,2,
\end{align}
where $\delta= 8 c_1 M\eps \log \eps^{-1}.$

By Lemma~\ref{lem:sigma_contained_U} and by relabeling $\eta_1$ and $\eta_2$ if necessary, suppose that there exists a geodesic $\sigma^2$ from $z_j$ to $\eta_2(0)$ which is contained in $\ol U$ and let $\sigma=\sigma^2$. See Figure~\ref{fig:proof_X1}. 
By symmetry, we also know that there exist $i\in\{1,2\}$ and a geodesic $\wt \sigma^i$ from $z_j$ to $\eta_i(T_i)$ which is contained in $\ol U$. In the present proof, we suppose that $i=1$, namely there exists a geodesic $\wt \sigma^1$ from $z_j$ to $\eta_1(T_1)$ which is contained in $\ol U$. Let $\wt \sigma=\wt \sigma^1$. However, the proof works essentially the same if $i=2$.
Let $S$ be the length of $\sigma$ and let $\wt S$ be the length of $\wt \sigma$. We also reorient $\sigma, \wt\sigma$ in such a way that $\sigma(0)=\eta_2(0)$ and $\wt\sigma(0)= \eta_1(T_1)$.

Let $m_1$ be the smallest $m\in\N$ such that $\xi_{m}$ intersects $\wt \sigma$. Due to the symmetric roles of $\sigma$ and $\wt\sigma$, we can assume without loss of generality that $m_1\ge m_0/2-1$.

For each $0 \le m \le m_1-1$, $\xi_m$ intersects $\sigma$ and we let $\tau_m\in[0,S]$ be such that $\sigma(\tau_m)\in \sigma \cap \xi_m$. For each $m_1 \le m \le m_0$, $\xi_m$ intersects $\wt\sigma$ and we let $\wt \tau_m\in[0,S]$ be such that $\wt \sigma(\wt \tau_m)\in \wt\sigma \cap \xi_m$.  For each $0 \le m \le m_1-1$, we have
\begin{align}\label{eq:tau_m}
|\tau_m - v_m|=| d(\sigma(0), \sigma(\tau_m)) - d(\eta_2(0), \eta_2(v_m))|
\le  d(\sigma(\tau_m), \eta_2(v_m)) \le L \le 128 \eps_k.
\end{align}
For all $0 \le m \le m_1-2$, noting that $v_{m+1}- v_m =280 \eps_k$, we have
\begin{align}\label{eq:m_space}
152 \eps_k\le \tau_{m+1} -\tau_m \le 308 \eps_k.
\end{align}
Noting that $v_0=20\eps$, we also have
\begin{align*}
\tau_0 \le 20 \eps +128 \eps_k.
\end{align*}
Since we have assumed that $T_2 \ge 16 c_1 M \eps \log \eps^{-1}$ and $m_1 \ge m_0/2-1$, it follows that
$
m_1 >( 8 c_1 M \eps \log \eps^{-1} -20\eps )/(280 \eps_k) -2.
$
By~\eqref{eq:tau_m}, we know that by possibly decreasing $\eps_0$, we have
\begin{align*}
\tau_{m_1-1} \ge v_{m_1-1} -128 \eps_k = 20\eps+ 280 \eps_k (m_1-1) -128 \eps_k \ge 7 c_1 M \eps \log \eps^{-1}.
\end{align*}

\begin{figure}
\begin{center}
\includegraphics[width=\textwidth]{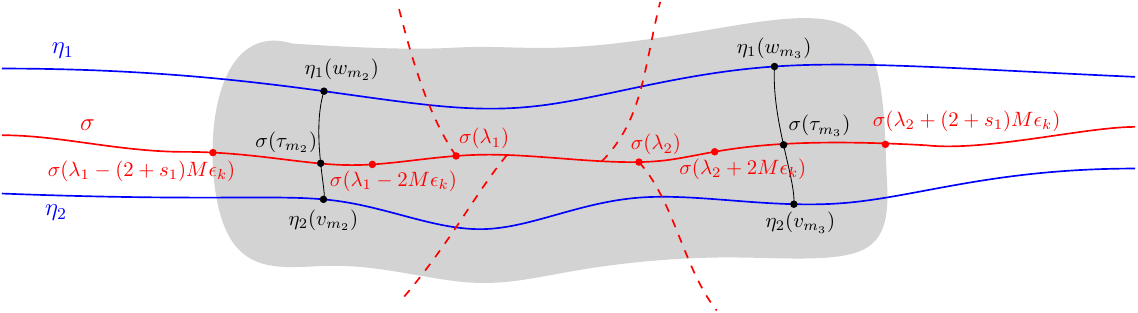}	
\end{center}
\caption{\label{fig:strong_conf_proof} Illustration of Step 2 of the proof of Proposition~\ref{prop:strong_confluence}.  The grey area represents $\CF(U_0)$, the filled $s_1M \eps_k$-neighborhood of $\sigma([\lambda_1- 2M\eps_k, \lambda_2 + 2M\eps_k])$. We show that $\eta_1$ and $\eta_2$ must intersect the same \X.}
\end{figure}

Fix $M \ge \max (1200 /s_1, 200/s_0)$. 
Let 
$$t_1= \tau_0 + ((s_1 +2) M - 800) \eps_k, \quad  t_2 =t_1+ c_1 (M \eps_k) \log (M \eps_k)^{-1}.$$ By possibly decreasing $\eps_0$, we have
\begin{align}\label{eq:t1t2}
t_2 \le \tau_{m_1-1} - ((s_1+ 2)M -800) \eps_k.
\end{align}
By Lemma~\ref{lem:eps_net_points}, we know that $\sigma|_{[t_1, t_2]}$ must pass through an \X\ which is $(s_0 M\eps_k, s_1 M\eps_k, 2M\eps_k)$-good.  See Figure~\ref{fig:strong_conf_proof}.
We suppose that the center of this \X\ is $\sigma([\lambda_1, \lambda_2])$ where $t_1 \le \lambda_1 < \lambda_2 \le  t_2$. Let $U_0$ be the $s_1M \eps_k$-neighborhood of $\sigma([\lambda_1- 2M\eps_k, \lambda_2 + 2M\eps_k])$. Let $\CF(U_0)$ be the complement in $\CS$ of the $z_j$-containing connected component of $\CS \setminus U_0$. Then none of the four branches of the \X\ is contained in $\CF(U_0)$.
We claim that there exist $0 \le m_2 < m_3 \le m_1-1$ so that 
\begin{align}\label{eq:lambda12_interval}
\lambda_1 - (s_1 +2) M \eps_k +800 \eps_k \le \tau_{m_2} \le \lambda_1 - 2M \eps_k, \quad
\lambda_2 + 2M \eps_k \le \tau_{m_3} \le \lambda_2 + (s_1+2 ) M \eps_k -800\eps_k.
\end{align}
Indeed, on the one hand, since $t_1 \le \lambda_1 < \lambda_2 \le  t_2$, we have by~\eqref{eq:t1t2} that
\begin{align*}
\lambda_1 - (s_1 +2) M \eps_k + 800\eps_k \ge \tau_0, \quad \lambda_2 + (s_1+2 ) M \eps_k - 800\eps_k\le \tau_{m_1-1}.
\end{align*}
On the other hand, the intervals $[\lambda_1 - (s_1 +2) M \eps_k + 800\eps_k,   \lambda_1 - 2M \eps_k]$ and $[\lambda_2 + 2M \eps_k, \lambda_2 + (s_1+2 ) M \eps_k - 800\eps_k]$ both have length $(s_1 M -800) \eps_k \ge 400 \eps_k$ which is larger than the spacing between the $\tau_m$'s  given by~\eqref{eq:m_space}.

Let $W$ denote the region which is surrounded clockwise by the concatenation of $\eta_1([w_{m_2}, w_{m_3}])$, the time reversal of $\xi_{m_3}$, the time reversal of $\eta_2([v_{m_2} , v_{m_3}])$ and $\xi_{m_2}$. 
Let us show that $W \subset \CF(U_0)$.
First note that by~\eqref{eq:lambda12_interval} we have
$$d(\sigma(\tau_{m_2}), \sigma(\lambda_1-2 M \eps_k)) \le s_1  M \eps_k -800 \eps_k, \quad d(\sigma(\tau_{m_3}), \sigma(\lambda_2 +2 M \eps_k)) \le s_1  M \eps_k -800 \eps_k.$$
This implies that for each $m_2\le m\le m_3$, we have
\begin{align*}
d(\sigma(\tau_m), \sigma([\lambda_1-2M \eps_k, \lambda_2 + 2 M \eps_k]) \le s_1  M \eps_k -800 \eps_k.
\end{align*}
For any point $z\in \xi_{m_2}$, we have
\begin{align*}
d(z, \sigma(\lambda_1-2 M \eps_k)) \le &d(z, \sigma(\tau_{m_2})) + d(\sigma(\tau_{m_2}), \sigma(\lambda_1-2 M \eps_k))\\
 \le& L + s_1  M \eps_k -800 \eps_k \le s_1 M\eps_k-672 \eps_k.
\end{align*}
Therefore $z \in U_0$, hence $\xi_{m_2} \subset  U_0$.
Similarly, we can deduce that $\xi_{m_3} \subset U_0$. 
For each $m_2\le m\le m_3-1$ and $v_{m} \le t \le v_{m+1}$,  we have
\begin{align*}
d(\eta_2(t),\sigma([\lambda_1 -2 M \eps_k, \lambda_2 + 2 M \eps_k])) \le & d (\eta_2(t), \eta_2(v_m)) + d(\eta_2(v_m), \sigma([\lambda_1 -2 M \eps_k, \lambda_2 + 2 M \eps_k]))\\
\le & 280\eps_k + L+ s_1 M\eps_k-800 \eps_k \le s_1 M \eps_k.
\end{align*}
This implies that $\eta_2([v_{m_2}, v_{m_3}]) \subset U_0$. For each $m_2\le m\le m_3-1$ and $w_{m} \le t \le w_{m+1}$,  we have
\begin{align*}
d(\eta_1(t),\sigma([\lambda_1 -2 M \eps_k, \lambda_2 + 2 M \eps_k])) \le & d (\eta_1(t), \eta_1(w_m)) + d(\eta_1(w_m), \sigma([\lambda_1 -2 M \eps_k, \lambda_2 + 2 M \eps_k]))\\
\le & 536 \eps_k + L + s_1 M\eps_k- 800 \eps_k \le s_1 M \eps_k,
\end{align*}
where the number $536 \eps_k$ in the above line comes from \eqref{eq:w_m_spacing}.
This implies that $\eta_1([w_{m_2}, w_{m_3}]) \subset U_0$.   
Altogether, we have shown that $\partial W\subset U_0$, hence $W \subset \CF(U_0)$.

Since none of the four branches of the \X\ is contained in $\CF(U_0)$, they are also not contained in $W$. In addition, any branch of this \X\ cannot intersect $\xi_{m_2}$. Indeed, by~\eqref{itm:goodX_dist} in the definition of the goodness of the \X, since $\tau_{m_2} \le \lambda_1 - M \eps_k$, the distance from $\sigma(\tau_{m_2})$ to any branch of this \X\ is at least $\lambda_1 - M \eps_k -\tau_{m_2} + s_0 M \eps_k \ge s_0 M \eps_k \ge 200 \eps_k$.  However, if one of the branches of the \X\ intersects $\xi_{m_2}$, then the distance from $\sigma(\tau_{m_2})$ to this branch would be at most the length of $\xi_{m_2}$ which is at most $128\eps_k$, leading to a contradiction.
Similarly, no branch of this \X\ can intersect $\xi_{m_3}$. Therefore, each of the four branches of the \X\ must exit $W$ from $\eta_1$ or $\eta_2.$
By the uniqueness of the geodesics which make up the two sides of each {\X} (recall the definition of an {\X}), it follows that both $\eta_1$ and $\eta_2$ must intersect a common part of $\sigma([\lambda_1, \lambda_2])$. 
In other words, there exists $t\in [\lambda_1, \lambda_2]$ such that $\sigma(t)\in\eta_1\cap\eta_2$. Since
$d(\sigma(t), \eta_2(0) ) =t \le t_2  \le \delta$ and $d(\sigma(t), \eta_1(0) ) \le d(\sigma(t), \eta_2(0) ) + 5 \eps \le \delta$, we have proved~\eqref{eq:prop1}.

\emph{Step 3. End of the proof.} It remains to prove the following
\begin{align}\label{eq:prop2}
 \eta_i( [T_i- \delta, T_i)) \cap \eta_{3-i} \neq \emptyset  \quad \text{for} \quad i=1,2,
\end{align}
We will treat separately the cases $T_2  -v_{m_1} < 4 c_1 M \eps \log \eps^{-1}$ and $T_2  - v_{m_1} \ge 4 c_1 M \eps \log \eps^{-1}$.

\begin{enumerate}[1.]
\item Suppose that $T_2  -v_{m_1}  < 4 c_1 M \eps \log \eps^{-1}$, which means that $\sigma(S)$ is near $\eta_2(T_2)$. In this case, we can apply the same argument as in Step 2 for the other end of $\sigma$.
By~\eqref{eq:tau_m} and~\eqref{eq:m_space}, we have 
$$T_2 -\tau_{m_1-1} \le T_2 - v_{m_1-1} + L \le T_2 - v_{m_1} + 280\eps_k +128\eps_k \le 5 c_1 M \eps \log \eps^{-1}.$$
Let 
$$t_2=\tau_{m_1-1} - ((s_1+2) M -800) \eps_k, \quad t_1=t_2 - c_1 (M\eps_k) \log (M \eps_k)^{-1}.$$
Arguing like before, we know that $\sigma|_{[t_1, t_2]}$ must pass through an  \X\ which is $(s_0 M\eps_k, s_1 M\eps_k, 2M\eps_k)$-good, and both $\eta_1$ and $\eta_2$  intersect a common part of $\sigma([t_1, t_2])$. Suppose that for some $t\in[t_1, t_2]$, $\sigma(t) \in \eta_1 \cap \eta_2$, then 
$$d(\sigma(t), \eta_2(T_2))=T_2 - t \le T_2 -t_1 \le 7 c_1 M \eps \log \eps^{-1} \le \delta.$$ 
and $d(\sigma(t), \eta_1(T_1)) \le d(\sigma(t), \eta_2(T_2)) + 5\eps \le \delta$. This implies~\eqref{eq:prop2}.

\item Suppose that $T_2  - v_{m_1} \ge 4 c_1 M \eps \log \eps^{-1}$, which means that $\sigma(S)$ is far away from $\eta_2(T_2)$. In this case, we can apply  the same argument as in Step 2 for $\wt\sigma$. Note that for each $m_1\le m\le m_0$, we have
\begin{equation}\label{eq:wttau}
\begin{split}
&|(\wt \tau_m - \wt \tau_{m_0}) - (v_{m_0} - v_m) | = |d(\wt\sigma(\wt\tau_m), \wt\sigma(\wt\tau_{m_0})) -  d(\eta_2(v_{m}), \eta_2(v_{m_0}))| \\
\le & d(\wt\sigma(\wt\tau_m),\eta_2(v_m))  + d(\wt\sigma(\wt\tau_{m_0}),  \eta_2(v_{m_0})) \le 2L \le 256 \eps_k.
\end{split}
\end{equation}
Noting that $v_{m+1} - v_m =280 \eps_k$, it follows that for all $m_1\le m \le m_0-1$, we have
\begin{align*}
24 \eps_k \le \wt \tau_{m} - \wt \tau_{m+1} \le 536 \eps_k.
\end{align*}
We also have
\begin{align*}
&|\wt \tau_{m_0} - (T_2 -v_{m_0})| = |d(\wt\sigma(\wt \tau_{m_0}), \eta_1(T_1)) - d(\eta_2(v_{m_0}), \eta_2(T_2))| \\
\le& d(\wt\sigma(\wt \tau_{m_0}), \eta_2(v_{m_0})) + d(\eta_1(T_1), \eta_2(T_2)) \le L +5\eps \le 128 \eps_k + 5\eps.
\end{align*}
It follows that by possibly decreasing $\eps_0$, we have $\wt\tau_{m_0} \le c_1 M \eps \log \eps^{-1}$.
Let 
$$t_1= \wt\tau_{m_0} + ((s_1+2) M -800) \eps_k, \quad t_2 =  t_1 + c_1 (M\eps_k) \log (M \eps_k)^{-1}. $$ 
Noting that
\begin{align*}
m_0 \ge (T_2 -40\eps) /(280 \eps_k) -2,
\end{align*}
putting $m=m_1$ into~\eqref{eq:wttau}, we get
\begin{align*}
\wt\tau_{m_1} - \wt\tau_{m_0} \ge &(m_0 -m_1) 280 \eps_k - 256 \eps_k \ge T_2 -40\eps - 560 \eps_k - 280 m_1 \eps_k -256 \eps_k\\
\ge & T_2 - v_{m_1} - c_1 M \eps \log\eps^{-1} \ge
3 c_1 M \eps \log \eps^{-1}.
\end{align*}
Therefore, by possibly decreasing the value of $\eps_0$, we have that 
$$t_2 \le  \wt\tau_{m_1} - ((s_1+2)M -800) \eps_k.$$
We can argue as before that $\wt \sigma|_{[t_1, t_2]}$ must pass through an  \X\ which is $(s_0 M\eps_k, s_1 M\eps_k, 2M\eps_k)$-good, and both $\eta_1$ and $\eta_2$  intersect a common part of $\wt \sigma([t_1, t_2])$. Suppose that for some $t\in[t_1, t_2]$, $\wt \sigma(t) \in \eta_1 \cap \eta_2$, then 
$d(\wt\sigma(t), \eta_1(T_1))= t  \le t_2  \le \delta$ 
and $d(\wt \sigma(t), \eta_2(T_2)) \le d(\wt\sigma(t), \eta_1(T_1))+ 5 \eps \le \delta$. This again implies~\eqref{eq:prop2}.
\end{enumerate}
Altogether, we have thus completed the proof of Proposition~\ref{prop:strong_confluence}.
\end{proof}

\section{Finite number of geodesics}
\label{sec:finite_number_of_geodesics}

We are now going to show that the number of disjoint geodesics which can emanate from any given point in the Brownian map is at most a deterministic constant (Section~\ref{subsec:finite_from_a_point}) and likewise that the number of geodesics between any pair of points is at most a deterministic constant (Sections~\ref{subsec:approx_tail_bounds}, \ref{subsec:finite_between_points}).  
These two results will be important in Sections~\ref{sec:geodesic_structure}, \ref{sec:exponent_disjoint_geodesics}, \ref{sec:dimension} when we complete the proofs of the main theorems.
In particular, the proof of the latter result is based on the computation of the exponent for a collection of geodesics between a pair of points to have a given number of splitting points which will be an important input in the proof of Theorem~\ref{thm:finite_number_of_geodesics}.

\subsection{Finite number of geodesics from a point}
\label{subsec:finite_from_a_point}
The goal of this section is to prove the following proposition, which is a weaker version of Theorem~\ref{thm:maximum_geodesics}.

\begin{proposition}
\label{prop:single_point_number_of_geodesics}
There exists a constant $C > 0$ so that the following is true.  For $\bminflaw$ a.e.\ instance $(\CS,d,\nu,x,y)$ and $z \in \CS$, the number of geodesics starting from $z$ which are otherwise disjoint is at most $C$.
\end{proposition}

The main step in the proof of Proposition~\ref{prop:single_point_number_of_geodesics} is the following lemma, which gives that the number of disjoint geodesics which can cross a metric band is a.s.\ finite. The main ingredients of its proof are Proposition~\ref{prop:strong_confluence} and a compactness argument. We emphasize that the proof of Proposition~\ref{prop:single_point_number_of_geodesics} will not lead to the optimal value of $C$ in the statement of Proposition~\ref{prop:single_point_number_of_geodesics}.  We will later determine as a consequence of the main result of Section~\ref{sec:exponent_disjoint_geodesics} the correct exponent for having $k$ geodesics emanate from a point which are otherwise disjoint. 

\begin{lemma}
\label{lem:number_of_crossings}
For each $p \in (0,1)$ and $w > 0$ there exists $k_0 \in \N$ so that the following is true.  Suppose that $(\CB,d_\CB,\nu_\CB,z)$ has law $\bandlaw{1}{w}$.  Let $A_k$ be the event that there exist at least $k$ disjoint geodesics in $\CB$ which connect a point on $\innerboundary \CB$ to a point on $\outerboundary \CB$.  Then $\p[A_k] \leq p$ for all $k \geq k_0$.
\end{lemma}
We will prove Lemma~\ref{lem:number_of_crossings} by realizing an instance from the law $\bandlaw{1}{w}$ inside of the ambient space given by an instance from the distribution $\bminflaw$.  The reason for doing so is to avoid the issue of whether the interior-internal metric associated with a metric band is compact (though this follows from the fact that the law of a Brownian disk is compact \cite{bm2017disk} and the equivalence of Brownian disks as defined in \cite{bm2017disk} and as the complement of a metric ball as in \cite{ms2015axiomatic} as established in \cite{lg2019disksnake}).
\begin{proof}[Proof of Lemma~\ref{lem:number_of_crossings}]
Let $(\CS,d,\nu,x,y)$ be distributed according to $\bminflaw$ and for each $r > 0$ we let $Y_r$ be the boundary length of $\partial \fb{y}{x}{d(x,y)-r}$.  Let $\tau = \inf\{r \geq 0 : Y_r = 1\}$.  Conditionally on $\tau < \infty$, we have that $\CB = \fb{y}{x}{d(x,y)-\tau} \setminus \fb{y}{x}{d(x,y)-\tau-w}$ has law $\bandlaw{1}{w}$.  We let $d|_{\ol{\CB}}$ be the restriction of the metric $d$ on $\CS$ to $\ol{\CB}$.  Since $(\ol{\CB},d|_{\ol{\CB}})$ is a.s.\ a compact metric space (as $(\CS,d)$ is a compact metric space and $\ol{\CB}$ is closed in $\CS$), it follows that the Hausdorff topology on closed subsets of $(\ol{\CB},d|_{\ol{\CB}})$ is also compact.  In particular, the Hausdorff topology on $(\ol{\CB},d|_{\ol{\CB}})$ is totally bounded.  Fix $\epsilon > 0$; we will adjust its value later in the proof.  For each $k \in \N$, let $B_k$ be the event that there exist at least $k$ geodesics which cross $\CB$ from $\innerboundary \CB$ to $\outerboundary \CB$ so that the Hausdorff distance (computed using $d|_{\ol{\CB}}$) between any pair of them is at least $\epsilon$.  By the total boundedness of $(\ol{\CB},d|_{\ol{\CB}})$, we can then choose $k_0 \in \N$ so that $k \geq k_0$ implies that $\p[B_k] \leq p/2$.  On $B_k^c$, for any $k$ geodesics which cross $\CB$ from $\innerboundary \CB$ to $\outerboundary \CB$, there exist at least two of them that are within Hausdorff distance $\eps$ from each other.  By Proposition~\ref{prop:strong_confluence}, we can choose $\epsilon > 0$ sufficiently small so that with probability $1-p/2$ we have $\eps<\eps_0$ (where $\eps_0$ is the random variable from Proposition~\ref{prop:strong_confluence}), which implies that these two geodesics are not disjoint.  This implies that, with probability at least $1-p$, there do not exist $k$ disjoint geodesics which cross $\CB$ from $\innerboundary \CB$ to $\outerboundary \CB$.
\end{proof}

\begin{lemma}
\label{lem:geodesic_exponent}
For each $\alpha >0$ there exists $k_0 \in \N$ so that the following is true.
Suppose that $(\CS,d,\nu,x,y)$ is distributed according to $\bminflaw$.  For each $r > 0$, let $Y_r$ be the boundary length of $\partial \fb{y}{x}{d(x,y)-r}$.  Fix $\ell > 0$. Let $\tau_0 = \inf\{r \geq 0 : Y_r = \ell\}$. Let $\sigma_0=\inf\{r\ge \tau_0: Y_r=\eps\}$.  We condition on the event that $\{\tau_0 < \infty\}$ and $Y_r < 2\ell$ for all $r\in(\tau_0, \sigma_0)$. 
The conditional probability that there are at least $k_0$ disjoint geodesics in $\CS$ which connect $\partial \fb{y}{x}{d(x,y)-\tau_0}$ to $\partial \fb{y}{x}{d(x,y)-\sigma_0}$ is $O((\epsilon/\ell)^\alpha)$, where the implicit constant is uniform in $\ell >0$.
\end{lemma}
\begin{proof}
Throughout, we shall be working with $(\CS,d,\nu,x,y)$ conditioned on $\{\tau_0 < \infty\}$. Let $\CB$ denote the metric band $\fb{y}{x}{d(x,y)-\tau_0} \setminus \fb{y}{x}{d(x,y)-\sigma_0}$ conditioned on $Y_r < 2\ell$ for all $r\in(\tau_0, \sigma_0)$. 
We have made such a conditioning because $\CB$ naturally occurs in the Brownian map: If we do not condition on $Y_r < 2\ell$ for all $r\in(\tau_0, \sigma_0)$, then the process $Y_r$ would make a finite number of up-crossings from $\ell$ to $2\ell$ before reaching $\eps$. If this number is $0$, namely $Y_r$ does not reach $2\ell$ for $r\in(\tau_0, \sigma_0)$, then let $\theta= \tau_0$. Otherwise, let $\theta$ be the first time that $Y_r$ reaches $\ell$ after the last time that it gets above $2\ell$ for  $r\in(\tau_0, \sigma_0)$. The metric band $\fb{y}{x}{d(x,y)-\theta} \setminus \fb{y}{x}{d(x,y)-\sigma_0}$ then has the same law as $\CB$.
Even though $\theta$ is not a stopping time, the law of the metric band $\fb{y}{x}{d(x,y)-\theta} \setminus \fb{y}{x}{d(x,y)-\sigma_0}$ is still independent from $\CS \setminus \fb{y}{x}{d(x,y)-\theta}$ conditionally on $\theta$.

From now on, we condition on $\{\tau_0 < \infty\}$ and $Y_r < 2\ell$ for all $r\in(\tau_0, \sigma_0)$.  For each $j \in \N_0$, we let $\tau_{j+1} = \inf\{r > \tau_{j} : Y_r = Y_{\tau_{j}}/2 = 2^{-j-1} \ell\}$. For $r\in(\tau_j, \tau_{j+1})$, the process $Y_r$ makes a finite number of up-crossings from $2^{-j}\ell$ to $2^{-j+1} \ell$ before reaching $2^{-j-1} \ell$. If this number is $0$, namely $Y_r$ does not reach $2^{-j+1} \ell$ for $r\in(\tau_j, \tau_{j+1})$, then let $\theta_j= \tau_j$. Otherwise, let $\theta_j$ be the first time that $Y_r$ reaches $2^{-j}\ell$ after the last time that it gets above $2^{-j+1}\ell$ for  $r\in(\tau_j, \tau_{j+1})$. 
Let $\CB_j$ be a metric band with inner boundary length $2^{-j}\ell$ stopped when the boundary length reaches $2^{-j-1}\ell$, and let $\p_{\CB_j}$ denote its probability law. Then the band $\fb{y}{x}{d(x,y)-\theta_j} \setminus \fb{y}{x}{d(x,y)-\tau_{j+1}}$ has the law of $\CB_j$ conditioned on the event $E_j$ such that the boundary length of $\CB_j$ never gets above $2^{-j+1}\ell$.
Let $\sigma_{j+1} = \theta_j + w Y_{\tau_j}^{1/2}$. Fix $p \in (0,1)$.  We assume that $w > 0$ is chosen sufficiently small so that the probability that $\tau_{j+1} \leq \sigma_{j+1}$ is at most $p/2$. 
Given $\CB_j$, let $\CB_j^1$ be the metric band which consists of all the points in $\CB_j$ which are not disconnected from $\innerboundary \CB_j$ by the $(d_0- w (2^{-j}\ell)^{1/2})$-neighborhood of $\outerboundary \CB_j$, where $d_0$ is the width of $\CB_j$. Let $A_{k_0}^j$ be the event that there are at least $k_0$ disjoint geodesics in the metric band $\CB_j^1$ from $\innerboundary \CB^1_j$ to $\outerboundary \CB^1_j$. Then there exists  $k_0 \in \N$ such that
\begin{align*}
\p_{\CB_j} (A^j_{k_0} \mid E_j) \le \p_{\CB_j} (A^j_{k_0}) / \p_{\CB_j} (E_j) \le p/2.
\end{align*}
Indeed, $C:= \p_{\CB_j} (E_j)$ is a positive constant independent of $j$ and $\ell$. Lemma~\ref{lem:number_of_crossings} (and the scaling property for metric bands) implies that for there exists $k_0 \in \N$ so that $\p_{\CB_j} (A^j_{k_0}) \le Cp/2$.
This implies that the probability that there are at least $k_0$ disjoint geodesics which cross $\fb{y}{x}{d(x,y)-\theta_j} \setminus \fb{y}{x}{d(x,y) - \tau_{j+1}}$ is at most $p$.

If there are at least $k_0$ disjoint geodesics which cross $\fb{y}{x}{d(x,y) - \tau_0} \setminus \fb{y}{x}{d(x,y) - \tau_n}$, then there would be at least $k_0$ disjoint geodesics which cross $\fb{y}{x}{d(x,y) - \theta_j} \setminus \fb{y}{x}{d(x,y) - \tau_{j+1}}$ for all $0\le j\le n-1$. Taking $n=\lfloor \log_2 (\ell/\eps) \rfloor$ so that $\tau_n < \sigma_0$, we deduce that the probability that there are at least $k_0$ disjoint geodesics which cross $\fb{y}{x}{d(x,y) - \tau_0} \setminus \fb{y}{x}{d(x,y) - \sigma_0}$ is at most $p^n$. Choosing $p=2^{-\alpha}$ implies the lemma.
\end{proof}

\begin{lemma}
\label{lem:diameter_of_complement}
For $\bminflaw$ a.e.\ instance of $(\CS,d,\nu,x,y)$ the following is true.  For every $0<R <\diam(\CS)$ there exists $r_0 > 0$ so that for all $r \in (0,r_0)$ and $z \in \CS$ there is at most one connected component of $\CS \setminus B(z,r)$ with diameter at least $R$.
\end{lemma}
\begin{proof}
Suppose that $(\CS,d,\nu,x,y)$ has distribution $\bminflaw$.  Suppose that the assertion of the lemma is not true.  Then there exists $R > 0$ and a positive sequence $(r_n)$ with $r_n \to 0$ as $n \to \infty$ as well as a sequence $(z_n)$ in $\CS$ so that $\CS \setminus B(z_n,r_n)$ for every $n \in \N$ has at least two components $U_n, V_n$ both of which have diameter at least $R > 0$.  Since $\partial B(z_n,r_n)$ has diameter at most $2r_n$, it follows that there exists $n_0 \in \N$ so that for each $n \geq n_0$ we can find $u_n \in U_n$ so that $B(u_n,R/4) \subseteq U_n$ and $v_n \in V_n$ so that $B(v_n,R/4) \subseteq V_n$.  By passing to a subsequence if necessary, we may assume without loss of generality that $u_n,v_n,z_n$ respectively converge to distinct points $u,v,z$.  Since every path from $u_n$ to $v_n$ passes through $\partial B(z_n,r_n)$, it follows that every path from $u$ to $v$ passes through $z$.  This is a contradiction to the fact that $\CS$ is homeomorphic to $\s^2$, which completes the proof.
\end{proof}

\begin{proof}[Proof of Proposition~\ref{prop:single_point_number_of_geodesics}]
Fix $\sp{a} > 0$ and let $\alpha = 8+3 \sp{a}$.  Let $k_0 \in \N$ be so that the exponent from Lemma~\ref{lem:geodesic_exponent} is at least $\alpha$.  Let $(z_j)$ be an i.i.d.\ sequence chosen from $\nu$.  Lemma~\ref{lem:typical_points_dense} implies that there exists $\epsilon_0 > 0$ so that for all $\epsilon \in (0,\epsilon_0)$ the following is true.  If we let $N = \epsilon^{-4-\sp{a}}$ then for every $z \in \CS$ there exists $1 \leq j \leq N$ so that $d(z,z_j) \leq \epsilon$.

Fix $\epsilon > 0$. For each $1\le j \le N$ and $r > 0$, let $Y_r^j$ be the boundary length of $\partial \fb{y}{z_j}{d(z_j,y)-r}$. Let $\tau_j = \inf\{r>0 : Y_r^{j} = \epsilon^{1/2}\}$.  
On the event $\{\tau_j<\infty\}$, the process $(Y^j_r, r> \tau_j)$ makes a finite number of up-crossings from $\eps$ to $\eps^{1/2}$ before reaching $0$. If this number is $0$, then let $\omega_j=\tau_j$. Otherwise, let $\omega_j$ be the last time that  $(Y^j_r, r> \tau_j)$ completes an up-crossing  from $\eps$ to $\eps^{1/2}$. Let $\sigma_j:=\inf\{r>\omega_j: Y^j_r =\eps\}$.
Conditionally on $\omega_j$, $\fb{y}{z_j}{d(z_j,y) - \omega_j} \setminus \fb{y}{z_j}{d(z_j,y) - \sigma_j}$ is distributed as an (unconditioned) metric band with inner boundary length $\eps^{1/2}$ and stopped when the boundary length $Y^j_r$ first reaches $\eps$, and $\fb{y}{z_j}{d(z_j,y) - \sigma_j}$  is distributed as a metric band with inner boundary length $\eps$ and stopped when the boundary length $Y^j_r$ first reaches $0$, conditioned on the event that the boundary length $Y^j_r$ stays below $\eps^{1/2}$.

On the event $\{\tau_j<\infty\}$, the process $(Y^j_r,  \omega_j <r < \sigma_j)$ makes a finite number of up-crossings
from $\eps^{1/2}$ to $2\eps^{1/2}$ before reaching $\eps$. If this number is $0$, then let $\theta_j=\omega_j$. Otherwise, let $\theta_j$  be the first time that $(Y^j_r: \omega_j<r <\sigma_j)$ reaches $\eps^{1/2}$ after making the last up-crossing from $\eps^{1/2}$ to $2\eps^{1/2}$. By  Lemma~\ref{lem:geodesic_exponent}, we know that conditionally on $\{\tau_j<\infty\}$, the probability that there are at least $k_0$ disjoint geodesics in $\CS$ which which connect $\partial \fb{y}{z_j}{d(z_j,y) - \theta_j}$ to $\partial\fb{y}{z_j}{d(z_j,y) - \sigma_j}$ is $O(\eps^{\alpha/2})$.
By performing a union bound with respect to $j$ and then applying the Borel-Cantelli lemma, we can deduce that one can find $\eps_1\in(0, \eps_0)$ so that for all $\eps\in(0,\eps_1)$, for all $1 \leq j \leq N$, if $\tau_j<\infty$, then there are at most $k_0-1$ disjoint geodesics in $\CS$ which which connect $\partial \fb{y}{z_j}{d(z_j,y) - \theta_j}$ to $\partial\fb{y}{z_j}{d(z_j,y) - \sigma_j}$.

For each $1\le j \le N$,  arguing as in the proof of Lemma~\ref{lem:x_concentration}, we know that for some $c>0$
\[
\bminflaw (\tau_j=\infty, d(y,z_j) \geq \epsilon^{1/8}) \le \exp(-c \epsilon^{-1/8}).
\] 
Indeed, in each round of time of length $\epsilon^{1/4}$ in which $Y^j_r \leq \epsilon^{1/2}$ there is a positive chance that $Y^j_r$ hits $0$ and there are at least $\epsilon^{-1/8}$ such rounds when $d(y,z_j) \geq \epsilon^{1/8}$.  Therefore by performing a union bound and then applying the Borel-Cantelli lemma, we can deduce that one can find $\eps_2\in (0, \eps_1)$ so that for all $\eps\in(0, \eps_2)$, for all $1 \leq j \leq N$ with $d(y,z_j) \geq \epsilon^{1/8}$, the event $\tau_j<\infty$ holds.

On the event $\{\tau_j<\infty\}$, we have $Y^j_r < 2\eps^{1/2}$ for $r\in (\theta_j, \sigma_j)$. Arguing as in the previous paragraph, we have for some constant $c_1>0$
\begin{align*}
\bminflaw (\sigma_j - \theta_j \ge \eps^{1/8} \mid \tau_j<\infty)  \le  \exp (-c_1 \eps^{-1/8}).
\end{align*}
By the second paragraph of the proof, we also have $Y^j_r < \eps^{1/2}$ for $r > \sigma_j$. Again using the same arguments (which led to the previous equation), this implies 
\begin{align*}
\bminflaw (d(z_j, y) - \sigma_j \ge \eps^{1/8} \mid \tau_j<\infty)  \le  \exp (-c \eps^{-1/8}).
\end{align*}
On the other hand, by~\eqref{eqn:csbp_extinction_time}, we have for some constant $c_2>0$
\begin{align*}
\bminflaw (d(x,y)-\sigma_j \le \epsilon \mid \tau_j<\infty)  \le \exp(-c_2 \epsilon^{-2} \times \epsilon) = \exp(-c_2 \epsilon^{-1}).
\end{align*}
By performing a union bound and then applying the Borel-Cantelli lemma, we can deduce that one can find $\eps_3\in(0, \eps_2)$ so that for all $\eps\in(0, \eps_3)$, for all $1 \leq j \leq N$ with $d(y,z_j) \geq \epsilon^{1/8}$, the event $\{\tau_j<\infty\}\cap\{\sigma_j - \theta_j < \eps^{1/8}\} \cap \{d(z_j,y) -\sigma_j <\eps^{1/8}\} \cap\{d(z_j, y) - \sigma_j > \eps\}$ holds. On this event, the metric band $\fb{y}{z_j}{d(z_j,y) - \theta_j} \setminus \fb{y}{z_j}{d(z_j,y) - \sigma_j}$ is contained in $\fb{y}{z_j}{2\eps^{1/8}} \setminus \fb{y}{z_j}{\eps}$. 

Combined, we deduce that  for all $\eps\in(0, \eps_3)$, for all $1 \leq j \leq N$ with $d(y,z_j) \geq 3 \epsilon^{1/8}$,  there are at most $k_0-1$ disjoint geodesics in $\CS$ which which connect $\partial \fb{y}{z_j}{\eps}$ to $\partial\fb{y}{z_j}{\eps^{1/8}}$.

Let $U_\delta = \{z \in \CS : d(z,y) > \delta\}$.  Then  for all $\delta' \in (0,\delta/2)$ the number of geodesics which connect any $z \in U_\delta$ to $\partial \fb{y}{z}{\delta'}$ and are disjoint other than at $z$ is at most $k_0-1$.  Lemma~\ref{lem:diameter_of_complement} implies that if $z \in \CS \setminus \{y\}$ and there are at least $k_0-1$ geodesics starting from $z$ which are disjoint other than at $z$ then there exists $\delta' > 0$ so that these geodesics connect $z$ to $\partial \fb{y}{z}{\delta'}$.  Indeed, if the minimal length of the geodesics is $2R > 0$ and $\delta' > 0$ is sufficiently small then they must all enter the same component of $\CS \setminus B(z,\delta')$ as there is only one such component which has diameter at least $R$ for small enough $\delta' > 0$.  In particular, this component must be the $y$-containing component of $\CS \setminus B(z,r)$ for all $\delta' > 0$ sufficiently small.  Since $\delta > 0$ was arbitrary, this proves the result for all $z \in \CS \setminus \{y\}$.  If we sample $y'$ from $\nu$ independently of everything else, then the same also holds for all $z \in \CS \setminus \{y'\}$.  Since we a.e.\ have that $y \neq y'$, this proves the proposition.
\end{proof}

\subsection{Approximate geodesic tail bounds}
\label{subsec:approx_tail_bounds}

In this section, we will establish the following estimate for the probability that a metric band has $k+1$ points on its outer boundary whose distance to the marked point on the inner boundary is within $\epsilon$ of the width of the band.  This estimate will then feed into Section~\ref{subsec:finite_between_points} where it will be used to get the exponent for the collection of geodesics between a pair of points having a given number of splitting points.  In order to read the sequel, one only needs to read the statement of the following lemmas and can skip its proof.

\begin{lemma}
\label{lem:metric_band_epsilon_point}
Fix $\xi \in (0,1)$ and $\epsilon > 0$.  Suppose that $\ell, w > 0$ and $(\CB,d_\CB,\nu,z)$ has law $\bandlaw{\ell}{w}$.  Fix $k\in\N$.
Let $F_{\epsilon,\xi,k}(\CB)$ (we also denote it by $F_{\epsilon,\xi,k}$ when there is no ambiguity) be the event that there exist points $z_1,\ldots,z_{k+1} \in \outerboundary \CB$ ordered counterclockwise so that the boundary length distance from each $z_i$ to $z_{i+1}$ (with $z_{k+2} = z_1$) is at least $\xi^2$ and $d_{\CB}(z_i,z) \le w+ \epsilon$ for $1 \leq i \leq k+1$.
Then  $$\bandlaw{\ell}{w} [F_{\epsilon,\xi,k}] = O((\epsilon/\xi)^{k+o(1)}) \text{ as } \eps\to 0,$$
where the implicit constant does not depend on $\ell, \xi, w$. 
\end{lemma}
\begin{remark}
Throughout the paper, when we make a statement such as ``$f(x,r)=O(x^{k+o(1)})$ as $x\to 0$ where the implicit constant does not depend on $r$'', we mean that for all $\delta>0$, there is a constant $C_\delta$ which does not depend on $r$ such that $f(x,r)\le C_\delta x^{k-\delta}$ for all $x$ small enough.
\end{remark}

We remark that the event $F_{\epsilon,\xi,k}$ in the statement of Lemma~\ref{lem:metric_band_epsilon_point} implies that $\outerboundary\CB$ is not equal to a point and that the distance between the inner and outer boundaries of $\CB$ is equal to $w$.

We will deduce Lemma~\ref{lem:metric_band_epsilon_point} from the following estimate that an instance $(\CD,d,\nu,\nu_{\partial \CD},y)$ of $\bdisklawweighted{\ell}$ has $k+1$ distinct points on its boundary whose distance to $y$ is within $\epsilon$ of being minimal.  The difference between the statement of the following lemma and Lemma~\ref{lem:metric_band_epsilon_point} is that the Brownian disk case arises as a forward exploration of the Brownian map while the metric band case arises from a reverse exploration.

\begin{lemma}
\label{lem:num_pinch_points_disk}
Let $\epsilon > 0$ and suppose that $\ell \geq \xi^2 \geq \epsilon^2$.  Suppose that $(\CD,d,\nu,\nu_{\partial \CD},y)$ has law $\bdisklawweighted{\ell}$.  Let $E_{\epsilon,\xi,k}$ be the event that there exist points $x_1,\ldots,x_{k+1} \in \partial \CD$ ordered counterclockwise so that the boundary length distance between $x_i$ and $x_{i+1}$ is at least $\xi^2$ for each $i$ (with $x_{k+2} = x_1$) and $d(x_i,y) \leq d(\partial \CD,y) + \epsilon$ for each $i$.  Then 
$$\p[E_{\epsilon,\xi,k}] = O((\epsilon/\xi)^{k+o(1)}) \text{ as } \eps\to 0,$$
where the implicit constant does not depend on $\ell$, $\xi$.  The same statement holds if we condition on $\nu(\CD)$ or if we assume instead that $d(x_i,x_j) \geq \xi$ for each $i,j$ distinct.
\end{lemma}
\begin{proof}
Let $b$ be the Brownian bridge on $[0,\ell]$ in the Brownian snake construction of $(\CD,d,\nu,y)$.  Let $t_*$ be the a.s.\ unique time at which $b$ attains its infimum.  The event in the lemma statement is equivalent to the existence of times $0 < s_1 < \cdots < s_{k+1} < \ell$ so that $b(s_j) - b(t_*) \leq \epsilon$ and $s_j-s_{j-1} \mod \ell \geq \xi^2$ for each $j$.  For $t \in [0,1]$, we let $e(t) = \ell^{-1/2}( b(\ol{t_*+ \ell t}) - b(t_*))$ where $\ol{s}$ denotes $s$ modulo $\ell$.  Then $e$ is a normalized Brownian excursion and the event in the proposition statement implies that there exist times $0 < s_1 < \cdots < s_k < 1$ so that $e(s_j) \leq \ell^{-1/2} \epsilon$ and $s_j-s_{j-1} \geq \xi^2/\ell$ for each $1 \leq j \leq k$.  Let $\tau_1 = \inf\{t \geq \xi^2/\ell : e(t) \leq \epsilon \ell^{-1/2} \}$.  Given that $\tau_1,\ldots,\tau_j$ have been defined, let $\tau_{j+1} = \inf\{t \geq \tau_j + \xi^2/\ell : e(t) \leq \epsilon \ell^{-1/2}\}$.  To prove the lemma, we will show that $\p[ \tau_k < 1-\xi^2/\ell] = O( (\epsilon/\xi)^{k+o(1)})$.

Let $Z$ be the Lebesgue measure of $\{t \in [\xi^2/\ell,1-\xi^2/\ell] : e(t) \leq \epsilon/\ell^{1/2}\}$.  Recall that the density at time $t$ for $e$ at $x > 0$ is given by
\[ 2\sqrt{2\pi} q_t(x) q_{1-t}(x) \quad\text{where}\quad q_t(x) = \frac{x}{\sqrt{2\pi t^3}} \exp\left( - \frac{x^2}{2t} \right).\]
It therefore follows that $\E[Z] = \Theta(\epsilon^3/(\xi \ell))$.  We also have that $\E[ Z \giv \tau_1 < 1-\xi^2/\ell] = \Theta(\epsilon^2/\ell)$.  Therefore
\[ \p[ \tau_1 < 1- \xi^2/\ell] = \frac{\E[Z]}{\E[Z \giv \tau_1 < 1-\xi^2/\ell]} = \Theta(\epsilon/\xi).\]

For each $t \in [0,1]$, let $\CF_t = \sigma(e(s) : s \leq t)$.  Given $\tau_j < 1-\xi^2/\ell$, we let $T_j = 1-\tau_j$ and $e_j(t) = T_j^{-1/2}(e(t \cdot T_j+\tau_j ))$.  Note that $e_j(0) = \epsilon / (T_j \ell)^{1/2}$.  Then the event that $\tau_{j+1} < 1-\xi^2/\ell$ is equivalent to the event that 
$$\sigma_{j+1} = \inf\{t \geq \xi^2/(\ell T_j) : e_j(t) \leq \epsilon/(\ell T_j)^{1/2}\}  \leq 1 - \xi^2/(\ell T_j).$$  Let $Z_j$ be the Lebesgue measure of $\{t \in [\xi^2/(2\ell T_j),1-\xi^2/(2 \ell T_j)] : e_j(t) \leq \epsilon/(\ell T_j)^{1/2}\}$.  Arguing as above, we have that $$\E[ Z_j \giv \CF_{\tau_j}, \tau_j < 1-\xi^2/\ell] =\Theta(\epsilon^3/(\xi \ell T_j)).$$  We also have that 
$$\E[ Z_j \giv \CF_{\tau_j},\ \sigma_{j+1} < 1 - \xi^2/(\ell T_j)] = \Theta(\epsilon^2/ (\ell T_j)).$$  Altogether, this gives that $\p[ \sigma_{j+1} < 1- \xi^2/(\ell T_j) \giv \CF_{\tau_j}, \tau_j < 1-\xi^2/\ell] =\Theta(\epsilon/\xi)$.  Combining, we see that 
$$\p[ \tau_{j+1} < 1- \xi^2/\ell \giv \CF_{\tau_j}, \tau_j < 1-\xi^2/\ell] = \Theta(\epsilon/\xi).$$  By iterating, we thus see that $\p[ \tau_k < 1-\xi^2/\ell] = O( (\epsilon/\xi)^k)$, as desired.

The result in the case that we condition on $\nu(\CD)$ follows because the event in the lemma statement is determined by the Brownian bridge $b$ in the construction of $(\CD,d,\nu,\nu_\partial,y)$ which, in turn, is independent of $\nu(\CD)$.

The result in the case that we assume $d(x_i,x_j) \geq \xi$ for each $i,j$ distinct follows from the case in which we assume that the boundary length distance from each $x_i$ to $x_{i+1}$ is at least $\xi^2$ and Lemma~\ref{lem:boundary_length_distance}.
\end{proof}

\begin{proof}[Proof of Lemma~\ref{lem:metric_band_epsilon_point}]
Fix $\sp{a},\xi \in (0,1)$, $\epsilon > 0$, and $\ell > 0$.  Let $(\CS,d,\nu,x,y)$ be sampled from $\bminflaw$.  We will prove the result by realizing the metric band instance $(\CB,d_\CB,\nu_\CB,z)$ inside of $(\CS,d,\nu,x,y)$.

For each $r > 0$, we let $Y_r$ be the boundary length of $\partial \fb{y}{x}{d(x,y)-r}$.  Let $\tau = \inf\{r \geq 0 : Y_r = \ell\}$.  
Let $U$ be a uniform random variable in $[0,1]$ independent of everything else.  
Let $\CB(U)= \fb{y}{x}{d(x,y) - \tau} \setminus \fb{y}{x}{d(x,y)-\tau-1-U}.$
On $\{\tau < \infty\}$ and conditionally on $U$, the metric band $\CB(U)$ equipped with its interior-internal metric and the restriction of $\nu$ has law $\bandlaw{\ell}{1+U}$.  Let us now show that  given $\tau < \infty$ and the distance between the inner and outer boundaries of $\CB(U)$ is $1+U$, the conditional probability that $\nu(\CB(U)) \leq (\epsilon/\xi)^{\sp{a}}$ decays to $0$ as $\epsilon/\xi \to 0$ faster than any power of $\epsilon/\xi$.  Let $\sigma:= \inf\{r \ge \tau: Y_r \ge (\eps/\xi)^{\sp{a}/4}\}$. Note that we have $\sigma \in [\tau, \tau+1+U/2]$ off an event with conditional probability tending to $0$ faster than any power of $\eps/ \xi$.
Indeed, for each round of time of length $(\eps/\xi)^{\sp{a}/8}$ such that $Y_r \le (\eps/\xi)^{\sp{a}/4}$, there is a positive probability that the process $Y_r$ hits $0$, and there are $O((\eps/\xi)^{-\sp{a}/8})$ such rounds if $Y_r \le (\eps/\xi)^{\sp{a}/4}$ for all $r \in [\tau, \tau+1+U/2]$. On the event that $\sigma \in [\tau, \tau+1+U/2]$, we can divide $\partial \fb{y}{x}{d(x,y)- \sigma}$ into $\lfloor (\eps/\xi)^{-\sp{a}/8} \rfloor$ equally spaced parts, so that each part has boundary length at least $(\eps/\xi)^{3\sp{a}/8}$. Each of these boundary parts is the inner boundary of a slice with width $\tau+1+U -\sigma \ge U/2$.  The slices are independent with each other, and each slice has $\nu$-volume at least $(\eps/\xi)^{3\sp{a}/4}$ with a positive probability. Therefore, the probability that $\nu(\CB(U)) \leq (\epsilon/\xi)^{\sp{a}}$ is less than the probability that every slice has $\nu$-volume at most $(\eps/\xi)^{3\sp{a}/4}$, which decays to $0$  as $\epsilon/\xi \to 0$ faster than any power of $\epsilon/\xi$.

For a metric band $\CB$, let $F_{\epsilon,\xi,k, \sp{a}} (\CB)$ be the event that $F_{\epsilon,\xi,k} (\CB)$ holds and the area of $\CB$ is at least $(\epsilon/\xi)^{\sp{a}}$. It suffices to show that 
\begin{align}\label{eq:p_band_goal}
\bandlaw{\ell}{1}[F_{\epsilon,\xi,k, \sp{a}}]=O( (\epsilon/\xi)^{k-\sp{a}+o(1)}) \quad\text{as}\quad \epsilon \to 0
\end{align}
where the implicit constant does not depend on $\ell, \xi$. This will imply the lemma since we can take $\sp{a}>0$ arbitrarily close to $0$.  We will consider the possibilities that $\ell \in (0,1]$ and $\ell > 1$ separately.

We first suppose that $\ell \in (0,1]$.  
In this case, there exists $c_0 > 0$ so that $\bminflaw(\tau < \infty) \geq c_0$ for all such $\ell$.  Our first goal is to show that
\begin{align}\label{eq:mu_BM_FB}
\bminflaw(F_{\epsilon,\xi,k, \sp{a}}(\CB(U)) \cap \{\tau<\infty\})=O( (\epsilon/\xi)^{k-\sp{a}+o(1)}),
\end{align}
which implies
\begin{align}\label{eq:EP_band}
\E[ \bandlaw{\ell}{1+U}[F_{\epsilon,\xi,k, \sp{a}}]] = O ( (\epsilon/\xi)^{k-\sp{a}+o(1)} ) \quad\text{as}\quad \eps\to 0.
\end{align}
For each $r > 0$, we let $\CD_r$ be the $y$-containing component of $\CS \setminus B(x,r)$.  We also let $E_{\epsilon,\xi,k, r, \sp{a}}$ be the intersection of the event $E_{\eps, \xi, k}$ from Lemma~\ref{lem:num_pinch_points_disk} for $\CD_r$ and the event that $\nu(\CD_r) \geq (\epsilon/\xi)^{\sp{a}}$.  Given the boundary length of $\CD_r$ is at least $\xi^2$ and $\nu(\CD_r) \geq (\epsilon/\xi)^{\sp{a}}$, the conditional probability of $E_{\epsilon,\xi,k, r, \sp{a}}$ is $O((\epsilon/\xi)^{k+o(1)})$ by Lemma~\ref{lem:num_pinch_points_disk}.  Let $X$ be the set of $r >0$ so that $E_{\epsilon,\xi,k, r, \sp{a}}$ occurs.  We have that
\begin{align}
\label{eq:leb_X}
\bminflaw[\leb (X)] = \int_0^\infty \bminflaw(E_{\epsilon,\xi,k, r, \sp{a}}) dr = O( (\epsilon/\xi)^{k-\sp{a}+o(1)})
\end{align}
where $\leb$ denotes the Lebesgue measure.  Here, we have used that the $\bminflaw$ measure of the set of Brownian map instances with area at least $(\epsilon/\xi)^{\sp{a}}$ is $O( (\epsilon/\xi)^{-\sp{a}})$.  Note that $F_{\epsilon,\xi,k, \sp{a}}(\CB(U)) \cap\{\tau < \infty\} $ is contained in the event  $\{d(x,y)-(\tau+1+U) \in X\}$. As $U$ is uniform in $[0,1]$ independently of everything else, it follows that the $\bminflaw$ measure of $F_{\epsilon,\xi,k, \sp{a}}(\CB(U)) \cap\{\tau < \infty\} $ is at most \eqref{eq:leb_X}, thus proving~\eqref{eq:mu_BM_FB} and~\eqref{eq:EP_band}.

From~\eqref{eq:EP_band}, we deduce that there exists $u_0 \in [1/2,1]$ so that 
\[\bandlaw{\ell}{1+u_0}[F_{\epsilon,\xi,k, \sp{a}}] = O( (\epsilon/\xi)^{k-\sp{a}+o(1)} ).\]
By the scaling property of metric bands, the same is also true for $\bandlaw{\ell/(1+u_0)^2}{1}$.  Suppose that $(\wh\CB,d,\nu,z)$ has law $\bandlaw{\ell/(1+u_0)^2}{1}$.  Consider a reverse metric exploration from $\innerboundary \wh\CB$ to $\outerboundary \wh\CB$.   That is, for each $r \in(0,1)$ we let $\CB_r$ be the set of points which have distance at most $d_0-r$ from $\outerboundary \wh{\CB}$ or are disconnected from $\innerboundary \wh{\CB}$ by such points where $d_0$ is the distance from $\innerboundary \wh{\CB}$ to $\outerboundary \wh{\CB}$.  Then $\CB_r$ is again a metric band and has width $1-r$.  Let $y_r$ be the point on $\innerboundary \CB_r$ visited by the a.s.\ unique geodesic from $z$ to $\outerboundary \wh\CB$.   If we let $Y_r$ be the boundary length of $\innerboundary \CB_r$, then given $Y_r$ the conditional law of $(\CB_r,d_r,\nu_r,y_r)$ is given by $\bandlaw{Y_r}{1-r}$.  
For any $r\in (0,1)$, if the event $F_{\epsilon,\xi,k, \sp{a}} (\CB_r)$ holds, then  $F_{\epsilon,\xi,k, \sp{a}}(\wh\CB)$ also holds.
Note that $Y_0 = \ell/(1+u_0)^2 < \ell$.  Let $\tau = \inf\{r > 0 : Y_r = (1-r)^2 \ell\}$.  
Then the probability that $\tau < 1/2$ is at least $p > 0$. 
On the event $\{\tau<1/2\}$, the conditional law of $(\CB_\tau,d_\tau,\nu_\tau,z_\tau)$ given $\tau$ and $Y_\tau$ is given by $\bandlaw{Y_\tau}{1-\tau}$.  If we rescale boundary lengths by $(1-\tau)^{-2}$ and distances by $(1-\tau)^{-1}$ then we obtain an instance of $\bandlaw{\ell}{1}$.  It therefore follows that
\begin{align*}
\bandlaw{\ell/(1+u_0)^2}{1}[F_{\epsilon,\xi,k, \sp{a}}(\wh\CB)]
&\geq \E\!\left[\one_{\tau < 1/2} \bandlaw{Y_{\tau}}{1-\tau}[F_{\epsilon,\xi,k, \sp{a}}(\CB_\tau)] \right] 
 \geq p \bandlaw{\ell}{1}[ F_{2 \epsilon,2\xi,k, \sp{a}}],
\end{align*}
where the expectation in the middle above is with respect to $\bandlaw{\ell/(1+u_0)^2}{1}$.  Since the left hand side is $O((\epsilon/\xi)^{k-\sp{a}+o(1)})$, it follows that the right hand side is as well.  By replacing $\epsilon$ with $\epsilon/2$ and~$\xi$ with $\xi/2$, we can deduce~\eqref{eq:p_band_goal}. 
This completes the proof of the assertion of the lemma for $\ell \in (0,1]$.

We now consider the case that $\ell > 1$.  By the Lamperti transform~\eqref{eqn:lamperti_csbp_to_levy}, $\bminflaw(\tau < \infty)$ is equal to the mass assigned under the infinite measure on $3/2$-stable L\'evy excursions with only upward jumps to the set of excursions whose supremum is at least $\ell$.  This, in turn, is equal to a constant times $\ell^{-1}$.  Therefore the bound that we obtain by dividing~\eqref{eq:mu_BM_FB} by $\bminflaw(\tau < \infty)$ and applying the argument as above for $\bandlaw{\ell}{1}[F_{\epsilon,\xi,k}]$ is $O(\ell (\epsilon/\xi)^{k+o(1)})$ as $\epsilon \to 0$.  We will now explain how to eliminate the factor of $\ell$.  Fix $\sp{a} > 0$ and assume that $\ell \geq (\epsilon/\xi)^{-\sp{a}}$.  Then we can let $N = \lfloor (\epsilon/\xi)^{-\sp{a}} \rfloor$ and place equally spaced points $x_1,\ldots,x_N$ on $\innerboundary \CB$ with $x_{N/2} = z$ where the counterclockwise boundary length distance from $x_i$ to $x_{i+1}$ is equal to $1$.  Then the slices bounded by the geodesics from the $x_i$ to $\outerboundary \CB$ are i.i.d.\ samples from $\slicelaw{1}{1}$.  In particular, there exists $p > 0$ so that in each slice, the probability that the distance between the two geodesics is at least $1$ is at least $p$.  This implies that the probability that there exists a point $z_0$ on the outer boundary of the clockwise slice between $x_1$ and $x_N$ with  $d_{\CB}(z_0,z) \le 1+ \epsilon$ decays to $0$ as $\epsilon/\xi \to 0$ faster than any power of $\epsilon/\xi$.  
On the event that there does not exist such a point $z_0$, we can bound from above the probability of the event from the lemma statement by gluing the two sides of the counterclockwise slice between $x_1$ and $x_N$ together to obtain a band with inner boundary length $\ell= (\epsilon/\xi)^{-\sp{a}}$. Consequently, $\bandlaw{\ell}{1}[F_{\epsilon,\xi,k}] = O((\epsilon/\xi)^{k-\sp{a}+o(1)})$ as $\epsilon \to 0$ for all $\ell \geq (\epsilon/\xi)^{-\sp{a}}$.  Since $\sp{a} > 0$ was arbitrary, we obtain that $\bandlaw{\ell}{1}[F_{\epsilon,\xi,k}] = O((\epsilon/\xi)^{k+o(1)})$, as desired.
\end{proof}

We will now deduce a variant of Lemma~\ref{lem:metric_band_epsilon_point} which will be used in Section~\ref{sec:exponent_disjoint_geodesics}.  This result will not be used in the remainder of the current section.
 
\begin{lemma}
\label{lem:metric_band_epsilon_point2}
Fix $\xi \in (0,1)$ and $\epsilon > 0$.  Suppose that $\ell > 0$ and $(\CB,d_\CB,\nu,z)$ has law $\bandlaw{\ell}{1}$. For each $t \in[0,1]$, let $\CB_t$  be the set of points in $\CB$ which are not disconnected from $\innerboundary \CB$ by the $(d-t)$-neighborhood of $\outerboundary \CB$ where $d$ is the distance between $\outerboundary \CB$ and $\innerboundary \CB$.  Then $\CB_t$ is a metric band of width $t$.

Fix $k\in \N_0$. Let $H_{\epsilon,\xi,k}$ be the event that there exist $t\in(0,1]$ and points $z_1,\ldots,z_{k+1} \in \outerboundary \CB_t$ ordered counterclockwise so that the boundary length distance from each $z_i$ to $z_{i+1}$ (with $z_{k+2} = z_1$) is at least $\xi^2$ and $d_{\CB_t}(z_i,z) \le t+ \epsilon$ for $1 \leq i \leq k+1$.  Then 
\[ \bandlaw{\ell}{1} [H_{\epsilon,\xi,k}] = O((\epsilon/\xi)^{k+o(1)} \xi^{-1}) \text{ as } \eps\to 0,\]
where the implicit constant does not depend on $\ell,\xi$.
\end{lemma}
\begin{proof}
Let $\sigma$ be the first time $t>0$ such that there exist $z_1,\ldots,z_{k+1} \in \outerboundary \CB_t$ ordered counterclockwise so that the boundary length distance from each $z_i$ to $z_{i+1}$ (with $z_{k+2} = z_1$) is at least $\xi^2$ and $d_{\CB_t}(z_i,z) \le t+ \epsilon$ for $1 \leq i \leq k+1$.  Let $\sigma=\infty$ if there is no such $t$ in $[0,1]$. Note that $\sigma$ is a stopping time with respect to the filtration generated by $\CB_t$.

We condition on the event $\sigma\le 1$.  Let $T = \xi \lceil \sigma/\xi \rceil$.  Let $\eta_1, \ldots, \eta_{k+1}$ be the $k+1$ geodesics respectively from 
$z_1,\ldots,z_{k+1}$ to $\outerboundary \CB_T$ which hit $\outerboundary \CB_T$ respectively at $w_1, \ldots, w_{k+1}$.  Then with conditional probability at least $p>0$, the boundary length distance from each $w_i$ to $w_{i+1}$ (with $w_{k+2} = w_1$) is at least $\xi^2/4$. This is because the boundary length distances from $\eta_i$ to $\eta_{i+1}$ for $i=1,\ldots, k+1$ evolve like independent $3/2$-CSBP processes and there is a positive probability that none of them has reached a quarter of its initial length before $\xi$ units of time.  On the other hand, we have $d_{\CB_T}(w_i, z) \le d_{\CB_T}(w_i, z_i) + d_{\CB_\sigma}(z_i, z) \le T + \eps$ for each $1\le i\le k+1$. This implies that the event $F_{\eps,\xi/2, k}$ holds for $\CB_T$. Therefore
\begin{align*}
\p[ F_{\eps,\xi/2, k}(\CB_T) \mid \sigma\le 1 ] \ge p.
\end{align*}
It follows that
\begin{align*}
\p[ \sigma\le 1 ] \le \frac{1}{p} \p[ F_{\eps,\xi/2, k}(\CB_T)] \leq \frac{1}{p} \sum_{j=1}^{\xi^{-1}} \p[ F_{\eps,\xi/2,k}(\CB_{\xi j}) ] = O((\eps/\xi)^{k+o(1)} \xi^{-1})
\end{align*}
Note that $H_{\epsilon,\xi,k}$ is equal to the event $\{\sigma\le 1\}$, hence the proof is complete.
\end{proof}

\subsection{Finite number of geodesics between pairs of points}
\label{subsec:finite_between_points}

Suppose that $(\CS,d,\nu,x,y)$ is an instance of the Brownian map and $u,v \in \CS$ are distinct.  Recall that we say $t \in (0,d(u,v))$ is a \emph{splitting time} for $v$ towards $u$ of multiplicity at least $k \in \N$ if there exists $r > 0$ and geodesics $\eta_1, \ldots,\eta_{k+1}$ from $v$ to $u$ so that $\eta_i|_{[t-r,t]} = \eta_j|_{[t-r,t]}$ and $\eta_i((t,t+r]) \cap \eta_j((t,t+r]) = \emptyset$ for all $i \neq j$.  We call $z = \eta_1(t)$ a \emph{splitting point} for $v$ towards $u$.  The \emph{multiplicity} of $z$ is equal to the largest integer $k$ so that the multiplicity of $z$ is at least $k$.  The \emph{multiplicity} of a splitting time $t$ is equal to the sum of the multiplicities of all of the splitting points for $v$ towards $u$ which have distance $t$ to $v$.

We aim to prove the following proposition in this section.
\begin{proposition}
\label{prop:two_point_number_of_geodesics}
For $\bminflaw$ a.e.\ instance $(\CS,d,\nu,x,y)$, the following holds for all $u,v \in \CS$ distinct. The sum of the multiplicities of all the splitting times from $v$ to $u$ is at most $8$. Moreover, there exists a deterministic constant $C$ so that the number of geodesics which connect $u$ to $v$ is at most $C$.
\end{proposition}

The main input into the proof of Proposition~\ref{prop:two_point_number_of_geodesics} is Lemma~\ref{lem:k_geodesics_nearby}, which is stated a bit further below, and gives the exponent for there being a given number of splitting points for the collection of geodesics connecting two points.  The idea to obtain this exponent is the following.  Suppose that $(\CS,d,\nu,x,y)$ has distribution $\bminflaw$ and $u,v \in \CS$ are respectively within distance $\epsilon$ of $x$ and $y$.  Assume that the total multiplicity of the splitting points for the collection of geodesics between $u$ and $v$ is equal to $k$ and for simplicity that each splitting time corresponds to a single splitting point and each splitting point has multiplicity $1$.  Then there must exist metric bands in the reverse metric exploration from $y$ to $x$ with the property that the marked point on the inner boundary has distance at most the width of the band plus $\epsilon$ to two points on the outer boundary.  The exponent for this event was determined in Lemma~\ref{lem:metric_band_epsilon_point} and altogether gives us an exponent of $k$.  The proof will be a bit more involved since we will consider the case that there can be multiple splitting points for a given splitting time and also that each splitting point can have multiplicity which is larger than $2$.

We need to collect Lemmas~\ref{lem:band_boundary_distance}, \ref{lem:complement_ball}, \ref{lem:disk_crossing} and~\ref{lem:k_geodesics_nearby} before completing the proof of Proposition~\ref{prop:two_point_number_of_geodesics}. We first record in the following lemma a result from \cite[Proposition~2.8]{ms2015axiomatic} which we will use to prove Lemma~\ref{lem:band_boundary_distance}.

\begin{lemma}[\cite{ms2015axiomatic} Proposition 2.8]
\label{lem:axiom_prop2.8}
Suppose that $(\CG,d_\CG,\nu)$ has distribution $\slicelaw{1}{\infty}$. For each $a\in[0,1]$, let $z(a)$ be the point on $\innerboundary\CG$ whose counterclockwise boundary distance to the right extremity of $\innerboundary\CG$ is $a$. For $a, b \in [0, 1]$, let $d(a,b)$ be the distance from $\innerboundary \CG$ at which the leftmost geodesic from $z(a)$ to $\outerboundary\CG$ and the leftmost geodesic from $z(b)$ to $\outerboundary\CG$ merge. 

The random function $d$ can equivalently be constructed by the following procedure.
First choose a collection of pairs $(s,x)$ as a Poisson point process on $[0, 1] \times \R_+$ with intensity $ds \otimes x^{-3} dx$ where $ds, dx$ respectively denote Lebesgue measure on $[0,1]$ and $\R_+$. Let $d(a,b)$ be the largest value of $x$ such that $(s,x)$ is a point in this point process for some $s\in(a,b)$.
\end{lemma}

\begin{figure}[h!]
\centering
\includegraphics[width=.67\textwidth]{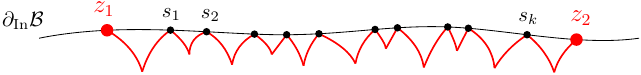}
\caption{We depict a path from $z_1$ to $z_2$ which is contained in $\ol \CB$ and has length at most $3 \eps^{1-2\sp{u}}$, constructed as the concatenation of geodesics from $\innerboundary\CB$ to $\outerboundary\CB$.}
\label{fig:boundary_distance}
\end{figure}

\begin{lemma}
\label{lem:band_boundary_distance}
Fix $\epsilon> 0$ and $0< \sp{u} < \sp{a} <1$. Suppose that $\ell_0(\eps)$ is a function of $\eps$ such that $\ell_0(\eps) \le \exp(\eps^{-\sp{u}/4})$ and $(\CB,d_\CB,\nu, z)$ has distribution $\bandlaw{\ell_0(\eps)}{\epsilon^{1-\sp{a}}}$. 
The probability that there exist $z_1,z_2 \in \innerboundary \CB$ such that either the clockwise or counterclockwise boundary length distance from $z_1$ to $z_2$ is at most $\epsilon^2$ and $d_\CB(z_1, z_2) \geq \epsilon^{1-\sp{u}}$ decays to $0$ as $\epsilon \to 0$ faster than any power of $\epsilon$, where the implicit constants in the decay rate depends only on $\sp{u}, \sp{a}$ but not on the choice of $\ell_0$.
\end{lemma}
\begin{proof}
Let $\Lambda$ be a Poisson point process on $[0, \ell_0(\eps)] \times \R_+$ with intensity measure given by $ds \otimes x^{-3}dx$.  
Fix $N = \lfloor \eps^{-2} \ell_0(\eps) \rfloor$. For $1\le n\le N$, let $x_n = n \eps^2$ and let $x_{N+1}=x_1$.  
For each $1\le n\le N$, the number of $s\in[x_n, x_{n+1}]$ with $(s,x) \in\Lambda$ where $x \ge \eps^{1-\sp{u}}$ is Poisson with parameter given by a constant times $\epsilon^{2\sp{u}}$.  
We recall the following tail bound on a Poisson random variable $X$ with parameter $\lambda$:
\begin{align*}
\p[X \ge \lambda(1+v)] \le \exp(-\lambda h(v)),
\end{align*}
where $v>0$ and $h(v)=(1+v) \log (1+v)-v$. 
Therefore, for each $1\le n\le N$, the probability that  there are at most $\epsilon^{-\sp{u}}$ such points in $[x_n, x_{n+1}]$ is at most $\exp(-\eps^{-\sp{u}/2}\sp{u})$.
Therefore, off an event with probability at most $\eps^{-2} \exp(\eps^{-\sp{u}/4}) \exp(-\eps^{-\sp{u}/2}\sp{u})$ (which decays to $0$ as $\epsilon \to 0$ faster than any power of $\epsilon$), for every $1\le n \le N$, there are at most $\epsilon^{-\sp{u}}$ such points in $[x_n, x_{n+1}]$.

Suppose that $z_1, z_2$ are points on $\innerboundary \CB$ with respective counterclockwise boundary length distance from $z$ given by $0 < a < b < 1$, and that $b-a \le \eps^2$. The previous paragraph implies that off an event whose probability tends to $0$ as $\epsilon \to 0$, there are at most $2 \eps^{-\sp{u}}$ points  $s\in [a,b]$ with $(s,x) \in\Lambda$ where $x \ge \eps^{1-\sp{u}}$. On this event, we denote these points by $a \le s_1 < \cdots < s_k \le b$ where $k \le 2\eps^{-\sp{u}}$.
We can construct a path from $z_1$ to $z_2$ in $\ol\CB$ as follows. See Figure~\ref{fig:boundary_distance}. 
Let $s_0=a$ and $s_{k+1} =b$.
By Lemma~\ref{lem:axiom_prop2.8}, since the interval $(s_i, s_{i+1})$ for each $0\le i \le k$ does not contain any $s$ with $(s,x) \in\Lambda$ such that $x \ge \eps^{1-\sp{u}}$,  the rightmost geodesic from $s_i$ to $\outerboundary\CB$ (which can be obtained as the limit of the leftmost geodesics from $s$ to $\outerboundary\CB$ as $s\to s_i$ from the right) and the leftmost geodesic from $s_{i+1}$ to $\outerboundary\CB$ should merge within distance $\eps^{1-\sp{u}}$ to $\innerboundary\CB$ (hence these geodesics are contained in $\ol\CB$). The concatenation of these geodesics forms a path with length at most $2 \eps^{1-\sp{u}} (k+1) \le 4 \eps^{1-3\sp{u}}$. This path intersects $\innerboundary\CB$ finitely many times, and can be approximated by paths from $z_1$ to $z_2$ which stay in $\CB$ (except their endpoints).
This implies $d_\CB(z_1, z_2) \le 4 \eps^{1-3\sp{u}}$. Since $\sp{u}$ can be chosen to be arbitrarily close to $0$, this completes the proof of the lemma.
\end{proof}

\begin{lemma}\label{lem:complement_ball}
Suppose that $(\CS,d,\nu,x,y)$ has distribution $\bminflaw$.  Fix $r, \eps, \xi>0$. 
Let $E_{r, \eps, \xi, k}$ be the event that $d(x,y)>r$ and there exist points $x_1, \ldots, x_{k+1} \in \partial \fb{y}{x}{d(x,y)-r}$ ordered counterclockwise so that the boundary  length distance between $x_i$ and $x_{i+1}$ is at least $\xi^2$ for each $i$ (with $x_{k+2} = x_1$) and $d_{\CS \setminus  \fb{y}{x}{d(x,y)-r}} (x_i,y) \leq r + \epsilon$ for each $i$. Then
\begin{align*}
\bminflaw[E_{r, \eps, \xi, k}] = O((\epsilon/\xi)^{k+o(1)}) \text{ as } \eps\to 0,
\end{align*}
where the implicit constant does not depend on $\xi, r$. 
\end{lemma}
Note that this lemma is very similar to Lemma~\ref{lem:num_pinch_points_disk}, except that it is a statement about $\CS \setminus  \fb{y}{x}{d(x,y)-r}$ instead of the Brownian disk. Its proof uses Lemmas~\ref{lem:metric_band_epsilon_point} and~\ref{lem:band_boundary_distance}.
\begin{proof}
 Fix $\eps>0$. Let $\tau_\eps$ be the first time $t>0$ that the process $Y_t$ which is given by the boundary length of $\partial \fb{y}{x}{d(x,y)-t}$ hits $\eps^2$. Fix $\eps_0:=\min(\eps, \tau_\eps)$.
For $1\le i \le k+1$, let $\eta_i$ be a geodesic w.r.t.\ $d_{\CS \setminus  \fb{y}{x}{d(x,y)-r}}$ from $x_i$ to $y$.
Let $z_i$ be the first point where  $\eta_i$ hits $\partial \fb{y}{x}{d(x,y)-\eps_0}$. Then the part of $\eta_i$ between $x_i$ and $z_i$ is entirely contained in $\fb{y}{x}{d(x,y)-\eps_0} \setminus   \fb{y}{x}{d(x,y)-r}$, and has length at most $r+\eps$.
Let $w$ be the unique point where the unique geodesic from $x$ to $y$ intersects $\partial \fb{y}{x}{d(x,y)-\eps_0}$.
By definition, the boundary length of $\partial \fb{y}{x}{d(x,y)-\eps_0}$ is at most $\eps^2$.

By Lemma~\ref{lem:band_boundary_distance}, we further know that for any $\sp{a} >0$, off an event with $\bminflaw$ measure decaying faster than any power of $\eps$,
\begin{align*}
d_{\fb{y}{x}{d(x,y)-\eps} \setminus   \fb{y}{x}{d(x,y)-r}} (z_i, w) \le \eps^{1-\sp{a}}.
\end{align*}
This implies that
\begin{align*}
&d_{\fb{y}{x}{d(x,y)-\eps} \setminus   \fb{y}{x}{d(x,y)-r}} (x_i, w) \\
\le &d_{\fb{y}{x}{d(x,y)-\eps} \setminus   \fb{y}{x}{d(x,y)-r}} (x_i, z_i) + d_{\fb{y}{x}{d(x,y)-\eps} \setminus   \fb{y}{x}{d(x,y)-r}} (z_i, w)\\
\le & r + \eps+ \eps^{1-\sp{a}} \le r +\eps^{1-2\sp{a}}.
\end{align*}
Applying Lemma~\ref{lem:metric_band_epsilon_point} to the band $\fb{y}{x}{d(x,y)-\eps} \setminus   \fb{y}{x}{d(x,y)-r}$, we deduce that
\begin{align*}
\bminflaw[E_{r, \eps, \xi, k}] = O((\epsilon^{1-2\sp{a}}/\xi)^{k+o(1)}) \quad\text{as}\quad \eps\to 0.
\end{align*}
Since $\sp{a}$ is arbitrarily close to $0$, we have proved the lemma.
\end{proof}

\begin{lemma}
\label{lem:disk_crossing}
Suppose that $(\CS,d,\nu,x,y)$ has distribution $\bminflaw$. Fix $\eps>0$, $r>0$. Fix $\sp{u} \in(0,1)$ which should be thought of as being arbitrarily close to $1$. Off an event with  $\bminflaw$ measure decaying faster than any power of $\eps$ as $\eps\to 0$, the following holds. For any $z_1, z_2 \in\partial \fb{y}{x}{r}$ with $d(z_1, z_2)\le \eps$, we have $d_{\CS\setminus \fb{y}{x}{r}}(z_1, z_2) \le \eps^{\sp{u}}$.
\end{lemma}

\begin{proof}
In order to prove the lemma, we will construct a rectifiable path from $z_1$ to $z_2$ which is  entirely contained in $\CS \setminus \fb{y}{x}{r}$ and has length at most $\eps^\sp{u}$. As a first step, we will construct a (non-rectifiable) path $\xi$ from $z_1$ to $z_2$ which is contained in $\ol{\CS \setminus \fb{y}{x}{r}}$. In the second step, we show that $\xi \subseteq B(z_1, 3\eps)$. In the last step, we finally construct the desired rectifiable path using points on $\xi$.

\emph{Step 1. Construction of a path $\xi$ from $z_1$ to $z_2$.}
Let $\eta$ be a geodesic between $z_1$ and $z_2$ in $\CS$.  Note that both $\eta$ and $\partial \fb{y}{x}{r}$ are simple curves, and that their intersection is a closed set that does not contain any curve (which is not a point). Indeed, any sub-curve of  $\partial \fb{y}{x}{r}$ has Hausdorff dimension $2$ (since it is the boundary of the  Brownian disk $\CS \setminus \fb{y}{x}{r}$), hence cannot be a subset of $\eta$. This implies that $\eta$ is cut by $\partial \fb{y}{x}{r}$ into a countable set of pieces $\{\eta_i, i\in I\}$, in a way that each piece $\eta_i$ is either entirely contained in  $\fb{y}{x}{r}$ or in $\CS \setminus \ol{\fb{y}{x}{r}}$ (except their endpoints). 
Let $ I_0$ be the set of $i\in I$ such that $\eta_i \subseteq \ol{\fb{y}{x}{r}}$.
Note that the pieces $\eta_i$ for $i \in I$ are ordered naturally according to their distance to $z_1$. If we concatenate the $\eta_i$'s for $i\in I$ in this order, then we recover~$\eta$.

In the same way, $\partial \fb{y}{x}{r}$ is also cut by $\eta$ into a countable set of pieces $\{\zeta_j, j \in J\}$. The configuration can be rather complicated, see Figure~\ref{fig:disk_crossing2.pdf}.
For $j\in J$, let $a_j$ and $b_j$ be the two endpoints of $\zeta_j$, so that $a_j$ is closer to $z_1$ than $b_j$ along $\eta$. Let $[a_j, b_j]$ be the part of $\eta$ between $a_j$ and $b_j$. 
Let $J_1$ (resp.\ $J_2$) be the set of $j\in J$ such that $\zeta_j$ is a subset of the clockwise (resp.\ counterclockwise) arc on $\partial \fb{y}{x}{r}$  from $z_1$ to $z_2$. For $j\in J$, we say that $\zeta_j$ is a \emph{good arc}, if 
\begin{itemize}
\item for any point $z\in \zeta_j$, any path from $z$ to $x$ which stays in $\ol{ \fb{y}{x}{r}}$ has to intersect $\eta$;
\item the arc $\zeta_j$ run from $a_j$ to $b_j$ is in the clockwise (resp.\ counterclockwise) direction if $j\in J_1$ (resp.\ $j \in J_2$).
\end{itemize}
For a good arc $\zeta_j$, we say that $\zeta_j$ is \emph{maximal}, if there does not exist any other good arc $\zeta_{j'}$ such that  $[a_j, b_j]$ is a subset of $[a_{j'}, b_{j'}]$.
Let $J_0 \subseteq J$ be the set of maximal good arcs in $J$.
For any distinct $j_1, j_2 \in J_0$, the intervals $[a_{j_1}, b_{j_1}]$ and $[a_{j_2}, b_{j_2}]$ must be disjoint.

\begin{figure}
\centering
\includegraphics[width=.5\textwidth]{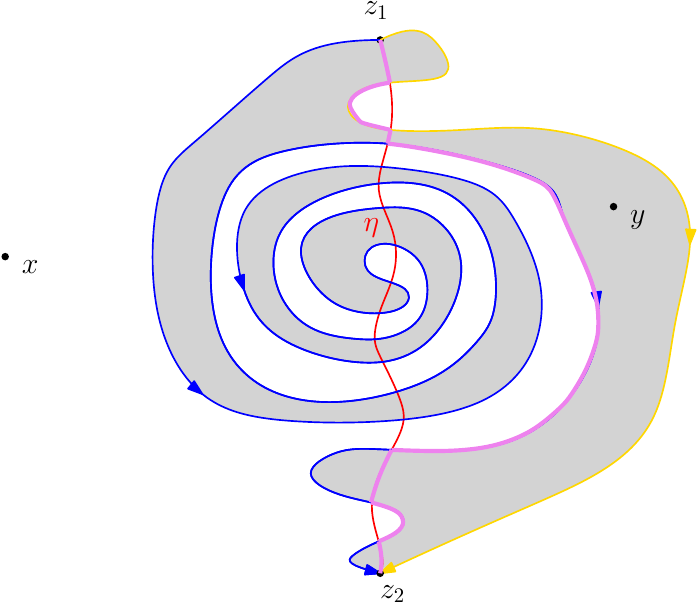}
\caption{The purple curve represents the path $\xi$.}
\label{fig:disk_crossing2.pdf}
\end{figure}

Note that the union of $[a_j, b_j]$ for $j\in J_0$ covers every piece $\eta_i$ for $i\in I_0$. Indeed, for $i\in I_0$, let $a$ be the endpoint of $\eta_i$ which is closest to $z_1$. If $a$ is on the clockwise (resp.\ counterclockwise) arc from $z_1$ to $z_2$ on $\partial \fb{y}{x}{r}$, then let $\zeta$ be the arc on $\partial \fb{y}{x}{r}$ starting from $a$, and going in the clockwise (resp.\ counterclockwise) direction until it hits $\eta$ again at some point $z$ with $d(z,z_2) < d(a, z_2)$. The arc $\zeta$ then must contain some good arc $\zeta_j$. Then either $\zeta_j$ is maximal, or $[a_j, b_j]$ is contained in some $[a_{j'}, b_{j'}]$ for some maximal arc $\zeta_{j'}$.

Let $I_1$ be the set of $i\in I\setminus I_0$ such that $\eta_i$ is not covered by the union of $[a_j, b_j]$ for $j\in J_0$. For $i\in I_1$, let $a_i$ and $b_i$ be the two endpoints of $\eta_i$ so that $a_i$ is closer to $z_1$ than $b_i$. By definition, the closure of the union of $[a_i, b_i]$ for $i\in I_1 \cup J_0$ is equal to $\eta$. Moreover, the intervals  $[a_i, b_i]$ for $i\in I_1\cup J_0$ are disjoint (except possibly at their endpoints), hence can be ordered according to their distance to $z_1$. This induces a natural order on the index set $I_1\cup J_0$.  Let $\xi$ be the concatenation of $\eta_i$ for $i\in I_1$ and $\zeta_j$ for $j\in J_0$ according to the natural order on the countable set $I_1\cup J_0$, and parametrized as follows: 
\begin{itemize}
\item For each $i\in I_1$, let $u_i:= d(a_i, b_i)$ and $\gamma_i: [0, u_i] \to [a_i, b_i]$ be a parametrization of the curve $[a_i, b_i]$ by the geodesic distance, from $a_i$ to $b_i$.

\item It is shown in \cite[Proposition 2.1]{ms2015axiomatic} that $\partial \fb{y}{x}{r}$ is a curve homeomorphic to $\s^1$, so there exists a continuous parametrization $\gamma: [0,1] \to \partial \fb{y}{x}{r}$ so that $\gamma(0)=\gamma(1)=z_1$ and $\gamma$ goes along $\partial \fb{y}{x}{r}$ in the counterclockwise direction. For each $j\in J_0$, let $[s_j, t_j]\subseteq [0,1]$ be the time interval which parametrizes the piece $\zeta_j$. Let $u_j:=t_j-s_j$. Let $\gamma_j : [0, u_j] \to \zeta_j$ be a continuous function defined by $\gamma_j(t) :=\gamma(t + s_j)$ if $\gamma(s_j)=a_j$, and  $\gamma_j(t): =\gamma(t_j-t)$ if $\gamma(s_j)=b_j$. Note that $\gamma_j$ parametrizes $\zeta_j$ from $a_j$ to $b_j$.

\item We then concatenate the functions $\gamma_i$ for $i\in I_1\cup J_0$, according to the natural order on $ I_1\cup J_0$. More precisely, let $T$ be the sum of $u_i$ for all $i\in I_1 \cup J_0$, then $T\le d(z_1, z_2) +1$ is finite. For each $k\in I_1 \cup J_0$, let $T_k$ be the sum of  $u_i$ for all $i\in I_1 \cup J_0$ that are ordered strictly before $k$. Let $\xi(0):=z_1$.
For each $t\in [0,T]$, we are in either of the two situations below.
\begin{itemize}
\item There exists $k \in I_1\cup J_0$ so that $T_k \le t \le T_k + u_k$, then we let $\xi (t):= \gamma_k(t- T_k)$.
\item There exists an infinite sequence $(k_n)_{n\ge 1}$ of indices in $I_1 \cup J_0$, so that $T_{k_n}$ increases to $t$ as $n\to \infty$. Let $\xi(t):=\lim_{n\to \infty}\gamma_{k_n}(0)$. Note that this limit exists, because the points $\gamma_{k_n}(0)$ are on $\eta$ and ordered according to increasing distance to $z_1$.
\end{itemize}
\end{itemize}
It is clear that $\xi$ is continuous at points $t$ such that  $T_k < t < T_k + u_k$ for some $k\in I_1\cup J_0$, because $\xi(t)$ is in the middle of a continuous curve $\gamma_k$. For the same reason, $\xi$ is right continuous at $t=T_k$ and left continuous at $t=T_k+u_k$.
Now suppose that  $t = \lim_{n\to\infty} T_{k_n}$ for an infinite sequence $(k_n)_{n\ge 1}$, let us show that $\xi$ is left continuous. 
Fix $\delta>0$.  Since $\xi(t):=\lim_{n\to \infty}\gamma_{k_n}(0)$, we can also find $n_1>0$ so that $d(\gamma_{k_n}(0), \xi(t)) \le \delta/4$ for all $n\ge n_1$. This implies that for all $i\in I_1 \cap \{k_{n_1}, k_{n_1+1}, \ldots \}$, we have that $[a_i, b_i]$ is contained in $B(\xi(t), \delta)$. Since $\gamma: [0,1] \to \partial \fb{y}{x}{r}$ is a continuous function on a compact interval, it is also uniformly continuous. Thus, there exists $\iota>0$ so that for all $v_1, v_2 \in  [0,1]$ with $|v_1 -v_2| \le \iota$, we have $d(\gamma(v_1), \gamma(v_2)) \le \delta/4$. 
Let $n_2\ge n_1$ be such that for all $n\ge n_2$, $t - T_{k_n} \le \iota$.
This implies that for all $j\in J_0 \cap \{k_{n_2}, k_{n_2+1}, \ldots \}$, we have $u_j \le \iota$, so $d(\gamma_{j}(0), \gamma_j(s)) \le \delta/4$ for all $0\le s\le u_j$, so $\zeta_j$ is also contained in $B(\xi(t), \delta)$. Combined, it implies that $\xi(s) \in B(\xi(t), \delta)$ for all $T_{k_{n_2}} \le s \le t$. This implies that $\xi$ is left continuous at $t$. So $\xi$ is left continuous at all points in $[0, T]$. Similarly, we can show that $\xi$ is right continuous at all points in $[0, T]$, hence continuous everywhere.

Finally, note that $\eta_i\subseteq \ol{\CS \setminus \fb{y}{x}{r}}$ for every $i\in I_1$. All together, we have shown that $\xi$ is a continuous curve  contained in $\ol{\CS \setminus \fb{y}{x}{r}}$ from $z_1$ to $z_2$.

\emph{Step 2. Let us prove that $\xi$ is contained in $B(z_1, 3 \eps)$.} Since $\eta \subseteq B(z_1, 3 \eps)$, it is enough to prove that for each $j \in J_0$, the arc $\zeta_j$ is contained in $B(z_1, 3 \eps)$. For each point $z\in\zeta_j$, the geodesic $\rho$ from $z$ to $x$ is contained in $\ol{\fb{y}{x}{r}}$ and must intersect $\eta$ at some point $z_0$, because $\zeta_j$ is a good arc. We have
\begin{align*}
&d(x, z_0) \ge d (x, z_1) - d(z_0, z_1) \ge r  - \eps. \\
&d(z, z_0) = d(x, z) - d(x, z_0) = r - d(x, z_0) \le \eps. \\
& d (z, z_1) \le d(z, z_0) + d(z_0, z_1) \le \eps +\eps = 2\eps.
\end{align*}
This proves the claim. 

\emph{Step 3. Construction of a rectifiable path from $z_1$ to $z_2$ using points on $\xi$.} Let us now define a sequence of $N$ points $w_1, \ldots, w_N$  on $\xi$, ordered  from $z_1$ to $z_2$ along $\xi$.
\begin{itemize}
\item Let $w_1$ be the first point on $\xi$ (going from $z_1$ to $z_2$) that has $d_{\CS \setminus \fb{y}{x}{r}}$ distance $\eps$ to $z_1$. 
\item For each $n\ge 1$, if $d_{\CS \setminus \fb{y}{x}{r}} (w_n, z_2)\ge 2\eps$, then
 let $w_{n+1}$ be the first point on $\xi$ after $w_n$ (going from $z_1$ to $z_2$) that has $d_{\CS \setminus \fb{y}{x}{r}}$ distance $2\eps$ to $z_1, w_1, \ldots, w_{n-1}, w_n$.
 \item If $d_{\CS \setminus \fb{y}{x}{r}} (w_n, z_2) < 2\eps$, then let $N:=n$.
\end{itemize}
By definition, one can find a collection of $n$ distinct points $w_{i_1}, \ldots, w_{i_n}$  among the $N$ points $w_1, \ldots, w_N$, so that 
\begin{itemize}
\item $w_{i_1}$ is connected to $z_1$ by a path contained in $\ol{\CS \setminus \fb{y}{x}{r}}$ of length at most $2\eps$;
\item for each $1\le k \le n-1$, $w_{i_{k}}$ and $w_{i_{k+1}}$ are connected by a path contained in $\ol{\CS \setminus \fb{y}{x}{r}}$ of length at most $2\eps$;
\item $w_{i_n}$ is connected to $z_2$ by a path contained in $\ol{\CS \setminus \fb{y}{x}{r}}$ of length at most $2\eps$.
\end{itemize} Concatenating all the paths, we have created a path  in $\ol{\CS \setminus \fb{y}{x}{r}}$ from $z_1$ to $z_2$ of length at most $(2N+1) \eps$. In order to prove the lemma, it suffices to show that for any $\sp{a}>0$, off an event with $\bminflaw$  measure decaying faster than any power of $\eps$, we have
\begin{align}\label{eq:est_N}
N \le \eps^{-\sp{a}}.
\end{align}
Since $\xi \subseteq \ol{\CS \setminus \fb{y}{x}{r}}$, by \cite[Lemma 3.3]{gm2019gluing}, off an event with $\bminflaw$ measure decaying faster than any power of $\eps$, there exists $C>0$, such that for each $1\le n\le N$, the $\nu$ area of the ball of radius $\eps$ around $w_n$ w.r.t.\ $d_{\CS \setminus \fb{y}{x}{r}}$ is at least $C \eps^{4+\sp{a}}$. Note that these $N$ balls around the points $w_1,\ldots, w_N$ are disjoint by definition. Moreover, these $N$ balls are all contained in $B(z_1, 4\eps)$ by the result of \emph{Step 2.} This implies that $\nu(B(z_1, 4\eps)) \ge CN \eps^{4+\sp{a}}$.  On the other hand, by Lemma~\ref{lem:le_gall_volume}, we know that for any $\sp{b}>0$, off an event with $\bminflaw$ measure decaying faster than any power of $\eps$, the $\nu$-measure of $B(z_1, 4 \eps)$ is at most $\eps^{4 -\sp{b}}$. Combined, this implies \eqref{eq:est_N} and proves the lemma. 
\end{proof}

\begin{lemma}
\label{lem:k_geodesics_nearby}
Suppose that $(\CS,d,\nu,x,y)$ has distribution $\bminflaw$. Fix  $\delta >0$, $k\in\N$ and $2k$ numbers $0<r_1<s_1 < \cdots <r_k<s_k$ such that $s_j-r_j <\delta/10$ for all $1\le j\le k$.
 Fix $\xi\in (0, \delta/100)$, $\zeta\in(0, \delta/10)$, $\epsilon\in(0, \delta /200)$ and $K \in \N$.  
 Fix $n_1, \dots, n_k \in\N$. 
 Let 
 \begin{align}\label{eq:G}
 G \left((r_j, s_j, n_j)_{1\le j \le k}; \eps, \delta, \xi, \zeta, K\right)
 \end{align}
 be the event that there exist $u \in B(x,\epsilon)$ and $v \in B(y,\epsilon)$ and splitting times $t_1<t_2<\cdots <t_k$ from $v$ towards $u$ so that  the sum of the multiplicities of the $t_j$ over $1\le j\le k$ is equal to $K$ and that
 the following conditions hold for each $1\le j\le k$ (by abuse of notation, we do not always write $j$, but all the objects in the following are defined for a fixed $j$).
 \begin{enumerate}[(i)]
\item\label{1} There are exactly $n_j$ splitting points $z_1, \ldots, z_{n_j}$ from $v$ to $u$ at time $t_j$, and they are all contained in the metric band $\CB_j= \fb{y}{x}{d(x,y)-r_j}\setminus \fb{y}{x}{d(x,y)-s_j}$.

\item\label{3} The distances of each splitting point at time $t_j$ to $\innerboundary\CB_j$ and to $\outerboundary\CB_j$ are at least $\zeta$.
\item\label{4} The distance between any two splitting points at the same time $t_j$ is at least $\delta$.
\item\label{5} 
Let us first make the following definitions before stating the condition (see Figure~\ref{fig:splitting_lemma}).
For each $1\le i\le n_j$, let $m_i$ be the multiplicity of $z_i$. It follows that there exist $\sum_{i=1}^{n_j} (m_i+1)$ different geodesics from $v$ to $u$, denoted by $\gamma_{i,\ell}$ for $(i,\ell)\in\{1\le i\le n_j, 1\le \ell \le m_i+1\}$ and $r>0$, such that 
for each $1\le i\le n_j$, 
\begin{itemize}
\item $\gamma_{i, \ell }|_{[t_j -r, t_j]}$ agree for all $1\le \ell\le m_i+1$ 
\item $\gamma_{i,\ell}((t_j, t_j+r]) \cap \gamma_{i,\ell'}((t_j, t_j+r]) =\emptyset$ for any $\ell, \ell'$ distinct in $[1, m_i +1]$.
\end{itemize}
Let $a_{i}$ be the point on the intersection with $\innerboundary \CB_j$ of the part of $\gamma_{i,1}$ between $v$ and $z_i$ -- if there are several such points,  then choose $a_i$ to be the closest point to $u$. 
By possibly relabeling the points $z_1, \ldots, z_{n_j}$, we can assume that $a_1, \ldots, a_{n_j}$ are ordered counterclockwise on $\innerboundary \CB_j$.
Let $ b_{i,\ell}$ be the point on the intersection of the part of $\gamma_{i,\ell}$ between $z_i$ and $u$ with $\outerboundary \CB_j$ which is the closest to $v$. 
We are now ready to state the condition: for each $1\le i\le n_j$ and $\ell, \ell' \in [1, m_i+1]$ distinct, we have $d(b_{i,\ell}, b_{i,\ell'})\ge \xi$.
\end{enumerate}
 Then we have 
 \begin{align}\label{eq:mu_bm_G}
 \bminflaw \left[G \left((r_j, s_j, n_j)_{1\le j \le k}; \eps, \delta, \xi, \zeta, K\right)\right] =O(\epsilon^{K+o(1)}) \text{ as } \epsilon \to 0
 \end{align}
 where the implicit constant depends on $(r_j, s_j, n_j)_{1\le j \le k}$ and $\delta, \xi, \zeta, K$.
\end{lemma}

\begin{figure}[h]
\begin{center}
\includegraphics[scale=0.7]{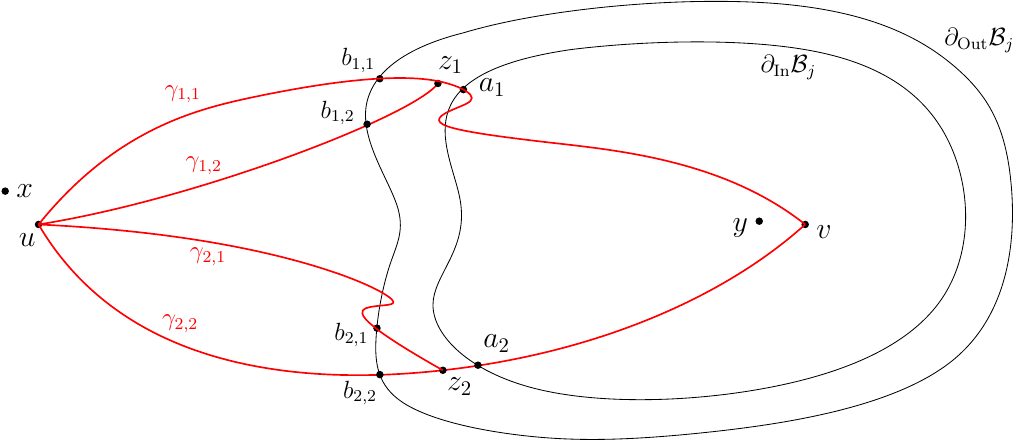}	
\end{center}
\caption{\label{fig:splitting_lemma} Illustration of the setup of Lemma~\ref{lem:k_geodesics_nearby}.}
\end{figure}

\begin{proof}
Throughout, we fix $(r_j, s_j, n_j)_{1\le j \le k}$ and $\delta, \xi, \zeta, K$ as in the statement of the lemma. Suppose that we are working on the event~\eqref{eq:G} and we denote it simply by $G$ in the rest of the proof. Let $u,v$ be as in the statement of the lemma.

\emph{Step 1. Ruling out a bad situation.} We continue to use the notation in item \eqref{5} in the lemma. Let $M_j:=\sum_{i=1}^{n_j} m_i$ be the multiplicity of $t_j$. Our first goal is to prove that on the event $G$, 
for all $(i,\ell)$,  the part of $\gamma_{i,\ell}$ between $b_{i,\ell}$ and $a_i$ is contained in $\CB_j$ (ruling out situations as in Figure~\ref{fig:splitting_lemma2}). Upon proving this, we can show that
\begin{align}\label{eq:dist1}
 d_{\CB_j}(a_i, b_{i,\ell})=d(a_i, b_{i,\ell}) \le s_j-r_j +2\eps.
\end{align}
This is because $ d (a_i, u) \le d(a_i, x) +\eps =d(x,y)-r_j+\eps, d(b_{i,\ell}, u) \ge d(b_{i,l}, x) -\eps=d(x,y)-s_j-\eps$, and in addition $a_i, b_{i,\ell}$ are both on the same geodesic $\gamma_{i,\ell}$ from $u$.
Assume on the contrary that the part of $\gamma_{i,\ell}$ between $b_{i,\ell}$ and $a_i$ is not contained in $\CB_j$.
Note that by definition, the part of $\gamma_{i,\ell}$ between $b_{i,\ell}$ and $z_i$ is disjoint from $\outerboundary\CB_j$ (except at $b_{i,\ell}$) and the part of $\gamma_{i,\ell}$ between $z_i$ and $a_i$ is disjoint from $\innerboundary\CB_j$ (except at $a_i$). Consequently, either
the part of  $\gamma_{i,\ell}$ between $b_{i,\ell}$ and $z_i$ exits $\CB_j$ from $\innerboundary\CB_j$, implying that
\[ d(b_{i,\ell}, z_i) \ge d(b_{i,\ell}, \innerboundary \CB_j) + d( \innerboundary \CB_j, z_i)\ge s_j -r_j +\zeta;\]
or the part of $\gamma_{i,\ell}$ between $z_i$ and $a_i$ exits $\CB_j$ from $\outerboundary\CB_j$, implying that
\[ d(a_{i}, z_i) \ge d(a_{i}, \outerboundary \CB_j) + d( \outerboundary \CB_j, z_i)\ge s_j -r_j +\zeta.\]
Redoing the argument under \eqref{eq:dist1} and combining it with condition \eqref{3} in the lemma, we have
\begin{align*}
d(b_{i,\ell}, z_i) = d (b_{i,\ell}, a_i) -d(a_i, z_i) \le s_j-r_j+2\eps -\zeta,\\
d(a_i, z_i) = d (b_{i,\ell}, a_i) -d(b_{i,\ell}, z_i) \le s_j-r_j+2\eps -\zeta,
\end{align*}
leading to a contradiction. Therefore, the part of $\gamma_{i,\ell}$ between $b_{i,\ell}$ and $a_i$ is contained in $\CB_j$ and~\eqref{eq:dist1} holds.

\begin{figure}[h]
\begin{center}
\includegraphics[scale=0.7]{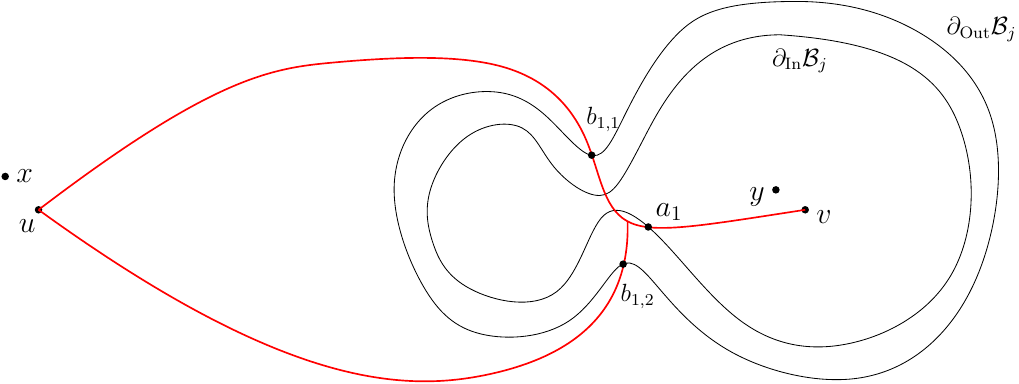}	
\end{center}
\caption{\label{fig:splitting_lemma2} This situation is ruled out.  Our setup implies that the part of $\gamma_{i,\ell}$ between $b_{i,\ell}$ and $a_i$ is always contained in $\CB_j$.}
\end{figure}

\emph{Step 2.} Now, we will use the conditional independence between the successive bands $\CB_j$ and also use Lemma~\ref{lem:metric_band_epsilon_point} to control the cost for each band.
We will carry out the following arguments for each $1\le j\le k$, in the order of the breadth-first exploration from $y$ to $x$.
 In the remainder of Step 2, we will deal with a fixed $j\in[1,k]$ (except when said otherwise) and sometimes omit the index $j$. 
For technical reasons that will be made clear later on, we define $\wh r_j$ to be a uniformly chosen point in $[r_j, r_j+\eps]$. Let  $\wh\CB_j= \fb{y}{x}{d(x,y)-\wh r_j}\setminus \fb{y}{x}{d(x,y)-s_j}$. Note that $\wh \CB_j \subseteq \CB_j$.
Let $\wh a_i$ be the point on the intersection with $\innerboundary \wh\CB_j$ of the part of $\gamma_{i,1}$ between $v$ and $z_i$ which is the closest to $u$.
 When we say ``probability'', we mean the conditional probability of $\bminflaw$ given $\CS \setminus \fb{y}{x}{d(x,y)-\wh r_j}$. 
We continue to use the notation in item \eqref{5} in the lemma. Let $L_j$ be the boundary length of $\innerboundary \wh\CB_j$. Let $M_j:=\sum_{i=1}^{n_j} m_i$ be the multiplicity of $t_j$.

\emph{Step 2.1. Finding the marked points $\wt a_i$.} Fix $\sp{a} \in(0,1)$. The goal of this step is to find points $\wt a_i \in \innerboundary\wh \CB_j$ for $1\le i \le n_j$ such that the boundary length between $\wt a_i$ and $\wh a_i$ is at most $\eps^{2\sp{a}}$. The points $\wt a_i$ should be chosen in a  way which depends on $\fb{y}{x}{d(x,y)-\wh r_j}$ only through the boundary  $\partial\fb{y}{x}{d(x,y)-\wh r_j}$.

Let $F_j(N)$ be the event that we can cover the set of points on $\innerboundary\wh\CB_j$ with $d_{\CS \setminus \fb{y}{x}{d(x,y)- \wh r_j}}$ distance at most $ r_j +8 \epsilon$ to $y$ with at most $N$ intervals $I_1,\ldots,I_{N}$ each with boundary length at most $\eps^{2\sp{a}}$.
Applying Lemma~\ref{lem:complement_ball}, we can deduce that $F_j(N)^c$ occurs with probability $O(\epsilon^{N (1-\sp{a}) +o(1)})$. Letting $N=\lceil K/ (1-\sp{a}) \rceil$, to prove the lemma, it is enough to prove 
\begin{align}\label{eq:G_enough}
\bminflaw \big(G \cap \cap_{j=1}^k F_j(N)\big)=O(\eps^{K+o(1)}).
\end{align}
From now on, for each $j\in[1,k]$, suppose that we are working on the event $G\cap F_j(N)$.

For some $\wt N_j  \le N$, let $I_1, \ldots, I_{\wt N_j}$ be the $\wt N_j$ intervals with boundary length at most $\eps^{2\sp{a}}$ that cover  the set of points on $\innerboundary\wh\CB_j$ with $d_{\CS \setminus \fb{y}{x}{d(x,y)-\wh r_j}}$ distance at most $ r_j +8 \epsilon$ to $y$. Let us show that for each $1\le i \le n_j$, $\wh a_i$ is covered by one of the $\wt N_j$ intervals.
Let $c_i$ be the point where $\gamma_{i,1}$ (from $v$ to $u$) first hits  $\innerboundary \wh \CB_j$.   By definition, the part of $\gamma_{i,1}$  from $v$ to $c_i$ is entirely contained in $\CS \setminus \fb{y}{x}{d(x,y)- \wh r_j}$, so
$$d_{\CS \setminus \fb{y}{x}{d(x,y)-\wh r_j}} (c_i, v) = d(c_i, v) =d (u,v) - d(u,c_i) \le d(x,y) +2\eps - (d(x, c_i) - \eps) =\wh r_j + 3\eps \le r_j+4\eps. $$
This implies that $c_i$ is covered by one of the $\wt N_j$ intervals.
Also note that 
\begin{align}\label{eq:d_a_c_j}
d (\wh a_j, c_j) = d(u, c_j) -d (u,\wh a_j) \le d(x, c_j) + \eps - (d(x, \wh a_j) -\eps) \le 2\eps.
\end{align}

In the following, we will argue using Lemma~\ref{lem:disk_crossing} that for any $\sp{a}>0$, off an event with probability decaying faster than any power of $\eps$, we have
\begin{align}\label{eq:goal_d_a_c}
d_{\CS \setminus \fb{y}{x}{d(x,y)-\wh r_j}} (\wh a_j, c_j) \le \eps^{1-\sp{a}}.
\end{align}
Let $\CR_\eps$ be the set of $r>0$ such that there exist $z_1, z_2 \in\partial \fb{y}{x}{r}$ with $d(z_1, z_2)\le \eps$ such that $d_{\CS\setminus \fb{y}{x}{r}}(z_1, z_2) > \eps^{1- \sp{a}}$. Lemma~\ref{lem:disk_crossing} implies that for any $n\in\N$, we have 
$$\bminflaw [ \leb (\CR_\eps)] \le \eps^n \quad \text{for}\quad \eps > 0 \quad \text{sufficiently small.}$$
Since $\wh r_j$ is a uniform point in $[r_j, r_j+\eps]$ chosen independently of anything else, we know that
\begin{align*}
\bminflaw [d(x,y) -\wh r_j \in \CR_\eps] =O( \eps^n) \quad \text{as}\quad \eps\to 0 \quad \text{for all}\quad n \in\N.
\end{align*}
On the event that $d(x,y) -\wh r_j \not \in \CR_\eps$, \eqref{eq:d_a_c_j} implies~\eqref{eq:goal_d_a_c}.

It then follows that, off an event with probability decaying faster than any power of $\eps$, we have
\begin{align*}
d_{\CS \setminus \fb{y}{x}{d(x,y)-\wh r_j}} (\wh a_i, y) \le  &d_{\CS \setminus \fb{y}{x}{d(x,y)-\wh r_j}} (\wh a_i, c_i) + d_{\CS \setminus \fb{y}{x}{d(x,y)-\wh r_j}} (c_i, v) + d_{\CS \setminus \fb{y}{x}{d(x,y)-\wh r_j}} (v, y) \\
\le & 2\eps + r_j + 4\eps +\eps \le r_j + 7\eps.
\end{align*}
This implies that $\wh a_i$ is also covered by one of the $\wt N_j$ intervals.

Let $1\le k_1, \ldots, k_{n_j} \le \wt N_j$ be the set of indices such that $\wh a_i \in I_{k_i}$ for all $1\le i \le n_j$.
We also choose a set of $n_j$ indices $1\le \wt k_1, \ldots, \wt k_{n_j} \le \wt N_j$ uniformly among all possible sets of $n_j$ indices, and independently of everything else. Let $\Q_j$ be the probability measure of the choice of $n_j$ points out of $\wt N_j$.
Let $H_j$ be the event that $\wt k_i = k_i$ for all $1\le i \le K$. Then the event $H_j$ occurs with  $\Q_j$-probability at least $N_j ^{-n_j}$.
For each $1\le i \le n_j$, let $\wt a_1$ be the middle point of the interval $I_{k_i}$.  The choice of the points $\wt a_1, \ldots, \wt a_{n_j}$ is measurable w.r.t.\ $\CS \setminus \fb{y}{x}{d(x,y)-\wh r_j}$.
Since $H_j$ has positive $\Q_j$-probability (at least $N^{-K}$) and is independent of $\bminflaw$, in order to prove~\eqref{eq:G_enough}, it is enough to prove that
\begin{align}\label{eq:G_enough2}
\bminflaw \otimes \prod_{j=1}^k \Q_j \big(G \cap \cap_{j=1}^k F_j(N) \cap \cap_{j=1}^k H_j \big)=O(\eps^{K+o(1)}).
\end{align}
By abuse of notation, we will sometimes drop $\Q_j$ and only refer to the $\bminflaw$-measure of the event in~\eqref{eq:G_enough2}.
We will from now on work on the event $G\cap F_j(N) \cap H_j$ at each layer $j$. Note that on the event  $G\cap F_j(N) \cap H_j$, for each $1\le i \le n_j$ the boundary length between $\wt a_i$ and $\wh a_i$ is at most $\eps^{2\sp{a}}$. 
By Lemma~\ref{lem:band_boundary_distance}, we know that for any $\sp{b}_1\in (0, \sp{a})$, on $G\cap F_j(N) \cap H_j$ minus an event with probability decaying faster than any power of $\eps$, we have
\begin{align}\label{eq:d_wh_wt}
d_{\wh \CB_j}(\wt a_i, \wh a_i)\le \eps^{\sp{a} -\sp{b}_1}.
\end{align}

\emph{Step 2.2. Finding the right collection of boundary intervals.} 
The goal of this step is to find a collection  of intervals $\wt I_1, \ldots, \wt I_{n_j}$ on $\innerboundary\wh\CB_j$ which satisfies the following conditions:
\begin{enumerate}
\item\label{it:good1} The intervals $\wt I_1, \ldots, \wt I_{n_j}$ are disjoint except possibly at their endpoints.
\item\label{it:good2} For each $1\le i \le n_j$, we have $\wt a_i \in \wt I_i$.
\item\label{it:good3} For each $1\le i \le n_j$, let $\wt b_i$ and $\wt c_i$ be the two endpoints of $\wt I_i$, so that  $\wt I_i$ is equal to the counterclockwise arc from $\wt b_i$ to $\wt c_i$.  We have
\begin{align}\label{eq:b-i-c-i}
d_{\wh \CB_j}(\wt a_i, \wt b_i) \ge  \delta/ 4, \quad  d_{\wh\CB_j}(\wt a_i, \wt c_i) \ge \delta/ 4.
\end{align}
\end{enumerate}
Our choice of the intervals will be made in a way which depends on $\fb{y}{x}{d(x,y)-\wh r_j}$ only through the boundary  $\partial\fb{y}{x}{d(x,y)-\wh r_j}$.

Let $w_j$ be the unique point where the unique geodesic from $y$ to $x$ hits $\innerboundary\wh\CB_j$. Fix $\sp{b}>0$.  Let $\CC_j$ be a collection of boundary intervals on $\innerboundary\wh\CB_j$ which cover $\innerboundary\wh\CB_j$, defined as follows. Let $\CC_j$ be the unique collection of $N_j= \lfloor L_j \eps^{-2\sp{b}} \rfloor$ intervals of equal boundary length so that every two intervals are disjoint (except possibly at their endpoints) and $w_j$ is  exactly the endpoint of two of the intervals. Then each interval in $\CC_j$ has boundary length $L_j / N_j \in[\eps^{2\sp{b}}, 2 \eps^{2\sp{b}}]$, if $\eps$ is sufficiently small.

We will choose the  intervals $\wt I_1, \ldots, \wt I_{n_j}$  in a way that each of them is the union of several adjacent intervals in $\CC_j$. The total number of such choices (i.e., $n_j$ disjoint intervals each made of the union of adjacent intervals in $\CC_j$) is at most $N_j^{2 n_j} = O(\eps^{-4 \sp{b}n_j})$, because choosing $n_j$ intervals boils down to choosing $2n_j$ endpoints among the $N_j$ endpoints of the intervals in $\CC_j$. We will make a choice according to the uniform measure $\wt \p_j$ among all such choices, independently of everything else.
It is clear that such a choice depends on $\fb{y}{x}{d(x,y)-\wh r_j}$ only through the boundary  $\partial\fb{y}{x}{d(x,y)-\wh r_j}$.

We say that a choice is \emph{good} if it satisfies the conditions \eqref{it:good1}, \eqref{it:good2}, and \eqref{it:good3} specified earlier.
Let $E_j$ be the event that we have made the good choice.
In the sequel, we will show that there exists at least one good choice, which will imply that $\wt \p_j(E_j)\ge\eps^{4\sp{b}n_j}$. In order to prove~\eqref{eq:G_enough2}, it is enough to prove that
\begin{align}\label{eq:G_enough3}
\bminflaw \otimes \prod_{j=1}^k (\Q_j \otimes \wt \p_j) \big(G \cap \cap_{j=1}^k F_j(N) \cap \cap_{j=1}^k H_j  \cap \cap_{j=1}^k E_j \big)=O(\eps^{K+o(1)}).
\end{align}
Indeed, \eqref{eq:G_enough3} implies that 
\begin{align*}
\bminflaw \otimes \prod_{j=1}^k \Q_j \big(G \cap \cap_{j=1}^k F_j(N) \cap \cap_{j=1}^k H_j \big)=O(\eps^{(1-\sp{b})K+o(1)}).
\end{align*}
Since $\sp{b}$ is arbitrarily close to $0$, it implies \eqref{eq:G_enough2}.

It remains to show that on the event $G \cap F_j(N) \cap H_j$ minus an event with probability decaying faster than any power of $\eps$, there exists at least one good choice.
Note that for any $(i,\ell)\in \{1\le i\le n_j, 1\le \ell \le m_i+1\}$ we have 
\begin{align}\label{eq:dist_d(a,b)}
d(\wh a_i, b_{i, \ell} )\le s_j -\wh r_j +2\eps \le s_j-r_j+ 2\eps.
\end{align}
This is because $d(\wh a_i, u) \le d(\wh a_i, x)+ \eps = d(x,y)-\wh r_j +\eps$, $d(b_{i,\ell}, u) \ge d(b_{i,\ell}, x) -\eps=d(x,y)-s_j -\eps,$
 and in addition $\wh a_i$, $b_{i,\ell}$ are both on the same geodesic $\gamma_{i,\ell}$ from $u$. 
For any $1\le i\le n_j$, we have
 \begin{align}
\notag d(\wh a_i, \wh a_{i+1}) \ge d(z_i, z_{i+1}) -d (\wh a_i, z_i) - d(\wh a_{i+1}, z_{i+1}) \ge \delta -d(\wh a_i, b_{i,1}) -d (\wh a_{i+1}, b_{i+1,1})\\
\label{eq:a_i_a_i}
  \ge \delta -(\delta/10+2\eps) -(\delta/10+2\eps)  \ge 4\delta/5 -4\eps \ge 39 \delta/ 50,
 \end{align}
 where the inequality in~\eqref{eq:a_i_a_i} follows from~\eqref{eq:dist_d(a,b)} and the fact that $s_j -r_j < \delta/10$. 
 Combined with \eqref{eq:d_wh_wt}, we have for $\eps$ small
 \begin{align}\label{eq:d_s/b_aa}
d(\wt a_i, \wt a_{i+1}) \ge \,  d (\wh a_i, \wh a_{i+1})- d (\wt a_i, \wh a_{i})  -d(\wh a_{i+1}, \wt a_{i+1}) 
\ge \, 39\delta/50- 2\eps^{\sp{a}-\sp{b}_1} \ge 38\delta/50.
\end{align}
Note that $\innerboundary\wh \CB_j$ between $\wt a_i$ and $\wt a_{i+1}$ is a continuous curve, so we can let $p_i$ (resp.\ $q_i$) be the first point that has $d$-distance $0.37\delta$ from $\wt a_i$ that one encounters as one goes along $\innerboundary\wh \CB_j$ from $\wt a_i$ in the clockwise (resp.\ counterclockwise) direction.  
Let $I(q_i)$ (resp.\ $I(p_i)$) be the interval in $\CC_j$ which contains $q_i$ (resp.\ $p_i$) and which is the closest to $a_i$ w.r.t.\ the boundary distance. Let $\wt b_i$ (resp.\ $\wt c_i$) be the endpoint of $I(p_i)$ (resp.\ $I(q_i)$) which is the closest to $a_i$ w.r.t.\ the boundary distance. See Figure~\ref{fig:boundary_order}.

\begin{figure}
\begin{center}
\includegraphics[scale=1.1]{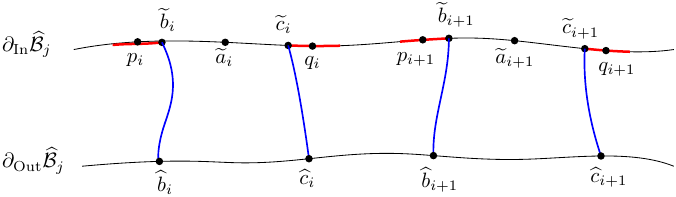}	
\end{center}
\caption{\label{fig:boundary_order} We depict the points on $\innerboundary \wh\CB_j$. The red intervals represent $I(p_i), I(q_i), I(p_{i+1})$ and $I(q_{i+1})$. We also depict the slices cut out in Step 2.3.}
\end{figure}

Let us show that,  on $G \cap F_j(N) \cap H_j$ minus an event with probability decaying faster than any power of $\eps$,  for each $1\le i \le n_j$, the points $\wt b_i, \wt a_i, \wt c_i, \wt b_{i+1}$ are ordered in a counterclockwise manner on $\innerboundary \wh\CB_j$. 
First, note that $d(\wt a_i, q_i) + d (\wt a_{i+1}, p_{i+1}) = 0. 74 \delta \le d(\wt a_i, \wt a_{i+1})$ by \eqref{eq:d_s/b_aa}, which implies that $\wt a_i, q_i, p_{i+1}, \wt a_{i+1}$ are ordered in a counterclockwise manner.

Then note that the distance between $\wt a_i$ and $p_i$ inside the band  $\fb{y}{x}{d(x,y)-\wh r_j} \setminus \fb{y}{x}{d(x,y)-\wh r_j-\delta/300}$ is at least $d(\wt a_i, p_i)=0.37 \delta$. 
Conditionally on the boundary length $L_j$ of $\innerboundary \wh\CB_j$, the band $\fb{y}{x}{d(x,y)-\wh r_j} \setminus \fb{y}{x}{d(x,y)-\wh r_j-\delta/300}$ is distributed according to $\bandlaw{L_j}{\delta/300}$.
There exists an absolute constant $c_0>0$, such that for any $\sp{a}>0$, we have $\bminflaw(L_j \ge \exp (\eps^{-\sp{a}})) \le \exp(-c_0 \eps^{-\sp{a}})$. On the event $\{L_j < \exp (\eps^{-\sp{a}})\}$, we can apply Lemma~\ref{lem:band_boundary_distance} to the band $\fb{y}{x}{d(x,y)-\wh r_j} \setminus \fb{y}{x}{d(x,y)-\wh r_j-\delta/300}$. This implies that  for any $\sp{b}_1>0$, off an event with probability decaying faster than any power of $\eps$, the boundary length between $\wt a_i$ and $p_i$ is at least $\eps^{\sp{b}_1}$. In particular, we can take $\sp{b}_1 =\sp{b}/4$. Since the boundary length between $\wt b_i$ and $p_i$ is at most $2\eps^{2\sp{b}}$,  $\wt b_i$ must be on the clockwise arc from $\wt a_i$ to $p_i$. The same arguments apply to $\wt c_i$ and $q_i$. This proves the claim about the order of the points $\wt b_i, \wt a_i, \wt c_i, \wt b_{i+1}$.
 For $1\le i \le n_j$, let $\wt I_i$ be the clockwise arc from $\wt b_i$ to $\wt c_i$. This choice of $\wt I_1, \ldots, \wt I_{n_j}$ satisfies the conditions (1) and (2). In order to show that it is a good choice, we only need to show \eqref{eq:b-i-c-i}.
 
Applying Lemma~\ref{lem:band_boundary_distance} as in the previous paragraph, we know that for each $\sp{b}_2\in (0, \sp{b})$, on  $G \cap F_j(N) \cap H_j$ minus an event with probability decaying faster than any power of $\eps$, we have
\begin{align*}
d_{\wh \CB_j}(p_i, \wt b_i) \le \eps^{\sp{b} -\sp{b}_2}, \quad d_{\wh \CB_j}(q_i, \wt c_i) \le \eps^{\sp{b} - \sp{b}_2}.
\end{align*}
Therefore
\begin{align*}
d_{\wh\CB_j} (\wt a_i, \wt b_i) \ge d_{\wh\CB_j}(\wt a_i, p_i) - d_{\wh \CB_j}(p_i, \wt b_i) \ge 0.37 \delta -  \eps^{\sp{b} -\sp{b}_2},
\end{align*}
which is bigger than $\delta/4$ as $\eps\to 0$. The same applies to $d_{\wh\CB_j} (\wt a_i, \wt c_i)$. 
This implies \eqref{eq:b-i-c-i}, hence  $\wt I_1, \ldots, \wt I_{n_j}$ is a good choice.

\emph{Step 2.3. Cutting $\wh\CB_j$ into independent slices.}
Let $I_1, \ldots, I_{n_j}$ be the $n_j$ intervals chosen in the previous step. Suppose that  $G \cap F_j(N) \cap H_j\cap E_j$ holds.
Let us cut $\wh\CB_j$ into independent slices along the geodesics from $\wt b_i$ and $\wt c_i$ to $x$. 
For each $1\le i\le n_j$, let $\eta_{L, i}$ (resp.\ $\eta_{R, i}$) be the geodesic from $\wt b_i$ (resp.\ $\wt c_i$) to $\outerboundary \wh\CB_j$. Let $ \wh b_i$ (resp.\ $ \wh c_i$) be the point where $\eta_{L, i}$ (resp.\ $\eta_{R, i}$) hits $\outerboundary \wh\CB_j$.  Let $\CG_{j,i}$ be the slice of $\wh\CB_j$ between $\eta_{L,i}$ and $\eta_{R,i}$ which contains $\wt a_i$.
Let us show the following.
\begin{equation}\label{sji}
\text{For all } 1 \le i\le n_j \text{ and } 1\le \ell \le m_i +1, \text{ the part of } \gamma_{i,\ell} \text{ between } b_{i,\ell} \text{ and } \wh a_i \text{ is contained in } \CG_{j,i}.
\end{equation}
Assume that there exist $1\le i\le n_j$ and $1\le \ell \le m_i +1$ such that it is not the case, and let us try to get a contradiction.
Since the part of $\gamma_{i,\ell}$ between $b_{i,\ell}$ and $\wh a_i$ is contained in $\wh\CB_j$, this part must intersect $\eta_{L,i}$ or $\eta_{R, i}$ in order to exit $\CG_{j,i}$ (let it be $\eta_{L,i}$ without loss of generality). 
However this is impossible since by~\eqref{eq:b-i-c-i} and~\eqref{eq:d_wh_wt}, for $\eps$ small
$$d(\wh a_i, \eta_{L, i}) \ge d(\wh a_i, \wt b_{i}) - (s_j -r_j) \ge d(\wt a_i, \wt b_{i})- d(\wh a_i, \wt a_i)- (s_j -r_j)  \ge \delta/4  - \eps^{\sp{a} -\sp{b}_1} - \delta/10\ge 0.14\delta$$ and 
$$d(\wh a_i, b_{i,\ell})\le \delta/10 +\delta/50=0.12\delta$$ by~\eqref{eq:dist_d(a,b)}. This proves \eqref{sji}.
In addition, it also follows that
\begin{align}\label{eq: d(b,q)}
d(b_{i,\ell},  \wh b_{i})\ge d(\wh a_i,  \wh b_i) -d (\wh a_i, b_{i,\ell})\ge 0.02\delta \quad \text{and similarly} \quad d(b_{i,\ell},  \wh c_{i}) \ge 0.02\delta.
\end{align}

Now fix $1\le i\le n_j$ and cut out the slice $\CG_{j,i}$. Gluing its left and right boundaries, we get a new metric band $\wh\CB_{j,i}$, with a certain inner boundary length that we denote by $L_{j,i}$ and width $s_j-\wh r_j$. Moreover, $\wh\CB_{j,i}$ satisfies the following conditions (\ref{aa}') and (\ref{dd}').
\begin{enumerate}[(a)]
\item\label{aa} Due to \eqref{sji} and \eqref{eq:dist1}, we have that for all $1\le \ell \le m_i+1$, 
\begin{align*}
d_{\wh\CB_{j,i}} (b_{i,\ell}, \wh a_i) \le s_j-r_j +2\eps.
\end{align*}
Then by~\eqref{eq:d_wh_wt}, we have
\begin{align*}
d_{\wh\CB_{j,i}} (b_{i,\ell}, \wt a_i) \le s_j-r_j + 2\eps^{\sp{a}-\sp{b}_1}.
\end{align*}
\item\label{dd}  Fix $\sp{b}_2 \in(0,1)$. By condition \eqref{5} and Lemma~\ref{lem:band_boundary_distance}, we know that for any $1\le \ell \le m_i$,
the boundary length of the counterclockwise arc on $\outerboundary\wh\CB_{j,i}$ from $b_{i,\ell}$ to $b_{i, \ell+1}$  is at least $\xi^2 \epsilon^{2\sp{b}_2}$.
By \eqref{eq: d(b,q)}  and Lemma~\ref{lem:band_boundary_distance}, we further know that the boundary length of the counterclockwise (resp.\ clockwise) arc on $\outerboundary\wh\CB_{j}$ from $b_{i,m_i+1}$ (resp.\ $b_{i,1}$) to $\wh c_i$ (resp.\ $\wh b_i$) is at least $(0.02\delta)^2 \epsilon^{2\sp{b}_2}$. Since the points $\wh b_i$ and $\wh c_i$ are identified in $\wh\CB_{j,i}$, we deduce that the boundary length of the counterclockwise arc on $\outerboundary\wh\CB_{j}$ from $b_{i,m_i+1}$ to $b_{i,1}$ is at least $2 (0.02\delta)^2 \epsilon^{2\sp{b}_2}$.
\end{enumerate}

Now we rescale the width of the band $\CB_{j,i}$ by $(s_j-\wh r_j)^{-1}$ to get a new band with width $1$ and inner boundary length $L_{j,i} (s_j-\wh r_j)^{-2}$. We apply Lemma~\ref{lem:metric_band_epsilon_point} to the rescaled band, and get that the $\bandlaw{L_{j,i}}{s_j-\wh r_j}$ probability that (\ref{aa}) and (\ref{dd}) simultaneously hold for $\wh \CB_{j,i}$ is at most $O(\eps^{m_i(\sp{a}-\sp{b}_1-\sp{b}_2)+o(1)})$. Since this is true for each $1\le i\le n_j$ and the slices $\wh\CB_{j,i}$ are independent for different $i$, we get that conditionally on $L_j$, the probability that \eqref{aa} and \eqref{dd} hold for every $1\le i\le n_j$ is $O(\eps^{\sum_{i=1}^{n_j} m_i (\sp{a}-\sp{b}_1-\sp{b}_2) + o(1)})=O(\eps^{M_j (\sp{a}-\sp{b}_1-\sp{b}_2)+o(1)})$.

\emph{Step 3. Conclusion.}
Since the above is true for all $1\le j\le k$, by the conditional independence of the metric bands $\wh\CB_j$ given their boundary lengths, it follows that 
$$\bminflaw \big( G \cap \cap_{j=1}^k F_j(N) \cap \cap_{j=1}^k H_j  \cap \cap_{j=1}^k E_j  \big)=O(\eps^{K(\sp{a}-\sp{b}_1-\sp{b}_2) +o(1)}).$$ 
Since $\sp{b}_1, \sp{b}_2$ can be chosen to be arbitrarily close to $0$ and $\sp{a}$ can be chosen to be arbitrarily close to $1$, this proves~\eqref{eq:G_enough3} and completes the proof of the lemma.
\end{proof}

\begin{proof}[Proof of Proposition~\ref{prop:two_point_number_of_geodesics}]
Suppose that $(\CS,d,\nu,x,y)$ has distribution $\bminflaw$.  We will prove that for every distinct pair $u,v \in \CS$, the sum of the multiplicities of all the splitting times from $v$ to $u$ is at most $8$.  This combined with Proposition~\ref{prop:single_point_number_of_geodesics} will imply that there is a deterministic constant $C$ so that the number of geodesics connecting $u$ and $v$ is at most $C$.

Fix  $\delta >0$, $k\in\N$ and rational numbers $(r_j, s_j, n_j)_{1\le j \le k}$, $\delta, \xi, \zeta$ and $K$ as in Lemma~\ref{lem:k_geodesics_nearby}. 
Fix $\eps_0>0$ that we will adjust later. Let 
$$P\left((r_j, s_j, n_j)_{1\le j \le k} ; \eps_0, \delta, \xi, \zeta, K\right)$$
(which we will also denote by $P$ for simplicity) be the set of pairs of points $(u,v)\in\CS^2$ such that there are splitting times $t_1<t_2<\cdots <t_k$ from $v$ towards $u$, so that  the sum of the multiplicities of the $t_j$ over $1\le j\le k$ is equal to $K$ and that
 the following conditions hold for each $1\le j\le k$ (by abuse of notation, we do not always write $j$, but all the objects in the following are defined for a fixed $j$).
\begin{enumerate}[(i)]
\item There are exactly $n_j$ splitting points $z_1, \ldots, z_{n_j}$ from $v$ to $u$ at time $t_j$, and they are all contained in the metric band $\CB_j (u,v)= \fb{v}{u}{d(u,v)-r_j-3\eps_0}\setminus \fb{v}{u}{d(u,v)-s_j+3\eps_0}$.

\item The distance of each splitting point at time $t_j$ to $\innerboundary\CB_j(u,v)$ and to $\outerboundary\CB_j(u,v)$ is at least $\zeta$.

\item The distance between any two splitting points at the same time $t_j$ is at least $\delta$.

\item As in (\ref{5}) of Lemma~\ref{lem:k_geodesics_nearby}, we define the following quantities. For each $1\le i\le n_j$, let $m_i$ be the multiplicity of $z_i$. It follows that there exist $\sum_{i=1}^{n_j} (m_i+1)$ different geodesics from $v$ to $u$, denoted by $\gamma_{i,\ell}$ for $(i,\ell)\in\{1\le i\le n_j, 1\le \ell \le m_i+1\}$ and $r>0$, such that 
for each $1\le i\le n_j$, 
\begin{itemize}
\item $\gamma_{i, \ell }|_{[t_j -r, t_j]}$ agree for all $1\le \ell\le m_i+1$ 
\item $\gamma_{i,\ell}((t_j, t_j+r]) \cap \gamma_{i,\ell'}((t_j, t_j+r]) =\emptyset$ for any $\ell, \ell'$ distinct in $[1, m_i +1]$.
\end{itemize}
Let $\wt a_{i}$ be the point on the intersection with $\innerboundary \CB_j(u,v)$ of the part of $\gamma_{i,1}$ between $v$ and $z_i$ -- if there are several such points,  then choose $\wt a_i$ to be the closest point to $u$. 
By possibly relabelling the points $z_1, \ldots, z_{n_j}$, we can assume that $\wt a_1, \ldots, \wt a_{n_j}$ are ordered counterclockwise on $\innerboundary \CB_j$.
Let $ \wt b_{i,\ell}$ be the point on the intersection of the part of $\gamma_{i,\ell}$ between $z_i$ and $u$ with $\outerboundary \CB_j(u,v)$ which is the closest to $v$. 
We are now ready to state the condition: for each $1\le i\le n_j$ and $\ell, \ell' \in [1, m_i+1]$ distinct, we have $d(\wt b_{i,\ell}, \wt b_{i,\ell'})\ge \xi + 12 \eps_0$.
\end{enumerate}

Choose $\sp{a} \in(0,1/3)$. 
Let $N_\epsilon = \epsilon^{-4-\sp{a}}$.  Let $x_1,\ldots,x_{N_\eps}$ be i.i.d.\ points chosen from $\nu$.  
By Lemma~\ref{lem:typical_points_dense}, we know that there a.e.\ exists $\epsilon_1 > 0$ so that for all $\epsilon \in (0,\epsilon_1)$ we have that $\CS \subseteq \cup_{i} B(x_i,\epsilon)$. Choose $\eps_0\in(0, \min(\eps_1, \delta/200))$ in the definition of $P$. Let 
$$G_{i,j} \left((r_j, s_j, n_j)_{1\le j \le k}; \eps, \delta, \xi, \zeta, 9\right)$$ 
(which we will denote by $G_{i,j}(\eps)$ for simplicity) be the event from Lemma~\ref{lem:k_geodesics_nearby} for $(\CS,d,\nu,x_i,x_j)$.  

If $(u,v)\in P$, then for all $\eps\in(0, \eps_0)$, there exist $1\le i,j\le N_\eps$  such that $u\in B(x_i, \eps), v\in B(x_j, \eps)$. Let us show that $(u,v) \in G_{i,j}(\eps)$.
First prove that 
\begin{align}\label{eq:B_uv_x}
\CB_j (u,v) \subseteq \fb{x_j}{x_i}{d(x_i, x_j) - r_j} \setminus \fb{x_j}{x_i}{d(x_i, x_j) - s_j}.
\end{align}
If $z\in  B(x_i, d(x_i, x_j) - s_j)$, then 
$d(z, u) \le d (z, x_i) + d(x_i, u) \le d(x_i, x_j) - s_j + \eps \le d(u,v) -s_j+3\eps.$
It follows that $B(x_i, d(x_i, x_j) - s_j) \subseteq B(u,  d(u,v) -s_j+3\eps)$. Note that $v$ and $y$ are in the same connected component of $B(u,  d(u,v) -s_j+3\eps)$, as well as of $B(x_i, d(x_i, x_j) - s_j)$. Therefore
we also have $\fb{y_i}{x_i}{d(x_i, x_j) - s_j} \subseteq \fb{v}{u}{ d(u,v) -s_j+3\eps}$. 
With similar arguments, we can show that $\fb{v}{u}{ d(u,v) -r_j-3\eps} \subseteq \fb{y_i}{x_i}{d(x_i, x_j) - r_j}$.
Combining the two inclusions and noting $\eps<\eps_0$, we have shown \eqref{eq:B_uv_x}. Therefore, the condition (i) for the set $P$ implies the condition \eqref{1} of  Lemma~\ref{lem:k_geodesics_nearby}.
Due to~\eqref{eq:B_uv_x}, we also know that for each splitting point $z_i$,
\begin{align*}
d(z_i, \partial \fb{x_j}{x_i}{d(x_i, x_j) - r_j}) \ge d(z_i, \innerboundary\CB_j(u,v)) \ge \zeta.\\
d(z_i, \partial \fb{x_j}{x_i}{d(x_i, x_j) - s_j}) \ge d(z_i, \outerboundary\CB_j(u,v)) \ge \zeta.
\end{align*}
Therefore the condition \eqref{3} of  Lemma~\ref{lem:k_geodesics_nearby} is also fulfilled.
The condition (iii) for the set $P$ is the same as the condition \eqref{4} of Lemma~\ref{lem:k_geodesics_nearby}.  

It remains to check the condition \eqref{5} of  Lemma~\ref{lem:k_geodesics_nearby}. Note that $ b_{i,\ell}$ and $\wt b_{i, \ell}$ are both on the geodesic $\gamma_{i, \ell}$, so
\begin{align*}
d( b_{i,\ell}, \wt b_{i, \ell}) = d(\wt b_{i, \ell}, u)  - d( b_{i,\ell}, u) \le &d(u,v) -s_j+3\eps_0 - (d( b_{i,\ell}, x_i) -\eps)\\
= &d(u,v) - d(x_i, x_j)+4\eps_0\le 6\eps_0.
\end{align*}
Therefore for any $\ell, \ell' \in [1, m_i+1]$ distinct, we have
\begin{align*}
d( b_{i,\ell},b_{i, \ell'}) \ge d( \wt b_{i,\ell},  \wt b_{i, \ell'})-d( b_{i,\ell}, \wt b_{i, \ell})   - d( b_{i,\ell'}, \wt b_{i, \ell'})\ge \xi. 
\end{align*}
the condition \eqref{5} of  Lemma~\ref{lem:k_geodesics_nearby}. Combined, we have proved $(u,v) \in G_{i,j}(\eps)$.

Then we know that $\bminflaw[G_{i,j}(\eps)] = O(\epsilon^{K-\sp{a}})$ as $\epsilon \to 0$.  If $K\ge 9$, then by a union bound over $1 \leq i,j \leq N_\epsilon$ we have that $\bminflaw [\cup_{i,j} G_{i,j}(\eps)] = o(\eps^{1-3\sp{a}})$ as $\epsilon \to 0$.  
By the Borel-Cantelli lemma, there exists a sequence of $\eps_m$ going to $0$ and $M>0$, such that the event $\cup_{i,j}G_{i,j}(\eps_m)$ a.e.\  does not occur for any $m\ge M$. Therefore, the set $P$ is $\bminflaw$ a.e.\ empty.

The union of $P\left((r_j, s_j, n_j)_{1\le j \le k}; \eps_0, \delta, \xi, \zeta, K\right)$  for any set of rational numbers $(r_j, s_j, n_j)_{1\le j \le k}$, $\delta$, $\xi$, $\zeta$ chosen as in Lemma~\ref{lem:k_geodesics_nearby}, and for any $\eps_0\in(0, \min(\eps_1, \delta/200))$ contains all pairs of points $u,v\in\CS$ such that the sum of the multiplicities of all the splitting times from $v$ to $u$ is equal to $K$. This completes the proof.
\end{proof}

\section{Proof of strong confluence and topology of geodesics}
\label{sec:geodesic_structure}

This section is structured as follows.  First, in Section~\ref{subsec:intersection_behavior}, we will give the proof of Theorem~\ref{thm:intersection_of_geodesics} which restricts the intersection behavior of geodesics.  This, in turn, will allow us to complete the proof of Theorem~\ref{thm:strong_confluence}.  Next, in Section~\ref{subsec:proof_th1}, we will prove Theorem~\ref{thm:strong_confluence2} using Theorem~\ref{thm:strong_confluence}, which will allow us to complete the proofs of Theorem~\ref{thm:ghost} and Corollary~\ref{cor:geodesic_frame}.

\subsection{Intersection behavior}
\label{subsec:intersection_behavior}

We will now give the proof of Theorem~\ref{thm:intersection_of_geodesics}, which we have restated in an equivalent manner in the following proposition.

\begin{proposition}
\label{prop:no_geodesic_bump}
The following is true for $\bminflaw$ a.e.\ instance of $(\CS,d,\nu,x,y)$.  If $\eta \colon [0,T] \to \CS$ is a geodesic and $0 < s < t < T$ then $\eta|_{[s,t]}$ is the unique geodesic in $\CS$ which connects $\eta(s)$ to $\eta(t)$.
\end{proposition}
The idea of the proof of Proposition~\ref{prop:no_geodesic_bump} is to assume the contrary and deduce that this implies the existence of a pair of points between which there are infinitely many geodesics.  This gives the result as it contradicts Proposition~\ref{prop:two_point_number_of_geodesics}.

\begin{figure}[h]
\begin{center}
\includegraphics[scale=1]{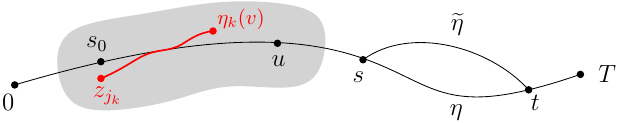}	
\end{center}
\caption{\label{fig:proof_intersection} Proof of Proposition~\ref{prop:no_geodesic_bump}. The grey area is the filled $(r_0/2)$-neighborhood of $\eta([s_0, u])$. We depict (part of) the geodesic $\eta_k$ in red in the case that it crosses $\eta$.}
\end{figure}

\begin{proof}[Proof of Proposition~\ref{prop:no_geodesic_bump}]
Suppose that $(\CS,d,\nu,x,y)$ has distribution $\bminflaw$ and $\eta \colon [0,T] \to \CS$ is a geodesic.  Suppose that there exist $0 < s < t < T$ and a geodesic $\wt{\eta}$ from $\eta(s)$ to $\eta(t)$ which is distinct from $\eta|_{[s,t]}$.  We may assume without loss of generality that $\wt{\eta}$ intersects $\eta$ only at its endpoints. See Figure~\ref{fig:proof_intersection}.  We will derive a contradiction to Proposition~\ref{prop:two_point_number_of_geodesics} by showing that this implies there are infinitely many geodesics in $\CS$ which connect $\eta(0)$ and $\eta(T)$.

Fix $s_0 \in (0,s)$.  Suppose that $(z_j)$ is a sequence in $\CS$ chosen i.i.d.\ from $\nu$.  Let $(z_{j_k})$ be a subsequence which converges to $\eta(s_0)$.  For each $k$, let $\eta_k \colon [0,T_k] \to \CS$ be a geodesic from $z_{j_k}$ to $\eta(T)$.  By passing to a further subsequence if necessary, we may assume without loss of generality that the sequence $\eta_k$ converges to a geodesic $\eta^{s_0}$ from $\eta(s_0)$ to $\eta(T)$.  It must be that $\eta^{s_0} \neq \eta|_{[s_0,T]}$.  For the sake of contradiction, suppose that $\eta_k$ converges to $\eta|_{[s_0,T]}$ in the Hausdorff topology.  Let us show that there exists $k\in\N$ such that $\eta_k \cap \eta|_{[0, s]}\not=\emptyset$. Fix $u\in(s_0,s).$ Let $r_0$ be the infimum over $r>0$ such that the $r$-neighborhood of $\eta([s_0, u])$  disconnects  $\eta(0)$ from  $\eta(s)$ in $\CS$.
There exists $n_0 \in \N$ so that for each $k \geq n_0$, the geodesic $\eta_k|_{[0, u-s_0]}$ must stay in the $(r_0/2)$-neighborhood of $\eta([s_0, u])$. There are two possibilities:
\begin{enumerate}
\item There exists $k\ge n_0$, such that $\eta_k|_{[0, u-s_0]}$ is not close to $\eta|_{[s_0, u]}$ in the one-sided Hausdorff distance. Suppose that $\eta_k(0)=z_{j_k}$ is close to the right side of $\eta$, then there exists $v\in [0, s_0-u]$ such that $\eta_k(v)$ is close to the left side of $\eta$. See Figure~\ref{fig:proof_intersection}. Note that $\eta_k$ needs to stay in the  $(r_0/2)$-neighborhood of $\eta([s_0, u])$, hence cannot go from $z_{j_k}$ to $\eta_k(v)$ without crossing $\eta([0, u])$. This implies $\eta_k \cap \eta|_{[0, s]}\not=\emptyset$.

\item The sequence of geodesics  $(\eta_k|_{[0, u-s_0]})_{k\ge n_0}$ converges to $\eta|_{[s_0, u]}$ in the one-sided Hausdorff distance. Then Proposition~\ref{prop:strong_confluence} implies that there exists $n_1 \ge n_0$ so that $k \geq n_1$ implies that $\eta_k \cap \eta|_{[s_0, s]} \neq \emptyset$. 
\end{enumerate}

Let $k\in\N$ be such that $\eta_k \cap \eta|_{[0, s]}\not=\emptyset$. 
 Let $u_k \in [0, u-s_0]$  and $v_k\in[0,s]$ be so that $\eta(v_k) = \eta_k(u_k)$.  Then the concatenation of $\eta_k|_{[0,u_k]}$, $\eta|_{[v_k,s]}$, $\wt{\eta}$, and $\eta|_{[t,T]}$ is a geodesic from $\eta_k(0)$ to $\eta_k(T_k) = \eta(T)$.  Likewise, the concatenation of $\eta_k|_{[0,u_k]}$ and $\eta|_{[v_k,T]}$ is another geodesic from $\eta_k(0)$ to $\eta(T)$.  These two geodesics agree on their initial and terminal segments.  This is a contradiction since the starting point of $\eta_k$ is typical and the results of \cite{lg2010geodesics} (explained in Figure~\ref{fig:root_geodesics}) imply that this a.s.\ cannot happen.  Therefore $\eta^{s_0}$ is distinct from $\eta|_{[s_0,T]}$.  
 
 Let $(s_k)$ be a positive sequence starting from $s_0$ and which decreases to $0$.  Arguing as above, we can construct for each $k$ a geodesic $\eta^{s_k}$ from $\eta(s_k)$ to $\eta(T)$ which is not equal to the concatenation of $\eta|_{[s_k, s_j]}$ with $\eta^{s_j}$ for $1 \leq j \leq k-1$ and also not equal to $\eta|_{[s_k,T]}$.  Therefore $\eta(0)$ and $\eta(T)$ are connected by infinitely many geodesics.  This contradicts Proposition~\ref{prop:two_point_number_of_geodesics}, which completes the proof.
\end{proof}

We can now complete the proof of Theorem~\ref{thm:strong_confluence}.

\begin{proof}[Proof of Theorem~\ref{thm:strong_confluence}]
This immediately follows by combining Propositions~\ref{prop:strong_confluence} and~\ref{prop:no_geodesic_bump}.  Indeed, Proposition~\ref{prop:no_geodesic_bump} implies that if we have geodesics $\eta_i \colon [0,T_i] \to \CS$ for $i=1,2$ and $\delta > 0$ so that $\eta_i(\delta),\eta_i(T_i-\delta) \in \eta_{3-i}$ then $\eta_i([\delta,T_i-\delta]) \subseteq \eta_{3-i}$.
\end{proof}

\subsection{Proof of Theorem~\ref{thm:strong_confluence2}}
\label{subsec:proof_th1}

In this subsection, we will prove Theorem~\ref{thm:strong_confluence2}, which then allows us to prove Theorem~\ref{thm:ghost} and Corollary~\ref{cor:geodesic_frame}.
We begin with the following estimate on the number of bottlenecks (as illustrated by Figure~\ref{fig:bottleneck})  in the Brownian map along a geodesic between typical points, and then combine this with Theorem~\ref{thm:strong_confluence} to complete the proof of Theorem~\ref{thm:strong_confluence2}.

\begin{lemma}[Bottleneck estimate]
\label{lem:bottlenecks}
Suppose that $(\CS,d,\nu,x,y)$ has distribution $\bminflaw$.  Let $\eta$ be the a.e.\ unique geodesic from $y$ to $x$.  Fix $\epsilon,\sp{u} > 0$.  The measure of the event that the number of $m \in \N$ so that $x$ and $y$ lie in two different connected components of $\CS\setminus B(\eta(m \epsilon),2\epsilon)$ is at least $\epsilon^{-\sp{u}}$ decays to $0$ as $\epsilon \to 0$ faster than any power of $\epsilon$.
\end{lemma}
\begin{proof}
Fix $\epsilon > 0, \sp{u} > 0$.  For each $r > 0$, we let $Y_r$ be the boundary length of $\partial \fb{y}{x}{d(x,y)-r}$.
Fix $m \in \N_{\ge 3}$.  We are going to argue that if $B(\eta(m \epsilon),2\epsilon)$ disconnects $x$ from $y$ in $\CS$, then off an event whose probability tends to $0$ as $\epsilon \to 0$ faster than any power of $\epsilon$ we have that $Y_{(m-2) \epsilon}$ is at most $\epsilon^{2-\sp{u}}$.  Suppose that $Y_{(m-2) \epsilon}$ is at least $\epsilon^{2-\sp{u}}$, and let us deduce a contradiction.  

On the one hand, we have off an event whose probability tends to $0$ as $\epsilon \to 0$ faster than any power of $\epsilon$ that the amount of volume in the $\epsilon$-neighborhood of $\partial \fb{y}{x}{d(x,y)-(m-2) \epsilon}$ is at least a constant times $\epsilon^{4-\sp{u}}$.  This can be seen because we can place $\epsilon^{-\sp{u}}$ equally spaced points on $\partial \fb{y}{x}{d(x,y)-(m-2) \epsilon}$ with boundary length spacing $\epsilon^2$.  The geodesic slices between these points are independent and in $\epsilon$ units of time, each has a positive chance of having mass a constant times $\epsilon^4$.  Therefore by binomial concentration, the total area is at least a constant times $\epsilon^{4-\sp{u}}$ off an event whose probability tends to $0$ as $\epsilon \to 0$ faster than any power of $\epsilon$.  

On the other hand, if $B(\eta(m \epsilon),2\epsilon)$ disconnects $x$ from $y$, then the $\epsilon$-neighborhood of $\partial \fb{y}{x}{d(x,y)-(m-2) \epsilon}$ would be contained in $B(\eta(m \epsilon), 7\epsilon)$. Indeed, we know that $\fb{y}{x}{d(x,y)-(m-2) \epsilon}$ contains $B(\eta(m \eps), 2\eps)$, since any point in $B(\eta(m \eps), 2\eps)$ has distance at most $d(x,y)-(m-2) \epsilon$ to $x$. Therefore $\partial \fb{y}{x}{d(x,y)-(m-2) \epsilon}$ is contained in the $y$-containing connected component of $\CS \setminus B(\eta(m\eps), 2\eps)$. For any point $z\in \partial \fb{y}{x}{d(x,y)-(m-2) \epsilon}$, any geodesic from $z$ to $x$ should pass through $B(\eta(m\eps), 2\eps)$, because $B(\eta(m\eps), 2\eps)$ disconnects $x$ from $y$. Suppose $z_0 \in B(\eta(m\eps), 2\eps)$ is a point on a geodesic from $z$ to $x$, then $d(x,z_0)\ge d(x,y) -m\eps - 2\eps$. We also have $d(x, z)  = d(x,z_0) + d (z_0, z) = d(x,y) -(m-2) \eps$. Combined, we have $d(z_0, z) \le 4\eps$. This implies that $d(z, \eta(2m\eps)) \le 6\eps$, hence $\partial \fb{y}{x}{d(x,y)-(m-2) \epsilon}$ is contained in $B(\eta(m\eps), 6\eps)$. Therefore, the $\epsilon$-neighborhood  of $\partial \fb{y}{x}{d(x,y)-(m-2) \epsilon}$ is contained in $B(\eta(m\eps), 7\eps)$.
We know by Lemma~\ref{lem:le_gall_volume} that the volume of $B(\eta(m \epsilon), 7\epsilon)$ is at most $\epsilon^{4-\sp{u}/2}$ off an event whose probability tends to $0$ as $\epsilon \to 0$ faster than any power of $\epsilon$. This contradicts the previous paragraph, and proves the desired claim.

To complete the proof of the lemma, it suffices to show that the number of $m \in \N$ so that $Y_{m \epsilon}$ is at most $\epsilon^{2-\sp{u}}$ is at most $\epsilon^{-\sp{u}}$ off an event whose probability tends to $0$ as $\epsilon \to 0$ faster than any power of $\epsilon$.  Suppose that $Y_{m \epsilon} \leq \epsilon^{2-\sp{u}}$.  Then there is a positive chance that $Y$ hits $0$ in the next $\epsilon^{1-\sp{u}/2}$ units of time.  That is, we have that the number of such $m$ is at most $\epsilon^{-\sp{u}/2}$ times a geometric random variable, which proves the result.
\end{proof}

\begin{proof}[Proof of Theorem~\ref{thm:strong_confluence2}]
We will prove the assertion of Theorem~\ref{thm:strong_confluence2} for a.e.\ $(\CS,d,\nu,x,y)$ which is sampled from $\bminflaw$ conditioned on $\{1 \leq \nu(\CS) \leq 2\}$.  It will then follow that for Lebesgue a.e.\ value of $a \in [1,2]$, the same result holds a.s.\ for $(\CS,d,\nu,x,y)$ with law $\bmlaw{a}$. The result in the case of a sample from $\bmlaw{a}$ for every value of $a > 0$ thus holds by the scaling property of the Brownian map.  This will complete the proof.

Let $(\CS, d, \nu, x, y)$ be sampled from $\bminflaw$ conditioned on $\{1 \leq \nu(\CS) \leq 2\}$. 
Let $(z_i)$ be a sequence of points chosen i.i.d.\ with respect to the $\nu$ measure in $\CS$. Fix $\sp{u}>0$ and $\eps>0$.  Let $\delta=\eps^{1-\sp{u}}$.

Fix $i, j \ge 1$ distinct. 
Let $E_{i,j}(\eps)$ be the following event.
There exist two geodesics $\eta_1 \colon [0,T_1] \to \CS$ and $\eta_2 \colon [0,T_2] \to \CS$ such that 
\begin{align}
\label{eq:confluence3}
& \eta_1(0), \eta_2(0) \in B(z_i, 6\eps) \quad \text{and} \quad \eta_1(T_1), \eta_2(T_2) \in B(z_j, 6\eps),\\
\label{eq:confluence1}
& T_1 = d(\eta_1(0), \eta_1(T_1)) \geq 2\delta \quad \text{and} \quad T_2 = d(\eta_2(0), \eta_2(T_2)) \geq 2\delta, \\
\label{eq:confluence2}
&\distH(\eta_1([0, T_1]),\eta_2([0,T_2])) \leq \epsilon, \\
\label{eq:confluence4}
& \eta_1([\delta, T_1 -\delta]) \not\subseteq \eta_2 \quad \text{or} \quad \eta_2([\delta, T_2 -\delta]) \not\subseteq \eta_1.
\end{align}
Let $\eta$ be the unique geodesic from $z_i$ to $z_j$. Let us show that there a.s.\ exists $\eps_0>0$ such that for all $\eps\in(0,\eps_0)$, on the event $E_{i,j}(\eps)$, the number of $m \in \N$ so that $B(\eta(30 m \eps), 60\eps)$ disconnects $z_i$ from $z_j$ is at least $\epsilon^{-\sp{u}/2}/8$. Upon showing this,  Lemma~\ref{lem:bottlenecks} will imply that the probability of $E_{i,j}(\eps)$ decays to $0$ as $\epsilon \to 0$ faster than any power of $\epsilon$.

Suppose that we are on the event $E_{i,j}(\eps)$. By~\eqref{eq:confluence4} and Proposition~\ref{prop:no_geodesic_bump}, we actually have 
\[\eta_1([0, \delta/2]) \cap \eta_2 =\emptyset\quad \text{or} \quad\eta_1([T_1-\delta/2, T_1]) \cap \eta_2 =\emptyset.\]
By possibly reversing the times of $\eta_1$ and $\eta_2$, we can assume that 
\begin{align}\label{eq:asum_eta_1}
\eta_1([0, \delta/2]) \cap \eta_2 =\emptyset.
\end{align}
For $1\le n \le \lfloor \eps^{- \sp{u}/2}/4 \rfloor$, let $\beta_n = n \eps^{1- \sp{u}/2}$.
Due to~\eqref{eq:confluence2}, there exists $\alpha_n \in [0, T_1]$ such that $d(\eta_1(\alpha_n), \eta_2(\beta_n))\le \eps$. 
For each $1\le n \le \lfloor \eps^{- \sp{u}/2}/4 \rfloor$, we also have
\begin{align}
\label{eq:distH<eps}
\distH(\eta_1([\alpha_n, \alpha_{n+1}]), \eta_2([\beta_n, \beta_{n+1}])) \le 5\eps.
\end{align}
Indeed, by the same arguments as in Lemma~\ref{lem:hausdorff_close}, we have $|\alpha_n -\beta_n| \le 2\eps$ for each $n$. For each $t\in[0, T_2]$, due to~\eqref{eq:confluence2}, there exists $\wh s(t) \in [0, T_1]$ such that 
$d(\eta_1(\wh s(t)), \eta_2(t)) = d (\eta_2(t), \eta_1) \le \eps.$ Similarly we have $|\wh s(t) -t| \le 2\eps$. This implies that for each $n$ and $t\in[\beta_n + 4\eps, \beta_{n+1} -4 \eps]$, we have $\wh s(t) \in [\alpha_n, \alpha_{n+1}]$. The same statement holds if we switch the roles of $\eta_1$ and $\eta_2$, hence~\eqref{eq:distH<eps} holds. On the other hand,
by Theorem~\ref{thm:strong_confluence}, if $\eps$ is small enough, we must have  
\begin{align}\label{eq:distHos>eps}
\distHos (\eta_1([\alpha_n, \alpha_{n+1}]), \eta_2([\beta_n, \beta_{n+1}])) >15 \eps.
\end{align}
since otherwise $\eta_1([\alpha_n, \alpha_{n+1}])$ and $\eta_2([\beta_n, \beta_{n+1}])$ would intersect each other, contradicting~\eqref{eq:asum_eta_1}.

By~\eqref{eq:distH<eps}, for each $t\in[\beta_n, \beta_{n+1}]$, there exist $s(t) \in [\alpha_n, \alpha_{n+1}]$ and a geodesic $\xi_t$  from $\eta_1(s(t))$ to $\eta_2(t)$ whose length is equal to $d(\eta_2(t), \eta_1([\alpha_n, \alpha_{n+1}])) \le 5 \eps$. Moreover, this geodesic is disjoint from $\eta_1([\alpha_n, \alpha_{n+1}])$ except at its starting point. Note that $\xi_t$ hits $\eta_1|_{[\alpha_n, \alpha_{n+1}]}$ at either its left or right side (we consider that it hits $\eta_1|_{[\alpha_n, \alpha_{n+1}]}$ at both its left and right sides if and only if $\xi_t(0)$ is equal to $\eta_1(\alpha_1)$ or $\eta_1(\alpha_2)$). For each $t$, it is possible that there is more than one choice of $s(t)$ and $\xi_t$. However, we will show that it is not possible to choose $s(t)$ and $\xi_t$ in a way so that $\xi_t$ hits $\eta_1|_{[\alpha_n, \alpha_{n+1}]}$ on the same side for all $t\in[\beta_n, \beta_{n+1}]$. 

Let us assume that we can choose  $s(t)$ and $\xi_t$ for each $t\in[\beta_n, \beta_{n+1}]$ so that $\xi_t$ always hits on the right side of $\eta_1|_{[\alpha_n, \alpha_{n+1}]}$, and try to get a contradiction to~\eqref{eq:distHos>eps}.  Recall that in the definition of $\distHos$, we let $\eta_1^{\rm{L}}$ and $\eta_1^{\rm{R}}$ denote the left and the right sides of $\eta_1$. For each $s\in [0,T_1]$, we will further let $\eta_1^{\rm{L}}(s)$ (resp.\ $\eta_1^{\rm{R}}(s)$) denote the prime end in $\CS\setminus\eta_1$ corresponding to $\eta_1(s)$ on $\eta_1^{\rm{L}}$ (resp.\ $\eta_1^{\rm{R}}$).  Under our assumption, it is clear that
\begin{align}\label{eq:eta_2_inclu_eta_1}
\eta_2 ([\beta_n, \beta_{n+1}]) \subseteq B_{\CS\setminus\eta_1([\alpha_n, \alpha_{n+1}])}(\eta_1^{\rm R}([\alpha_n, \alpha_{n+1}]), 5\eps),
\end{align}
where the right hand side denotes the $\eps$-neighborhood of $\eta_1^{\rm R}([\alpha_n, \alpha_{n+1}])$ in $\CS\setminus\eta_1([\alpha_n, \alpha_{n+1}])$ with respect to the interior-internal metric $d_{\CS\setminus\eta_1([\alpha_n, \alpha_{n+1}])}$.
We can in addition assume that the geodesics $\xi_t$ for $t\in[\beta_n, \beta_{n+1}]$ do not cross each other, so that for each $t_1< t_2$, $\xi_{t_1}$ stays to the right of $\xi_{t_2}$ (this can be achieved by exchanging the trajectories of geodesics $\xi_{t_1}$ and $\xi_{t_2}$ at their crossing point, if they ever cross each other). This implies that $s(t)$ is increasing in $t$. 
For each $s\in[\alpha_n, \alpha_{n+1}]$, if $s=s(t)$ for some $t\in[\beta_n, \beta_{n+1}]$, then we clearly have 
$$d_{\CS\setminus\eta_1([\alpha_n, \alpha_{n+1}])}(\eta_1^{\rm R}(s), \eta_2([\beta_n, \beta_{n+1}]))\le 5\eps.$$ 
Otherwise, let $s_1=\sup\{\sigma < s: \exists t\in[\beta_n, \beta_{n+1}], \sigma=s(t) \}$ and $s_2=\inf\{\sigma > s: \exists t\in[\beta_n, \beta_{n+1}], \sigma=s(t)\}$. Let $t_0= \sup\{t>0: s(t)\le s_1\}$, then we must also have $t_0= \inf\{t>0: s(t)\ge s_2\}$. 
Let $\gamma_1$ (resp.\ $\gamma_2$) be the geodesic from $\eta_1(s_1)$ (resp.\ $\eta_1(s_2)$) to $\eta_1(t_0)$ which is obtained as the limit of the geodesics $\xi_t$ as $t$ tends to $t_0$ from its left (resp.\ right) side. 
It is obvious that $\gamma_1$ and $\gamma_2$ are disjoint from $\eta_1([\alpha_n, \alpha_{n+1}])$ except at their starting points and they hit $\eta_1|_{[\alpha_n, \alpha_{n+1}]}$ on its right side. Moreover, both the lengths of $\gamma_1$ and $\gamma_2$ are at most $5\eps$. By the triangle inequality, $s_2 -s_1$ is at most the sum of the lengths of $\gamma_1$ and $\gamma_2$ which is at most $10\eps$.
By another application of the triangle inequality, it follows that 
$$d_{\CS\setminus\eta_1([\alpha_n, \alpha_{n+1}])}(\eta_1^{\rm R}(s), \eta_2([\beta_n, \beta_{n+1}]))\le 15 \eps.$$ 
Altogether, we have proved
$$\eta_1^{\rm R}([\alpha_n, \alpha_{n+1}]) \subseteq B_{\CS\setminus\eta_1([\alpha_n, \alpha_{n+1}])}(\eta_2([\beta_n, \beta_{n+1}]), 15 \eps).$$ 
Combined with~\eqref{eq:eta_2_inclu_eta_1}, we have proved that the Hausdorff distance from $\eta_1^{\rm R}([\alpha_n, \alpha_{n+1}])$ to $\eta_2([\beta_n, \beta_{n+1}])$ with respect to $d_{\CS\setminus\eta_1([\alpha_n, \alpha_{n+1}])}$ is at most $15\eps$, which contradicts~\eqref{eq:distHos>eps}. 
Similarly, it is impossible to choose  $s(t)$ and $\xi_t$ for each $t\in[\beta_n, \beta_{n+1}]$ so that $\xi_t$ always hits at the left side of $\eta_1|_{[\alpha_n, \alpha_{n+1}]}$.

Let $t_1$ (resp.\ $t_2$) be the infimum over all  $t\in[\beta_n, \beta_{n+1}]$ such that there does not exist $s \in [\alpha_n, \alpha_{n+1}]$ with 
\begin{align*}
&d_{\CS \setminus \eta_1([\alpha_n, \alpha_{n+1}])}(\eta_2(t), \eta_1^{\rm L}(s))= d (\eta_2(t), \eta_1 ([\alpha_n, \alpha_{n+1}])) \\\text{resp. }  &d_{\CS \setminus \eta_1([\alpha_n, \alpha_{n+1}])}(\eta_2(t), \eta_1^{\rm R}(s))= d (\eta_2(t), \eta_1 ([\alpha_n, \alpha_{n+1}])).
\end{align*}
 By the previous paragraph, we know that $t_1, t_2 < \beta_{n+1}$. On the other hand, due to~\eqref{eq:distH<eps}, we must have $t_1>\beta_n$ or $t_2>\beta_n$. Without loss of generality, suppose that $t_1>\beta_n$.
Then there is a sequence $(t^n)$ tending to $t_1$ from the left such that we can choose $s(t^n)$ and $\xi_{t^n}$ in a way that $\xi_{t^n}$ hits $\eta_1$ at its left side. Let $\gamma_3$ be the limit of $\xi_{t^n}$.
There is also a sequence of $(\wt t^n)$ tending to $t_1$ from the right such that we can choose $s(\wt t^n)$ and $\xi_{\wt t^n}$ in a way that $\xi_{\wt t^n}$ hits $\eta_1$ at its right side.  Let $\gamma_4$ be the limit of $\xi_{\wt t^n}$.
Then there exist $s_3, s_4$ such that $\gamma_3$ is a geodesic from $\eta_1^{\rm L} (s_3)$ to $\eta_2(t_1)$ and $\gamma_4$ is a geodesic from $\eta_1^{\rm R} (s_4)$ to $\eta_2(t_1)$. The lengths of $\gamma_3$ and $\gamma_4$ are both at most $5\eps$. We also have $|s_3 -s_4| \le 10\eps$. Note that the concatenation of the time-reversal of $\gamma_3$, $\eta_1([s_3, s_4])$ and $\gamma_4$ forms a closed loop which disconnects $z_i$ and $z_j$, and this loop is contained in $B (\eta_2(t_1), 15 \eps)$.
This implies that $B (\eta_2(t_1), 15 \eps)$ disconnects $z_i$ and $z_j$ inside $\CS$. 
Therefore, we have proved that there exists $t \in (\beta_n, \beta_{n+1})$ such that $z_i$ and $z_j$ are in two different connected components of $\CS \setminus B(\eta_2(t), 15 \eps)$.

For each $1\le n \le \lfloor \eps^{-\sp{u}/2}/4 \rfloor$, let $t_n\in (\beta_n, \beta_{n+1})$ be such that $B(\eta_2(t_n), 15 \eps)$ disconnects $z_i$ and $z_j$.
Recall that $\eta$ is the unique geodesic from $z_i$ to $z_j$, thus $\eta$ must intersect $B(\eta_2(t_n), 15 \eps)$. Suppose that $\eta(r) \in B(\eta_2(t_n), 15 \eps)$. Let $m=\lfloor r/( 30\eps) \rfloor$. Then $B(\eta(30 m \eps), 60\eps)$ also disconnects $z_i$ and $z_j$ inside $\CS$. Since this is true for all $1\le n \le \lfloor \eps^{-\sp{u}/2}/4 \rfloor$, the number of $m \in \N$ so that $B(\eta(30 m \eps), 60\eps)$ disconnects $z_i$ from $z_j$ is at least $\epsilon^{-\sp{u}/2}/8$. By Lemma~\ref{lem:bottlenecks}, the probability of $E_{i,j}(\eps)$ decays to $0$ as $\epsilon \to 0$ faster than any power of $\epsilon$.

Fix $\sp{a}>0$ and let $N= \eps^{-8-\sp{a}}$.
By the union bound, we can deduce that by possibly decreasing $\eps_0$, we have that for all $\eps\in(0, \eps_0)$ and $1\le i, j \le N$ distinct, $E_{i,j}(\eps)$ does not occur.
Let $\delta=\eps^{1-\sp{u}}$. Suppose that $\eta_i : [0, T_i] \to \CS$ for $i=1,2$ are two geodesics such that $T_i \ge 2\delta$ and $ \distH(\eta_1([0, T_1]),\eta_2([0,T_2])) \leq \epsilon$. By Lemma~\ref{lem:typical_points_dense}, by possibly decreasing the value of $\eps_0$, there exist $1\le i,j \le N$ such that $\eta_1(0) \in B(z_i, \eps)$ and $\eta_1(T_1) \in B(z_j, \eps)$. By Lemma~\ref{lem:hausdorff_close}, we know that $\eta_2(0) \in B(z_i, 6\eps)$ and $\eta_2(T_2) \in B(z_j, 6\eps)$. Since $E_{i,j}(\eps)$ does not occur, \eqref{eq:confluence4} must not hold. Therefore, we have $\eta_i([\delta, T_i-\delta]) \subseteq \eta_{3-i}$ for $i=1,2$.
This completes the proof.
\end{proof}

We are now ready to prove Theorem~\ref{thm:ghost} and Corollary~\ref{cor:geodesic_frame}.

\begin{proof}[Proof of Theorem~\ref{thm:ghost}]
Suppose that $\eta \colon [0,T] \to \CS$ is a geodesic and fix $0 < s < t < T$.  Let $(\eta_n)$ be a sequence of geodesics $\eta_n \colon [0,T_n] \to \CS$ with $\eta_n(0) \to \eta(s)$ and $\eta_n(T_n) \to \eta(t)$ as $n \to \infty$.  Suppose that $(\eta_{n_k})$ is a subsequence of $(\eta_n)$.  We can then pass to a further subsequence $(\eta_{n_{k_j}})$ which converges to a geodesic $\wh{\eta}$ connecting $\eta(s)$ and $\eta(t)$.  Theorem~\ref{thm:intersection_of_geodesics} implies that $\wh{\eta}$ is equal to $\eta|_{[s,t]}$.  Thus we have shown that every subsequence of $(\eta_n)$ has a further subsequence which converges to $\eta|_{[s,t]}$ and therefore $(\eta_n)$ converges to $\eta|_{[s,t]}$.  The assertion of the theorem then follows from Theorem~\ref{thm:strong_confluence2}.
\end{proof}

\begin{proof}[Proof of Corollary~\ref{cor:geodesic_frame}]
Let $(z_j)$ be an i.i.d.\ sequence chosen from $\nu$ and, for each $i,j$, we let $\eta_{i,j}$ be the a.s.\ unique geodesic which connects $z_i$ and $z_j$.  Note that $\dimH(\cup_{i,j} \eta_{i,j}) = 1$ a.s.\ since $\dimH(\eta_{i,j}) = 1$ a.s.\ for each $i,j \in \N$.  We will finish the proof by showing that $\gf(\CS) \subseteq \cup_{i,j} \eta_{i,j}$.  Let $\eta \colon [0,T] \to \CS$ be a geodesic and fix $\epsilon > 0$.  Then there exist subsequences $(z_{n_j})$, $(z_{m_j})$ of $(z_j)$ so that $z_{n_j} \to \eta(\epsilon)$ and $z_{m_j} \to \eta(T-\epsilon)$.  Theorem~\ref{thm:ghost} implies that there exists $k_0 \in \N$ so that $k \geq k_0$ implies that $\eta([2\epsilon,T-2\epsilon]) \subseteq \eta_{n_k, m_k}$.  That is, $\eta([2\epsilon,T-2\epsilon]) \subseteq \cup_{i,j} \eta_{i,j}$.  This proves the result since~$\eta$ and $\epsilon > 0$ were arbitrary.
\end{proof}

\section{Exponent for disjoint geodesics from a point}
\label{sec:exponent_disjoint_geodesics}

The purpose of this section is to derive the exponent for the event that the root $x$ of the Brownian map is within distance $\epsilon$ of a point $z$ from which there are $k$ geodesics which are disjoint except for at $z$.

\begin{proposition}
\label{prop:disjoint_geodesics}
Fix $r,\rho >0$ and $\sp{b}_0\in(0,1)$.  Suppose that $(\CS,d,\nu,x,y)$ has distribution $\bminflaw$.  For each $s \geq 0$, let $Y_s$ be the boundary length of $\partial \fb{y}{x}{d(x,y)-s}$ and let $\tau_0 = \inf\{s \geq r : Y_s = \rho\}$.  Fix $k \ge 2$. For each $\epsilon > 0$, let $E(\epsilon,k, r,\rho)$ be the event that $\tau_0<\infty$ and there exists $z \in B(x,\eps)$ such that 
the maximal number of geodesics  from points on $\partial \fb{y}{x}{d(x,y)-\tau_0}$ to $z$ which are disjoint except at $z$ is equal to $k$. 
For each $\rho\in( \epsilon^{2\sp{b}_0}, \eps^{-2\sp{b}_0})$, we have that 
\begin{align}\label{eq:disjoint_geodesics}
\bminflaw[E (\epsilon,k, r,\rho)  \mid d(x,y)>r] =  O((\epsilon/\rho^{1/2})^{k-1 +o(1)}) \quad\mathrm{as}\quad \epsilon \to 0
\end{align}
where the implicit constant depends only on $k, \sp{b}_0$ and does not depend on $r$ or $\rho$.
\end{proposition}

The basic idea to prove Proposition~\ref{prop:disjoint_geodesics} is to use the fact that the boundary length between geodesics from the boundary of a filled metric ball back to the root $x$ evolve as independent $3/2$-stable CSBPs. See Figure~\ref{fig:outline_disjoint_geodesics}. The heuristic is that if there are $k$ geodesics towards a point $z$ which has distance at most $\epsilon$ from $x$ which are disjoint except for at $z$ then one would expect that there are $k$ geodesics towards $x$ which are disjoint until hitting $B(x,\epsilon)$.  This, in turn, would correspond to there being $k$ independent $3/2$-stable CSBPs with initial value of constant order which all hit $0$ for the first time within time $\epsilon$ of each other.  If one conditions on when the first CSBP hits $0$, the conditional probability that the $k-1$ others hit $0$ within $\epsilon$ of this time will then be of order $\epsilon^{k-1}$.  There is some technical work in turning this heuristic into a proof because the $k$ geodesics towards $z$ will not always coincide with $k$ geodesics towards $x$.  This will involve showing that a version of this statement holds at most scales.

We will first collect some preliminary estimates in Section~\ref{subsec:band_estimate} in order to give the proof of Proposition~\ref{prop:disjoint_geodesics}.  More concretely, we will consider an exploration of the Brownian map in which we look at successive bands in the reverse exploration where the boundary length goes down by a factor of two.  The main estimate in Section~\ref{subsec:band_estimate} is Lemma~\ref{lem:bad_band} which bounds from above the probability that such a band has a pinch point in its inner boundary (with respect to the ambient Brownian map metric).  
We will also give a lower bound for the width of such a band in Lemma~\ref{lem:CSBP_stopping}.
In order to understand the proof of Proposition~\ref{prop:disjoint_geodesics}, one can on a first reading only read the statement of Lemmas~\ref{lem:bad_band} and~\ref{lem:CSBP_stopping}, and skip the rest of Section~\ref{subsec:band_estimate}.  
We will complete the proof of Proposition~\ref{prop:disjoint_geodesics} in Section~\ref{subsec:completion_of_proof}.

\subsection{Band estimate}
\label{subsec:band_estimate}
The main part of this subsection is devoted to the proof of the following lemma. At the end of this subsection, we will also prove Lemma~\ref{lem:CSBP_stopping}.

\begin{lemma} 
\label{lem:bad_band}
Suppose that $(\CS,d,\nu,x,y)$ has distribution $\bminflaw$.  For each $r \geq 0$ we let $Y_r$ be the boundary length of $\partial \fb{y}{x}{d(x,y)-r}$.  Let $\tau_1 = \inf\{ r \geq 0 : Y_r = 1\}$, $\tau_2 = \inf\{r \geq \tau_1 : Y_r = 1/2\}$, and $\tau_3 = \inf\{ r \geq \tau_2 : Y_r = 1/4\}$.  Fix $\sp{a} \in (0,1/3)$.  Given $\tau_1 < \infty$, the probability that there exist $z,w \in \partial \fb{y}{x}{d(x,y)-\tau_2}$ with $d(z,w) \leq \epsilon$ and both the clockwise and counterclockwise boundary length distance from $z$ to $w$ are at least $\epsilon^{2\sp{a}}$ is $O(\epsilon^{4/3-4\sp{a}+o(1)})$ as $\epsilon \to 0$.
\end{lemma}

We will make use of the following strategy to prove Lemma~\ref{lem:bad_band}.  We will first (Lemma~\ref{lem:mass_disconnected}) bound the amount of area near $\partial \fb{y}{x}{d(x,y)-\tau_2}$.  We will in particular show that the area in $\fb{y}{x}{d(x,y)-\tau_2+\epsilon} \setminus \fb{y}{x}{d(x,y)-\tau_2}$ is $O(\epsilon^{4/3+o(1)})$.  This is smaller than for a ``typical'' $\epsilon$-width band because when exploring towards $y$ from $\partial \fb{y}{x}{d(x,y)-\tau_2}$ the boundary length process is prevented from making a large downward jump.  This result will be used in the proof of Lemma~\ref{lem:bad_band} because we will argue that the existence of a pinch point would imply that the amount of area in $\fb{y}{x}{d(x,y)-\tau_2+\epsilon} \setminus \fb{y}{x}{d(x,y)-\tau_2}$ is of order $\epsilon^{4\sp{a}}$ as it is likely to contain a ball of radius of order $\epsilon^{\sp{a}}$ (Lemma~\ref{lem:mass_in_filled_metric_ball_complement}).

\begin{lemma}
\label{lem:mass_disconnected}
Suppose that we have the setup of Lemma~\ref{lem:bad_band}.  There exists an event $G$ so that $\p[G^c]$ tends to $0$ as $\epsilon \to 0$ faster than any power of $\epsilon$ so that the following is true.  Given $\tau_1 < \infty$, the conditional expectation of $\nu( \fb{y}{x}{d(x,y)-\tau_2+\epsilon} \setminus \fb{y}{x}{d(x,y)-\tau_2}) \one_G$ is $O(\epsilon^{4/3+o(1)})$.
\end{lemma}
\begin{proof}
Let $X$ be given by $s \mapsto Y_{\tau_1+t(s)}$ where $t(s)$ is the inverse of the Lamperti transform \eqref{eqn:lamperti_csbp_to_levy}.  Then $X$ is a $3/2$-stable L\'evy process with only upward jumps and $X_0 = 1$.  Fix $\sp{u} > 0$.  Let $\sigma_1 = \inf\{ s \geq 0 : X_s = 1/2+\epsilon^{2/3-\sp{u}}\}$ and let $\tau_1 = \inf\{s \geq \sigma_1 : X_s \notin (1/2,1/2+2\epsilon^{2/3-\sp{u}})\}$.  Given that $\sigma_1,\tau_1,\ldots,\sigma_k,\tau_k$ have been defined, we let $\sigma_{k+1} = \inf\{s \geq \tau_k : X_s = 1/2+\epsilon^{2/3-\sp{u}}\}$ and $\tau_{k+1} = \inf\{s \geq \sigma_{k+1} : X_s \notin (1/2,1/2+2\epsilon^{2/3-\sp{u}})\}$.  Let $N = \min\{k : X_{\tau_k} = 1/2\}$.  By the definition of the stopping times $\sigma_k,\tau_k$ we have that $X|_{[\sigma_k,\tau_k]} \leq 1$ (provided $\epsilon > 0$ is small enough) hence
\[ \int_{\sigma_k}^{\tau_k} \frac{1}{X_s} ds \geq \tau_k - \sigma_k,\]
which implies by~\eqref{eqn:lamperti_levy_to_csbp} that $t(\tau_k) - t(\sigma_k) \ge \tau_k -\sigma_k$.

Let $G_1$ be the event that $\tau_N -\sigma_N>\eps$. Then on $G_1$, we have
\[ \inf_{\sigma_N \leq s \leq \sigma_N+\epsilon} X_s \geq 1/2.\]
Since the probability that $\inf_{\sigma_k \leq s \leq \sigma_k+\epsilon} X_s \leq 1/2$ decays to $0$ as $\epsilon \to 0$ faster than any power of $\epsilon$ and the probability that $N$ is larger than $(\log \epsilon^{-1})^2$ decays to $0$ as $\epsilon \to 0$ faster than any power of $\epsilon$, we altogether have that $\p[ G_1^c] \to 0$ as $\epsilon \to 0$ faster than any power of $\epsilon$.

Let $S_k$ be the sum of the squares of the jumps made by $X$ in $[\sigma_k,\tau_k]$ of size at most $2\epsilon^{2/3-\sp{u}}$.  Note that $X|_{[\sigma_N,\tau_N]}$ cannot make a jump of size larger than $2\epsilon^{2/3-\sp{u}}$ for otherwise $X|_{[\sigma_N,\tau_N]}$ would first exit the interval $(1/2,1/2+2\epsilon^{2/3-\sp{u}})$ by exceeding the value $1/2 + 2\epsilon^{2/3-\sp{u}}$.  It therefore follows that on $G_1$, the conditional expectation of $A = \nu( \fb{y}{x}{d(x,y)-\tau_2+\epsilon} \setminus \fb{y}{x}{d(x,y)-\tau_2})$ given $X$ is at most a constant times $S_N$.

Let $G_2$ be the event that $\tau_k - \sigma_k$ is at most $\epsilon^{1-2\sp{u}}$ for each $1 \leq k \leq N$.  Then $\p[G_2^c] \to 0$ as $\epsilon \to 0$ faster than any power of $\epsilon$ since in each round of length $\epsilon^{1-3\sp{u}/2}$ the process has a positive chance of leaving $(1/2,1/2+2 \epsilon^{2/3-\sp{u}})$ no matter where it starts.  Let $G = G_1 \cap G_2$.  Then it suffices to show that
\[ \E\left[ \one_G S_N \right] =  O(\epsilon^{4/3-\sp{c}})\]
where $\sp{c} > 0$ can be made arbitrarily close to $0$ by making $\sp{u} > 0$ arbitrarily close to $0$.  
Note that
\begin{align*}
 \E\left[ \one_G S_N \right] = \E\left[ \one_G \sum_{k=1}^\infty S_k \one_{\{N=k\}} \right] \le \sum_{k=1}^\infty \E\!\left[S_k^p \one_G \right]^{1/p} \p[N=k]^{(p-1)/p}
\end{align*}
for any $p>1$.  Since $N$ is a geometric random variable, $\p[N=k]^{(p-1)/p}$ is summable in $k$. Therefore, it is enough to prove that there exists $p>1$ such that
\begin{align*}
\E\left[ S_1^p \one_G \right]^{1/p} = O(\eps^{4/3-\sp{c}}).
\end{align*}
By the definition of $G$, it suffices to show  that
\[  \E[ S^p]^{1/p}  = O(\epsilon^{4/3-\sp{c}})\]
where $S$ is the sum of the squares of the jumps made by $X$ of size at most $2\epsilon^{2/3-\sp{u}}$ in the time interval $[0,\epsilon^{1-2\sp{u}}]$.  We note that for each $k \in \Z$, the number $N_k$ of jumps that $X$ makes in time $\epsilon^{1-2\sp{u}}$ of size in $[e^k,e^{k+1}]$ is Poisson with mean $\lambda_k$ proportional to $\epsilon^{1-2\sp{u}} e^{-3 k/2}$.  
Fix $\delta > 0$ and let
\[ \sum_{k \in \Z, k \leq (2/3-\sp{u}) \log \epsilon} e^{\delta k} \leq c_0 \epsilon^{\delta(2/3-\sp{u})} = c_1.\]
We have that
\begin{align*}
  \E[ S^p ]
&\leq \E\left[ \left( \sum_{k\in\Z, k \leq (2/3-\sp{u}) \log \epsilon}  e^{2k+2} N_k \right)^p\, \right]\\
&= \E\left[ \left( \sum_{k\in\Z, k \leq (2/3-\sp{u}) \log \epsilon}  e^{(2-\delta)k+2} N_k e^{\delta k} \right)^p\, \right]\\
&\leq c_1^{p-1}  \sum_{k\in\Z,k \leq (2/3-\sp{u}) \log \epsilon} \big( e^{(2-\delta)kp +2p} \E[ N_k^p] \big)  e^{\delta k} \quad\text{(Jensen's inequality)}.
\end{align*}
If $k \in \Z$ is such that $\lambda_k \geq 1$ then $\E[ N_k^p]$ is at most a constant times $\lambda_k^p$.  Altogether, we see that for each $\sp{c} > 0$ we can find $\delta, \sp{u} > 0$ sufficiently close to $0$ and $p > 1$ sufficiently close to $1$ so that $(\E[ S^p ])^{1/p} = O(\epsilon^{4/3-\sp{c}})$ as $\epsilon \to 0$.
\end{proof}

\begin{lemma}
\label{lem:mass_in_filled_metric_ball_complement}
Suppose that we have the setup of Lemma~\ref{lem:bad_band} and we condition on $\tau_1 < \infty$.  Fix $\sp{u} > 0$.  The probability of the event that there exist $z,w \in \partial \fb{y}{x}{d(x,y)-\tau_2}$ with $d(z,w) \leq \epsilon$ and both the clockwise and counterclockwise boundary length distance from $z$ to $w$ are at least $\epsilon^{2\sp{a}}$ and there does not exist $z_0$ so that $B(z_0,\epsilon^{\sp{a}+\sp{u}}) \subseteq \fb{y}{x}{d(x,y)-\tau_2+8\epsilon} \setminus \fb{y}{x}{d(x,y)-\tau_2}$ decays to $0$ as $\epsilon \to 0$ faster than any power of $\epsilon$.
\end{lemma}

\begin{proof}
Suppose that $(\CS,d,\nu,x,y)$ has distribution $\bminflaw$.  We let $\sigma$ be the smallest $r > 0$ so that the boundary length of $\partial \fb{y}{x}{r}$ is equal to $1/2$ and we condition on $\sigma < \infty$.  Fix $\sp{c} \geq 1$.  Then for each $k \geq 0$ we have that $\CS \setminus \fb{y}{x}{\sigma+k\epsilon^{\sp{c}}}$ is a Brownian disk weighted by its area.  It follows that given $d(x,y) \geq \sigma + k \epsilon^{\sp{c}}$, we have off an event whose probability decays to $0$ as $\epsilon \to 0$ faster than any power of $\epsilon$ that the following is true.  For each $p,q \in \partial \fb{y}{x}{\sigma + k \epsilon^{\sp{c}}}$ the measure of the $\epsilon^{\sp{a}+\sp{u}}$ neighborhood of the counterclockwise arc from $p$ to $q$ is between $\epsilon^{2\sp{a}+3\sp{u}}$ and $\epsilon^{2\sp{a}+\sp{u}}$ times the length of the arc \cite[Lemmas~3.2--3.4]{gm2019gluing}.  Therefore this holds for all $k$ such that $d(x,y) \geq \sigma + k \epsilon^{\sp{c}}$ simultaneously off an event whose probability tends to $0$ as $\epsilon \to 0$ faster than any power of $\epsilon$.

\begin{figure}[h!]
\centering
\includegraphics[width=.8\textwidth]{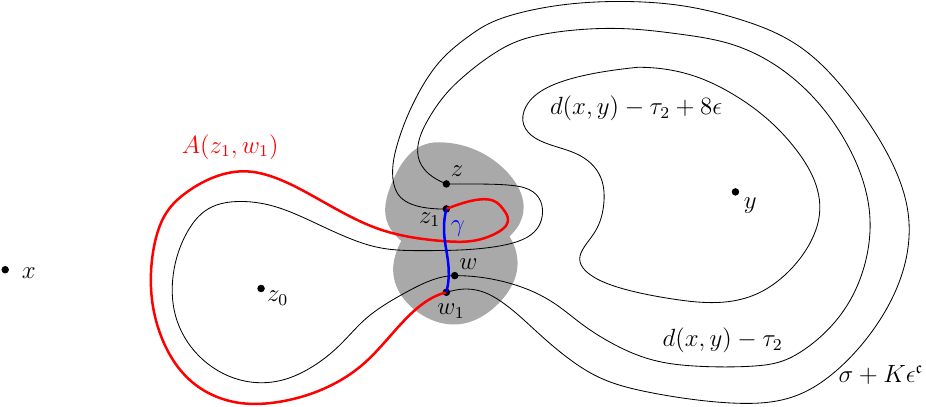}
\caption{Illustration of the setup and proof of Lemma~\ref{lem:mass_in_filled_metric_ball_complement}. We depict in dark grey $\fb{y}{z}{8\eps} \cup \fb{y}{w}{8\eps}$. The arc $A(z_1, w_1)$ is drawn in red and its $\eps^{\sp{a}+\sp{u}/2}$-neighborhood with respect to $d_K$ is drawn in light grey.}
\label{fig:crossing_layer4}
\end{figure}

Let $K$ be such that $\sigma + K \epsilon^{\sp{c}} \leq d(x,y)-\tau_2 \leq \sigma + (K+1)\epsilon^{\sp{c}}$.  
Let $d_K$ denote the interior-internal metric of $\CS \setminus \fb{y}{x}{\sigma+K\eps^{\sp{c}}}$.
Let $z_1, w_1$ be points on $\partial \fb{y}{x}{\sigma + K \epsilon^{\sp{c}}}$ which are closest to $z,w$.  
See Figure~\ref{fig:crossing_layer4}. 
Let $C_0$ be the connected component containing $y$ of $\CS\setminus (\fb{y}{x}{\sigma+K\eps^{\sp{c}}} \cup \fb{y}{z}{8\eps} \cup \fb{y}{w}{8\eps})$. Let us now show that, among the two arcs (clockwise and counterclockwise) on $\partial \fb{y}{x}{\sigma + K \epsilon^{\sp{c}}}$ from $z_1$ to $w_1$, at least one of them is disjoint from $\ol{C}_0$. 
In order to show that at most one of the two arcs intersect $\partial C_0$, it suffices to  show that $\partial C_0 \setminus (\partial \fb{y}{z}{8\eps} \cup \partial \fb{y}{w}{8\eps})$ contains at most one connected component. We will prove this statement by contradiction in the next paragraph.

Suppose that $\partial C_0 \setminus (\partial \fb{y}{z}{8\eps} \cup \partial \fb{y}{w}{8\eps}))$ contains at least two different connected components $A_1$ and $A_2$. See Figure~\ref{fig:crossing_layer6}. Then $A_1$ and $A_2$ are separated by two other arcs $A_3$ and $A_4$ which are both contained in $\partial C_0 \cap (\partial \fb{y}{z}{8\eps} \cup \partial \fb{y}{w}{8\eps})$. Note that $A_3$ and $A_4$ are connected via some path $\xi \subset  \fb{y}{z}{8\eps} \cup  \fb{y}{w}{8\eps}$. Since $d(z,w) \le \eps$ and $d(z_1, w_1) \le 3 \eps$, $\fb{y}{z}{8\eps} \cup \fb{y}{w}{8\eps}$ must contain a geodesic $\gamma$ connecting $z_1$ to $w_1$. 
Therefore we can choose $\xi$ as the concatenation of a curve from $A_3$ to $z_1$, and $\gamma$, and a curve from $w_1$ to $A_4$.
Since $\xi$ is disjoint from $C_0 \cup \{x\}$, it must disconnect either $A_1$ or $A_2$ from $x$ inside $\CS \setminus C_0$. Without loss of generality, suppose that $\xi$ disconnects $A_1$ from $x$ inside $\CS \setminus C_0$. 
For any point $u\in A_1$, a geodesic from $u$ to $x$ must stay in $\fb{y}{x}{\sigma+K\eps^{\sp{c}}}$, hence must intersect $\gamma$ at some point $u_0$.
Note that 
$$d(u_0, x) \ge d (z_1,x) - d(z_1, u_0) \ge d (z_1,x) - 3\eps.$$
Therefore 
$$d(u, u_0) = d(u,x) - d(u_0, x) \le d(u,x) - d (z_1,x) + 3\eps = 3\eps.$$
Therefore $d(u, z) \le d(u,u_0) + d(u_0, z_1) + d(z_1, z) \le 7\eps$, hence $u \in \fb{y}{z}{8\eps}$. This implies that $A_1 \subset  \fb{y}{z}{8\eps}$, which contradicts the fact that $A_1 \subset \partial C_0 \setminus (\partial \fb{y}{z}{8\eps} \cup \partial \fb{y}{w}{8\eps}))$.

\begin{figure}[h!]
\centering
\includegraphics[width=.65\textwidth]{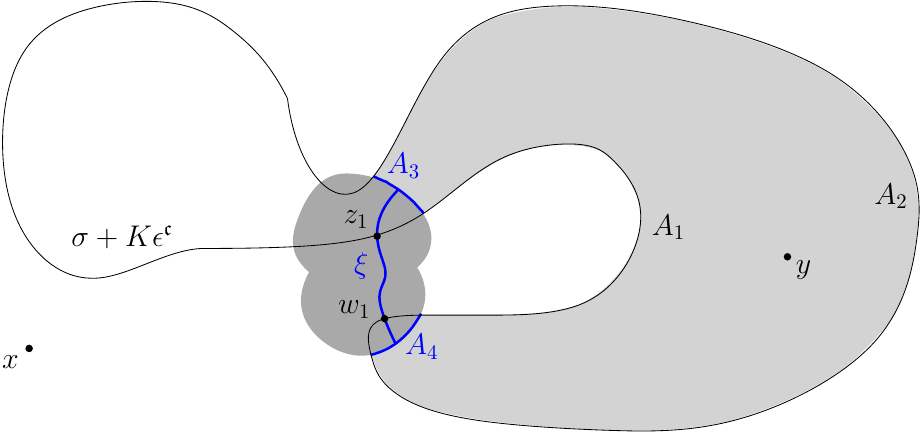}
\caption{Illustration of the proof of Lemma~\ref{lem:mass_in_filled_metric_ball_complement}. We depict in dark grey $\fb{y}{z}{8\eps} \cup \fb{y}{w}{8\eps}$, and in light grey $C_0$.}
\label{fig:crossing_layer6}
\end{figure}

We have now proved that at least one the two arcs (clockwise and counterclockwise) on $\partial \fb{y}{x}{\sigma + K \epsilon^{\sp{c}}}$ from $z_1$ to $w_1$ is disjoint from $\ol{C}_0$. Let $A(z_1, w_1)$ be such an arc.  
We also remark that $C_0$ in fact contains $\CS\setminus \fb{y}{x}{d(x,y) -\tau_2 +8\eps}$, hence all other connected components of $\CS\setminus (\fb{y}{x}{\sigma+K\eps^{\sp{c}}} \cup \fb{y}{z}{8\eps} \cup \fb{y}{w}{8\eps})$ are disjoint from  $\CS\setminus \fb{y}{x}{d(x,y) -\tau_2 +8\eps}$.

Suppose that there does not exist a point $z_0$ as in the statement of the lemma. Then any point in $\CS\setminus (\fb{y}{x}{\sigma+K\eps^{\sp{c}}} \cup \fb{y}{z}{8\eps} \cup \fb{y}{w}{8\eps})$ which is not in $C_0$ has $d_K$-distance at most $\eps^{\sp{a}+\sp{u}} + \eps$ from $A(z_1, w_1)$. For any point $v \in C_0$, any path connecting $v$ to $A(z_1, w_1)$ inside $\CS \setminus \fb{y}{x}{\sigma+K\eps^{\sp{c}}}$ must first hit $\fb{y}{z}{8\eps} \cup \fb{y}{w}{8\eps}$. Therefore, the $d_K$-distance from $v$ to $ A(z_1, w_1)$ is at least  $\max (d(v, z_1), d(v, w_1)) - 8\eps$.

This implies that the $\epsilon^{\sp{a}+\sp{u}/8}$ neighborhood of $A(z_1, w_1)$ with respect to  $d_K$ is contained in the union $U$ of the $2\epsilon^{\sp{a}+\sp{u}}$ neighborhood of  $A(z_1, w_1)$ with respect to $d_K$ together with $B(z,2\epsilon^{\sp{a}+\sp{u}/8})$ and $B(w,2\epsilon^{\sp{a}+\sp{u}/8})$.  Let $L$ be the length of $A(z_1, w_1)$.  We are now going to give upper and lower bounds for $\nu(U)$ which are not in agreement with each other.  
\begin{itemize}
\item (Upper bound.) The $2\epsilon^{\sp{a}+\sp{u}}$ neighborhood of  $A(z_1, w_1)$ with respect to $d_K$ has $\nu$-measure at most $\epsilon^{2\sp{a}+\sp{u}} L$ off an event whose probability tends $0$ as $\epsilon \to 0$ faster than any power of $\epsilon$.
By Lemma~\ref{lem:le_gall_volume}, we know that the $\nu$-measure of $B(z,2\epsilon^{\sp{a}+\sp{u}/8})$ and $B(w,2\epsilon^{\sp{a}+\sp{u}/8})$ is at most $\epsilon^{4\sp{a}+7\sp{u}/16}$ off an event whose probability tends $0$ as $\epsilon \to 0$ faster than any power of $\epsilon$ (the choice $7/16$ is arbitrary; any multiple of $\sp{u}$ smaller than $1/2$ suffices).  Altogether, we deduce that  $\nu(U) \le \epsilon^{2a + \sp{u}} L + \epsilon^{4\sp{a}+7\sp{u}/16}$ off an event whose probability tends $0$ as $\epsilon \to 0$ faster than any power of $\epsilon$.
\item (Lower bound.) On the other hand, the $\nu$-measure of the $\epsilon^{\sp{a}+\sp{u}/8}$ neighborhood of $A(z_1, w_1)$ is at least $\epsilon^{2\sp{a}+3\sp{u}/8} L$ off an event whose probability tends $0$ as $\epsilon \to 0$ faster than any power of $\epsilon$.  Since $L \geq \epsilon^{2\sp{a}}/2$ off an event whose probability tends to $0$ as $\epsilon \to 0$ faster than any power of $\epsilon$, we have that $\nu$-measure of the $\epsilon^{\sp{a}+\sp{u}/8}$ neighborhood of $A(z_1, w_1)$ exceeds $\epsilon^{2a + \sp{u}} L + \epsilon^{4\sp{a}+7\sp{u}/16}$ off an event whose probability tends to $0$ as $\epsilon \to 0$ faster than any power of $\epsilon$.
\end{itemize}
Both of these possibilities cannot occur.  Altogether, this proves the lemma.
\end{proof}

\begin{proof}[Proof of Lemma~\ref{lem:bad_band}]
Suppose that we have the setup described in the statement of the lemma and $z,w \in \partial \fb{y}{x}{d(x,y)-\tau_2}$ have clockwise and counterclockwise boundary length distance at least $\epsilon^{2\sp{a}}$.  Lemma~\ref{lem:mass_in_filled_metric_ball_complement} implies that off an event whose probability decays to $0$ as $\epsilon \to 0$ faster than any power of $\epsilon$ there exists $z_0$ so that $B(z_0,\epsilon^{\sp{a}+\sp{u}}) \subseteq \fb{y}{x}{d(x,y)-\tau_2+\epsilon} \setminus \fb{y}{x}{d(x,y)-\tau_2}$.  We also have that $\nu(B(z_0,\epsilon^{\sp{a}+\sp{u}})) \geq \epsilon^{4\sp{a}+5\sp{u}}$ off an event whose probability tends to $0$ as $\epsilon \to 0$ faster than any power of $\epsilon$.  This implies that $\nu(\fb{y}{x}{d(x,y)-\tau_2+\epsilon} \setminus \fb{y}{x}{d(x,y)-\tau_2})$ is at least $\epsilon^{4\sp{a}+5\sp{u}}$ off an event whose probability tends to $0$ as $\epsilon \to 0$ faster than any power of $\epsilon$.  On the other hand, Lemma~\ref{lem:mass_disconnected} implies that there exists an event $G$ with $\p[G^c] \to 0$ faster than any power of $\epsilon$ so that the expected amount of mass in $\fb{y}{x}{d(x,y)-\tau_2+\epsilon} \setminus \fb{y}{x}{d(x,y)-\tau_2}$ on $G$ is $O(\epsilon^{4/3+o(1)})$.  Therefore by Markov's inequality, the probability of the event in question is $O(\epsilon^{4/3-4\sp{a}-5\sp{u} +o(1)})$.  Since $\sp{u} > 0$ was arbitrary, the result follows. 
\end{proof}

Finally, let us record an estimate on CSBP which implies that the width of the bands considered in Lemma~\ref{lem:bad_band} cannot be too thin.
\begin{lemma}
\label{lem:CSBP_stopping}
Fix $\alpha \in (1,2)$ and suppose that $Y$ is an $\alpha$-CSBP and $Y_0 >0$. Let $\tau$ be the first time that $Y$ hits $Y_0/2$. Then there exists a constant $c_1>0$ depending only on $\alpha$ such that
\begin{align*}
\p[\tau \le \eps Y_0^{\alpha-1}] =O\left( \exp(- c_1 \epsilon^{-1/\alpha}) \right) \quad\text{as}\quad \eps\to 0.
\end{align*}
\end{lemma}
\begin{proof}
By the scaling property of the $\alpha$-CSBP, it suffices to prove the lemma for $Y_0=1$.  Let $X$ be the $\alpha$-stable L\'evy process with only upward jumps obtained by applying the inverse of the Lamperti transform~\eqref{eqn:lamperti_csbp_to_levy} to $Y$.  Let $\sigma_0 = 0$, $\tau_0 = \inf\{t \geq 0 : X_t \notin (1/2,2)\}$, and $\sigma_1 = \inf\{t \geq \tau_0 : X_t = 1\}$.  Given that $\tau_0,\sigma_1,\ldots,\sigma_k,\tau_k$ have been defined, we let $\sigma_{k+1} = \inf\{t \geq \tau_k : X_t = 1\}$ and $\tau_{k+1} = \inf\{t \geq \sigma_{k+1} : X_t \notin (1/2,2)\}$.  Let $N = \min\{k \geq 0 : X_{\tau_k} = 1/2\}$.  Then we have that $N$ is a geometric random variable with parameter $p \in (0,1)$.  Note that for each $k \in \N_0$ we have that
\[ \int_{\sigma_k}^{\tau_k} \frac{1}{X_t} dt \geq \frac{1}{2} (\tau_k-\sigma_k).\]
Therefore
\[ \tau \geq \tfrac{1}{2}(\tau_N- \sigma_N).\]
Moreover, the probability that $\tau_k - \sigma_k \leq 2\epsilon$ on the event that $X|_{[\sigma_k,\infty)}$ hits $1/2$ before $2$ is at most $\exp(-c \epsilon^{-1/\alpha})$ where $c > 0$ is a constant which depends only on $\alpha$ \cite[Chapter~VII, Corollary~2]{b1996levy}.  Altogether, we have that
\begin{align*}
\p[ \tau \leq \epsilon] 
&\leq \p[ N \geq n ] + n \p[ \tau_0 \leq 2 \epsilon]
 \leq p^n + n \exp(-c \epsilon^{-1/\alpha} ).
\end{align*}
Taking $n = \epsilon^{-1/\alpha}$ implies the result.
\end{proof}

\subsection{Proof of Proposition~\ref{prop:disjoint_geodesics}}
\label{subsec:completion_of_proof}

We will divide the proof of Proposition~\ref{prop:disjoint_geodesics} into eight steps each of which are carried out below.  Before we proceed to the details, we will first provide an overview of these steps.  The overall idea of the proof is to relate the behavior of the geodesics towards the point $z$ to the behavior of the geodesics towards $x$.  Namely, we will consider a reverse metric exploration towards $x$ and consider successive metric bands which correspond to when the boundary length drops by a factor of $2$.  We will then argue that off an event whose probability tends to $0$ as $\epsilon \to 0$ sufficiently fast, the geodesics towards $z$ will coincide with geodesics towards $x$ except in possibly a finite number of these metric bands (also called layers).  This will allow us to relate the event considered in Proposition~\ref{prop:disjoint_geodesics} to the event that $k$ independent $3/2$-stable CSBPs hit $0$ within time $\epsilon$ of each other.

In Step 1, we will describe the setup.  The purpose of Steps 2-4 is to describe the three types of ``bad'' layers.  These correspond to ``fat layers'' (the amount of time it takes for the boundary length to go down by a factor of $2$ is too large), ``crossing layers'' (the geodesics towards $z$ intersect the boundary of a layer in a problematic way), and ``non-merging layers'' (the geodesics towards $z$ fail to merge with the geodesics towards $x$).  Since conditioning on the location of the point $z$ as well as on the behavior of the geodesics towards $z$ is not compatible with the Markovian nature of the reverse exploration, we will need to define an exploration which has a good chance of following the geodesics towards $z$.  The purpose of Step 5 is to define such a Markovian exploration. In Step 6 we will show that on an event with sufficiently large probability, we can couple the Markovian exploration with the non-Markovian ``exploration'' that we used to explore the geodesics towards $z$.  In Step 7 we will prove that all but a finite number of layers are good layers off an event which occurs with negligible probability.  Finally, in Step 8, we will explain how to relate the exponent for $k$ independent $3/2$-stable CSBPs hitting $0$ of each other (applied to the good layers) to the event in the statement of Proposition~\ref{prop:disjoint_geodesics}.

Throughout the proof, we will condition  $(\CS,d,\nu,x,y)$ on $d(x,y) >r$.  We will also fix $\sp{b}_1\in(0,1)$ and $\sp{a} \in (0,(1-\sp{b}_1)/6)$. For each $m\in \N_0$, let $\sp{a}_m= \sp{a} (m+2)/(m+1)$. Note that $\sp{a}_m\in (\sp{a}, 2\sp{a}]$ and is decreasing in $m$.

\noindent{\it Step 1: Setup.}
For each $t \geq 0$, we let $Y_t$ be the boundary length of $\partial \fb{y}{x}{d(x,y)-t}$. We inductively define radii as follows.  We let $\tau_0  = \inf\{ s \geq r : Y_s = \rho\}$.  Given that $\tau_0,\ldots,\tau_m$ have been defined, we define $\tau_{m+1}:=\inf\{t \geq \tau_m:\, Y_t \le Y_{\tau_m}/2\}$. Since the process $Y_t$ only has upward jumps, for each $m \in\N$ we have $Y_{\tau_{m}} = Y_{\tau_{m-1}}/2 =2^{-m}\rho$.

Let $m_0$ be the smallest integer $m$ so that $\partial  \fb{y}{x}{d(x,y)-\tau_m}$ has boundary length at most $\epsilon^{2 \sp{b}_1}$.  Since $Y_{\tau_{m_0}} = 2^{-m_0} \rho$ and $\rho\in(\eps^{2\sp{b}_0}, \eps^{-2\sp{b}_0})$, we have
\begin{align*}
 2(\sp{b}_1- \sp{b}_0) \log_2 \eps^{-1} \le \log_2 (\eps^{-2 \sp{b}_1} \rho)  \le m_0 \le 1+ \log_2 (\eps^{-2 \sp{b}_1} \rho) \le 1 +2(\sp{b}_1+\sp{b}_0) \log_2 \eps^{-1}.
\end{align*}
By~\eqref{eqn:csbp_extinction_time}, off an event whose probability decays faster than any power of $\eps$, we have 
\[ B(x, \eps) \subseteq  \fb{y}{x}{d(x,y)-\tau_{m_0}}.\]

For $m\in \{0,\ldots,m_0-1\}$, let $\CF_m$ denote the $\sigma$-algebra generated by the metric measure space $\CS \setminus \fb{y}{x}{d(x,y) -\tau_{m}}$.  Let $\CB_m$ be the metric band $\fb{y}{x}{d(x,y) - \tau_{m}} \setminus \fb{y}{x}{d(x,y) - \tau_{m+1}}$ so that $\CB_m$ is $\CF_{m+1}$-measurable. For each $t\in[r, d(x,y))$, let $\CL_t:= \partial \fb{y}{x}{d(x,y)-t}$.  

Let $A_0$ be the set of points $z\in B(x, \eps)$ such that there are at least $k$ disjoint (except at $z$) geodesics  from $\partial  \fb{y}{x}{d(x,y)-\tau_0}$ to $z$. On the event $E(\eps, k, r, \rho)$ (we will denote this event by $E$ in the sequel), $A_0$ is non empty. 
Let us specify a measurable way of choosing a particular point $z\in \ol{A_0}$ and $k$ disjoint (except at $z$) geodesics $\eta_1, \ldots, \eta_k$ from $\partial  \fb{y}{x}{d(x,y)-\tau_1}$ to $z$.

\emph{Choice of $z$.} 
We first prove that every point $z_0$ in the closure $\ol {A_0}$ also has at least $k$ disjoint (except possibly at $z_0$) geodesics from $\partial  \fb{y}{x}{d(x,y)-\tau_1}$ to $z_0$.
Let $(z_n)_{n\ge 1}$ be a sequence of points in $A_0$ that converges to some $z_0 \in \ol{A_0}$. For each $n\ge 1$ there are $k$ disjoint except at $z_n$ geodesics $\eta_1^{n}, \ldots, \eta_k^{n}$ from  $\partial  \fb{y}{x}{d(x,y)-\tau_0}$ to $z_{n}$.  By possibly extracting subsequences, we can assume that for each $1\le j \le k$ the geodesics $(\eta_j^{n})_{i\ge 1}$ converge to a geodesic $\eta^0_j$ from $\partial  \fb{y}{x}{d(x,y)-\tau_0}$ to $z_0$. By strong confluence (Theorem~\ref{thm:strong_confluence2}), the geodesics $\eta^0_1, \ldots, \eta^0_k$ must be disjoint except at their endpoints: they all intersect at $z_0$, and can possibly intersect at their other endpoints. Consequently, there are $k$ geodesics from $\partial  \fb{y}{x}{d(x,y)-\tau_1}$ to $z_0$ which are all disjoint except at $z_0$. 

For each $\wt z \in \ol{A_0}$, let $\xi_{\wt z}$ be the left-most geodesic from $\wt z$ to $x$, where \emph{left-most} is relative to the unique geodesic from $y$ to $x$. 
We define an order on the set $\ol{A_0}$, so that $z_1 \preceq z_2$ if the geodesic $\xi_{z_1}$ is  \emph{left} to the geodesic $\xi_{z_2}$, namely if $\xi_{z_2}$ merges on the right side of $\xi_{z_1}$ or on the left side of the unique geodesic from $y$ to $x$. It is easy to check that $z_1\preceq z_2$ and $z_2\preceq z_3$ implies $z_1\preceq z_3$. We have $z_1 \preceq z_2$ and $z_2\preceq z_1$ if and only if $z_1\in \xi_{z_2}$ or $z_2\in \xi_{z_1}$. Let $A_1$ be the subset of $\ol{A_0}$ which contains all the minimal points with respect to the relation $\preceq$. Since $\ol{A_0}$ is compact, $A_1$ is non-empty. All the points in $A_1$ belong to a same geodesic to $x$. Let $z\in A_1$ be the point in $A_1$ which has the maximal distance to $x$.

\emph{Choice of $k$ disjoint (except at $z$) geodesics to $z$.} For each $z_0 \in \partial  \fb{y}{x}{d(x,y)-\tau_1}$, let $\ell_{z_0}$ be  the counterclockwise boundary length along $ \partial  \fb{y}{x}{d(x,y)-\tau_1}$ from $z_0$ to the point where the unique geodesic from $y$ to $x$ hits $ \partial  \fb{y}{x}{d(x,y)-\tau_1}$.
Let $U_0$ be the set of $k$-tuples $(z_1, \ldots, z_k)$ such that $z_1, \ldots, z_k\in \partial  \fb{y}{x}{d(x,y)-\tau_1}$ are ordered in a way  that $\ell_{z_i}$ is decreasing with respect to $1\le i \le k$, and that there are $k$ disjoint geodesics (except at $z$) from $z_i$ to $z$ for $1\le i \le k$. Let $\delta(z_1, \ldots, z_k):=\min_{1\le i <j \le k} d(z_i, z_j) > 0$. Let $U_1$ be the set of  $k$-tuples $(z_1, \ldots, z_k)$ in $U_0$ such that $\delta(z_1, \ldots, z_k)$ attains the maximum of $\delta(\wt z_1, \ldots, \wt z_k)$ for all $(\wt z_1, \ldots, \wt z_k)$ in $U_0$. The set $U_1$ is non empty by the argument of extracting a subsequence and using strong confluence. 
Let $U_2$ be the set of $k$-tuples $(z_1, \ldots, z_k)\in U_1$ such that $\ell_{z_1}$ attains $\min_{(\wt z_1, \ldots, \wt z_k)\in U_0} \ell_{\wt z_1}$. For each $1\le i \le k-1$, suppose that we have defined $U_{i+1}$, let $U_{i+2}$ be  set of $k$-tuples $(z_1, \ldots, z_k)\in U_{i+1}$ such that $\ell_{z_{i+1}}$ attains $\min_{(\wt z_1, \ldots, \wt z_k)\in U_0} \ell_{\wt z_{i+1}}$. The sets $U_2, \ldots, U_{k+1}$ are non empty by the argument of extracting a subsequence and using strong confluence. Moreover, the set $U_{k+1}$ contains a unique $k$-tuple $(z^0_1, \ldots, z^0_k)$. Let $V_0$ be the set of $k$-tuples $(\eta_1, \ldots, \eta_k)$ such that  $\eta_1, \ldots, \eta_k$ are $k$ disjoint geodesics (except  at $z$) respectively from $z^0_i$ to $z$ for $1\le i \le k$. 
For $0\le i\le k-1$, suppose that we have defined $V_i$, let $V_{i+1}$ be  the set of  $k$-tuples $(\eta_1, \ldots, \eta_k)$  in $V_i$ such that $\eta_{i+1}$ is the left-most geodesic among all $\wt \eta_{i+1}$ for $(\wt \eta_1, \ldots, \wt \eta_k) \in V_i$. The sets $V_1, \ldots, V_{k}$ are non empty by the argument of extracting a subsequence and using strong confluence. Moreover, the set $V_k$ contains a unique $k$-tuple $(\eta_1, \ldots, \eta_k)$, which is our final choice of $k$ disjoint (except at $z$) geodesics to $z$.

For each $t\in[d(x,y)-\tau_1, d(x,y)-\tau_{m_0}]$ and $1\le j\le k$, let $u_{t, j}$ (resp.\ $v_{t,j}$) be the point in $\eta_{j} \cap \CL_t$ which is the closest (resp.\ furthest) to $z$ (note that there can be more than one point in the intersection). 

Note that for every $1\le j \le k$ and $r\le t < d(x,y)$, for any two points $e_1, e_2 \in \eta_j \cap \CL_t$, we have
\begin{align}\label{eq:d_e1_e2}
d(e_1, e_2) \le 2 \eps.
\end{align}
This is because $|d(e_i, z) - (d(x,y) -t) | \le \eps$ for $i=1,2$, hence $|d(e_1, z) - d(e_2,z)| \le 2\eps$.
Since $e_1, e_2$ are on the same geodesic to $z$, we have~\eqref{eq:d_e1_e2}.
In particular, we have $d(u_{t,j}, v_{t,j} ) \le 2\eps.$

\noindent{\it Step 2: Definition of the fat layers.}
For each $\beta>0$, we call $2\le m \le m_0$ a \emph{$\beta$-fat layer} if 
$$\tau_m -\tau_{m-1} \ge \eps^{-\beta} Y_{\tau_{m-1}}^{1/2}.$$
By~\eqref{eqn:csbp_extinction_time}, we know that the probability for a $3/2$-CSBP started at $1$ to survive until time $t$ is  $O(t^{-2})$ as $t \to \infty$. Since $\tau_m -\tau_{m-1}$ is at most the extinction time, we deduce that the probability for $m$ to be a $\beta$-fat layer is $O(\eps^{2\beta})$.

Fix $\beta_0= (1+\sp{b}_0)(k-1)$. 
Let $D_0$ be the event that there are no $\beta_0$-fat layers among $2 \le m\le m_0$. Then for $\rho\in(\eps^{2\sp{b}_0}, \eps^{-2\sp{b}_0})$, by a union bound on the $m_0$ layers, we have
$$\p[D_0^c] = O(\eps^{2(1+\sp{b}_0)(k-1) +o(1)})= O((\epsilon/\rho^{1/2})^{k-1+o(1)}).$$ 
Fix  $K_0 \in\N$ and $H=(h_1, \ldots, h_{K_0}) \in \N^{K_0}$ with $2\le h_1< \dots, h_{K_0} \le m_0$. Let $D(H)$ be the event  that $h_1, \ldots, h_{K_0}$ are $\beta$-fat layers and all the other layers in $[1,  m_0]$ are not $\beta$-fat layers.
Then we have
$\p[D(H)] = O(\eps^{2 \beta K_0}).$
Fix $\beta\in(0, \beta_0)$ that we will adjust later. Fix $N_0= \lceil (1+\sp{b}_0)k/ (2 \beta) \rceil$ and let $H_{N_0}$ be the union of $D(H)$ such that $|H| > N_0$. Then
\begin{align*}
\p[H_{N_0}] \le \sum_{n={N_0}+1}^{m_0} m_0^n O(\eps^{2\beta n}) = O\big(\eps^{(1+\sp{b}_0)k+o(1)}\big) = O((\epsilon/\rho^{1/2})^{k-1+o(1)}).
\end{align*}
Therefore, to prove the result, it is enough to prove that
$$
\p[E\cap D_0 \cap H_{N_0}^c] = O((\epsilon/\rho^{1/2})^{k-1+o(1)}).
$$
Note that the number of choices of $H$ such that $|H| \le N$ is at most 
$\sum_{n=0}^{N_0} m_0^{n} = \eps^{o(1)}.$
Therefore, to prove the result, it is enough to show that for each $K_0\le N_0$ and each given $H=(h_1, \ldots, h_{K_0})$, we have
\begin{align}\label{eq:fat_layer}
\p[ E \cap D_0 \cap D(H)] = O((\epsilon/\rho^{1/2})^{k-1+o(1)}).
\end{align}

\noindent{\it Step 3: Definition of the crossing layers.}
We call $2 \le m \le m_0$ a \emph{crossing layer} if there exists $1\le j \le k$ such that the boundary length distance between $u_{\tau_m,j}$ and $v_{\tau_m,j}$ is at least $\eps^{2\sp{a}_m} Y_{\tau_m}$.
Fix $K_{1}\in\N$ and $I= (i_1,\ldots, i_{K_1}) \in \N^{K_1}$ with $2\le  i_1 < \cdots < i_{K_1}\le m_0$. Let $C(I)$ be the event that $E$ holds, 
and that $i_1, \ldots, i_{K_1}$ are crossing layers and all the other layers in $[1,  m_0]$ are not crossing layers.
Let us prove that 
\begin{align}\label{eq:pC(I)}
\p[C(I)] =  O\big(\eps^{2(1-\sp{b}_1-6\sp{a}) K_1/3+o(1)} \big).
\end{align}
Note that the exponent on $\eps$ is positive, because we have chosen $\sp{b}_1\in(0,1)$ and $\sp{a} \in(0,(1-\sp{b}_1)/6)$.

Suppose that $m$ is a crossing layer. By~\eqref{eq:d_e1_e2}, we have $d(u_{\tau_m,j}, v_{\tau_m,j}) \le 2\eps$.
Rescaling the distance of the metric band $\fb{y}{x}{d(x,y)-\tau_{m-1}} \setminus \fb{y}{x}{d(x,y)-\tau_{m+1}}$ by $Y_{\tau_{m-1}}^{-1/2}$ and then applying Lemma~\ref{lem:bad_band}, we deduce that the probability of $m$ being a crossing layer is
$$O\left(\big(\eps  Y_{\tau_{m-1}}^{-1/2} \big)^{4/3} \eps^{-4\sp{a}_m + o(1)}\right).$$
For each $0\le m \le m_0$, we have $Y_{\tau_m} \ge \eps^{2\sp{b}_1}/2$, hence the above probability is at most 
$$O(\eps ^{4(1-\sp{b}_1)/3 -4 \sp{a}_m+o(1)}) \le O(\eps ^{4(1-\sp{b}_1)/3 - 8 \sp{a} +o(1)}).$$

Note that there are at least $K_1/2$ layers among $i_1, \ldots, i_{K_1}$ which are all odd or all even. 
By the conditional independence of the metric bands $\fb{y}{x}{d(x,y)-\tau_{2n}} \setminus \fb{y}{x}{d(x,y)-\tau_{2n+2}}$, 
the probability of having $K_1/2$ odd crossing layers is $O(\eps ^{(4(1-\sp{b}_1)/3 - 8 \sp{a}) K_1/2 +o(1)})$.
Similarly,  by the conditional independence of  the metric bands $\fb{y}{x}{d(x,y)-\tau_{2n+1}} \setminus \fb{y}{x}{d(x,y)-\tau_{2n+3}}$, the probability of having $K_1/2$ even crossing layers is also $O(\eps ^{(4(1-\sp{b}_1)/3 - 8 \sp{a}) K_1/2 +o(1)})$.
We have thus proved~\eqref{eq:pC(I)}. 

Fix $N_1=\lceil (3/2) (1+\sp{b}_0) k/ (1-\sp{b}_1-6\sp{a}) \rceil$. Using~\eqref{eq:pC(I)}, arguing like at the end of Step 2, and combining with~\eqref{eq:fat_layer}, we deduce that to prove the proposition, it is enough to show that for each given $H$, $I$ with $|H| \le N_0$ and $|I| \le N_1$, we have
\begin{align}\label{eq:fat_crossing_layer}
\p[D_0 \cap D(H) \cap C(I) ] = O((\epsilon/\rho^{1/2})^{k-1+o(1)}).
\end{align}

\noindent{\it Step 4: Definition of the non-merging layers.}
Fix $K_0\le N_0$ and $K_1 \le N_1$. Fix $H= (h_1,\ldots, h_{K_0}) \in \N^{K_0}$ with $2\le h_1 < \cdots < h_{K_0} \le m_0$. Fix $I= (i_1,\ldots, i_{K_1}) \in \N^{K_1}$ with $2\le i_1 < \cdots < i_{K_1} \le m_0$. Suppose that we are working on the event $D_0 \cap D(H) \cap C(I)$.

For $1\le j\le k$, let $w_{0,j}$ be a point independently and uniformly chosen from the interval on $\partial\fb{y}{x}{d(x,y) -r}$ of length $2 \eps^{2\sp{a}_0} \rho$ centered at $\eta_j(0)$.
For $m\in \{0,\ldots,m_0-1\} \setminus (H\cup I)$, suppose that we have defined $w_{m,1},\ldots, w_{m,k}$. 
For each $1\le j \le k$ let $\eta_{m,j}$ be the geodesic from $w_{m,j}$ to $x$. For each $t>0$, let $e^m_{t,j}$ be the intersection between $\eta_{m,j}$ with $\CL_t$.
\begin{enumerate}[(i)]
\item\label{itm:define_non_merging} If $m+1$ is not a fat layer or a crossing layer (i.e., $m+1\not\in (H \cup I)$), then we have the following two possibilities.
If $e^{m}_{\tau_{m+1}, j} = v_{\tau_{m+1}, j}$ for every $j$, then let $w_{m+1, j}:=e^{m}_{\tau_{m+1}, j} $ for all $j$. Otherwise, we say that $m+1$ is a \emph{non-merging layer}.

\item\label{itm:define_marked_points} If $m+1$ is a non-merging layer, a fat layer or a crossing layer,  let $\ell \ge m+2$ be the first layer after $m+1$ which is not a fat layer or a crossing layer.
Let us define $w_{\ell, 1}, \ldots, w_{\ell,k}$ as follows. 
Fix $1\le j\le k$. Let $N\ge 0$ be the minimal number such that there exist $N$ points  $\wt w_1, \ldots, \wt w_N$ on $\partial\fb{y}{x}{d(x,y) - \tau_{\ell}}$ such that the $N+1$ intervals centered at $e^m_{\tau_{\ell},j}$ and $\wt w_1, \ldots, \wt w_N$ each with boundary length $\eps^{2 \sp{a}_{\ell}} Y_{\tau_\ell}$ cover the set of points on $\CL_{\tau_{\ell}}$ which have distance at most $\tau_{\ell} -\tau_m + \eps^{\sp{a}_m} Y_{\tau_m}^{1/2}$ to $w_{m,j}$. When there is more than one possible choice of $\wt w_1, \ldots, \wt w_N$, we choose them randomly in a way that we will specify later in Step 6.
Let $w_{\ell,j}$ be one of the $N+1$ points $e^m_{\tau_{\ell},j}$ and $\wt w_1, \ldots, \wt w_N$ which has boundary length distance at most $\eps^{2 \sp{a}_{\ell}}Y_{\tau_\ell}$ to $v_{\tau_{\ell},j}$.
\end{enumerate}

We also call a layer a \emph{bad layer} if it is either a fat layer, a crossing layer or a non-merging layer, otherwise we call it a \emph{good layer}.
By induction, we have defined $w_{m,1}, \ldots, w_{m,k}$ for all the good layers in $\{1, \ldots, m_0-1\}$. See Figure~\ref{fig:layers} for an illustration.  We remark that the definition of the non-merging layers depends on both the Brownian map instance $(\CS, d,\nu, x,y)$ and the randomness coming from the choice of the marked points. We also remark that the ``exploration process'' as defined above is not Markovian, because the choice of the marked points $w_{m,j}$ depends on the geodesics $\eta_j$ going to $z$ which are not determined by $\CF_m$.

\begin{figure}[h!]
\centering
\includegraphics[width=.7\textwidth]{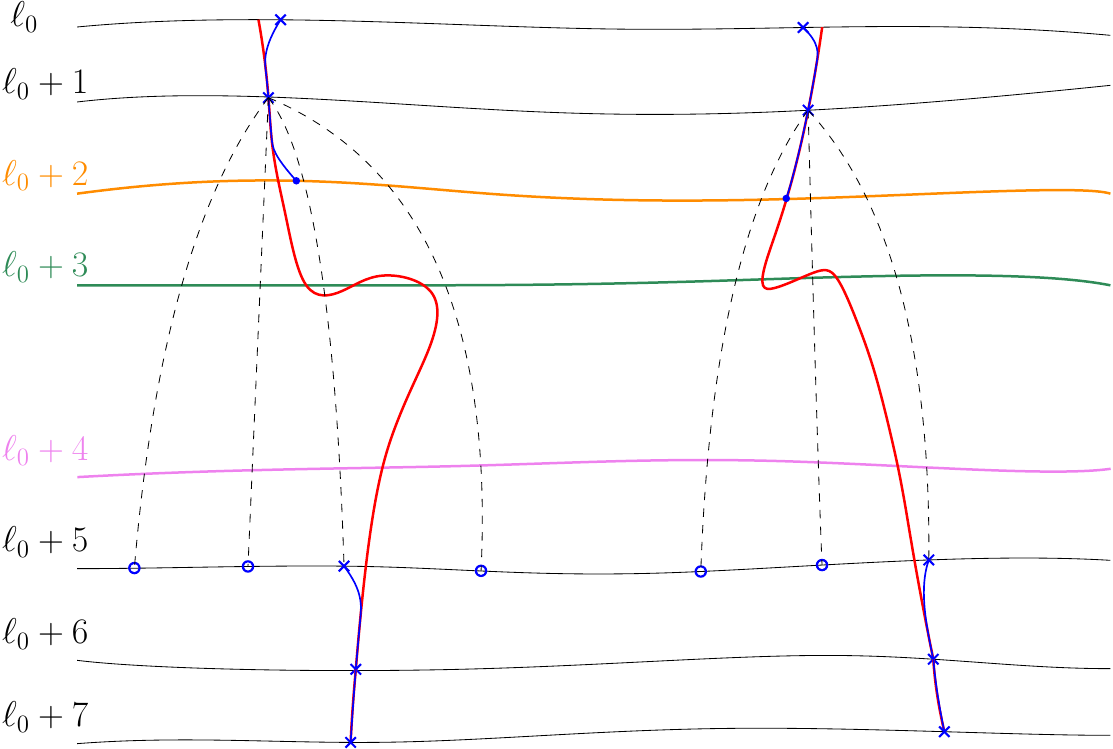}
\caption{We depict the case of $k=2$. The geodesics $\eta_1$ and $\eta_2$ are in red.  In this picture, the layer $\ell_0+2$ is a non-merging layer, the layer $\ell_0+3$ is a crossing layer, and the layer $\ell_0+4$ is a fat layer. 
For each good layer $m$, the points $w_{m,j}$ are drawn with a blue cross.
The points $w_{\ell_0+1, 1}$ and $w_{\ell_0+1, 2}$ have respectively $4$ and $3$ children on the layer $\ell_0+5$ and we have chosen for each of them the child which is the closest to $v_{\tau_{\ell_0+5},j}$.}
\label{fig:layers}
\end{figure}

Fix $K_2\in\N$ and $J= (j_1,\ldots, j_{K_2}) \in \N^{K_2}$  with $2\le  j_1 < \cdots < j_{K_2} \le m_0$.
Let $F(H, I,J)$ be the event that $D_0 \cap D(H) \cap C(I)$ holds and that $j_1, \ldots, j_{K_2}$ are non-merging layers and all the other layers are not non-merging layers.

\noindent{\it Step 5: Definition of a Markovian exploration process.}
Fix $K\in\N$ and $S= (s_1,\ldots, s_{K}) \in \N^{K}$ with $1 = s_1 < \cdots < s_K< s_{K+1} = m_0$. For each $1 \le n \le K$, we will randomly choose $k$ points $x_{s_n, 1}, \ldots, x_{s_n, k} \in \CL_{\tau_{s_n}}$ in a way which is independent from $ \fb{y}{x}{d(x,y) - \tau_{s_n}}$ conditionally on $\CF_{s_n}$.

Let $z_1, \ldots, z_k$ be $k$  points on $\CL_{\tau_0}$ chosen independently and uniformly from the boundary measure and then ordered to be counterclockwise. 
Assuming that we have defined points $z_J \in \CL_{\tau_{s_n}}$ indexed by elements of $\N^{n}$, we inductively define points in $ \CL_{\tau_{s_{n+1}}} $ indexed by elements of $\N^{n+1}$ as follows.  Suppose that $J \in \N^{n}$ and $z_J \in \CL_{\tau_{s_n}}$. For $i \in \N$, we then let $Ji \in \N^{n+1}$ be given by concatenating $J$ and $i$.  
Let $\eta_{z_J}$ be the unique geodesic from $z_J$ to $x$ and let $z_{J0}$ be where it hits $\CL_{\tau_{s_{n+1}}}$.
We then let $z_{J1},\ldots,z_{J\ell}$ be a minimal collection so that the intervals of boundary length $\epsilon^{2\sp{a}_{s_{n+1}}} Y_{\tau_{s_{n+1}}}$ on $ \CL_{\tau_{s_{n+1}}}$ centered at the points $z_{J0}, z_{J1},\ldots,z_{J\ell}$ cover the set of points on $ \CL_{\tau_{s_{n+1}}}$ whose distance to $z_J$ with respect to the interior-internal metric of the band $\fb{y}{x}{d(x,y) -\tau_{s_n}} \setminus \fb{y}{x}{d(x,y) -\tau_{s_{n+1}}}$ is at most $\tau_{s_{n+1}} -\tau_{s_n} + \eps^{\sp{a}_{s_n}} Y_{\tau_{s_n}}^{1/2}$.  The choice of $z_{J1},\ldots,z_{J\ell}$ can be made in a way which is independent of $\fb{y}{x}{d(x,y) -\tau_{s_{n+1}}}$ conditionally on $\CF_{s_{n+1}}$.
We assume that $z_{J0}, z_{J1},\ldots,z_{Jl}$ are ordered counterclockwise and we call them the children of $z_J$.  

We let $x_{0,1}=z_1,\ldots,x_{0,k} = z_k$.  Given that we have defined $x_{s_n,1},\ldots,x_{s_n,k}$ for some $n\in\N_0$, we let $x_{s_{n+1},j}$ for $1 \leq j \leq k$ be one of the children of $x_{s_n,j}$ chosen uniformly at random and independently of everything else among all of the children of $x_{s_n,j}$. It is clear that for each $1\le n \le K$, the points $x_{s_n, 1}, \ldots, x_{s_n, k} $ are conditionally independent of $\fb{y}{x}{d(x,y)-\tau_{s_n}}$ given $\CF_{s_n}$.

Suppose that we are on the event $D_0 \cap D(H)$ for some given $H$ with $|H|\le N_0$. Then for each $1\le n \le K$, the width of the band  $\fb{y}{x}{d(x,y) -\tau_{s_n}} \setminus \fb{y}{x}{d(x,y) -\tau_{s_{n+1}}}$ is at most $Y_{\tau_{s_n}}^{1/2} (\eps^{-\beta_0}N_0 + \eps^{-\beta} (s_{n+1} -s_n)) \le 2 Y_{\tau_{s_n}}^{1/2} \eps^{-\beta_0}N_0$.
If we rescale the metric of this band by $\eps^{\beta_0} Y_{\tau_{s_n}}^{-1/2} N_0^{-1}$ and apply Lemma~\ref{lem:metric_band_epsilon_point2}, then
for each $J\in\N^n$, the event that $z_J$ has $L >1$ children has probability 
$$O\big(\eps^{(L-1)(\sp{a}_{s_n} -\sp{a}_{s_{n+1}}) -\sp{a}_{s_{n+1}} -\beta_0+o(1)} \big).$$  
Let $L_n= \lceil \left((1+\sp{b}_0) k +2\sp{a} + \beta_0\right)/(\sp{a}_{s_{n+1}} -\sp{a}_{s_n}) \rceil +1$. Let $\CL$ be the event that for each $1 \le n \le K$ and  $J\in\N^{n}$, $z_J$ has at most $L_n$ children. Then we have 
\begin{align*}
\p[D_0 \cap D(H) \cap \CL]=1- O(\eps^{(1+\sp{b}_0)k+o(1)}) =1- O((\epsilon/\rho^{1/2})^{k-1+o(1)}).
\end{align*}
Combined with~\eqref{eq:fat_crossing_layer}, we deduce that to prove the proposition, it is enough to show that for each given $H$, $I$ with $|H| \le N_0$ and $|I| \le N_1$, we have 
\begin{align}\label{eq:fat_crossing_layer_L}
\p[ D_0 \cap D(H) \cap \CL \cap C(I) ] = O((\epsilon/\rho^{1/2})^{k-1+o(1)}).
\end{align}

\noindent{\it Step 6: Coupling of the exploration processes in Steps 4 and 5.}
Fix $K_0, K_1, K_2\in\N$ and $H=(h_1, \ldots, h_{K_0})$, $I= (i_1,\ldots, i_{K_1}) \in \N^{K_1}$, $J= (j_1,\ldots, j_{K_2}) \in \N^{K_2}$  with $1\le h_1 < \cdots < h_{K_0} \le m_0$, $1\le i_1 < \cdots < i_{K_1} \le m_0$,  $1\le  j_1 < \cdots < j_{K_2} \le m_0$ such that the following conditions hold, which ensures that $F(H,I,J)$ as defined in Step 4 has positive probability.

\begin{enumerate}
\item\label{itm:HI_not_J} $(H\cup I)\cap J=\emptyset$.  This is because according to (\ref{itm:define_non_merging}), we only define a layer to be a non-merging layer if it is not a fat layer or a crossing layer.
\item  For any $1\le n \le K_0$  we have $h_{n}+1 \not\in J$.  For any $1\le n \le K_1$ we have $i_n+1 \not\in J$. For any $1\le n \le K_2$  we have $j_{n}+1 < j_{n+1}$.
These conditions are due to (\ref{itm:define_marked_points}). If $m$ is a non-merging layer, a fat layer or a crossing layer, then we put marked points $w_{\ell,1}, \ldots, w_{\ell,k}$ on the first layer $\ell\ge m+1$ which is not a fat layer or a crossing layer. In particular, $\ell$ is a good layer and the first non-merging layer after $m$ must be equal or after $\ell+1$ which is at least $m+2$.
\end{enumerate}
Given $H$ and $I$ as above, the union of $F(H, I, J)$ over all possible choices of $J$ satisfying the above conditions covers the event $D_0 \cap D(H) \cap C(I)$.

Let $S$ be the set of $m\in \{1,\ldots,m_0\} \setminus (H \cup I \cup J)$ such that either $m-1$ or $m+1$ is in $H\cup I\cup J$. Note that the layers  $1,\ldots,m_0$ are grouped into consecutive good layers and bad layers, and $S$ is the set of layers which are either the first or the last layer in each group of good layers.
Let $S_1$ be the set of $s\in S$ such that $s+1\in H \cup  I\cup J$. Let $S_2 := S\setminus S_1$. Then each layer in $S_2$ is the first layer in a group of at least two consecutive good layers.
We also consider $S$ as an ordered set $(s_1, \ldots, s_{K}) \in \N^K$ with $s_1<\cdots < s_K$. Note that we always have $s_1=0$ and $K \le 2 K_0+ 2 K_1 + 2 K_2$.

Now, we will show that, it is possible to couple the ``exploration process'' in Step 4 with the Markovian exploration process in Step 5, so that on a certain event $G(H, I, J)$, we have $x_{m,j}=w_{m,j}$  for all $m\in S$ and $1\le j \le k$.
The event $G(H, I, J)$ is also defined so that $F(H, I, J) \cap \CL$ holds on $G(H, I, J)$, and for some constant $c_1>0$ depending only on $k$, $\sp{b}_0$, $K_2$ and $a$, we have
\begin{align}\label{eq:GCK2}
\p[ G(H, I, J)  \mid F(H, I, J) \cap\CL ] \ge  c_1 (\log_2 \eps^{-1} )^{-2k K} \eps^{2k \sp{a}_0}.
\end{align}
Let us now describe the coupling and the event $G(H, I, J)$, and prove~\eqref{eq:GCK2}.
Let us first suppose that on $G(H, I, J)$, the event $F(H, I, J) \cap \CL$ holds, and the boundary length distance between $x_{0,j}$ and $v_{\tau_{0}, j}$ is at most $\epsilon^{2\sp{a}_0} \rho$ for all $1\le j \le k$. The latter event has probability $2^k\eps^{2 k \sp{a}_0} $. Conditionally on this event, we can choose $x_{0,j} = w_{0,j}$ for all $1\le j \le k$.
For $n\in\N$, assume that we have chosen $x_{s_n, j} = w_{s_n, j}$ for all $1\le j \le k$, let us describe the coupling and the event $G(H, I, J)$ in the layer $s_{n+1}$.
\begin{itemize}
\item If $s_n \in S_1$, then $s_{n+1}$ is the first layer after $s_n$ which is not in $H\cup I\cup J$.
For each $j$, let us choose the set of children for $w_{s_n,j}$ on $\CL_{s_{n+1}}$ as defined in Step 4 to be the same as the set of children for $x_{s_n, j}$ on $\CL_{s_{n+1}}$ as defined in Step 5. Moreover, we put into the definition of $G(H, I, J)$ that $x_{s_{n+1},j} = w_{s_{n+1},j}$ for each $1\le j \le k$. 
Since on the event $F(H, I, J) \cap \CL$ there are at most $L_{n}$ children of $x_{s_n, j}$, the conditional probability of choosing $x_{s_{n+1},j} = w_{s_{n+1},j}$ for each $1\le j \le k$ is at least $L_{n}^{-k}$.

\item Otherwise, $s_n\in S_2$ and $s_{n+1}$ is the first layer after $s_n$ which is in $S_1$. In particular, all the layers $\{s_n,\ldots, s_{n+1}\}$ are good layers.  This implies that $w_{s_{n+1}, j}$ is just the point where the geodesic from $w_{s_n, j}$ to $x$ hits $\CL_{\tau_{s_{n+1}}}$, namely $w_{s_{n+1}, j}=v_{\tau_{s_{n+1}},j}$. We put into the definition of  $G(H, I, J)$ that $x_{s_{n+1}, j} =  w_{s_{n+1}, j}$ for each $1\le j \le k$. Since we are on the event $F(H, I, J) \cap \CL$ and $w_{s_{n+1}, j}$ is by definition one of the children of $x_{s_n,j}$, the conditional probability of choosing $x_{s_{n+1}, j}= w_{s_{n+1}, j}$ for each $j$ is at least $L_n^{-k}$.
\end{itemize}
By induction, we have defined the event $G(H, I, J)$ and the coupling. We remark that on $G(H, I, J)$, the boundary length distance between $x_{s_{n},j}$ and $v_{\tau_{s_{n}},j}$ is at most $\epsilon^{2\sp{a}_{s_{n}}} Y_{\tau_{s_{n}}}$ for each $1\le j \le k$ and $1\le n\le K$. Moreover, by the conditional independence across the layers, we have
\[ \p[ G(H, I, J)  \mid F(H, I, J) \cap\CL ] \ge (L_1\cdots L_{K})^{-k}  \eps^{2 k \sp{a}_0 }.\]
Noting that $\sp{a}_{m+1} - \sp{a}_m =a(m+1)^{-1} (m+2)^{-1}$ and that $s_n \le m_0$, we deduce that there exist constants $C_1, C_2>0$ such that for each $1\le n\le K$, 
\[ L_n \le  C_1 (m_0+1) (m_0+2) \le C_2 (\log_2 \eps^{-1} )^2.\] 
Combined with the previous inequality, we have proved~\eqref{eq:GCK2}.

From~\eqref{eq:GCK2}, we can deduce  that  for all $\eps$ small enough, we have
\begin{align}\label{eq:pfij}
\p[F(H, I, J) \cap \CL ] \le \eps^{-2 \sp{a}_0 k+o(1)}  \, \p[G(H, I, J)].
\end{align}

\noindent{\it Step 7: Most layers are good layers.}  Fix $H, I, J$ as in Step 6. We aim to obtain an upper bound on the probability of $G(H, I, J)$, by exploring from $\partial \fb{y}{x}{d(x,y) -r}$ to $x$. This will in turn imply an upper bound on $\p[F(H, I, J) \cap \CL]$ by~\eqref{eq:pfij}. We have already shown in Steps 2 and 3 that off an event with negligible probability, there are at most $N_0$ fat layers and $N_1$ crossing layers. In this step, by estimating the probability $F(H, I, J) \cap \CL $ when $|J|$ is large, we will show that there are also at most a constant number of non-merging layers off an event with negligible probability. All together, we will show that with high probability all but a constant number of layers are good layers.

Let us first show that, for any $m$ such that $m+1\in J$, conditionally on $\CF_{m}$ and on $x_{m, j} =w_{m,j}$ for each $j$, the probability that $m+1$ is a non-merging layer is $O(\eps^{\sp{a}_{m} -\beta +o(1)}).$
Recall that we denote by $\eta_{m,j}$ the geodesic from $w_{m,j}$ (hence also $x_{m, j}$) to $x$, and $e^{m}_{t,j}$ is the intersection of $\eta_{m,j}$ with $\CL_t$. We also let $\wt w_{m,j}$ be the intersection between $\eta_{m,j}$ and $\CL_{m+1}$.
If $m+1$ is a non-merging layer, then there exists  $1\le j \le k$ such that $\wt w_{m,j} \not= v_{\tau_{m+1}, j}$. 
We must be in one of the following two situations.

\begin{figure}[h]
\centering
\includegraphics[width=.52\textwidth]{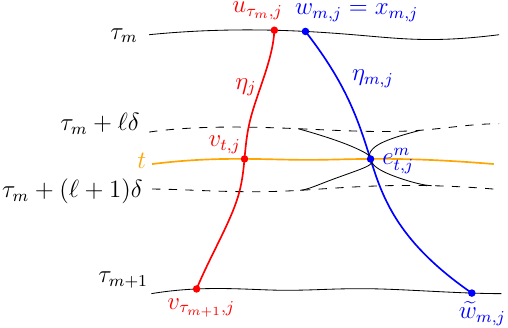}
\caption{One possible configuration when $m+1$ is a non-merging layer, corresponding to the case (\ref{itm:non_merging1}).}
\label{fig:bad_layer1}
\end{figure}

\begin{enumerate}[(1)]
\item\label{itm:non_merging1} The part of $\eta_{m,j}$ from $x_{m, j}$ to $\CL_{m+1}$ is disjoint from $\eta_j$. See Figure~\ref{fig:bad_layer1}.
Fix $0<\sp{b}_2<\sp{b}_3<\sp{b}_4<\sp{a}$ and let $\delta:=\eps^{\sp{b}_3} Y_{\tau_{m}}^{1/2}$. Consider the metric bands 
\[\CB^n_\ell:= \fb{y}{x}{d(x,y) -\tau_{m}- \ell \delta} \setminus \fb{y}{x}{d(x,y) - \tau_{m}-(\ell+1) \delta}\]
for $0\le \ell \le L-1$ where $L:=\lfloor \eps^{-\sp{b}_2} \rfloor$.
By Lemma~\ref{lem:CSBP_stopping}, we know  that off an event whose probability decays faster than any power of $\eps$, we have $\tau_m + L \delta <\tau_{m+1}$, so that the bands $\CB^n_\ell$ are all contained in $\CB_m$.

\smallskip

\noindent
By Lemmas~\ref{lem:X_prob2} and~\ref{lem:good_band_map}, there exist $s_0, s_1, p_1>0$ such that with probability $1-(1-p_1)^L$, there exists $\ell$ such that there is an \X\  along $\eta_{m,j}$ contained in the band $\CB^n_\ell$ which is $(s_0 \delta, s_1 \delta,  2 \delta)$-good.
Let us show that, on the event of this situation (i.e., the part of $\eta_{m,j}$ from $x_{m, j}$ to $\CL_{m+1}$ is disjoint from $\eta_j$) and off an event with probability decaying to $0$ faster than any power of $\eps$, there exists $t\in[\tau_m+ \ell\delta, \tau_m+(\ell+1)\delta]$ such that the boundary distance between $v_{t,j}$ and $e^m_{t,j}$ along $\CL_t$ is greater than $ \eps^{2\sp{b}_4} Y_{\tau_{m}}$.
Suppose that it is not the case, so that there exists $K_4>0$ such that the event $\wt E$ defined in the following occurs with probability at least $\eps^{K_4}$:
\begin{itemize}
\item the part of $\eta_{m,j}$ from $x_{m, j}$ to $\CL_{m+1}$ is disjoint from $\eta_j$
\item and the boundary distance between $v_{t,j}$ and $e^m_{t,j}$ along $\CL_t$ is at most $ \eps^{2\sp{b}_4} Y_{\tau_{m}}$ for all $t\in[\tau_m+ \ell\delta, \tau_m+(\ell+1)\delta]$.
\end{itemize}
For each $k\in \N$ and $\iota\in(0, \sp{b}_4)$, let $E_k$ be the event that there exist two points $u, v$ on $\partial \fb{y}{x}{k\eps}$ such that the boundary length between $u$ and $v$ is at most $\eps^{2 \sp{b}_4 -\iota}$ and $d(u,v) \ge \eps^{\sp{b}_4 -\iota}$. By Lemma~\ref{lem:boundary_length_distance}, we know that $\p(E_k)$ decays to $0$ faster than any power of $\eps$.
Since $\tau_m \le \inf \{s\ge 0, Y_s=0\}$, by \eqref{eqn:csbp_extinction_time} we know that for every $K_2>0$, there exist $K_1>0$ such that $\p(\tau_m + (\ell+1) \delta \ge \eps^{-K_1}) =o (\eps^{K_2})$.
Let $K_3> K_1$ be an arbitrary number. Then we have
\begin{align*}
&\sum_{k\in\N} \p ( k \eps \in [\tau_m+ \ell\delta, \tau_m+(\ell+1)\delta], E_k)\\
\le &  \sum_{k=1}^{\lceil \eps^{- K_1 -1} \rceil}   \p ( k \eps \in [\tau_m+ \ell\delta, \tau_m+(\ell+1)\delta], E_k) +  (2\delta/\eps)\p ( \tau_m+(\ell+1)\delta \ge \eps^{- K_1})\\
\le &  \sum_{k=1}^{\lceil \eps^{- K_1 -1} \rceil}  \p ( k \eps \in [\tau_m+ \ell\delta, \tau_m+(\ell+1)\delta])^{1/2} \p(E_k)^{1/2} + o(\eps^{K_2-1})\\
\le & o(\eps^{K_3})  \sum_{k=1}^{\lceil \eps^{- K_1 -1} \rceil}   \p ( k \eps \in [\tau_m+ \ell\delta, \tau_m+(\ell+1)\delta])^{1/2} + o(\eps^{K_2-1})\\
\le & o(\eps^{K_3 -K_1 -1})  + o(\eps^{K_2-1}).
\end{align*}
Since $K_2, K_3$ are arbitrary, we know that off an event with probability decaying to $0$ faster than any power of $\eps$, we have that for each $k\in\N$ with $k \eps \in[\tau_m+ \ell\delta, \tau_m+(\ell+1)\delta]$, the event $E_k$ does not occur. 
It follows that on $\wt E$ minus an event with probability decaying to $0$ faster than any power of $\eps$, for each $k\in\N$ with $k \eps \in[\tau_m+ \ell\delta, \tau_m+(\ell+1)\delta]$, if we let $t_k= k\eps$, then
$$d(v_{t_k,j}, e^m_{t_k,j}) \le \eps^{\sp{b}_4+o(1)} Y_{\tau_m}^{1/2} < s_0 \eps^{\sp{b}_3} Y_{\tau_{m}}^{1/2} /2 =s_0 \delta/2.$$
Combined with~\eqref{eq:d_e1_e2}, it follows that on this event $\eta_j \cap \CB^n_\ell$ is contained in the $s_0 \delta$-neighborhood of $\eta_{m,j} \cap \CB^n_\ell$. Therefore, on this event, the two branches of \X\ which is at the side of $\eta_j$ must both intersect $\eta_j$, which implies that $\eta_j$ would intersect $\eta_{m,j}$ inside $\CB^n_\ell$, which contradicts the definition of $\wt E$.
\smallskip

\noindent
Consequently, on the event of this situation and off an event with probability decaying faster than any power of $\eps$, in the metric band $\fb{y}{x}{d(x,y) -\tau_m} \setminus \fb{y}{x}{d(x,y) -\tau_m - L \delta}$, there exists $\tau_{m}<t < \tau_{m} + L \delta$ such that the boundary distance between $e^{m}_{t,j}$ and $v_{t, j}$ is greater than $ \eps^{2\sp{b}_4} Y_{\tau_{m}}$.
On the other hand, the boundary length along $\partial \fb{y}{x}{d(x,y) -\tau_{m}} $ between $x_{m,j}$ and $v_{\tau_m,j}$ is at most $\eps^{2\sp{a}_m} Y_{\tau_m}$ by construction.
By Lemma~\ref{lem:band_boundary_distance}, in the band $\CB:=\fb{y}{x}{d(x,y) -\tau_{m}} \setminus \fb{y}{x}{d(x,y)  -t}$, the distance with respect to the interior-internal metric of the band $d_\CB$ between $x_{m,j}$ and $v_{\tau_m,j}$ is at most $\eps^{\sp{a}_m} Y_{\tau_{m}}^{1/2}$, off an event with probability decaying faster than any power of $\eps$. It follows that 
$$d_\CB(v_{t,j}, x_{m,j}) \le d_\CB(v_{t,j}, v_{\tau_m,j}) + d_\CB(v_{\tau_m,j}, x_{m,j}) \le t-\tau_{m} + \eps^{\sp{a}_{m}} Y_{\tau_{m}}^{1/2}.$$ We also know that $d_\CB (e^{m}_{t, j}, x_{m, j})=t-\tau_{m}$. 
If we rescale the band $\fb{y}{x}{d(x,y) -\tau_m} \setminus \fb{y}{x}{d(x,y) -\tau_m - L \delta}$ by $(L\delta)^{-1}$ and apply Lemma~\ref{lem:metric_band_epsilon_point2}, we can deduce that the event of this situation happens with probability 
$O(\eps^{\sp{a}_{m}-2\sp{b}_4 + \sp{b}_3 - \sp{b}_2 + o(1)})$. 
Since $\sp{b}_2, \sp{b}_3, \sp{b}_4$ can be arbitrarily close to $0$, we have shown that this situation occurs with probability  $O(\eps^{\sp{a}_{m} + o(1)})$.

\item Otherwise, the part of $\eta_{m,j}$ from $x_{m, j}$ to $\CL_{m+1}$ intersects $\eta_j$, but then leaves $\eta_j$. Due to Proposition~\ref{prop:no_geodesic_bump}, the part of $\eta_{m,j}$ from $\wt w_{m,j}$ to $\CL_{m+2}$ is disjoint from $\eta_j$. We can then consider the metric bands  
\[ \fb{y}{x}{d(x,y) -\tau_{m+1}- \ell \delta} \setminus \fb{y}{x}{d(x,y) - \tau_{m+1}-(\ell+1) \delta}\]
for $0\le \ell \le L-1$. 
Similarly, off an event with probability decaying faster than any power of $\eps$, there exists an \X\ along $\eta_{m,j}$ in one of those metric bands. Arguing like in the previous case, this implies that the boundary length distance between $\eta_j$ and $\eta_{m,j}$ must get $\eps^{2\sp{b}_4} Y_{\tau_m}$ apart at some layer inside the metric band  $\CB_{m+1}$.
\smallskip

\noindent 
Note that the width of the metric band $\fb{y}{x}{d(x,y) -\tau_{m}} \setminus \fb{y}{x}{d(x,y) - \tau_{m+1} - L \delta}$ is at most $\eps^{-\beta} Y_{\tau_m}^{1/2} + L \delta$, because $m+1 \in J$ is not a fat layer, due to~\eqref{itm:HI_not_J} in Step 6.
If we rescale this metric band by $(\eps^{-\beta} Y_{\tau_m}^{1/2} + L \delta)^{-1}$ and apply Lemma~\ref{lem:metric_band_epsilon_point2}, we can deduce that the event of this situation happens with probability $O(\eps^{\sp{a}_{m} -2 \sp{b}_4 - \max(\beta, \sp{b}_2 - \sp{b}_3)+ o(1)})$. Since $\sp{b}_2, \sp{b}_3, \sp{b}_4$ can be arbitrarily close to $0$, we have shown that this situation occurs with probability  $O(\eps^{\sp{a}_{m} -\beta + o(1)})$.
\end{enumerate}
Altogether, we have proved that conditionally on $\CF_{m}$ and on $x_{m, j} =w_{m,j}$ for each $j$, the probability that $m+1$ is a non-merging layer is $O(\eps^{\sp{a}_{m} - \beta +o(1)}).$
From now on, we fix $\beta= \min(\beta_0, \sp{a}/2)$, so that the preceding probability is $O(\eps^{\sp{a}/2})$.

By conditional independence across the layers, this implies that 
\[ \p[ G(H, I, J)] =O(\eps^{K_2 \sp{a}/2 }).\]
By~\eqref{eq:pfij}, we have
\begin{align*}
\p[ F(H, I, J) \cap \CL] =  O\big(\eps^{-2\sp{a}_0k+ K_2\sp{a}/2 +o(1)} \big).
\end{align*}
Recall that the union of $F(H, I, J)$ over all possible choices of $J$ satisfying the conditions in Step 6 covers the event $D_0 \cap D(H) \cap C(I)$. Let $N_2 =\lceil 2\sp{a}_0 k+ 2(1+\sp{b}_0) k/\sp{a} \rceil$. Arguing like at the end of Step 2, and combining with~\eqref{eq:fat_crossing_layer_L}, we deduce that to prove the proposition, it is enough to show that for each given $H, I, J$ with $|H| \le N_0$, $|I|\le N_1$ and $|J| \le N_2$, we have
\begin{align}\label{eq:enough0}
\p[ F(H, I, J) \cap \CL ] = O((\epsilon/\rho^{1/2})^{k-1+o(1)}).
\end{align}

\noindent{\it Step 8: Conclusion of the proof.}
Suppose that we have fixed $H, I, J$ as in Step 6 with the additional condition that $|H| \le N_0$, $|I|\le N_1$ and $|J| \le N_2$. Our goal is to show
\begin{align}\label{eq:enough}
\p[ F(H, I, J) ] = O((\epsilon/\rho^{1/2})^{k-1+o(1)}).
\end{align}
which implies~\eqref{eq:enough0}. Suppose that we are working on the event $F(H, I, J)$.
Let $S=(s_1,\ldots, s_K)$ and $S_1, S_2$ be defined as in Step 6.  We denote by $\eta_{s_n, j}$ the geodesic  from $x_{s_n,j}$ to $x$, and let $e^{s_n}_{t,j}$ be where  $\eta_{s_n, j}$ hits $\CL_{t}$ for $t\ge \tau_{s_n}$. We also take the convention that $e^{s_n}_{t, k+1} = e^{s_n}_{t, 1}$.

For each $n$ such that $s_n\in S_2$, all the layers in $\{s_n,\ldots, s_{n+1}\}$ are good layers.  Therefore for each $j$, the geodesic $\eta_{s_n, j}$ merges with $\eta_j$ in the band $\fb{y}{x}{d(x,y) - \tau_{s_n+1} } \setminus \fb{y}{x}{d(x,y) - \tau_{s_{n+1}}}$, hence the geodesics $\eta_{s_n, 1}, \ldots, \eta_{s_n, k}$ are disjoint from each other in this band as well as in the larger band $\fb{y}{x}{d(x,y) - \tau_{s_n} } \setminus \fb{y}{x}{d(x,y) - \tau_{s_{n+1}}}$.  Let $Z_s^{s_n,j}$ be the boundary length of the counterclockwise segment of $\CL_{\tau_{s_n} +s}$ from $e^{s_n}_{\tau_{s_n} +s,j}$ to $e^{s_n}_{\tau_{s_n}+ s, j+1}$.  Then we know that the processes $Z^{s_n,1}, \ldots, Z^{s_n,k}$ evolve as independent $3/2$-stable CSBPs.
Let $\Delta_n = \tau_{s_{n+1}} - \tau_{s_n}$ and let $A_n$ be the event that none of the $Z^{s_n, 1},\ldots,Z^{s_n,k}$ hit $0$ in the time interval $[0,\Delta_n]$.  Then we have $F(H, I, J)\subseteq A_n$.
This implies that  
\begin{align}\label{eq:fij_subset}
F(H, I, J) \subseteq \bigcap_{n\in\N, s_n\in S_2} A_n.
\end{align}

On the event $A_n$, we have
\[ \sum_{j=1}^k Z_{\Delta_n}^{s_n, j} = Y_{\tau_{s_{n+1}}}= 2^{-s_{n+1}} \rho .\]
In particular,
\[ \max_{1 \leq j \leq k} Z_{\Delta_n}^{s_n, j} \leq   2^{-s_{n+1}} \rho.\]
It follows from the scaling property for a $3/2$-stable CSBP that the conditional probability that $Z^{s_n, 1},\ldots,Z^{s_n,k}$ all hit $0$ for the first time in the next $\rho^{1/2}  2^{-s_{n+1}/2}$ units of time after $\Delta_n$ is at least some $p_1 > 0$.  On the other hand, the probability that $Z^{s_n, 1},\ldots,Z^{s_n,k}$ all hit $0$ within $\rho^{1/2}  2^{-s_{n+1}/2}$ units of time of each other is at most $c_0 2^{-(s_{n+1}-s_n ) (k-1)/2}$ for some $c_0>0$, due to~\eqref{eqn:csbp_extinction_time} and the independence of these processes. It follows that 
\[ \p[ A_n] \le c_0 2^{-(s_{n+1}-s_n ) (k-1)/2}.\] 
Due to conditional independence across the layers, we have
\[ \p \bigg[ \bigcap_{n\in\N, s_n\in S_2} A_n \bigg] \le \prod_{n\in S_2} c_0 2^{-(s_{n+1}-s_n )  (k-1)/2}.\]
Since we have chosen $H, I, J$ so that $|H|\le N_0, |I|\le N_1, |J|\le N_2$, 
we have 
$$\sum_{n\in S_2} s_{n+1} -s_n \ge m_0 - 2(N_0 + N_1 + N_2) \ge \log_2 (\eps^{-2\sp{b}_1} \rho) -2(N_0 + N_1 + N_2).$$ This implies that
\begin{align*}
\p \bigg[ \bigcap_{n\in\N, s_n\in S_2} A_n \bigg]  = O((\eps^{\sp{b}_1}/\rho^{1/2})^{(k-1)}).
\end{align*}
Note that we can choose $\sp{b}_1$ to be arbitrarily close to $1$. Combined with~\eqref{eq:fij_subset}, we can deduce~\eqref{eq:enough} which completes the proof.\qed

\begin{remark}\label{remark}
Proposition~\ref{prop:disjoint_geodesics} is a main ingredient in Section~\ref{sec:dimension} to prove the dimension upper bounds for $k$-star points. 
We remark that similar techniques can be used to obtain the second moment estimate which is key to obtaining the dimension lower bounds for $k$-star points.

More precisely, for any two typical points $x_1, x_2$,
we need to obtain an upper bound on the probability that for $i=1,2$ there are $k$ geodesics going from the boundary of a certain metric ball centred at $x_i$ to $x_i$ which are disjoint except at $x_i$.
When $x_1$ and $x_2$ are far away from each other, what happens in the metric balls centred at $x_1$ and $x_2$ are almost independent. When $d(x_1, x_2)=\delta$ is small, similar techniques as in the proof of Proposition~\ref{prop:disjoint_geodesics} will allow us to show that the geodesics to $x_1$ and $x_2$ coincide at most layers, before reaching distance $O(\delta)$ of $x_1$ and $x_2$, and that afterwards the geodesics to $x_1$ and $x_2$ evolve independently.
\end{remark}

\section{Completion of proofs of dimension upper bounds}\label{sec:dimension}

In Section~\ref{subsec:n_of_geodesics} we will obtain the dimension upper bound for the set of points from which a given number of geodesics emanate which are otherwise disjoint (Theorem~\ref{thm:maximum_geodesics}). Next, in Section~\ref{subsec:n_of_geodesics2}, we will obtain the dimension upper bound for the pairs of points which are connected by a collection of geodesics with a given topology (Theorem~\ref{thm:finite_number_of_geodesics}), as well as those connected by a given number of geodesics (Theorem~\ref{thm:number_of_geodesics}).

\subsection{Number of disjoint geodesics from a point}  
\label{subsec:n_of_geodesics}
The goal of this section is to complete the proof of Theorem~\ref{thm:maximum_geodesics}. The main input in its proof is Proposition~\ref{prop:disjoint_geodesics}. We first record two lemmas that will be used in the proof.

\begin{lemma}
\label{lem:area_in_complement}
For $\bminflaw$ a.e.\ instance of $(\CS,d,\nu,x,y)$ the following is true.  There exists $r_0 > 0$ so that for all $r \in (0,r_0)$ and $z \in \CS$ the following is true.  Let $U$ be the connected component of $\CS \setminus B(z,r)$ with the largest diameter.  Then $\nu(U) \geq \nu(\CS)/2$.
\end{lemma}
\begin{proof}
Suppose that $(\CS,d,\nu,x,y)$ has distribution $\bminflaw$.  Fix $\sp{a} > 0$.  Lemma~\ref{lem:bm_volume_estimates} implies that there exists $\delta_0 > 0$ so that for all $\delta \in (0,\delta_0)$ and $z \in \CS$ we have that $\nu(B(z,\delta)) \leq \delta^{4-\sp{a}}$.  Fix $\delta \in (0,\delta_0)$ sufficiently small so that $\delta^{4-\sp{a}} \leq \nu(\CS)/2$.  Then for every $z \in \CS$ we have that $\nu(B(z,\delta)) \leq \nu(\CS)/2$.  Lemma~\ref{lem:diameter_of_complement} implies that there exists $r_0 > 0$ so that for every $r \in (0,r_0)$ we have that $\CS \setminus B(z,r)$ has at most one component of diameter at least $\delta/4$.  We may assume without loss of generality that $r_0 \leq \delta/4$.  All of the other components of $\CS \setminus B(z,r)$ are contained in $B(z,r+\delta/2) \subseteq B(z,\delta)$.  By our choice of $\delta$, it thus follows that the component of $\CS \setminus B(z,r)$ with the largest diameter has $\nu$-area at least $\nu(\CS)/2$.
\end{proof}

\begin{lemma}
\label{lem:boundary_length_bounded_from_below}
Fix $\sp{a},r > 0$.  Suppose that $(\CS,d,\nu,x,y)$ has distribution $\bminflaw$.  For each $s \geq 0$ we let $Y_s$ be the boundary length of $\partial \fb{y}{x}{d(x,y)-s}$.  The $\bminflaw$ measure of the event that $d(x,y) > r$ and $\sup_{s \in [0,r]} Y_s \leq \epsilon^{\sp{a}}$ decays to $0$ as $\epsilon \to 0$ faster than any power of $\epsilon$.
\end{lemma}
\begin{proof}
Suppose that $s > 0$.  On the event that $Y_s \leq \epsilon^{\sp{a}}$, the probability that $Y_s$ hits $0$ in the next $\epsilon^{\sp{a}/2}$ units of time is positive.  It follows that the probability of the event that $Y$ stays below $\epsilon^{\sp{a}}$ in a given interval of time of length $r$ is at most $\exp(-c \epsilon^{-\sp{a}/2})$ where $c > 0$ is a constant which depends only on $r$.
\end{proof}

We are now ready to complete the proof of Theorem~\ref{thm:maximum_geodesics}.

\begin{proof}[Proof of Theorem~\ref{thm:maximum_geodesics}]
Fix $\sp{a}, M>0$. It suffices to prove the theorem for $(\CS, d, \nu, x, y)$ such that $\diam(\CS) \leq M$.  

By Lemma~\ref{lem:diameter_of_complement} and Lemma~\ref{lem:area_in_complement}, for a.e.\ instance of  $(\CS, d, \nu, x, y)$ there exists $r_0>0$ such that for every $r \in (0,r_0)$ there is a unique component $U_{z,r}$ of $\CS \setminus B(z,r)$ with diameter at least $\diam(\CS)/2$, and moreover $\nu(U_{z,r}) \ge \nu(\CS)/2$.

For each $k \in \N$ and $r_1\in(0, r_0)$, let $\Psi_{k, r_1}$ be the set of $z \in \CS$ such that 
there are $k$ geodesics which start from $z$, are otherwise disjoint and all go to $\partial\fb{w}{z}{r_1}$ where $w$ is any point in $U_{z,r_1}$ (the set $\partial \fb{w}{z}{r_1}$ does not depend on the choice of $w \in U_{z,r_1}$).
For any $z\in \Psi_k$, there a.s.\ exists $r_2 \in(0, r_0)$ such that there are $k$ geodesics starting from $z$ of length at least $r_2$ which are otherwise disjoint. 
Then by  Lemma~\ref{lem:diameter_of_complement},  there a.s.\ exists $r_1\in(0, r_2)$ such that these $k$ geodesics all reach $\partial\fb{w}{z}{r_1}$ for $w \in U_{z, r_1}$.  
We therefore have that
\[\Psi_k \subseteq \bigcup_{r_1 \in(0, r_0)} \Psi_{k, r_1}.\]
Note that the set $ \Psi_{k, r_1}$ is increasing as $r_1$ decreases.
By the countable stability of Hausdorff dimensions, to prove that $\dimH(\Psi_k) \leq 5-k$, it thus suffices to show that $\dimH(\Psi_{k, r_1}) \leq 5-k$ for every $r_1> 0$.  

Fix $\epsilon > 0$, $\sp{b}>0$, let $N_\epsilon = \epsilon^{-4-\sp{b}}$, and $M_\epsilon = \epsilon^{-\sp{b}}$.  Let also $(x_i)$, $(y_i)$ be independent i.i.d.\ sequences chosen from $\nu$.  By Lemma~\ref{lem:typical_points_dense}, there exists $\epsilon_0 > 0$ so that for all $\epsilon \in (0,\epsilon_0)$ we have that $\CS \subseteq \cup_{j=1}^{N_\epsilon} B(x_j,\epsilon)$ and $\CS \subseteq \cup_{j=1}^{M_\epsilon} B(y_j,\diam(\CS)/10)$.  We also know that $(\CS,d,\nu,x_i,y_j)$ has the same distribution as $(\CS,d,\nu,x,y)$ for each $i,j$.  

Choose $\eps_0\in(0, r_1/2)$ and let $\eps\in(0, \eps_0)$. Let $\CI_{\eps, k, r_1}$ be the set of $1\le i\le N_\eps$ such that there exists $z\in \Psi_{k, r_1}$ with $d(x_i, z) <\eps$.   Fix $i\in \CI_{\eps, k, r_1}$ and $z\in  \Psi_{k, r_1} \cap B(x_i, \eps)$. By the definition of $r_1$, there exists $y_j \in U_{z,r_1}$ with distance at least $\diam(\CS)/4$ to $z$. Moreover, there are $k$ disjoint geodesics from $z$ to $\partial\fb{y_j}{x_i}{r_1/2}$.  For $0\le n \le \lceil 4M/r_1 \rceil$, let $t_n:= nr_1/4$.  On $\{\diam(\CS) \le M\}$, Lemma~\ref{lem:boundary_length_bounded_from_below} implies that off an event whose $\bminflaw$ measure tends to $0$ as $\epsilon \to 0$ faster than any power of $\epsilon$ the following event occurs for $(\CS,d,\nu,x_i,y_j)$ where $E(\eps, k, t, \rho)$ is the event from Proposition~\ref{prop:disjoint_geodesics} and $\sp{a}>0$.
\begin{align}\label{eq:event_E_x_y}
E(x_i, y_j):= \bigcup_{n=1}^{\lceil 4M/r_1 \rceil} E(\eps, k, t_n, \eps^{2\sp{a}}) \cap\{\diam(\CS) \le M\} \cap \{ d(x_i, y_j) > t_n \}.
\end{align}
For any $1\le n\le \lceil 4M/r_1 \rceil$, we have
\begin{align*}
& \bminflaw [E(\eps, k, t_n, \eps^{2\sp{a}}), \diam(\CS)\le M, d(x_i, y_j) > t_n]\\
\le & \bminflaw [E(\eps, k, t_n, \eps^{2\sp{a}}), \diam(\CS)\le M \mid  d(x_i, y_j) > t_n]\, \bminflaw[d(x_i, y_j) > r_1/4] \\
= & O(\eps^{(k-1)(1-\sp{a})+o(1)}) \,   \bminflaw[d(x_i, y_j) > r_1/4].
\end{align*}
Note that the last line above is due to Proposition~\ref{prop:disjoint_geodesics} and the implicit constant does not depend on $t_n$. Therefore by applying a union bound, we get that $\bminflaw[E(x_i, y_j)] = O(\eps^{(k-1)(1-\sp{a})+o(1)})$ where the implicit constant depends only on $k, M, r_1$.  By applying a union bound over $1\le j \le M_\eps$ and using that $\sp{a} > 0$ is arbitrary, we get that
\[ \bminflaw [i \in \CI_{\eps,k, r_1}, \,  \diam(\CS) \le M] =O(\eps^{-\sp{b}+k-1 +o(1)}).\]

For each $\alpha>0$, let $\CH_\alpha$ denote the Hausdorff-$\alpha$ measure. Note that for $\epsilon \in (0,\epsilon_0)$, the union of $B(x_i,\epsilon)$ for all $i \in \CI_{\eps, k, r_1}$ covers $\Psi_{k, r_1}$.  Therefore for all $\alpha>0$, we have
\begin{align*}
\int \CH_\alpha( \Psi_{k, r_1})   \one_{\eps<\eps_0} \one_{\diam(\CS) \leq M} d\bminflaw \leq (2\epsilon)^\alpha \int |\CI_{\epsilon,k, r_1}|   \one_{\diam(\CS) \leq M} d\bminflaw \\
= O(\epsilon^\alpha) \, N_\eps \, O(\eps^{-\sp{b}+k-1+o(1)}) = O(\eps^{\alpha-4-2\sp{b} + k-1+o(1)}).
\end{align*}
Suppose that $\alpha > 5-k$.  Assume that we have taken $\sp{b} > 0$ sufficiently small so that the exponent of $\epsilon$ above is positive.  Since $\bminflaw( \epsilon \geq \epsilon_0) \to 0$ as $\epsilon \to 0$, it follows that $\CH_\alpha(\Psi_{k, r_1}) = 0$ for $\bminflaw$ a.e.\ instance of $(\CS,d,\nu,x,y)$ on  $\{\diam(\CS) \le M\}$.  This proves the result for $k \leq 5$.  
For $k\ge 6$, we have
\begin{align*}
\bminflaw [ \eps<\eps_0, \diam(\CS)\le M, \Psi_{k, r_1}\not=\emptyset]=O(\eps^{-4-2\sp{b} + k-1+o(1)}),
\end{align*}
where the exponent on the right hand side can be made positive by taking $\sp{b}\in(0,1/2)$.  We can then conclude by letting $\eps$ go to $0$.
\end{proof}

\subsection{Number of geodesics between two points}
\label{subsec:n_of_geodesics2}
The goal of this section is to prove Theorem~\ref{thm:finite_number_of_geodesics} and Theorem~\ref{thm:number_of_geodesics}.

The main input into the proof of Theorem~\ref{thm:finite_number_of_geodesics} is the following lemma, which gives the exponent for the event that two marked points in the Brownian map are within distance $\epsilon$ of a pair of points which are connected by a family of geodesics with a specified structure.

\begin{lemma}
\label{lem:geodesics_splitting_lemma}
Suppose that $(\CS,d,\nu,x,y)$ has distribution $\bminflaw$.  Fix $0 < R_0< R_1 < R_2 $, $\delta, \sp{a}>0$, $k\in\N$ and $2k$ numbers $R_1<r_1<s_1 <r_2<s_2< \cdots <r_k<s_k < R_2$ such that $s_j-r_j <\delta/10$ for all $1\le j\le k$. 
Fix $\xi\in(0,\delta/100)$, $\zeta\in(0, \delta/10)$ and $\epsilon\in(0, \delta /100)$.  
 Fix $n_1, \dots, n_k \in\N$. 
Let 
\begin{align}\label{eq:F_IJK}
F_{I,J,K}\left(R_0, R_1,R_2; (r_j, s_j, n_j)_{1\le j \le k} ; \eps, \delta, \xi, \zeta, \sp{a}\right)
\end{align} 
be the following event.  There exist $u\in B(x, \eps)$ and $v\in B(y,\eps)$ such that
\begin{enumerate}[(i)]
\item\label{it:vu_disj} There are geodesics $\eta_1,\ldots,\eta_J$ from $v$ to $u$ so that the sets $\eta_i((0,R_1))$ for $1 \leq i \leq J$ are disjoint from each other.  Let $z_1, \ldots, z_J$ be where $\eta_1,\ldots,\eta_J$ last hit $\partial \fb{x}{y}{R_0}$. The boundary distance between any two points   $z_i, z_j$ for $i,j$ distinct is at least $\xi^2$.

\item\label{it:split} The event $G \left((r_j, s_j, n_j)_{1\le j \le k}; \eps, \delta, \xi, \zeta, K\right)$ from Lemma~\ref{lem:k_geodesics_nearby} occurs for $u, v$.

\item\label{it:uv_disj} The boundary length of $\partial  \fb{y}{x}{d(x,y)- R_2}$ is at least $\epsilon^{2\sp{a}}$.
There are geodesics $\eta_1,\ldots,\eta_I$ from $u$ to $v$ so that the sets $\eta_i((0,d(x,y)- R_2))$ for $1 \leq i \leq I$ are disjoint from each other.

\end{enumerate}
Then  for each $M>0$, we have that
\begin{align*}
&\bminflaw \left[F_{I,J,K}(R_0, R_1,R_2; (r_j, s_j, n_j)_{1\le j \le k}; \eps, \delta, \xi, \zeta, \sp{a}), \diam(\CS) \le M\right] \\
&= O(\epsilon^{I+2J+K-3 -c \sp{a}+o(1)})
\end{align*}
where $c > 0$ depends only on $I,J,K$ and the implicit constant in the $O$ depends on $R_0, R_1,R_2$, $(r_j, s_j, n_j)_{1\le j \le k}$,  $\delta, \xi, \zeta, \sp{a}$, $M$ and $I,J,K$.
\end{lemma}

\begin{figure}[ht!]
\begin{center}
\includegraphics[scale=0.9]{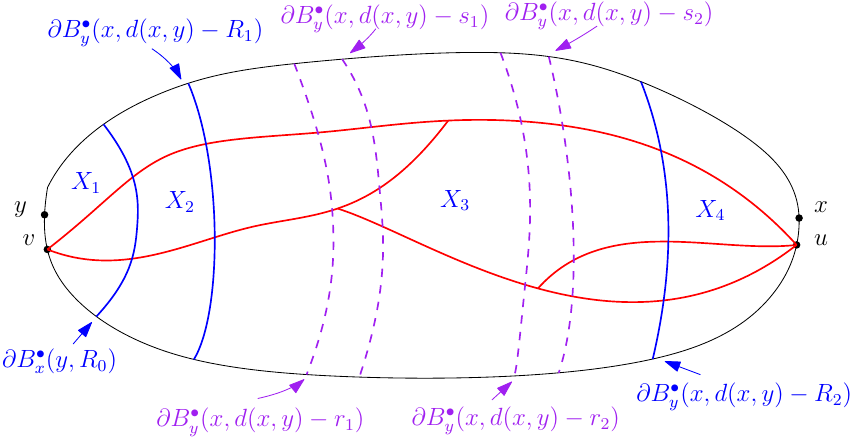}	
\end{center}
\caption{\label{fig:splitting_time_lemma} Illustration of the setup of Lemma~\ref{lem:geodesics_splitting_lemma}.  In the figure, $I=3$, $J=2$, and $K=2$. The event $F_{I,J,K}(R_1,R_2, r_1, s_1, \ldots, r_k, s_k; \xi, \eps, \delta, \sp{a})$ holds.}
\end{figure}

\begin{proof}

Suppose we are working on the event~\eqref{eq:F_IJK} (which we denote by $F_{I, J, K}$ in the rest of the proof) and on $\{\diam(\CS) \le M\}$.
See Figure~\ref{fig:splitting_time_lemma} for an illustration of the setup of the lemma. 

We fix $0<R_0<R_1<R_2$ and  divide $\CS$ into the four sets $X_1=\fb{x}{y}{R_0},$ $X_2 = \CS \setminus (\fb{x}{y}{R_0} \cup \fb{y}{x}{d(x,y)- R_1})$,  $X_3 = \fb{y}{x}{d(x,y)- R_1} \setminus \fb{y}{x}{d(x,y)- R_2}$ and $X_4 =  \fb{y}{x}{d(x,y)- R_2}$.

First of all, part~\eqref{it:vu_disj} implies that there are $J$ disjoint geodesics from $v$ to $\partial \fb{x}{y}{R_0}$. This event is entirely determined by $X_1$. Moreover, on this event, Lemma~\ref{lem:boundary_length_bounded_from_below} implies that off an event whose probability tends to $0$ as $\epsilon \to 0$ faster than any power of $\epsilon$ the following event holds where $E(\eps, k, t, \rho)$ is the event from Proposition~\ref{prop:disjoint_geodesics}.
\begin{align*}
E(x, y):= \bigcup_{n=1}^{\lceil 2M/ R_0 \rceil} E(\eps, k, nR_0/2, \epsilon^{2\sp{a}}) \cap\{\diam(\CS) \le M\} \cap \{ d(x, y) >  n R_0/2 \}.
\end{align*}
This has the same form as~\eqref{eq:event_E_x_y} and has probability $O(\eps^{(J-1)(1-\sp{a}) +o(1)})$.

Conditionally on the boundary length of $\partial \fb{x}{y}{R_0}$, $X_2$ is independent of $X_1$. Let $d_2$ denote the interior-internal metric of $X_2$.  On $F_{I,J,K}$, we have $d(x,y) > R_1$ and the following event holds. There exist $J$ points $z_1, \ldots, z_J$ on $\partial \fb{x}{y}{R_0}$ such that the boundary distance between any two points $z_i, z_j$ for $i,j$ distinct is at least $\xi^2$, and that the distance from $z_i$ to $\partial \fb{y}{x}{d(x,y)-R_1}$ with respect to $d_2$ is at most $R_1- R_0 +4\eps$ for each $i$.  This event depends entirely on $X_2$. Note that $\CS \setminus \fb{x}{y}{R_0}$ is a Brownian disk with marked point $x$. This event implies that the distance from $z_i$ to $x$ with respect to the interior-internal metric of the disk is at most $d(x, \fb{x}{y}{R_0}) + 4\eps$ for each $1\le i\le J$, hence by Lemma~\ref{lem:num_pinch_points_disk}, this event happens with probability $O(\eps^{J-1 + o(1)})$.

Conditionally on the boundary length of $\partial \fb{y}{x}{d(x,y)-R_1}$, $X_3$ is independent from $X_1$ and $X_2$.
In the proof of Lemma~\ref{lem:k_geodesics_nearby}, we have estimated the probability of the event~\eqref{eq:G} (which we will denote by $G$) by successively looking at the metric bands $\fb{y}{x}{d(x,y) - r_j} \setminus \fb{y}{x}{d(x,y) - s_j}$ for $j=1,\ldots,k$ and then using the conditional independence of these bands.
Even though the event $G$ does not only depend on $X_3$, the proof of Lemma~\ref{lem:k_geodesics_nearby} shows that $G$ implies that certain events happen for each of the metric bands $\fb{y}{x}{d(x,y) - r_j} \setminus \fb{y}{x}{d(x,y) - s_j}$ and the overall probability of these events is $O(\eps^{K+o(1)})$. Since these metric bands are all contained in $X_3$, we can deduce that $F_{I,J,K}$ implies that a certain event which depends entirely on $X_3$ happens with probability $O(\eps^{K+o(1)})$.

Finally, conditionally on the boundary length of $\partial \fb{y}{x}{d(x,y)-R_2}$, $X_4$ is independent of $X_1, X_2$ and $X_3$.  Part~\eqref{it:uv_disj} implies that there are $I$ disjoint geodesics from $u$ to $\partial \fb{y}{x}{d(x,y)-R_2}$. Since  the boundary length of $\partial \fb{y}{x}{d(x,y)-R_2}$ is at least $\epsilon^{2\sp{a}}$, we have $u\in X_4$ off an event whose probability decreases faster than any power of $\eps$, so that this event is entirely determined by $X_4$.  Moreover, by Proposition~\ref{prop:disjoint_geodesics} the conditional probability that there are  $I$ disjoint geodesics from $u$ to $\partial \fb{y}{x}{d(x,y)-R_2}$ is given by $O(\epsilon^{(I-1)(1-\sp{a})+o(1) })$.

Combining the above arguments and multiplying the four probabilities together gives the result.
\end{proof}

We are now ready to complete the proofs of Theorems~\ref{thm:finite_number_of_geodesics} and~\ref{thm:number_of_geodesics}.

\begin{proof}[Proof of Theorem~\ref{thm:finite_number_of_geodesics}]
Suppose that $(\CS,d,\nu,x,y)$ has distribution $\bminflaw$.  By scaling, it suffices to prove the result for $\diam(\CS) \le 1$.  Fix $R_0, R_1,R_2$, $(r_j, s_j, n_j)_{1\le j \le k}$, $\delta, \xi, \zeta, \sp{a}$ as in Lemma~\ref{lem:geodesics_splitting_lemma}. Fix $\eps_0\in(0, \delta/100)$. Fix $I, J \in\N$ and $K\in\N_0$. Let 
\begin{align}\label{eq:Z}
\CZ_{I,J,K}\left(R_0, R_1,R_2; (r_j, s_j, n_j)_{1\le j \le k}; \eps_0, \delta, \xi, \zeta, \sp{a}\right)
\end{align} 
be the set of pairs $(u,v) \in \CS$ such that parts~\eqref{it:vu_disj}--\eqref{it:uv_disj} of Lemma~\ref{lem:geodesics_splitting_lemma} hold for $\eps=\eps_0$.  
Note that the union of the sets~\eqref{eq:Z} over all possible choices of rational numbers $R_0, R_1,R_2$, $(r_j, s_j, n_j)_{1\le j \le k}$, $\delta, \xi, \zeta, \sp{a}, \eps_0$ is $\Phi_{I, J, K}$. 

Suppose that $I,J,K \geq 0$ are such that $11-(I+2J+K) \geq 0$. By the countable stability of the Hausdorff dimension, it is enough to show that the Hausdorff dimension of the set~\eqref{eq:Z} is at most $11 - (I+2J+K)$ for any
choice of rational numbers $R_0, R_1,R_2$, $(r_j, s_j, n_j)_{1\le j \le k}$, $\delta, \xi, \zeta, \sp{a}$. In the sequel, we fix these rational numbers and denote the set~\eqref{eq:Z} simply by $\CZ_{I, J, K}$.

Let $N = \epsilon^{-8-\sp{a}}$.  Let $(x_j)$, $(y_j)$ be independent i.i.d.\ sequences in $\CS$ chosen from $\nu$.  By Lemma~\ref{lem:typical_points_dense}, with full $\bminflaw$ measure there exists $\epsilon_1 \in(0, \eps_0)$ so that for all $\epsilon \in (0,\epsilon_1)$ we have for every $u,v \in \CS$ distinct, there exists $1 \leq i,j \leq N$ so that $u \in B(x_i,\epsilon)$ and $v \in B(y_j, \epsilon)$.  Let $G_{i,j}$ be the event~\eqref{eq:F_IJK} from Lemma~\ref{lem:geodesics_splitting_lemma} where we take $x = x_i$ and $y = y_j$.  Since $(\CS,d,\nu,x_i,y_j)$ has distribution $\bminflaw$, it follows from Lemma~\ref{lem:geodesics_splitting_lemma} that 
\[ \bminflaw [G_{i,j}, \diam(\CS) \le 1] = O(\epsilon^{I+2J+K-3-c\sp{a}+o(1)}).\]
Let $\CI$ be the set of pairs $(i,j)$ so that $G_{i,j}$ occurs.  For $\alpha > 0$, we then have that
\[ 
\int \CH_\alpha(\CZ_{I,J,K}) \one_{\epsilon < \epsilon_1} \one_{\diam(\CS)\le 1} d \bminflaw \leq (2\epsilon)^\alpha \int |\CI| \one_{\diam(\CS)\le 1} d\bminflaw = O(\epsilon^{\alpha+I+2J+K-11 - (1+c)\sp{a} +o(1)}).
\]
For any $\alpha > 11 - (I+2J+K)$ we can choose $\sp{a} > 0$ sufficiently small so that the right hand side of the equation above goes to $0$ as $\epsilon \to 0$.  It therefore follows that for $\bminflaw$ a.e.\ instance of $(\CS,d,\nu,x,y)$ we have that $\dimH(\CZ_{I,J,K}) \leq 11 - (I+2J+K)$, as desired.  The same argument implies that $\CZ_{I,J,K} = \emptyset$ if $I+2J+K > 11$.

To finish the proof, it is left to explain why every geodesic contains at most $2$ splitting points from one endpoint towards the other, and the multiplicity of any splitting point is $1$.  
This will follow from Theorem~\ref{thm:ghost} and a result of Le Gall \cite{lg2010geodesics} which says that any geodesic from the root of the map to another point has at most $2$ splitting points and the multiplicity of each splitting point is $1$ (see Figure~\ref{fig:root_geodesics}).
Suppose that $u,v \in \CS$  and $\eta_1,\ldots,\eta_k$ is a maximal collection of geodesics from $v$ to $u$ (note that $k$ is at most a finite and deterministic constant due to Proposition~\ref{prop:two_point_number_of_geodesics}). Suppose that for $1\le k_0 \le k$, $\eta_1, \ldots, \eta_{k_0}$ all pass through a splitting point from $v$ towards $u$ of multiplicity $k_0$.
Let $\epsilon > 0$ be such that $\eta_i|_{[t-\eps,t]} = \eta_j|_{[t-\eps,t]}$ for $1 \leq i,j \leq k_0$ and the sets $\eta_i|_{(t,t+\epsilon)}$ are disjoint.  Fix $s \in (t-\eps,t)$ and let $(x_{n_j})$ be a subsequence of $(x_n)$ defined above which converges to $\eta_1(s)$.  For each $j$, let $\eta_{n_j}$ be a geodesic from $x_{n_j}$ to $u$.  By passing to a subsequence if necessary, we may assume without loss of generality that $\eta_{n_j}$ converges to a geodesic from $\eta_1(s)$ to $u$.  By Theorem~\ref{thm:ghost}, it follows that $\eta_1(t)$ is a splitting point for $x_{n_j}$ towards $u$ for all $j$ large enough.  Since a splitting point from any typical point to any other point has multiplicity $1$, it follows that $\eta_1(t)$ has multiplicity  $1$. Similarly, suppose that $\eta$ is a geodesic from $v$ to $u$ which contains $n$ splitting points from $v$ towards $u$, and let $t>0$ be the smallest splitting time. We can again find a subsequence  $(x_{n_j})$ of $(x_n)$ converging to $\eta(t/2)$ and for some $j$ big enough a geodesic from $x_{n_j}$ to $u$ which contains $n$ splitting points from $v$ towards $u$. Since $x_{n_j}$ is a typical point, we must have $n\le 2$.
\end{proof}

\begin{proof}[Proof of the dimension upper bounds in Theorem~\ref{thm:number_of_geodesics}]
In view of Theorem~\ref{thm:finite_number_of_geodesics}, to prove the dimension upper bounds, we need to identify the possibilities for the number of splitting times and the number of disjoint terminal and initial segments of geodesics between pairs of points which are connected by $j$ geodesics.

Suppose that we have a given $I,J,K$ triple.  Let $M$ be the number of corresponding merging points (i.e., points at which two distinct geodesics connecting common points merge).  By Theorem~\ref{thm:intersection_of_geodesics}, it is not possible for a pair of geodesics which connect a given pair of points to split and then remerge.  We therefore must have that $I = J+K-M$.  In other words, $M = J+K-I$.  For a given $I,J,K$ triple, it is not difficult to see that the number of geodesics connecting $u$ and $v$ is maximized by the collection consisting of an normal $(M+1,K+1)$-network (recall the definition given in Section~\ref{subsec:background_intro}) together with $J-(M+1)$ geodesics which are disjoint except possibly at their endpoints.  Note that a normal $(i,j)$-network corresponds to $I = i$, $J = j$, and $K = i-1$ so that $I + 2J + K = 2(i+j)-1$.  Therefore if $n=I+2J+K$ is odd, then the configuration which maximizes the number of geodesics connecting two points with $n=I+2J+K$ is a normal network and if $n = I+2J+K$ is even it is a normal network together with an additional geodesic disjoint from the others except at its endpoints.  From this, we deduce that the maximal number of $g(n)$ geodesics (as a function of $n=I+2J+K$) that one can have for $3 \leq n \leq 11$ is given by the following:
\[ g(3) = 1,\ g(4) = 0,\ g(5) = 2,\ g(6) = 2,\ g(7) = 4,\ g(8) = 3,\ g(9) = 6,\ g(10) = 5,\ g(11) = 9.\]
By Theorem~\ref{thm:finite_number_of_geodesics}, we thus have that
\[ \dimH(\Phi_1) \leq 8,\ \dimH(\Phi_2) \leq 6,\ \dimH(\Phi_4) \leq 4,\ \dimH(\Phi_6) \leq 2,\ \dimH(\Phi_9) = 0.\]
It also implies that $\dimH(\Phi_3) \leq 4$ (since $g$ exceeds the value $2$ only for $n \geq 7$), $\dimH(\Phi_5) \leq 2$ (since $g$ exceeds the value $5$ only for $n \geq 9$), $\dimH(\Phi_7) = 0$, and $\dimH(\Phi_8) = 0$ (since $g$ exceeds the value $7$ only for $n \geq 11$).  It also implies that $\Phi_j = \emptyset$ for all $j \geq 10$ (since $g$ exceeds the value $10$ only for $n \geq 12$).
\end{proof}

\section{Points connected by $5,7$ and $8$ geodesics}
\label{sec:578}

The goal of this section is to complete the proof of Theorem~\ref{thm:number_of_geodesics}. We have already proved the upper bounds for $\dimH(\Phi_i)$ for all $1\le i \le 9$. The matching lower bounds of $\dimH(\Phi_i)$ for $i\in\{2,3,4,6\}$ and the fact that $\Phi_9$ is countably infinite follow from the results of~\cite{akm2017stability}.
It was also shown in~\cite{akm2017stability} that the sets $\Phi_i$ for $i\in\{2,3,4,6,9\}$ are dense in $\CS^2$.
The only missing part of Theorem~\ref{thm:number_of_geodesics} is the following proposition.

\begin{proposition}
\label{prop:five_geodesics_lbd}
The following holds for $\bminflaw$ a.e.\ instance $(\CS,d,\nu,x,y)$. We have $\dimH(\Phi_5) \geq 2$, and $\Phi_7$, $\Phi_8$ are both countably infinite. Moreover, for each $i\in\{5,7,8\}$, the set of $u \in\CS$ such that there exists $v \in \CS$ with $(u,v) \in \Phi_i$ is dense in $\CS$.
\end{proposition}

The main idea of the proof of Proposition~\ref{prop:five_geodesics_lbd} is to construct the collection of geodesics between $u$ and $v$ with $(u,v)\in\Phi_i$ for each $i\in\{5,7,8\}$ as a concatenation of two sets of independent geodesics, emanating respectively from $u$ and $v$. In Section~\ref{subsec:geo_v}, we will construct the set of geodesics emanating from $v$. To construct $\Phi_5$ (resp.\ $\Phi_7, \Phi_8$), we will need to construct a set $\wt\Phi_2$ (resp.\ $\wt\Phi_3$) of points such that for each $v\in\wt\Phi_2$ (resp.\ $v\in\wt\Phi_3$) there are two (resp.\ three) geodesics emanating from $v$ that are disjoint on some initial time interval (except at $v$) before merging into a common geodesic in a specific way.
In Section~\ref{subsec:geo_u}, we will construct a point $u$ from which emanate three geodesics that are disjoint on some initial time interval (except at $u$), and then try to merge the geodesics emanating from $v$ constructed in Section~\ref{subsec:geo_v} into the geodesics emanating from $u$, so they form the desired configurations of geodesics between $u$ and $v$.
This idea of separating the collection of geodesics between $u$ and $v$ into two independent sets of geodesics is similar to that in the proof of the dimension of the endpoints of normal networks in~\cite{akm2017stability}. However, our case is more complicated, because the set of geodesics between $u$ and $v$ do not pass through a common point (except their common endpoints). To construct the desired sets of geodesics, we will cut the Brownian map into metric bands and metric slices, construct the different parts of the desired geodesics in these (conditionally) independent metric spaces, and then glue them back together in a specific way.
Some technicalities arise, because we need to make sure that the concatenation of geodesics are still geodesics.

Throughout this section (except when we indicate differently), let $(\CB,d_\CB,\nu_\CB,z)$ be a metric band with law $\bandlaw{1}{\infty}$. Let $\eta$ be the a.s.\ unique geodesic from $z$ to $\outerboundary \CB$, which consists of a single point that we call $x_0$.  For each $t > 0$, we let $\CB_t$ be the metric band which corresponds to the set of points disconnected from $\innerboundary \CB$ by $B(x_0,d_0-t)$ where $d_0$ is the distance from $\innerboundary \CB$ to $x_0$. For $t \ge 0$, let $\CF_t$ be the $\sigma$-algebra generated by $\CB\setminus \CB_t$.
For $c\in(0,1)$, let $\tau_c$ be the first time $t$ such that the boundary length of $\innerboundary \CB_{t}$ is equal to $c$. Then $\tau_c$ is an $(\CF_t)$-stopping time.

\subsection{The sets $\wt \Phi_2$ and $\wt \Phi_3$}\label{subsec:geo_v}

\begin{figure}[ht!]
\begin{center}
\includegraphics[width=.95\textwidth]{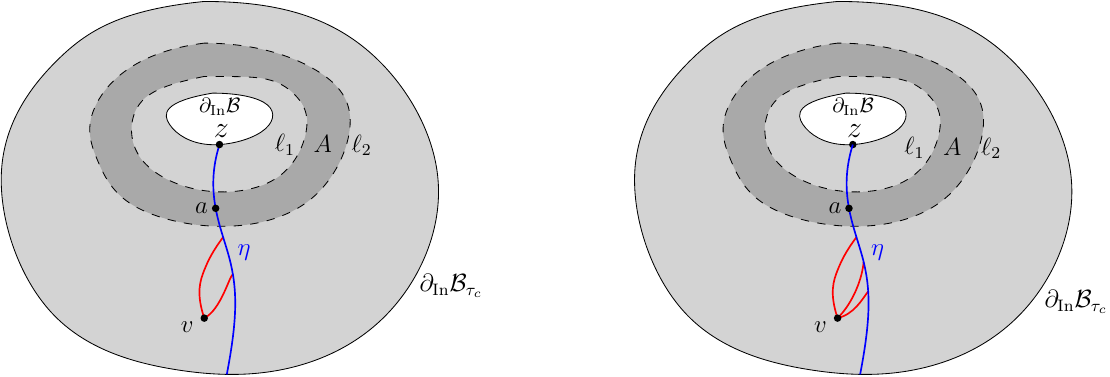}
\end{center}
\caption{\label{fig:geo_def} 
Illustration of the event $E$ in Lemma~\ref{lem:two_geo_dim_two}.
\textbf{Left:} Illustration of a point $v\in \wt\Phi_2$.
\textbf{Right:} Illustration of a point $v\in \wt\Phi_3$.}
\end{figure}

The goal of this subsection is to prove the following lemma.
\begin{lemma}
\label{lem:two_geo_dim_two}
There exist $c\in(0,1)$ and $d_1>0$ such that the following event $E$ occurs with positive probability for the band $ \CB \setminus \CB_{\tau_c}$ (see Figure~\ref{fig:geo_def}). 
\begin{enumerate}[(I)]
\item There exists an annulus $A \subseteq \CB \setminus \CB_{\tau_c}$ which encircles $\innerboundary\CB$ and satisfies the following properties. Let $d_A$ be the interior-internal metric of $d_\CB$ in $A$. Let $\ell_1$ and $\ell_2$ be respectively the inner and outer boundary of $A$.
There exists $a\in A$ such that every geodesic (w.r.t.\ $d_A$) from a point in $\ell_1$ to a point in $\ell_2$ passes through $a$.
\item There exists a set $\wt \Phi_2 \subseteq \CB \setminus \CB_{\tau_c}$ which is disjoint from $A$,  contained in the region between $\ell_2$ and $\innerboundary\CB_{\tau_c}$, and satisfies the following properties.
The distance from any point in  $\wt\Phi_2$ to $z$ is at most $d_1$. The distance from any point in $\wt\Phi_2$ to $\innerboundary\CB_{\tau_c}$ is at least $d_1$. 
We have $\dimH(\wt\Phi_2)\ge 2$. 
For each $v\in \wt\Phi_2$, there are exactly two geodesics $\eta_1$ and $\eta_2$ from $v$ to $z$. Moreover, $\eta_1$ and $\eta_2$ satisfy the following properties.
\begin{enumerate}[(i)]
\item There exists $r > 0$ so that $\eta_1((0,r)) \cap \eta_2((0,r)) = \emptyset$.
\item For $i=1,2$, $\eta_i$ intersects $\eta$ before (or at the same time as) intersecting $\eta_{3-i}$, and then merges with (the time-reversal of) $\eta$.
\item There exist two distinct points $b_1$ and $b_2$ on $\eta$ which remain the same for any choice of $v\in\wt\Phi_2$, such that for $i=1,2$, $\eta_i$ hits $\eta$ at its right side on $b_i$.
\end{enumerate}

\item There exists a set $\wt \Phi_3 \subseteq \CB \setminus \CB_{\tau_c}$ which is disjoint from $A$,  contained in the region between $\ell_2$ and $\innerboundary\CB_{\tau_c}$, and satisfies the following properties.
The distance from $\wt\Phi_3$ to $z$ is at most $d_1$. The distance from $\wt\Phi_3$ to $\innerboundary\CB_{\tau_c}$ is at least $d_1$. 
The set $\wt\Phi_3$ is non-empty.
For each $v\in \wt\Phi_3$, there are exactly three geodesics $\eta_1, \eta_2$ and $\eta_3$ from $v$ to $z$. Moreover, $\eta_1, \eta_2$ and $\eta_3$ satisfy the following properties.
\begin{enumerate}[(i)]
\item There exists $r > 0$ so that $\eta_i((0,r)) \cap \eta_j((0,r)) = \emptyset$ for any $i, j \in\{1,2,3\}$ distinct.
\item For $i, j \in\{1,2,3\}$ distinct, $\eta_i$ intersects $\eta$ before (or at the same time as) intersecting $\eta_{j}$, and then merges with (the time-reversal of) $\eta$.
\item For $i=1,2, 3$, $\eta_i$ hits $\eta$ at its right side.
\end{enumerate}

\end{enumerate}
\end{lemma}

In order to prove Lemma~\ref{lem:two_geo_dim_two}, we need to first prove the following lemma.

\begin{lemma}\label{lem:geo_merge}
There exists $c_0\in(0,1)$ such that the following event $F$ occurs with positive probability for $\CB \setminus \CB_{\tau_{c_0}}$.
There exists an annulus $A \subseteq \CB \setminus \CB_{\tau_{c_0}}$ which encircles $\innerboundary\CB$, and satisfies the following properties. Let $d_A$ be the interior-internal metric of $d_\CB$ in $A$. Let $\ell_1$ and $\ell_2$ be respectively the inner and outer boundary of $A$.
There exists $a\in A$ such that every geodesic (w.r.t.\ $d_A$) from a point in $\ell_1$ to a point in $\ell_2$ passes through $a$.
\end{lemma}

\begin{proof}
Let  $(\CS,d,\nu,x,y)$ be distributed according to $\bminflaw$, and let $\eta:[0, T] \to \CS$ be the a.e.\ unique geodesic between $x$ and $y$. We condition on the event $T =d(x,y) >1$ (which has positive and finite $\bminflaw$ measure).  For $n\in\N$, let $A_n$ be the annulus with boundaries $\partial \fb{y}{x}{1/n}$ and $\partial \fb{x}{y}{1/n}$.
We aim to show that there a.s.\ exists $N\in\N$ such that every geodesic (w.r.t.\ $d_{A_N}$) from a point in $\partial \fb{y}{x}{1/N}$ to a point in $\partial \fb{x}{y}{1/N}$ passes through $\eta(T/2)$.
Suppose that this is not the case. Then for every $n\in\N$, there exists $a_n \in \partial \fb{y}{x}{1/n}$, $b_n \in \partial \fb{x}{y}{1/n}$ and a geodesic $\eta_n$ from $a_n$ to $b_n$ (w.r.t.\ $d_{A_n}$) which does not pass through $\eta(T/2)$. By extracting a subsequence, we can assume that $\eta_n$ converges w.r.t.\ the Hausdorff distance to a limit $\xi$. Then $\xi$ must be a geodesic from $x$ to $y$ in $\CS$. Since there is a.s.\ a unique geodesic from $x$ to $y$, we have $\xi=\eta$.
By the strong confluence of geodesics (Theorem~\ref{thm:strong_confluence2}), there exists $n$ big enough such that $\eta_n$ coalesces with $\eta$ away from its endpoints. In particular, we have $\eta(T/2) \in \eta_n$, leading to a contradiction. We have thus proved the existence of $N$. 

It remains to cut out a band in $\CS$ which contains $A_N$. For each $c>0$, let $T_c$ be the first time $t$ such that the boundary length of $\partial\fb{x}{y}{d(x,y)-t}$ reaches $c$. Given $T_c$ and $c_0\in(0,1)$, let $T_{c, c_0}$ be the first time $t>T_c$ such that the boundary length of $\partial\fb{x}{y}{d(x,y)-t}$ reaches $cc_0$. The metric band $\fb{x}{y}{d(x,y)-T_c}\setminus \fb{x}{y}{d(x,y)-T_{c, c_0}}$ has the same law as $\CB \setminus \CB_{\tau_{c_0}}$ (after rescaling areas by $c^{-1/2}$, boundary lengths by $c^{-1}$, and distances by $c^{-2}$).
Note that $\partial\fb{x}{y}{d(x,y)-T_c}$ tends to $x$ as $c\to 0$, and $\partial  \fb{x}{y}{d(x,y)-T_{c, c_0}}$ tends to $y$ as $c_0\to 0$. Therefore, there exist $c, c_0$ such that $\fb{x}{y}{d(x,y)-T_c}\setminus \fb{x}{y}{d(x,y)-T_{c, c_0}}$ contains $A_N$ with positive probability. We have therefore constructed the event $F$ as in the statement of the lemma, for the band $\fb{x}{y}{d(x,y)-T_c}\setminus \fb{x}{y}{d(x,y)-T_{c, c_0}}$, the annulus $A_N$, and $a=\eta(T/2)$.
\end{proof}

\begin{figure}[ht!]
\begin{center}
\includegraphics[width=.9\textwidth]{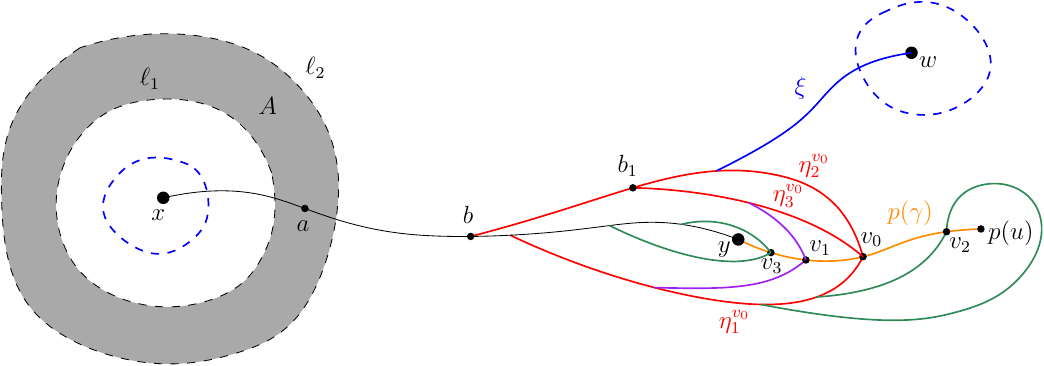}
\end{center}
\caption{\label{fig:two_geo_dim_two} Illustration of the setup to prove Lemma~\ref{lem:two_geo_dim_two}.  The path $p(\gamma)$ is shown in orange and geodesics from $v_0, v_1$ to the root $x$ are shown in red and purple.  The point $w$ is chosen according to $\nu$ independently of everything else.}
\end{figure}

We are now ready to prove Lemma~\ref{lem:two_geo_dim_two}.

\begin{proof}[Proof of Lemma~\ref{lem:two_geo_dim_two}]
We will prove the lemma in three steps.
Let  $(\CS,d,\nu,x,y)$ be distributed according to $\bminflaw$.  We will first construct an event for $(\CS,d,\nu,x,y)$ with positive $\bminflaw$ measure, as well as the sets $A$ and $\wt \Phi_2$,  $\wt \Phi_3$ (Steps 1 and 2). Then we will cut out a metric band from $\CS$ (Step 3) which contains $A$ and $\wt \Phi_2$, $\wt \Phi_3$ so that the event  in the statement of the lemma holds for this band.
See Figure~\ref{fig:two_geo_dim_two} for an illustration of the setup and proof.  

\emph{Step 1: Constructing the annulus $A$.}
We perform a reverse metric exploration from $x$ to $y$.  Fix $c>0$, and let $T_c$ be the first time $t$ such that the boundary length of $\partial\fb{x}{y}{d(x,y)-t}$ reaches $c$ (we condition on the event that $T_c <\infty$ which has positive and finite $\bminflaw$ measure).  Let $c_0\in(0,1)$ be given by Lemma~\ref{lem:geo_merge}, and let $T_{c, c_0}$ be the first time $t>T_c$ such that the boundary length of $\partial\fb{x}{y}{d(x,y)-t}$ reaches $cc_0$.
We further condition on the event $F$ from Lemma~\ref{lem:geo_merge} (which occurs with positive probability for the band $\fb{x}{y}{d(x,y)-T_c}\setminus \fb{x}{y}{d(x,y)-T_{c, c_0}}$). This constructs the annulus $A$, together with $\ell_1, \ell_2$ and $a\in A$.
Moreover, given the boundary length of $\partial \fb{x}{y}{d(x,y)-T_{c, c_0}}$, the space $\fb{x}{y}{d(x,y)-T_{c, c_0}}$ is independent of $\CS \setminus \fb{x}{y}{d(x,y)-T_{c, c_0}}$.

\emph{Step 2: constructing the sets $\wt\Phi_2$ and $\wt\Phi_3$.}
Let $(X,Y)$ be the Brownian snake encoding of $(\CS,d,\nu,x,y)$.  That is, $X \colon [0,T] \to \R_+$ is a Brownian excursion and its length is equal to $\nu(\CS)$.  Let $\CT$ be the CRT instance which is encoded by $X$ and let $p \colon \CT \to \CS$ be the projection map.  Let $\rho_\CT$ be the root of $\CT$ (so that $p(\rho_\CT) = y$).
Almost surely (w.r.t.\ the conditional probability), there exists a point $u \in \CT$ such that 
\begin{enumerate}[(i)]
\item $\CT \setminus \{u\}$ is disconnected, i.e., $u \in \CT$ is not a leaf (nor the root);
\item Let $\gamma$ be the geodesic in $\CT$ from $u$ to $\rho_\CT$.  Then $p(\gamma) \subseteq \fb{x}{y}{d(x,y)-T_{c, c_0}}$.
\end{enumerate}
This is because 
for any $u \in\CT$ which is not a leaf (nor the root), each point on the path $\gamma$ from $u$ to $\rho_\CT$ (except the point $\rho_\CT$) is also not a leaf (nor the root). If $p(\gamma)$ is not contained in $\fb{x}{y}{d(x,y)-T_{c, c_0}}$, then we can choose $u$ to be another point on $\gamma$ which is closer to $\rho_\CT$ so that the new path is contained in $\fb{x}{y}{d(x,y)-T_{c, c_0}}$.
From now on, we fix such $u,\gamma$ (we note that we can choose $u$, $\gamma$ in a measurable way by starting with the longest branch in $\CT$ from $\rho_\CT$ and then taking $\gamma$ to be the sub-branch starting from $u$ so that it satisfies the above properties).  

By \cite[Theorem~1.4]{lg2010geodesics}, we know that for each $v \in p(\gamma)$ there are either two or three geodesics from $v$ to $x$. Let $P_2(\gamma)$ (resp.\ $P_3(\gamma)$) be the set of points $v \in p(\gamma)$ from which there are exactly two (resp.\ three) geodesics to $x$. 
It follows from \cite[Proposition~3.3]{lg2010geodesics} that it is a.s.\ the case that for any $v_0, v_1 \in p(\gamma)$ distinct, the part of $p(\gamma)$ between $v_0$ and $v_1$ has Hausdorff dimension $2$.  The set $P_3(\gamma)$ is countable and dense in $p(\gamma)$. 
For any $v_0, v_1 \in p(\gamma)$ distinct, the points on the part of $p(\gamma)$ between $v_0$ and $v_1$ which belong to $P_2(\gamma)$ also has Hausdorff dimension $2$.
In addition, we have the following facts.

\begin{itemize}
\item For each $v\in P_2(\gamma)$, let $\eta_1^v$ and $\eta_2^v$ be the two geodesics from $v$ to $x$.
There exists $r>0$ such that $\eta_1^v((0,r)) \cap \eta_2^v((0,r)) = \emptyset$. 
By \cite[Lemma 7.3]{lg2010geodesics}, we know that $\eta_1^v,\eta_2^v$ do not intersect $p(\gamma)$ except at their starting point.
In fact, we further know that, by possibly relabelling, $\eta_1^v$ and $\eta_2^v$ respectively hit $p(\gamma)$ at its left and right sides. This is because that,  the geodesics $\eta_1^v$, $\eta_2^v$  are branches in the geodesic tree, and emanate from different corners of $v$ in the dual tree $\CT$, hence stay at the two sides of $p(\gamma)$. (This can for example be seen from the picture where one obtains the Brownian map by gluing the processes $(X,Y)$ as described in \cite{ms2015axiomatic}.)
\smallskip

\item For each $v\in P_3(\gamma)$, let $\eta_1^v, \eta_2^v$ and $\eta_3^v$ be the three geodesics from $v$ to $x$.
There exists $r>0$ such that $\eta_i^v((0,r)) \cap \eta_j^v((0,r)) = \emptyset$ for $i,j \in\{1,2,3\}$ distinct. By \cite[Lemma 7.3]{lg2010geodesics}, we know that $\eta_1^v,\eta_2^v, \eta_3^v$ do not intersect $p(\gamma)$ except at their starting point. In fact, we further know that, by possibly relabelling, $\eta_1^v$ and $\eta_2^v$ respectively hit $p(\gamma)$ at its left and right sides ($\eta_3^v$ can hit $p(\gamma)$ at either side). This is because that, the points $v\in P_3(\gamma)$ correspond to those points that have exactly three corners in the dual tree $\CT$. Exactly two of the three corners are at the same side of the branch $p(\gamma)$.
The geodesics $\eta_1^v$, $\eta_2^v$ and $\eta_3^v$  are branches in the geodesic tree, and exactly two of them would stay at the same side of $p(\gamma)$. (This can for example be seen from the picture where one obtains the Brownian map by gluing the processes $(X,Y)$ as described in \cite{ms2015axiomatic}. The points $v\in P_3(\gamma)$ correspond to the local minima of $X$.)

\end{itemize}

Fix  $v_0\in P_3(\gamma)$ that we will adjust later. We already know that there exist two geodesics $\eta^{v_0}_1, \eta^{v_0}_2$ from $v_0$ to $x$ that hit $p(\gamma)$ respectively at its left and right sides. Suppose that $\eta^{v_0}_1$ and $\eta^{v_0}_2$ merge at $b$. Then the two sub-geodesics of  $\eta^{v_0}_1, \eta^{v_0}_2$ from $v_0$ to $b$ form a closed contour which separates $\CS$ into two connected components $\CC_1$ and $\CC_2$. Exactly one of the two endpoints of $p(\gamma)$ will be in the same connected component as $x$. Let us choose $v_0$ so that $p(u)$ is in the same connected component as~$x$, and we denote this connected component by $\CC_2$. This choice implies that $y\in\CC_1$. (In Figure~\ref{fig:two_geo_dim_two}, the point $v_2$ does not satisfy the criteria of this choice.) Such a point $v_0$ exists, because otherwise, there would be a sequence of points $(v_n)_{n\ge 4}$ in $P_2(\gamma)$ tending to $y$, and geodesics $\eta_n$ from $v_n$ to $x$, such that a subsequence of $\eta_n$ would converge to a geodesic from $y$ to $x$ which is distinct from $\eta$, leading to a contradiction.
Without loss of generality, we can assume that the third geodesic $\eta^{v_0}_3$ merges with  $\eta^{v_0}_2$ (at a point that we denote by $b_1$) before merging with $\eta^{v_0}_1$. By possibly relabelling $\eta^{v_0}_2$ and $\eta^{v_0}_3$, we can suppose that the part of $\eta^{v_0}_3$ from $v_0$ to $b_1$ is contained in $\CC_1$. 
From now on, let $\wt\Phi_3=\{v_0\}$.

Now let us prove that there a.s.\ exists $v_1\in P_2(\gamma)$ which is on the path $p(\gamma)$ between~$y$ and~$v_0$, with the property that for any point $v\in p(\gamma)$ which is between~$v_1$ and~$v_0$,  every geodesic from $v$ to $x$ which hits $p(\gamma)$ at its left (resp.\ right) merges with $\eta^{v_0}_1$ (resp.\ $\eta^{v_0}_3$) before hitting $\eta^{v_0}_3$ (resp.\ $\eta^{v_0}_1$). 
(In Figure~\ref{fig:two_geo_dim_two}, the point $v_3$ does not satisfy this condition.) 
If it is not the case, then there exists a sequence $(v_n)_{n\ge 4}$ of points on the path $p(\gamma)$ between $y$ and $v_0$, tending to $v_0$, with the property that for each $n\ge 4$, there exists a geodesic $\eta^{v_n}$ from $v_n$ to $x$ which hits $p(\gamma)$ at its left (or right) but merges with $\eta^{v_0}_3$ (or $\eta^{v_0}_1$)  before hitting $\eta^{v_0}_1$ (or $\eta^{v_0}_3$). By passing to a subsequence, without loss of generality, we can assume that for all $n\ge 4$, $\eta^{v_n}$ hits $p(\gamma)$ at its right and merges with $\eta^{v_0}_1$ before hitting $\eta^{v_0}_3$. Let $\wt\eta$ be a subsequential limit of $\eta^{v_n}$, then $\wt\eta$ must be disjoint from the part of $\eta^{v_0}_3$ between $v_0$ and $b$, because otherwise by Theorem~\ref{thm:strong_confluence2}, there would exist $n\ge 4$ such that $\eta^{v_n}$ hits the part of $\eta^{v_0}_3$ between $v_0$ and $b$, hence $\eta^{v_n}$ merges with $\eta^{v_0}_3$ before $\eta^{v_0}_1$, which contradicts our assumption. By topological reasons, $\wt\eta$ must also be disjoint from the part of $\eta^{v_0}_2$ between $v_0$ and $b$. 
In particular, $\wt\eta$ is distinct from $\eta^{v_0}_2$ and $\eta^{v_0}_3$. Since $\wt\eta$ hits $p(\gamma)$ at its right, it is also distinct from $\eta^{v_0}_1$.
This leads to four distinct geodesics from $v_0$ to $x$, leading to a contradiction. Therefore, we have proved the existence of $v_1$. 

Let $\wt\Phi_2$ be the set of points on the path $p(\gamma)$ between $v_0$ and $v_1$ which belong to $P_2(\gamma)$.
We know that $\dimH(\wt\Phi_2)\ge 2$.
Since the distance from $\wt\Phi_2\cup \wt\Phi_3$ to $x$ is a.s.\ finite, there exists $d_1>0$ such that this distance is at most $d_1$ with positive probability.

\emph{Step 3: End of the proof.}
Let $U$ be the connected component of $\CS \setminus (\eta^{v_0}_1 \cup \eta^{v_0}_2)$ with $x$ on its boundary.  Let $w$ be a point sampled from $\nu$ independently of everything else.  Then it is a positive conditional probability event (given everything else) that the a.s.\ unique geodesic $\xi$ from $w$ to $x$ hits $\partial U$ first at a point on the part of $\eta^{v_0}_2$ between $v_0$ and $b_1$.
Note that $(\CS,d,\nu,x,w)$ has the same distribution as $(\CS,d,\nu,x,y)$ by the root invariance of the Brownian map.
Moreover, for each $v\in \wt\Phi_2$, there are exactly two geodesics $\eta^v_1$ and $\eta^v_2$ from $v$ to $x$, and they merge with $\xi$ in a way that satisfies the conditions (i) (ii) (iii) in the item (II) of the statement of the lemma, for $(\eta^v_1, \eta^v_2, \xi)$ instead of $(\eta_1, \eta_2, \eta)$. Note that $\eta^v_1$ and $\eta^v_2$ hit $\xi$ at two points $b_1$ and $b_2:=b$, for any choice of $v\in \wt \Phi_2$.
For the point $v_0\in\wt\Phi_3$, there are exactly three geodesics $\eta^v_1, \eta^v_2$ and $\eta^v_3$ from $v$ to $x$, and they merge with $\xi$ in a way that satisfies the conditions (i) (ii) (iii) in the item (III) of the statement of the lemma, for $(\eta^v_1, \eta^v_2, \eta^v_3, \xi)$ instead of $(\eta_1, \eta_2, \eta_3, \eta)$. 

It remains to cut out a metric band from $\CS$ which contains the union of $A$, $\wt\Phi_2, \wt\Phi_3$ and all the geodesics from points in $\wt\Phi_2 \cup \wt\Phi_3$ to $b$. Note that this union a.s.\ has positive distance from $x$ and $w$. Also note that $$\inf_{v\in\wt\Phi_2}\{ d(v,w) + d(w,x) -d(v,x)\} >0,$$
because otherwise there would exist $v$ in the closure of $\wt\Phi_2$ (which is equal to the part of $p(\gamma)$ from $v_1$ to $v_0$) such that there is a geodesic from $v$ to $x$ which passes through $w$, which is not the case.

Let us do a reverse metric exploration from $x$ to $w$.
 Fix $c_1>0$, and let $T_{1}$ be the first time $t$ such that the boundary length of $\partial\fb{x}{w}{d(x,w)-t}$ reaches $c_1$. Fix $c_2 \in(0,1)$, and let $T_2$  be the first time $t \ge  T_1$ such that the boundary length of $\partial\fb{x}{w}{d(x,w)-t}$ reaches $c_1c_2$. Then we have the following facts.
\begin{itemize}
\item As $c_1\to 0$, we have $T_{1}$ tends to $0$ and $\partial \fb{x}{w}{d(x,w)- T_{1}}$ shrinks to $x$. 
\item For each given $c_1>0$, as $c_2\to 0$,  we have $T_{2}$ tends to $d(x,w)$ and $\partial \fb{x}{w}{d(x,w)- T_{2}}$ shrinks to $w$.
\end{itemize}
This implies that there exist $c_1, c_2>0$ such that with positive (conditional) probability, $ \fb{x}{w}{d(x,w)-T_1} \setminus \fb{x}{w}{d(x,w)- T_{2}}$ contains the union of $A$, $\wt\Phi_2, \wt\Phi_3$ and all the geodesics from points in $\wt\Phi_2 \cup \wt\Phi_3$ to $b$, and 
\begin{align}\label{eq:geo_cond_length}
d(x,w) - T_{2} <\frac12\, \inf_{v\in\wt\Phi_2}\{ d(v,w) + d(w,x) -d(v,x)\}.
\end{align}
We further condition on this event. 
 
 Fix $c_3< c_2$, and let $T_3$ be the first time $t\ge T_2$ such that  the boundary length of $\partial\fb{x}{w}{d(x,w)-t}$ reaches $c_1 c_3$. The band $ \fb{x}{w}{d(x,w)-T_2} \setminus \fb{x}{w}{d(x,w)-T_3}$ is independent from $\CS \setminus \fb{x}{w}{d(x,w)-T_2}$, and it is a positive probability event that $T_3 - T_2 > d_1$. We further condition on this event. We remark that such a conditioning on $\CS$ can \emph{a priori} modify the geodesics from $v\in\wt\Phi_2 \cup \wt\Phi_3$ to $x$ in $\CS$, but~\eqref{eq:geo_cond_length} ensures that  every geodesic $\wt\xi$ from $v\in\wt\Phi_2 \cup \wt\Phi_3$ to $x$ in $\CS \setminus \fb{x}{w}{d(x,w)-T_2}$ is also a geodesic in $\CS \setminus \fb{x}{w}{d(x,w)-T_3}$, because $\wt\xi$ cannot intersect $\partial  \fb{x}{w}{d(x,w)-T_2}$ otherwise it will be too long.
 
 Altogether, we have constructed the event in the statement of the lemma for the band $ \fb{x}{w}{d(x,w)- T_{1}} \setminus \fb{x}{w}{d(x,w)-T_3}$. This event has positive probability. The band $ \fb{x}{w}{d(x,w)- T_{1}} \setminus \fb{x}{w}{d(x,w)-T_3}$, after rescaling, has the same law as $\CB \setminus \CB_{\tau_{c_3}}$. This completes the proof of the lemma. 
\end{proof}

\subsection{Completion of the proof}\label{subsec:geo_u}
In this subsection, we will first prove Lemmas~\ref{lem:five_geo_pos_prob}, \ref{lem:seven_geo_pos_prob} and~\ref{lem:eight_geo_pos_prob}, and then complete the proof of Proposition~\ref{prop:five_geodesics_lbd}.

\begin{figure}[ht!]
\begin{center}
\includegraphics[scale=0.85]{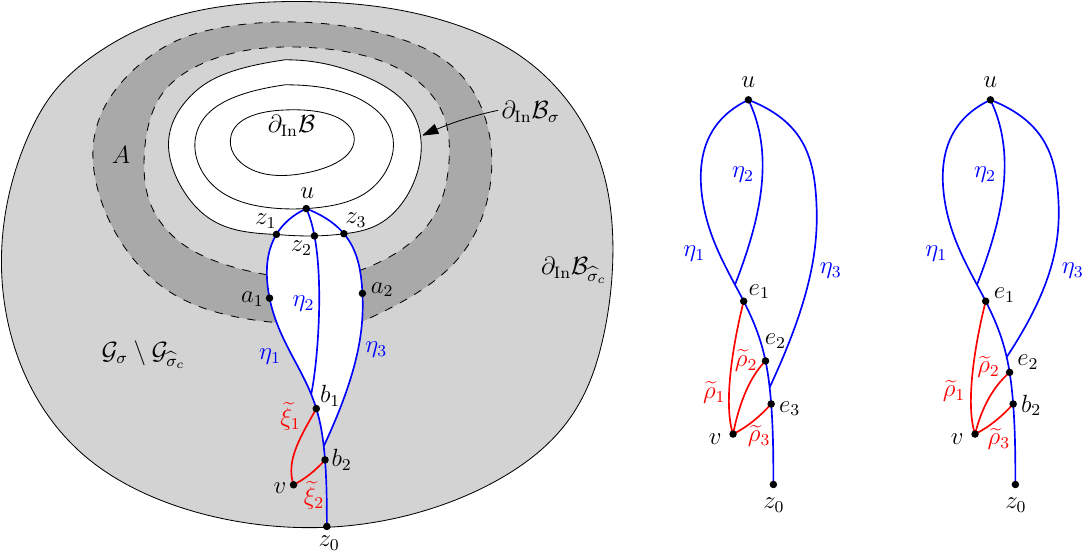}
\end{center}
\caption{\label{fig:5_7_8_geo_band} {\bf Left:} Setup of the proof and configuration to get $5$ geodesics.  {\bf Middle:}  Configuration to get $7$ geodesics.  {\bf Right:} Configuration to get $8$ geodesics.}
\end{figure}

\begin{lemma}
\label{lem:five_geo_pos_prob}
There exist $p, c_4 \in(0,1)$ and $d_2>0$ such that with probability at least $p$ the following event holds for the band $\CB\setminus \CB_{\tau_{c_4}}$.  There exists a set $\wh\Phi_5 \subseteq (\CB\setminus \CB_{\tau_{c_4}})^2$ such that $\dimH(\wh\Phi_5) \geq 2$, and that for each $(u,v) \in \wh\Phi_5$,
\begin{enumerate}[(i)]
\item $u$ and $v$ are connected by exactly $5$ geodesics; 
\item both $u$ and $v$ have distance at least $d_2$ from $\innerboundary\CB \cup \innerboundary\CB_{\tau_{c_4}}$, and  $d_{\CB}(u,v) \le d_2/2$.
\end{enumerate}
\end{lemma}

\begin{lemma}
\label{lem:seven_geo_pos_prob}
There exist $p, c_4 \in(0,1)$ and $d_2>0$ such that with probability at least $p$, the following event holds for the band $\CB\setminus \CB_{\tau_{c_4}}$.  There exists a non-empty set $\wh\Phi_7 \subseteq (\CB\setminus \CB_{\tau_{c_4}})^2$ such that for each $(u,v) \in \wh\Phi_7$,
\begin{enumerate}[(i)]
\item $u$ and $v$ are connected by exactly $7$ geodesics; 
\item both $u$ and $v$ have distance at least $d_2$ from $\innerboundary\CB \cup \innerboundary\CB_{\tau_{c_4}}$, and  $d_{\CB}(u,v) \le d_2/2$.
\end{enumerate}
\end{lemma}

\begin{lemma}
\label{lem:eight_geo_pos_prob}
There exist $p, c_4 \in(0,1)$ and $d_2>0$ such that with probability at least $p$, the following event holds for the band $\CB\setminus \CB_{\tau_{c_4}}$.  There exists a non-empty set $\wh\Phi_8 \subseteq (\CB\setminus \CB_{\tau_{c_4}})^2$ such that for each $(u,v) \in \wh\Phi_8$,
\begin{enumerate}[(i)]
\item $u$ and $v$ are connected by exactly $8$ geodesics; 
\item both $u$ and $v$ have distance at least $d_2$ from $\innerboundary\CB \cup \innerboundary\CB_{\tau_{c_4}}$, and  $d_{\CB}(u,v) \le d_2/2$.
\end{enumerate}
\end{lemma}

We remark that, for Lemmas~\ref{lem:five_geo_pos_prob}, \ref{lem:seven_geo_pos_prob} and~\ref{lem:eight_geo_pos_prob}, if we embed $\CB\setminus \CB_{\tau_{c_4}}$ into an ambient Brownian map, then the geodesics between $(u,v) \in \wh\Phi_i$ for $i=5,7,8$ are also geodesics in the Brownian map. Indeed, for a geodesic from $u$ to $v$ to exit $\CB\setminus \CB_{\tau_{c_4}}$, it must traverse distance at least $2d_2$, which is greater than $d_{\CB}(u,v)$.

\begin{proof}[Proof of Lemmas~\ref{lem:five_geo_pos_prob}, \ref{lem:seven_geo_pos_prob} and~\ref{lem:eight_geo_pos_prob}]
We will prove the three lemmas together. See Figure~\ref{fig:5_7_8_geo_band} for an illustration of the setup of the proof.  We will complete the proof in two steps.

\emph{Step 1: Construction of $\wh\Phi_5, \wh\Phi_7, \wh\Phi_8$ and proving that these sets satisfy the condition (ii) in each of Lemmas~\ref{lem:five_geo_pos_prob}, \ref{lem:seven_geo_pos_prob} and~\ref{lem:eight_geo_pos_prob}.}
This step is the same for Lemmas~\ref{lem:five_geo_pos_prob}, \ref{lem:seven_geo_pos_prob} and~\ref{lem:eight_geo_pos_prob}.

Fix $d_2>0$ that we will adjust later.  
First, we take the metric band $(\CB,d_\CB,\nu_\CB,z)$ with law $\bandlaw{1}{\infty}$, and condition on the positive probability event that $\tau_{3/4} > d_2$.
Conditionally on this event, the metric band $\CB_{\tau_{3/4}}$ is independent of $\CB \setminus \CB_{\tau_{3/4}}$ and has law $\bandlaw{3/4}{\infty}$.

Fix $\eps\in(0, d_2/100)$ that we will adjust later. Let $\sigma_0$ be the first time $t \ge \tau_{3/4}+ \eps$ that there exists $u \in \innerboundary \CB_t$ from which there emanate $3$ geodesics $\eta_1,\eta_2,\eta_3$ from $u$ to $\outerboundary \CB$ with $\eta_i((0,\epsilon)) \cap \eta_j((0,\epsilon)) = \emptyset$ for $i,j \in \{1,2,3\}$ distinct.   According to \cite{ms2015axiomatic}, the metric net of the Brownian map is distributed as a $3/2$-stable L\'evy net. The points on the metric net which have three distinct geodesics going to the root corresponds to the local minima of a $3/2$-stable L\'evy process (see in particular the proof of \cite[Proposition~3.4]{ms2015axiomatic} -- points with three distinct geodesics correspond to equivalence classes of Type~III). If we do a breadth-first exploration of the L\'evy net, then these local minima occur for a dense set of times. For each $\eps>0$, if we restrict to those local minima from which there are $3$ geodesics to the root which are disjoint on an initial time interval of length $\eps$, then we will encounter these local minima at a finite set of times. Let $\sigma = \sigma_0+\epsilon$.  
Note that  $\sigma$ is an $(\CF_t)$-stopping time. If we choose $\eps$ sufficiently small, then there exist $c_2\in(0,3/4)$ such that with positive probability $\sigma \leq \tau_{c_2}$.
From now on, we further condition on the event $\sigma \leq \tau_{c_2}$. The conditional law of $\CB_\sigma$ given $\CB \setminus \CB_\sigma$ is a metric band with boundary length equal to that of $\innerboundary \CB_\sigma$  and with width $\infty$.

Assume that $\eta_1,\eta_2,\eta_3$ are ordered from left to right, and that they hit $\innerboundary\CB_\sigma$ respectively at $z_1, z_2, z_3$.  Note that $\eta_1, \eta_2$ and $\eta_3$ divide $\CB_\sigma$ into three slices. Let $\CG_\sigma$ be the counterclockwise slice between $\eta_3$ and $\eta_1$. Given the boundary length of $\innerboundary \CG_\sigma$, the slice  $\CG_\sigma$ is independent of $\CB \setminus \CB_{\sigma}$ and $\eta_1, \eta_2, \eta_3$.  In addition, the boundary length of $\innerboundary\CG_\sigma$ is a positive random variable. Hence there exists $c_3>0$ such that with positive (conditional) probability, this boundary length is in the interval $(c_3, 9c_3/4)$.

For each $t > \sigma$, we let $\CG_t$ be the slice which corresponds to the set of points disconnected from $\innerboundary \CG_\sigma$ by the $d_0-t$ neighborhood of $\outerboundary \CG_\sigma$ where $d_0$ is the distance from $\innerboundary \CG_\sigma$ to $\outerboundary \CG_\sigma$.
Let $c\in(0,1)$ be the constant from Lemma~\ref{lem:two_geo_dim_two}. Let $\wh\tau_c$ be the first time $t>\sigma$ that the boundary length of $\innerboundary\CG_t$ reaches $c$ times the boundary length of $\innerboundary \CG_\sigma$. 

Now, take the slice $\CG_{\sigma} \setminus \CG_{\wh\tau_c}$, and glue its left and right boundaries together, so that we obtain a new metric band $\wt\CB$. The points $z_1$ and $z_3$ are identified, and form a new marked point $\wt z\in \innerboundary\wt\CB$. The geodesics $\eta_1$ and $\eta_3$ are also identified, and form the marked geodesic $\wt\eta$ from $\wt z$ to $\outerboundary\wt\CB$.  If we rescale the metric of $\wt\CB$ by the boundary length of $\innerboundary\CG_\sigma$  to the power $-1/2$ (and areas and distances accordingly), then its law is the same as $\CB \setminus \CB_{\tau_c}$.
Let $d_1$ be the constant from Lemma~\ref{lem:two_geo_dim_two} and let $\wt d_1$ be equal to $d_1$ times the boundary length of $\innerboundary\CG$ to the power $-1/2$.
Therefore, by Lemma~\ref{lem:two_geo_dim_two}, with positive probability, there exists an annulus $A \subseteq \wt\CB$, $a \in A$ and $\wt\Phi_2, \wt\Phi_3\subseteq \CB$ which satisfy the properties in the statement of Lemma~\ref{lem:two_geo_dim_two}, for $\wt d_1$ in place of $d_1$. In particular, the following facts hold.
\begin{itemize}
\item We must have $a\in \wt\eta$, since $\wt\eta$ is also a geodesic which traverses $A$. Let $a_1 \in\eta_1$ and $a_2 \in \eta_3$ be the points in $\CB$ which correspond to $a$.
\item For each $v\in \wt\Phi_2$, there are exactly two geodesics $\wt \xi_1$ and $\wt \xi_2$ (w.r.t.\ $d_{\wt\CB}$) from $v$ to $\wt z$.  For $i=1,2$, let $b_i$ be the point where $\wt\xi_i$ first intersects $\wt\eta$.
Since $\wt \xi_1$ and $\wt \xi_2$ hit $\wt\eta$ at its right side. Then $b_1$ and $b_2$ correspond to points on $\eta_1$ in $\CB$.

\noindent For each $v\in \wt\Phi_3$, there are exactly three geodesics $\wt \rho_1, \wt\rho_2$ and $\wt \rho_3$ (w.r.t.\ $d_{\wt\CB}$) from $v$ to $\wt z$.  For $i=1,2,3$, let $e_i$ be the point where $\wt\rho_i$ first intersects $\wt\eta$.
Since $\wt \rho_1, \wt\rho_2$ and $\wt \rho_3$ hit $\wt\eta$ at its right side. Then $e_1, e_2$ and $e_3$ correspond to points on $\eta_1$ in $\CB$. 

\item For each $v \in\wt\Phi_2 \cup \wt\Phi_3$, the distance from $v$ to $\wt z$ is at most $\wt d_1$. The distance from $\wt\Phi_2 \cup \wt\Phi_3$ to $\outerboundary\wt\CB$ is at least $\wt d_1$.
Recall that we have conditioned on the event that the boundary length of $\innerboundary\CG_{\sigma}$ is in the interval $(c_3, 9c_3/4)$, hence $2d_1/ 3c_3 < \wt d_1 <d_1/c_3$.
\end{itemize}
By an abuse of notation, we use the same notation for many objects in $\wt\CB$ and in $\CG_{\sigma} \setminus \CG_{\wh\tau_c}$. Note that $ \wt\xi_1$, $\wt\xi_2$ and $\wt\rho_1, \wt\rho_2, \wt\rho_3$ are also geodesics in $\CG_{\sigma} \setminus \CG_{\wh\tau_c}$ (w.r.t.\ $d_{\CG_{\sigma} \setminus \CG_{\wh\tau_c}}$). The above event for $\wt\CB$ naturally corresponds to an event for $\CG_{\sigma} \setminus \CG_{\wh\tau_c}$.
From now on, we further condition on this event.

Let $\wh\Phi_5$ be the set $\{u\} \times \wt\Phi_2$ and let $\wh\Phi_7, \wh\Phi_8$ be both equal to $\{u\}\times \wt\Phi_3$. Since $\dimH(\wt\Phi_2)\ge 2$, we also have $\dimH(\wh\Phi_5)\ge 2$. Since $\wt\Phi_3$ is non-empty, $\wh\Phi_7, \wh\Phi_8$ are also non-empty.

Let $d_2=4d_1/(5c_3)$. It follows that each $\wh\Phi_5$, $\wh\Phi_7$ and $\wh\Phi_8$ respectively satisfy the condition (ii) in each of Lemmas~\ref{lem:five_geo_pos_prob}, \ref{lem:seven_geo_pos_prob} and~\ref{lem:eight_geo_pos_prob}, for $\wh\tau_c$ in place of $\tau_{c_4}$ (we will later fix some $c_4\in(0,1)$ with $\tau_{c_4}> \wh\tau_c$). It remains to define $c_4$, and then further condition on an event with positive probability (such an event will be different for each of Lemmas~\ref{lem:five_geo_pos_prob}, \ref{lem:seven_geo_pos_prob} and~\ref{lem:eight_geo_pos_prob}) for which the condition (i) in the statement of each lemma holds for the band $\CB\setminus \CB_{\tau_4}$.

\emph{Step 2: Getting the $5,7,8$ geodesics for pairs of points in $\wh\Phi_5$, $\wh\Phi_7$ and $\wh\Phi_8$.}
Let us now consider the geodesic slice $\CB_{\sigma} \setminus \CG_{\sigma}$, which is independent from $\CG_{\sigma}$. The following events $E_1, E_2$ and $E_3$ depend only on the slice $\CB_{\sigma} \setminus \CG_{\sigma}$. Let $E_1$ be the event
\begin{itemize}
\item  $\eta_2$ merges with $\eta_1$ at some point between $z_1$ and $b_1$;
\item $\eta_3$ merges with $\eta_1$ at some point between $b_1$ and $b_2$.
\end{itemize}
 Let $E_2$ be the event
\begin{itemize}
\item  $\eta_2$ merges with $\eta_1$ at some point between $z_1$ and $e_1$;
\item $\eta_3$ merges with $\eta_1$ at some point between $e_2$ and $e_3$.
\end{itemize}
 Let $E_3$ be the event
\begin{itemize}
\item  $\eta_2$ merges with $\eta_1$ at some point between $z_1$ and $e_1$;
\item $\eta_3$ merges with $\eta_1$ at some point between $e_1$ and $e_2$.
\end{itemize}
The events $E_1, E_2$ and $E_3$ each occur with positive probability, because the times at which $\eta_2,\eta_3$ merge with $\eta_1$ are determined by the times at which the corresponding $3/2$-stable CSBPs die, which have a positive density with respect to the Lebesgue measure. 

\begin{enumerate}[(A)]
\item \emph{The following applies for the proof of Lemma~\ref{lem:five_geo_pos_prob}.}
From now on, we further condition on $E_1$.
For $i=1,2$, let $\xi_i$ be the concatenation of the part of $\eta_1$ from $z_1$ to $u$ with (the time reversal of) $\wt\xi_i$.  We aim to show that, for $i=1,2$, $\xi_i$ is also a geodesic from $u$ to $v$ in $\CB$ w.r.t.\ $d_{\CB}$. If this is true, then the union of $\xi_1, \xi_2$ together with $\eta_1, \eta_2, \eta_3$ contains exactly $5$ geodesics (w.r.t.\ $d_{\CB}$) from $u$ to $v$. We will also show that there are no other geodesics in $\CB$ from $u$ to $v$. 

\item \emph{The following applies for the proof of Lemma~\ref{lem:seven_geo_pos_prob}.}
From now on, we further condition on $E_2$.
For $i=1,2$, let $\rho_i$ be the concatenation of the part of $\eta_1$ from $z_1$ to $u$ with (the time reversal of) $\wt\rho_i$.  We aim to show that, for $i=1,2,3$, $\rho_i$ is also a geodesic from $u$ to $v$ in $\CB$ w.r.t.\ $d_{\CB}$. If this is true, then the union of $\rho_1, \rho_2, \rho_3$ together with $\eta_1, \eta_2, \eta_3$ contains exactly $7$ geodesics (w.r.t.\ $d_{\CB}$) from $u$ to $v$. We will also show that there are no other geodesics in $\CB$ from $u$ to $v$. 

\item \emph{The following applies for the proof of Lemma~\ref{lem:eight_geo_pos_prob}.}
From now on, we further condition on $E_3$.
For $i=1,2$, let $\xi_i$ be the concatenation of the part of $\eta_1$ from $z_1$ to $u$ with (the time reversal of) $\wt\xi_i$.  We aim to show that, for $i=1,2,3$, $\rho_i$ is also a geodesic from $u$ to $v$ in $\CB$ w.r.t.\ $d_{\CB}$. If this is true, then the union of $\rho_1, \rho_2, \rho_3$ together with $\eta_1, \eta_2, \eta_3$ contains exactly $8$ geodesics (w.r.t.\ $d_{\CB}$) from $u$ to $v$. We will also show that there are no other geodesics in $\CB$ from $u$ to $v$. 
\end{enumerate}

The proofs of the claims in (A), (B), (C) are almost the same, and we will carry them out in the following two steps.
To prove the claims of (A), suppose that $v\in\wt\Phi_2$. To prove the claims of (B), (C), suppose that $v\in\wt\Phi_3$.
\begin{enumerate}[1.]
\item If $\xi$ is a geodesic from $u$ to $v$, then there exists $r>0$, such that $\xi([0,r])$ is contained in union of $\eta_1, \eta_2$ and $\eta_3$.
Suppose that it is not the case, and that there is a  geodesic $\xi$ which starts being disjoint from $\eta_1, \eta_2, \eta_3$ (except at $u$). Then $\xi$ cannot hit $\eta_1, \eta_2$ or $\eta_3$ again, because otherwise there will be four geodesics from $u$ to $\outerboundary\CB$, which is impossible. So $\xi$ must be entirely disjoint from $\eta_1, \eta_2, \eta_3$ (except at $u$). In this case, for topological reasons, $\xi$ must traverse the region $A$. As a consequence, $\xi$ must intersect either $a_1$ or $a_2$. This leads to a contradiction.

\item Suppose that $\xi$ is a geodesic from $u$ to $v$, and let $z_4=\xi(T)$ where $T$ is  the last time $t$ that $\xi(t) \in \eta_1\cup\eta_2 \cup \eta_3$. Then for topological reasons,  the part of $\xi$ from $z_4$ to $v$ must be contained in $\CG_\sigma$, since it is not allowed to intersect $\eta_1, \eta_2, \eta_3$ again. In fact, we further have that the part of $\xi$ from $z_4$ to $v$ must be contained in $\CG_\sigma\setminus\CG_{\wh\tau_c}$. If it is not the case, then  the part of $\xi$ from $z_4$ to $v$
\begin{itemize}
\item either intersects $\innerboundary \CG_\sigma$, hence must traverse $A$ again before reaching $v$, and must intersect $a_1$ or $a_2$, leading to a contradiction.
\item either intersects $\innerboundary \CG_{\wh\tau_c}$. However, we know that the distance from $v$ to $\innerboundary \CG_{\wh\tau_c}$ is at least $d_2$. If $\xi$ intersects  $\innerboundary \CG_{\wh\tau_c}$, then it will have length more than $d_\CB(u,v)$, also leading to a contradiction.
\end{itemize}
For the case of Lemma~\ref{lem:five_geo_pos_prob}, this implies that  the part of $\xi$ from $z_4$ to $v$ must coincide with $\wt\xi_1$ or $\wt\xi_2$, because otherwise, there would be at least $3$ geodesics from $v$ to $\wt z$ in $\wt\CB$, which contradicts the definition of $\wt\Phi_2$.

\noindent For the cases of Lemmas~\ref{lem:seven_geo_pos_prob} and~\ref{lem:eight_geo_pos_prob}, this implies that  the part of $\xi$ from $z_4$ to $v$ must coincide with $\wt\rho_1, \wt\rho_2$ or $\wt\rho_3$, because otherwise, there would be at least $4$ geodesics from $v$ to $\wt z$ in $\wt\CB$, which contradicts the definition of $\wt\Phi_3$.
\end{enumerate}
Altogether, we have proved the claims in (A), (B), (C).
To complete the proof, note that there exists $c_4\in(0,1)$ such that with positive probability, $\tau_{c_4}> \wh\tau_c$. The event in each of Lemmas~\ref{lem:five_geo_pos_prob}, \ref{lem:seven_geo_pos_prob} and~\ref{lem:eight_geo_pos_prob} then holds with positive probability for the band $\CB\setminus\CB_{\tau_{c_4}}$. 
\end{proof}

\begin{figure}[ht!]
\begin{center}
\includegraphics[scale=0.85]{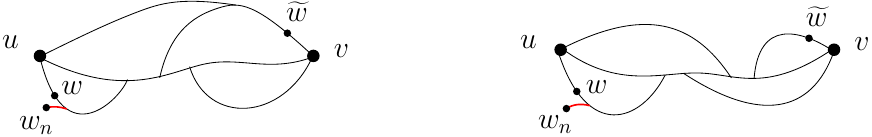}
\end{center}
\caption{\label{fig:geo_7_8} The configurations of $7$ and $8$ geodesics between $u$ and $v$. The endpoint $v$ belongs to the set of points which are connected to $w_n$ (which is a $\nu$-typical point) by exactly $3$ geodesics, which is countable. Similarly, the endpoint $u$ also belongs to a countable set.}
\end{figure}

\begin{proof}[Proof of Proposition~\ref{prop:five_geodesics_lbd}]
Suppose that $(\CS,d,\nu,x,y)$ has distribution $\bminflaw$. For each $r > 0$ we let $Y_r$ be the boundary length of $\partial \fb{y}{x}{d(x,y)-r}$.
Let $T_1 = \inf\{r\ge 0 : Y_r =1\}$. Without loss of generality, we can condition on the event that $T_1<\infty$. The statement for general $\CS$ follows by scaling.

Let $c,p$ be the constants from Lemma~\ref{lem:five_geo_pos_prob}.  For $n\in\N$, let $T_n= \inf\{r\ge 0; Y_r =c^n\}$.  Lemma~\ref{lem:five_geo_pos_prob} implies that for each $n\in\N$, with probability at least $p$,  there is a set of pairs of points in  $\fb{y}{x}{d(x,y)-T_n} \setminus \fb{y}{x}{d(x,y)-T_{n+1}}$ with dimension at least $2$ which are connected by exactly $5$ geodesics.
By the independence across these bands, we have $\dimH(\Phi_5)\ge 2$ with probability $1$.
Similarly, Lemmas~\ref{lem:seven_geo_pos_prob} and~\ref{lem:eight_geo_pos_prob} will imply that $\Phi_7, \Phi_8$ are non-empty with probability $1$.

Now, fix $j\in\{5,7,8\}$. Let us show that the set of $u \in\CS$ such that there exists $v \in \CS$ with $(u,v) \in \Phi_j$ is dense in $\CS$. Let $(z_i)_{i\in\N}$ be a sequence of points chosen i.i.d.\ according to $\nu$. Then for each $z_i$ and $\eps>0$, by considering successive metric bands centered at $z_i$ (coming from the reverse metric exploration from $y$ to $z_i$) and by arguing like in the previous paragraphs, we can deduce that there are a.s.\ two points $u,v \in \fb{y}{z_i}{\eps}$ which are connected by exactly $j$ geodesics. This proves the claim.

The previous paragraph implies in particular that the sets $\Phi_7$ and $\Phi_8$ are infinite. Let us now show that they are countable. Theorem~\ref{thm:finite_number_of_geodesics} implies that the only configurations which give rise to $7$ and $8$ geodesics are the ones in the following Figure~\ref{fig:geo_7_8}. Fix $j \in\{7,8 \}$. Suppose that $u$ and $v$ are connected by exactly $j$ geodesics. Then there exists a point $w$ which is  on exactly $3$ geodesics $\eta_1, \eta_2, \eta_3$ from $u$ to $v$, in a way that $\eta_1, \eta_2, \eta_3$ coincide with each other in a neighborhood of $u$ and splits into $3$ disjoint geodesics before meeting again at $v$ (see Figure~\ref{fig:geo_7_8}).  Let $(w_n)_{n\in\N}$ be a sequence of $\nu$-typical points tending to $w$. Let $\xi_n$ be a geodesic from $w_n$ to $v$.  Any subsequential limit of $(\xi_n)$ must coincide with one of the geodesics $\eta_1, \eta_2, \eta_3$, because otherwise there would be more than $j$ geodesics between $u$ and $v$. The strong confluence of geodesics (Theorem~\ref{thm:strong_confluence2}) then implies that for all $n$ sufficiently large, $\xi_n$ must coalesce with $\eta_1, \eta_2, \eta_3$ at a point very close to $w$. This gives rise to exactly $3$ geodesics between $w_n$ and $v$. Note that $w_n$ is a $\nu$-typical point, and it was shown in~\cite{lg2010geodesics} that the set of points which are connected to a $\nu$-typical point by $3$ geodesics is countable. Consequently, $v$ must be a subset of a countable set.  Similarly, one can find a point $\wt w$  which is  on exactly $3$ geodesics $\wt \eta_1, \wt \eta_2, \wt \eta_3$ from $v$ to $u$, in a way that $\wt\eta_1, \wt\eta_2, \wt\eta_3$ coincide with each other in a neighborhood of $v$ and splits into $3$ disjoint geodesics before meeting again at $u$ (see Figure~\ref{fig:geo_7_8}). The same arguments imply that $u$ must also be a subset of a countable set. Therefore, $(u,v)$ is also a subset of a countable set, hence $\Phi_7, \Phi_8$ are both countable. This completes the proof.
\end{proof}

\bibliographystyle{abbrv}
\bibliography{geodesics}

\end{document}